\documentclass[11pt,reqno]{amsart}
\usepackage{amsmath,amsthm,amssymb,mathrsfs,stmaryrd,color}
\usepackage[all]{xy}
\usepackage{url}
\usepackage{geometry}
\usepackage[utf8]{inputenc}
\usepackage[T1]{fontenc}

\usepackage{relsize}
\usepackage[bbgreekl]{mathbbol}
\usepackage{amsfonts}

\DeclareSymbolFontAlphabet{\mathbb}{AMSb}
\DeclareSymbolFontAlphabet{\mathbbl}{bbold}
\newcommand{\Prism}{{\mathlarger{\mathbbl{\Delta}}}}
\usepackage{enumitem}
\usepackage[colorlinks=true,hyperindex, linkcolor=magenta, pagebackref=false, citecolor=cyan,pdfpagelabels]{hyperref}
\usepackage[capitalize]{cleveref}
\setlist[enumerate]{itemsep=2pt,parsep=2pt,before={\parskip=2pt}}

\newcommand{\cosimp}[3]{\xymatrix@1{#1 \ar@<.4ex>[r] \ar@<-.4ex>[r] & {\ }#2 \ar@<0.8ex>[r] \ar[r] \ar@<-.8ex>[r] & {\ } #3 \ar@<1.2ex>[r] \ar@<.4ex>[r] \ar@<-.4ex>[r] \ar@<-1.2ex>[r] & \cdots }}

\newcommand{\colim}{\mathop{\mathrm{colim}}}

\newcommand{\adjunction}[4]{\xymatrix@1{#1{\ } \ar@<0.3ex>[r]^{ {\scriptstyle #2}} & {\ } #3 \ar@<0.3ex>[l]^{ {\scriptstyle #4}}}}

\begin{document}

\newtheorem{theorem}{Theorem}[section]
\newtheorem*{theorem*}{Theorem}
\newtheorem*{definition*}{Definition}
\newtheorem{proposition}[theorem]{Proposition}
\newtheorem{lemma}[theorem]{Lemma}
\newtheorem{corollary}[theorem]{Corollary}

\theoremstyle{definition}
\newtheorem{definition}[theorem]{Definition}
\newtheorem{question}[theorem]{Question}
\newtheorem{remark}[theorem]{Remark}
\newtheorem{warning}[theorem]{Warning}
\newtheorem{example}[theorem]{Example}
\newtheorem{notation}[theorem]{Notation}
\newtheorem{convention}[theorem]{Convention}
\newtheorem{construction}[theorem]{Construction}
\newtheorem{claim}[theorem]{Claim}
\newtheorem{assumption}[theorem]{Assumption}

\crefname{assumption}{assumption}{assumptions}
\crefname{construction}{construction}{constructions}

% notation for q-crystalline site
\newcommand{\qc}{q-\mathrm{crys}}

\newcommand{\Shv}{\mathrm{Shv}}
\newcommand{\et}{{\acute{e}t}}
\newcommand{\crys}{\mathrm{crys}}
\renewcommand{\inf}{\mathrm{inf}}
\newcommand{\Hom}{\mathrm{Hom}}
\newcommand{\Sch}{\mathrm{Sch}}
\newcommand{\Spf}{\mathrm{Spf}}
\newcommand{\Spa}{\mathrm{Spa}}
\newcommand{\Spec}{\mathrm{Spec}}
\newcommand{\perf}{\mathrm{perf}}
\newcommand{\qsyn}{\mathrm{qsyn}}
\newcommand{\perfd}{\mathrm{perfd}}
\newcommand{\arc}{{\rm arc}}
\newcommand{\conj}{\mathrm{conj}}
\newcommand{\rad}{\mathrm{rad}}

\setcounter{tocdepth}{1}

\title{Prisms and prismatic cohomology}
\author{Bhargav Bhatt}
\author{Peter Scholze}
\dedicatory{In the memory of Jean-Marc Fontaine}

\begin{abstract}
We introduce the notion of a prism, which may be regarded as a ``deperfection'' of the notion of a perfectoid ring. Using prisms, we attach a ringed site --- the prismatic site --- to a $p$-adic formal scheme. The resulting cohomology theory specializes to (and often refines) most known integral $p$-adic cohomology theories. 

As applications, we prove an improved version of the almost purity theorem allowing ramification along arbitrary closed subsets (without using adic spaces), give a co-ordinate free description of $q$-de Rham cohomology as conjectured by the second author in \cite{ScholzeqdeRham}, and settle a vanishing conjecture for the $p$-adic Tate twists $\mathbf{Z}_p(n)$ introduced in \cite{BMS2}.
\end{abstract}

\maketitle

\tableofcontents

\newpage 

\section{Introduction}

Fix a prime $p$. In this article, we give a new and unified construction of various $p$-adic cohomology theories, including \'etale, de~Rham and crystalline cohomology, as well as the more recent constructions in \cite{BMS1}, \cite{BMS2} and the so far conjectural $q$-de~Rham cohomology from \cite{ScholzeqdeRham}. The key innovation is the following definition, which is a ``deperfection" of perfectoid rings.

\begin{definition} 
\label{def:Prism}
A \emph{prism} is a pair $(A,I)$ where $A$ is a $\delta$-ring (see Remark~\ref{DeltaIntro}) and $I\subset A$ is an ideal defining a Cartier divisor in $\Spec(A)$, satisfying the following two conditions.
\begin{enumerate}
\item The ring $A$ is $(p,I)$-adically complete.\footnote{We actually allow $A$ to be merely derived $(p,I)$-adically complete below, although we will usually restrict to situations where the subtle difference between these notions does not come up.}
\item The ideal $I+\phi_A(I)A$ contains $p$; here $\phi_A$ refers to the lift of Frobenius on $A$ induced by its $\delta$-structure (Remark~\ref{DeltaIntro}).
\end{enumerate}
A map $(A,I) \to (B,J)$ of prisms is given by a map of $A \to B$ of $\delta$-rings carrying $I$ into $J$. 
\end{definition}

Before giving examples, we briefly comment on the notion of $\delta$-rings used above. 

\begin{remark}
\label{DeltaIntro}
The notion of $\delta$-rings was introduced by Joyal \cite{JoyalDelta} and studied extensively by Buium \cite{BuiumArithmeticAnalog} under the name of ``$p$-derivations''; see also \cite{BorgerCourse}. This notion provides a convenient language for discussing rings with a lift of the Frobenius, and is reviewed in depth in \S \ref{sec:DeltaRings}. Roughly, a $\delta$-ring structure on a commutative ring $A$ is a map $\delta_A:A \to A$ satisfying certain identities ensuring that the associated map $\phi_A:A \to A$ given by $\phi_A(x) = x^p + p\delta_A(x)$ is a ring homomorphism (which then necessarily lifts the Frobenius on $A/p$). In particular, if  $A$ is a $p$-torsionfree commutative ring, then a $\delta$-ring structure on $A$ is equivalent to the datum of a ring map $\phi_A: A\to A$ lifting the Frobenius on $A/p$. The name ``$p$-derivation'' is explained by observing that $\delta$ lowers the $p$-adic order of vanishing by $1$ in a suitable sense (see Remark~\ref{InitialDeltaRing} for an example).
\end{remark}

Let us give some examples of prisms. In the following examples (and, in fact, all known examples), the ideal $I=(d)$ is actually principal. Under \cref{def:Prism} (1), condition (2) is then equivalent to the condition that $\delta(d)\in A$ is a unit; we call such an element $d\in A$ a {\em distinguished element} of $A$. 

\begin{example} 
\label{PrismExIntro}

The choice of names below is largely dictated by the comparison theorems for the associated prismatic cohomology theory.
\begin{enumerate}
\item (Crystalline Prisms) Let $A$ be any $p$-torsionfree and $p$-complete ring with a Frobenius lift $\phi: A\to A$; this induces a unique $\delta$-ring structure on $A$. Then $(A,(p))$ is a prism.

\item (Perfect Prisms) A prism $(A,I)$ is called perfect if $A$ is perfect, i.e.~$\phi: A\to A$ is an isomorphism. The category of perfect prisms is equivalent to the category of perfectoid rings (as defined in say \cite[Definition 4.18]{BMS2}) via the functors $(A,I)\mapsto A/I$ and $R\mapsto (A_\inf(R),\ker(\theta))$ (Theorem~\ref{PerfdPrism}). Any prism admits a ``perfection'', so a general prism can be regarded as a ``not necessarily perfect'' analog of a perfectoid ring.

\item (Breuil-Kisin-type Prisms) Let $K/\mathbf Q_p$ be a complete discretely valued extension with perfect residue field $k$ and uniformizer $\pi\in K$. Let $A=\mathfrak S=W(k)[[u]]$, regarded as a $\delta$-ring via the Frobenius lift that extends the Frobenius on $W(k)$ and sends $u$ to $u^p$. Let $I\subset A$ be the kernel of the surjective map $A=W(k)[[u]]\to \mathcal O_K$ sending $u$ to $\pi$. Then $(A,I)$ is a prism.

\item ($q$-crystalline Prism) Let $A=\mathbf Z_p[[q-1]]$, the $(p,[p]_q)$-adic completion of $\mathbf Z[q]$, regarded as a $\delta$-ring via the Frobenius lift that sends $q$ to $q^p$. Let $I=([p]_q)$, where $[p]_q = \frac{q^p-1}{q-1} = 1+q+...+q^{p-1}$ is the $q$-analog of $p$. Then $(A,I)$ is a prism.
\end{enumerate}
In particular, combining the first example with any of the other examples shows that a given $\delta$-ring $A$ may support many different prism structures. At the opposite extreme, one can show that the local ring $W(\mathbf{F}_p[x]/(x^2))$ is naturally a $p$-adically complete $\delta$-ring (see Remark~\ref{ThetaWitt} for the $\delta$-structure on the ring of Witt vectors) that does not underlie any prism $(A,I)$: one can show that any nonzerodivisor $d \in W(\mathbf{F}_p[x]/(x^2))$ must be a unit\footnote{The maximal ideal  of $W(\mathbf{F}_p[x]/(x^2))$ is generated by $[x]$ and $VW(\mathbf{F}_p[x]/(x^2))$ and is annihilated by multiplication by the nonzero element $[x]$: we have $[x] \cdot [x] = [x^2] = 0$ and $[x] V(-) = V(F([x]) \cdot -) = 0$ as $F([x]) = [x^p] = 0$.}.
\end{example}

Our constructions with prisms will often require us to contemplate very large algebras and modules.  For technical reasons pertaining to the behaviour of completions in highly nonnoetherian situations (and, relatedly, to avoid introducing derived analogs of prisms), we will restrict attention to the following class of prisms, which includes all the ones in Example~\ref{PrismExIntro}.

\begin{definition} 
A prism $(A,I)$ is \emph{bounded} if $A/I$ has bounded $p^\infty$-torsion, i.e.~there is some integer $n$ such that $A/I[p^\infty] = A/I[p^n]$.
\end{definition}

Given a bounded prism $(A,I)$ and a smooth $p$-adic formal scheme $X$ over $A/I$, our next goal is to describe a ringed site $((X/A)_\Prism, \mathcal{O}_\Prism)$ that we call the prismatic site of $X$.  Roughly, this category is defined by ``probing'' $X$ using prisms over $(A,I)$. Before giving the formal definition, it is very useful to note the following rigidity property of prisms over $(A,I)$.

\begin{proposition}[Proposition~\ref{PrismMapTaut}]
\label{RigidityIntro} 
If $(A,I)\to (B,J)$ is a map of prisms, then $J=IB$.
\end{proposition}

In other words, when working over a fixed base prism, the ideal $I$ is not varying anymore. We can now give the promised definition. 

\begin{definition}
\label{PrismaticSiteIntro} Fix a bounded prism $(A,I)$ and a smooth $p$-adic formal scheme $X$ over $A/I$.
\begin{enumerate}
\item A map $(A,I)\to (B,J)$ of prisms is called {\em (faithfully) flat} if $A\to B$ is $(p,I)$-completely (faithfully) flat (see \S \ref{NotationGlobal} for a definition of the latter).

\item The {\em prismatic site} $(X/A)_\Prism$ is the opposite of the category of prisms $(B,J)$ with a map $(A,I)\to (B,J)$ and a map $\Spf(B/J)\to X$ over $\Spf(A/I)$, endowed with the faithfully flat covers as defined in (1).

\item The {\em structure sheaf} $\mathcal O_\Prism$ on $(X/A)_\Prism$ is the sheaf taking a pair $(B,J)$ (with maps $(A,I)\to (B,J)$ and $\Spf(B/J)\to X$) to $B$. This is a sheaf of $I$-torsionfree $\delta$-$A$-algebras. 

\item There is also another sheaf $\overline{\mathcal{O}}_\Prism$ of rings on $(X/A)_\Prism$ taking a pair $(B,J)$ to $B/J$. This is naturally a sheaf of $\mathcal{O}(X)$-algebras, and we have $\mathcal{O}_\Prism \otimes_A A/I \cong \overline{\mathcal{O}}_\Prism$ by \cref{RigidityIntro}.
\end{enumerate}
\end{definition}

\begin{remark}
There are many variants of Definition~\ref{PrismaticSiteIntro}. For example, one might use the \'etale or quasisyntomic topologies instead of the flat topology. More interestingly, one can define an ``absolute'' prismatic site $(X)_\Prism$ for any $p$-adic formal scheme $X$ by simply discarding the base $(A,I)$ in \cref{PrismaticSiteIntro}, i.e., as the category of prisms $(B,J)$ together with a map $\mathrm{Spf}(B/J) \to X$. While important in arithmetic considerations, this notion does not play a significant role in this paper, and is only used in our study of algebras over a perfectoid ring in an auxiliary way to make certain functorialities transparent. Another variant is the perfect prismatic site $(X/A)_\Prism^{\perf}$, which is obtained from $(X/A)_\Prism$ by restricting attention to perfect prisms $(B,J)$; this variant plays an important role in our results on perfectoid rings, such as Theorem~\ref{AlmostPurityIntro}. 
\end{remark}

A major goal of this paper is to explain why prismatic cohomology recovers (and often refines) most known integral $p$-adic cohomology theories:

\begin{theorem}\label{thm:A} Let $(A,I)$ be a bounded prism, and let $X$ be a smooth $p$-adic formal scheme over $A/I$. Let
\[
R\Gamma_\Prism(X/A) := R\Gamma((X/A)_\Prism,\mathcal O_\Prism)\ ,
\]
which is a commutative algebra in the derived category $D(A)$ of $A$-modules, and comes equipped with a $\phi_A$-linear endomorphism $\phi$.
\begin{enumerate}
\item {\em Crystalline Comparison (Theorem~\ref{CrysComp}):} If $I=(p)$, then there is a canonical $\phi$-equivariant isomorphism
\[
R\Gamma_\crys(X/A)\cong R\Gamma_\Prism(X/A) \widehat{\otimes}_{A,\phi_A}^L A\ .
\]
of commutative algebras in $D(A)$.

\item {\em Hodge-Tate Comparison (Theorem~\ref{HTCompPrismatic}):} If $X$ is affine, say $X=\Spf R$,  there is a canonical $R$-module isomorphism
\[
\Omega^i_{R/(A/I)}\{-i\}\cong H^i(R\Gamma_\Prism(X/A)\otimes^L_A A/I)\ .
\]
Here we write $M\{i\} = M\otimes_{A/I} (I/I^2)^{\otimes i}$ for any $A/I$-module $M$.

\item {\em de~Rham Comparison (Theorem~\ref{dRComp1}, Corollary~\ref{generaldeRham}):} There is a canonical isomorphism
\[
R\Gamma_{\mathrm{dR}}(X/(A/I))\cong R\Gamma_\Prism(X/A)\widehat{\otimes}^L_{A,\phi_A} A/I\ .
\]
of commutative algberas in $D(A)$. Moreover, it can be upgraded naturally to an isomorphism of commutative differential graded algebras.

\item {\em \'Etale Comparison (Theorem~\ref{EtaleCompThm}):} Assume $A$ is perfect. Let $X_\eta$ be the generic fibre of $X$ over $\mathbf Q_p$, as a (pre-)adic space. For any $n \geq 0$, there is a canonical isomorphism
\[
R\Gamma_\et(X_\eta,\mathbf Z/p^n\mathbf Z)\cong \left(R\Gamma_\Prism(X/A)/p^n[\tfrac 1I]\right)^{\phi=1}\ .
\]
of commutative algebras in $D(\mathbf{Z}/p^n)$.

\item {\em Base Change (Corollary~\ref{BaseChangePrismCoh}):} Let $(A,I)\to (B,J)$ be a map of bounded prisms, and let $Y=X\times_{\Spf(A/I)} \Spf(B/J)$. Then the natural map induces an isomorphism
\[
R\Gamma_\Prism(X/A)\widehat{\otimes}^L_A B\cong R\Gamma_\Prism(Y/B)\ ,
\]
where the completion on the left is the derived $(p,J)$-adic completion.

\item {\em Image of $\phi$ (Corollary~\ref{ImageofPhi}):} The linearization
\[
\phi_A^\ast R\Gamma_\Prism(X/A)\to R\Gamma_\Prism(X/A)
\]
of $\phi$ becomes an isomorphism after inverting $I$. More precisely, if $I=(d)$ is principal, there is a map $V_i: H^i_\Prism(X/A)\to H^i(\phi_A^\ast R\Gamma_\Prism(X/A))$ such that $V_i\phi = \phi V_i = d^i$.
\end{enumerate}

In particular, it follows from the Hodge-Tate comparison that $R\Gamma_\Prism(X/A)$ is a perfect complex of $A$-modules if $X$ is proper.
\end{theorem}

Let us discuss the consequences of \cref{thm:A} over the prisms discussed in \cref{PrismExIntro}.

\begin{example} 
\label{ExPrismaticClassical}

\begin{enumerate}
\item (Crystalline) Assume that $I=(p)$. In this case, prismatic cohomology gives a canonical Frobenius descent of crystalline cohomology by \cref{thm:A} (1). Note that this is no extra information when $A$ is perfect, while it is interesting extra information in the common case that $A$ is a $p$-completely smooth lift\footnote{In this special case where $A/p$ is a smooth $k$-algebra, the extra Frobenius descent provided by Theorem~\ref{thm:A} (1), at the level of cohomology groups, can also be deduced from Ogus-Vologodsky's description \cite[Theorem 2.8 (3)]{OgusVologodsky} of their local Cartier transform.} of a smooth $k$-algebra equipped with a Frobenius lift, for some perfect field $k$. This yields restrictions on the possible structure of the torsion in crystalline cohomology (e.g., its length, when finite, has to be a multiple of $p^{\dim(A/pA)}$).

\item (BMS1) Let $C/\mathbf Q_p$ be an algebraically closed complete extension with ring of integers $\mathcal O_C$, and let $A := A_\inf = A_\inf(\mathcal O_C)$ be Fontaine's ring with $I=\ker(\theta: A_\inf\to \mathcal O_C)$. The pair $(A,I)$ is a perfect prism corresponding to the perfectoid ring $\mathcal{O}_C$. In \cite{BMS1}, we defined a complex of $A_\inf$-modules $R\Gamma_{A_\inf}(X)$ equipped with a ``Frobenius'' operator $\phi$. We shall prove in \S \ref{sec:AOmegaComp} that there exists a canonical $\phi$-equivariant isomorphism
\[
R\Gamma_{A_\inf}(X)\cong \phi_A^\ast R\Gamma_\Prism(X/A_\inf)\ .
\]
Note that in this case $\phi_{A}$ is an isomorphism, so the pullback appearing above merely twists the $A$-module structure. The comparison results from \cite{BMS1} are then immediate consequences of the comparison results above (except for the results on de~Rham-Witt theory in \cite{BMS1} that we have not taken up here); note that the $A_{\crys}$-comparison theorem from \cite{BMS1} follows from \cref{thm:A} (1) thanks to \cref{thm:A} (5) applied to the map $(A_{\inf},\ker(\theta)) \to (A_{\crys},(p))$ of bounded prisms. 

\begin{remark}
Our search for the prismatic site was substantially motivated by a desire to obtain a site-theoretic construction of the $A_{\inf}$-cohomology. The ``Frobenius'' operator on this cohomology had slightly mysterious origins: it came, rather indirectly, via the tilting equivalence for perfectoids. Thanks to the relation to prismatic cohomology,  the origins of this operator are now clear: it comes from the Frobenius lift on $\mathcal{O}_\Prism$. %In fact, this perspective also says something slightly non-trivial about the Frobenius action on the crystalline cohomology $R\Gamma_{\crys}(X/W(k))$ of smooth varieties $X$ over a perfect field $k$ characteristic $p$: under the standard isomorphism $R\Gamma_{\crys}(X/W(k)) \otimes_{W(k)}^L k \simeq R\Gamma_{dR}(X/k)$ between crystalline cohomology modulo $p$ and de Rham cohomology, the natural Frobenius endomorphism of the left side (coming from its realization as the cohomology of a sheaf of $\mathbf{F}_p$-algebras) coincides with the one on the right induced by the Frobenius on $X$.
\end{remark}

\item (Breuil-Kisin, BMS2) Let $K/\mathbf Q_p$ be a complete discretely valued field with perfect residue field $k$ and uniformizer $\pi$, and let $A=\mathfrak S=W(k)[[u]]$ with $I\subset A$ the kernel of the map $A\to \mathcal O_K$ sending $u$ to $\pi$. Then in \cite{BMS2} we defined a complex of $A$-modules $R\Gamma_{\mathfrak S}(X)$ equipped with a Frobenius, using topological Hochschild homology. In \S \ref{sec:RelativeNygaard}, we construct a canonical $\phi$-equivariant isomorphism
\[
R\Gamma_{\mathfrak S}(X)\cong R\Gamma_\Prism(X/\mathfrak S)\ .
\]
The above comparison results recover the results of \cite{BMS2} on Breuil-Kisin cohomology.

\item ($q$-crystalline) Let $A=\mathbf Z_p[[q-1]]$ with $I=([p]_q)$ as above, so that $A/I = \mathbf Z_p[\zeta_p]$. If $R$ is a $p$-completely smooth $\mathbf Z_p$-algebra, $R^{(1)} := R\otimes_{\mathbf Z_p} \mathbf Z_p[\zeta_p]$ and $X=\Spf R^{(1)}$, then for any choice of a $p$-completely \'etale map
\[
\square: \mathbf Z_p\langle T_1,\ldots,T_d\rangle\to R\ ,
\]
we construct in \S \ref{sec:qcrys} a canonical isomorphism
\[
q\Omega_R^\square\cong R\Gamma_\Prism(X/A)
\]
with the $q$-de~Rham complexes from \cite{ScholzeqdeRham}, proving \cite[Conjecture 1.1]{ScholzeqdeRham} on the independence (up to canonical quasi-isomorphism) of $q\Omega_R^\square$ from the choice of $\square$. 
\end{enumerate}
\end{example}

\begin{remark}
In Theorem~\ref{thm:A} (2), we have formulated the Hodge-Tate comparison solely in the affine case. There are several reasons for this choice. First, it is simpler to formulate the result in the affine case as we can avoid introducing excessive formalism necessary to formulate the appropriate analog for non-affine formal schemes, e.g., the non-affine version of the Hodge-Tate comparison  in Theorem~\ref{HTCompPrismatic} needs the introduction of the functor $R\nu_*:\mathrm{Shv}((X/A)_\Prism) \to \mathrm{Shv}(X_{\et})$ from Construction~\ref{PrismtoEtale} for general $X$. Secondly, there is no loss of generality:  our constructions are functorial in the $\infty$-categorical sense, so the affine statement in Theorem~\ref{thm:A} (2) immediately implies the generalization to non-affine formal schemes in Theorem~\ref{HTCompPrismatic} (once the latter has been formulated). Relatedly, the statement for affines immediately yields analogs of interest beyond formal schemes, e.g., for formal stacks. Thus, in Theorem~\ref{thm:A} (2) --- and in fact for several other analogous results elsewhere in the paper --- we have restricted ourselves to the affine case when it leads to cleaner statements without sacrificing generality.
\end{remark}

In \cite{FontaineMessing}, Fontaine-Messing gave a description of crystalline cohomology in terms of syntomic cohomology of a sheaf of divided power envelopes. The corresponding idea in mixed characteristic was taken up in \cite{BMS2}, where it was rephrased in terms of quasisyntomic descent from a class of semiperfectoid rings. To adapt these results to the prismatic context, we prove the following result about the prismatic site of semiperfectoid rings; we note that part (2) below says that there is a notion of ``perfection" even in mixed characteristic (see also Theorem~\ref{AlmostPurityIntro} (2) for a more general statement).

\begin{theorem}\label{thm:B} Let $S$ be a semiperfectoid ring, i.e., $S$ is a (derived) $p$-adically complete quotient of a perfectoid ring.
\begin{enumerate}
\item {\em Existence of universal prisms (Proposition~\ref{InitialPrism}):} The category $(S)_\Prism$ of prisms $(A,I)$ with a map $S\to A/I$ admits an initial object $(\Prism_S^{\mathrm{init}},I)$, and $I=(d)$ is principal.
\item {\em The perfectoidization of $S$ (Corollary~\ref{PerfectionMixedChar}, Theorem~\ref{PerfectoidificationSurj}):} Let $(\Prism_S^{\mathrm{init}})_\perf$ be the $(p,I)$-completed perfection of $\Prism_S^{\mathrm{init}}$, and $S_\perfd = (\Prism_S^{\mathrm{init}})_\perf/I$. Then $S_\perfd$ is a perfectoid ring and the map $S \to S_\perfd$ is the universal map to a perfectoid ring from $S$. Moreover, the map $S \to S_\perfd$ is surjective. 
\end{enumerate}
\end{theorem}

The surjectivity of $S \to S_\perfd$ in Theorem~\ref{thm:B} (2) implies that the notions of ``Zariski closed" and ``strongly Zariski closed" subsets of affinoid perfectoid spaces agree, contrary to a claim made by the second author in \cite[Section II.2]{ScholzeTorsion}. The proof of this result uses an important flatness lemma of Andr\'e \cite{AndreDirectFactor} for perfectoid rings (for which we offer a direct prismatic proof in \cref{AndreFlatness}).

\begin{remark}[Almost mathematics with respect to any closed set]
\label{ZariskiClosedStrongly}
\cref{thm:B} (2) implies that there is a good notion of ``almost mathematics'' with respect to any ideal of a perfectoid ring, and not merely $\sqrt{pR}$ as in the classical setup. More precisely,  if $J \subset R$ is any $p$-complete ideal in a perfectoid ring, \cref{thm:B} (2) implies that $R \to (R/J)_\perfd$ is a surjective map of perfectoid rings. General properties of perfectoid rings then show that the kernel $J_\perfd := \ker(R \to (R/J)_\perfd)$ satisfies $J_\perfd \widehat{\otimes}_R^L J_\perfd \simeq J_\perfd$, which leads to a notion of ``$J$-almost mathematics'' for $p$-complete $R$-modules and $p$-complete objects of the derived category $D(R)$ (see \S \ref{ss:AlmostMathIdeal}).
\end{remark}

The prism $\Prism_A^{\mathrm{init}}$ from \cref{thm:B} is poorly behaved in general, and we do not know any explicit description. However, if $S$ has a well-behaved cotangent complex, then one can do better, as this prism can be described as the derived prismatic cohomology of $S$ itself.

\begin{theorem}
\label{thm:Bq}
Let $S$ be a semiperfectoid ring. Assume now that $S$ is quasiregular, i.e.~$S$ has bounded $p^\infty$-torsion and the cotangent complex $L_{S/\mathbf Z_p}[-1]$ is $p$-completely flat. Write\footnote{We have switched from $\Prism_S^{\mathrm{init}}$ to the more evocative $\Prism_S$ in Theorem~\ref{thm:Bq} as this object is well-behaved for $S$ quasiregular semiperfectoid (by Theorem~\ref{thm:Bq}!).} $\Prism_S = \Prism_S^{\mathrm{init}}$ for the object from \cref{thm:B}.

\begin{enumerate}[resume]

\item {\em The conjugate filtration (Proposition~\ref{QRSPPrism}, Construction~\ref{DerivedPrismatic}):} The $S$-algebra $\Prism_S/I$ is $p$-completely flat. Moreover, for any choice of a perfectoid $R$ mapping to $S$, the $S$-algebra $\Prism_S/I$ admits an increasing (``conjugate") filtration $\mathrm{Fil}^{\conj}_i \Prism_S/I$ by $S$-modules together with isomorphisms
\[ \mathrm{gr}^{\conj}_i \Prism_S/I\cong (\wedge^i L_{S/R}[-i])^\wedge \{-i\},\]
where the Breuil-Kisin twist appearing on the right is defined using the perfect prism attached to $R$ as in  Theorem~\ref{thm:A} (2).

\item {\em The Nygaard filtration (\S \ref{ss:NygaardSPerfdStructure}):} The ring $\Prism_S$ admits a natural decreasing (``Nygaard") filtration
\[
\mathrm{Fil}^j_N \Prism_S = \{x\in \Prism_S\mid \phi(x)\in I^j \Prism_S\}\ .
\]
For a generator $d\in I$ coming from a generator of $\ker\theta: A_\inf(R)\to R$ for $R$ as in (1), the composite map
\[
\mathrm{Fil}^j_N \Prism_S\xrightarrow{\phi/d^j} \Prism_S\to \Prism_S/I
\]
has image $\mathrm{Fil}^{\conj}_j \Prism_S/I$, inducing an isomorphism
\[
\mathrm{gr}^j_N \Prism_S\cong \mathrm{Fil}^{\mathrm{conj}}_j \Prism_S/I\ .
\]

\item {\em Comparison with the topological theory (\S \ref{sec:BMS2Comp}):} The ring $\widehat{\Prism}_S=\pi_0 \mathrm{TP}(S;\mathbf Z_p)$ defined in \cite{BMS2} is $\phi$-equivariantly isomorphic to the completion of $\Prism_S$ with respect to its Nygaard filtration, and in particular admits a functorial $\delta$-ring structure.
\end{enumerate}
\end{theorem}

\begin{remark}
Let $S$ be a quasiregular semiperfectoid ring (see \cref{thm:B}). The formalism of topological Hochschild homology endows the commutative ring $\pi_0 \mathrm{TP}(S;\mathbf{Z}_p)$ with an endomorphism $\phi_S$ often called the ``Frobenius'' (see \cite{NikolausScholze}). Despite the name, it is not clear from the definitions that $\phi_S$ lifts the Frobenius on $\pi_0 \mathrm{TP}(S;\mathbf{Z}_p)/p$. Thanks to \cref{thm:Bq} (3), an even better statement is now available: $\phi_S$ is the Frobenius lift attached to a $\delta$-structure. (Note that this assertion need not be true for more general $S$.)
\end{remark}

We give three applications of the above results on semiperfectoid rings. Our first application concerns the Nygaard filtration on prismatic cohomology in the smooth case; the relevance of semiperfectoid rings to this question is that the quasiregular semiperfectoid rings form a basis for the quasisyntomic site. Note that part (3) below implies the result on the image of $\phi$ in \cref{thm:A}.

\begin{theorem}[Theorem~\ref{thmCagain}]
\label{thm:C}
Let $(A,I)$ be a bounded prism and let $X=\Spf(R)$ be an affine smooth $p$-adic formal scheme over $A/I$. 

\begin{enumerate}
\item {\em (Existence of the Nygaard filtration)} On the quasisyntomic site $X_\qsyn$, one can define a sheaf of $(p,I)$-completely flat $\delta$-$A$-algebras $\Prism_{-/A}$ equipped with a Nygaard filtration on its Frobenius twist
\[
\Prism_{-/A}^{(1)} := \Prism_{-/A}\widehat{\otimes}^L_{A,\phi} A
\]
given by
\[
\mathrm{Fil}^i_N \Prism_{-/A}^{(1)} = \{x\in \Prism_{-/A}^{(1)}\mid \phi(x)\in I^i \Prism_{-/A}\}\ .
\]

\item {\em (Quasisyntomic descent)} There is a canonical isomorphism
\[
R\Gamma_\Prism(X/A)\cong R\Gamma(X_\qsyn,\Prism_{-/A})
\]
and we endow prismatic cohomology with the Nygaard filtration
\[
\mathrm{Fil}^i_N R\Gamma_\Prism(X/A)^{(1)} = R\Gamma(X_\qsyn,\mathrm{Fil}^i_N \Prism_{-/A}^{(1)})\ .
\]
\item {\em (Graded pieces of the Nygaard filtration)} There are natural isomorphisms
\[
\mathrm{gr}^i_N R\Gamma_\Prism(X/A)^{(1)} \cong \tau^{\leq i} \overline{\Prism}_{R/A}\{i\}
\]
for all $i\geq 0$. 
\item {\em (Frobenius is an isogeny)} The Frobenius $\phi$ on $R\Gamma_\Prism(X/A)$ factors as
\[
\phi_A^\ast R\Gamma_\Prism(X/A)=R\Gamma_\Prism(X/A)^{(1)}\xrightarrow{\tilde{\phi}} L\eta_I R\Gamma_\Prism(X/A)\to R\Gamma_\Prism(X/A)\ ,
\]
using the d\'ecalage functor $L\eta_I$ as e.g.~in \cite{BMS1}. The map
\[
\tilde{\phi}: \phi_A^\ast R\Gamma_\Prism(X/A)\to L\eta_I R\Gamma_\Prism(X/A)
\]
is an isomorphism.
\end{enumerate}
\end{theorem}

Our second application is to the study of perfectoid rings. Generalizing Theorem~\ref{thm:B} (2), we prove that there is a ``perfectoidization'' for {\em any} finite algebra over a perfectoid ring, and this operation behaves like the perfection in characteristic $p$ in many ways. 

\begin{theorem}
\label{AlmostPurityIntro}
Let $R$ be a perfectoid ring and let $S$ be an integral $R$-algebra.
\begin{enumerate}
\item {\em Perfectoidizing integral algebras (Theorem~\ref{PerfdDiscrete}):} There is a perfectoid ring $S_\perfd$ with a map $S\to S_\perfd$ such that any map from $S$ to a perfectoid ring factors uniquely over $S_\perfd$.

\item {\em Almost purity (Theorem~\ref{GeneralAlmostPurity}):} Assume $R \to S$ is finite \'etale away from $V(J) \subset \mathrm{Spec}(R)$, where $J \subset R$ is a finitely generated ideal. Then $R \to S_\perfd$ is $J$-almost finite \'etale (see Remark~\ref{ZariskiClosedStrongly} for the terminology, and \S \ref{sec:APT} for a precise formulation).
\end{enumerate}
\end{theorem}

Theorem~\ref{AlmostPurityIntro} (2) reduces to the usual almost purity theorem for perfectoid algebras when $J = (p)$, and improves on the perfectoid Abhyankar lemma \cite{AndreAbhyankar} when $J = (g)$ for some $g \in R$; moreover, our proof avoids adic spaces (and thus perfectoid spaces), so we get a new proof of both results. As a corollary, we deduce that for a perfectoid ring $R$, the $\mathbf{F}_p$-\'etale cohomological dimension of $\mathrm{Spec}(R[1/p])$ is $\leq 1$ (Theorem~\ref{CohDim}).

The third application of our results on semiperfectoid rings is a resolution of the vanishing conjecture on the complexes $\mathbf Z_p(n)$ made in \cite[Conjecture 7.18]{BMS2}. By the main results of \cite{BMS2} and \cite{ClausenMathewMorrow}, these are related to $p$-adic algebraic $K$-theory.

\begin{theorem}[Theorem~\ref{OddVanishing}]
\label{thm:bhattvanishing} The quasisyntomic sheaves of complexes $\mathbf Z_p(n)$ of \cite{BMS2} are concentrated in degree $0$ and $p$-torsionfree. In particular, locally in the quasisyntomic topology of any quasisyntomic $\mathbf Z_p$-algebra, the mod $p$ algebraic $K$-theory $K(-)/p$ functor is concentrated in even degrees.
\end{theorem}

\cref{thm:bhattvanishing} (or, rather, its proof) yields vanishing results in concrete situations. For instance, if $C/\mathbf{Q}_p$ is a complete and algebraically closed field, then $\pi_* K(\mathcal{O}_C/p^n;\mathbf{Z}_p)$ vanishes in odd degrees for any $n \geq 0$ (\cref{OddVanishingExamples}).

\subsection{Leitfaden of the paper}
We begin with the theory of $\delta$-rings in \S \ref{sec:DeltaRings}. Having established enough language, we introduce prisms in \S \ref{sec:Prisms} and the prismatic site in \S \ref{sec:PrismaticSite}. The crystalline and Hodge-Tate comparison results in characteristic $p$ are proven next in \S \ref{ss:HTProof}. From this, we deduce the Hodge-Tate comparison in general in \S \ref{sec:HTCompGeneral}, which implies the compatibility with base change, and the de~Rham comparison (under a small technical hypothesis). 

We then proceed in \S \ref{sec:SemiPerfd} to study semiperfectoid rings $S$ via derived prismatic cohomology (i.e., using simplicial resolution by smooth algebras) to prove Theorem~\ref{thm:B}. These results are then applied in \S \ref{generalperfoidization} --- \S \ref{sec:EtaleCD}  to studying the ``generic fibre''. In particular, we prove the \'etale comparison theorem for prismatic cohomology in \S \ref{sec:EtaleComp}, and \cref{AlmostPurityIntro} in \S \ref{sec:APT}. 

In \S \ref{sec:Nygaard}, we analyze the Nygaard filtration explicitly for quasiregular semiperfectoid rings. This analysis is used to prove the comparison with \cite{BMS2} in \S \ref{sec:BMS2Comp} and \cref{thm:bhattvanishing} in \S \ref{sec:OddVanishing}. Descending back to the smooth case, we deduce Theorem~\ref{thm:C} in \S \ref{sec:RelativeNygaard}, which also implies the results on the image of $\phi$ and the general version of the de~Rham comparison in \cref{thm:A}.

Finally, in \S \ref{sec:qcrys}, we introduce $q$-crystalline cohomology; this theory is computed by $q$-de Rham complexes in the presence of co-ordinates, which resolves some conjectures from \cite{ScholzeqdeRham}. This theory is also identified with prismatic cohomology in certain situations, which leads to concrete representatives computing prismatic cohomology. These concrete representatives are used in \S \ref{sec:AOmegaComp} to relate prismatic cohomology to the $A\Omega$-complexes from \cite{BMS1}.  We end in \S \ref{sec:UniqueComp} by proving a strong uniqueness result for comparison isomorphisms (so any diagram chase involving such comparison isomorphisms necessarily commutes).

\subsection{Notation} 
\label{NotationGlobal}
Given an abelian group $M$ equipped with a set of commuting endomorphisms $\{f_i: M\to M\}_{i \in I}$, we may regard $M$ as a module over the polynomial algebra $\mathbf Z[f_i|i\in I]$ by letting the variable $f_i$ act as the endomorphism $f_i$. If we are further given a total ordering of $I$, we define the (homological) Koszul complex
\begin{equation}
\mathrm{Kos}(M;(f_i)_{i\in I}) = \ldots \to \bigoplus_{i<j} M\to \bigoplus_{i\in I} M\xrightarrow{(f_i)_{i\in I}} M\to 0
\end{equation}
with usual differentials (see, e.g., \cite[Tag 0621]{Stacks}) which represents $M\otimes_{\mathbf Z[f_i|i\in I]}^L \mathbf Z$ in the derived category, and dually the cohomological Koszul complex
\begin{equation}
\mathrm{Kos}_c(M;(f_i)_{i\in I}) = 0\to M\xrightarrow{(f_i)_{i\in I}} \prod_{i\in I} M\to \prod_{i<j} M\to\ldots
\end{equation}
representing $R\Hom_{\mathbf Z[f_i|i\in I]} (\mathbf Z,M)$ in the derived category. Note that if $I$ is finite, then these two complexes are isomorphic up to shift.

We will often take various completions. These are taken in the derived sense unless otherwise specified. We refer to \cite[Tag 091N]{Stacks} for the following assertions about derived completion; other relevant references are \cite[\S 3.4, 3.5]{BSProetale} as well as \cite[\S 7, \S 8, and Appendix D]{LurieSAG} (especially for the $\infty$-categorical aspects). Recall that if $A$ is a ring with a finitely generated ideal $I\subset A$, then a complex $M$ of $A$-modules is derived $I$-adically complete (often abbreviated to derived $I$-complete) if the natural map
\[
M\to \widehat{M} := R\lim_n \mathrm{Kos}(M;f_1^n,\ldots,f_r^n)
\]
is an isomorphism, where $f_1,\ldots,f_r\in I$ are generators of $I$; the object $\widehat{M} \in D(A)$ and the map $M \to \widehat{M}$ (and thus the condition of derived $I$-completeness) are independent of the choice of generators of $I$.\footnote{Beware that in general the completion of $M$ is not given by $R\lim_n M\otimes^L_A A/I^n$; however, this happens if $I$ is generated by a regular sequence, or if $A$ is noetherian (by Artin-Rees).} We shall write $D_{I-comp}(A)$ for the full subcategory of $D(A)$ spanned by $I$-complete complexes; if the ideal $I$ is clear from context, we shall simply write $D_{comp}(A)$ instead. 

A complex $M$ of $A$-modules is derived $I$-complete if and only if each $H^i(M)$ is derived $I$-complete. In particular, the category of derived $I$-complete $A$-modules is an abelian category stable under passage to kernels, cokernels, images and extensions in the category of all $A$-modules. Any classically $I$-adically complete $A$-module $M$ is derived $I$-complete, and conversely if $M$ is derived $I$-complete and $I$-adically separated, then $M$ is classically $I$-adically complete. In general, for an $A$-module $M$ it can happen that its derived $I$-completion $\widehat{M}$ is not concentrated in degree zero. However, if $I$ is principal and $M$ has bounded $I^\infty$-torsion, i.e.~there is some integer $n$ such that $M[I^\infty]=M[I^n]$, then $\widehat{M}$ is concentrated in degree zero and agrees with the classical $I$-adic completion $\lim_n M/I^n M$ of $M$. More generally, for $I=(f_1,...,f_r)$ and an $A$-module $M$, the derived $I$-completion of $M$ and the classical $I$-adic completion of $M$ agree if the projective system $\{H_i(\mathrm{Kos}(M;f_1^n,\ldots,f_r^n)\}_{n \geq 1}$ is pro-zero for $i > 0$. In such cases, we shall simply write ``$I$-adic completion'' for either the classical or the derived completion of $M$. 

A complex $M$ of $A$-modules is $I$-completely flat if for every $I$-torsion $A$-module $N$, the derived tensor product $M\otimes^L_A N$ is concentrated in degree $0$. This implies in particular that $M\otimes^L_A A/I$ is concentrated in degree $0$ and is a flat $A/I$-module. Moreover, if $M$ is $I$-completely flat, then it is $J$-completely flat for any ideal $J$ that contains $I^n$ for some $n$. Note that if $M$ is a flat $A$-module, then the derived $I$-completion $\widehat{M}$ is still $I$-completely flat.  More generally, we say that $M$ has finite $I$-complete Tor amplitude if $M \otimes_A^L A/I$ has finite Tor amplitude in $D(A/I)$; this is equivalent to requiring that there exist a constant $c \geq 0$ such that $M \otimes_A^L -$ carries $I$-power torsion $A$-modules to $D^{[-c,c]}$ (with $c=0$ corresponding to $I$-complete flatness).

A complex $M$ of $A$-modules is $I$-completely faithfully flat if it is $I$-completely flat and $M\otimes^L_A A/I$ (which is automatically a flat $A/I$-module concentrated in degree zero by $I$-complete flatness) is a faithfully flat $A/I$-module. 

A derived $I$-complete $A$-algebra $R$ is called $I$-completely \'etale (resp.~$I$-completely smooth, $I$-completely ind-smooth) if $R\otimes^L_A A/I$ is \'etale (resp.~smooth, ind-smooth) over $A/I$, i.e., that $R\otimes^L_A A/I$ is concentrated in degree $0$, where it is given by an \'etale (resp.~smooth, ind-smooth) $A/I$-algebra.  We note that by Elkik's algebraization results\footnote{As Elkik's exposition assumes noetherianness, we sketch a direct explanation of the relevant algebraization result via derived deformation theory. Given a commutative ring $A$ with finitely generated ideal $I \subset A$, we shall show that any $I$-completely smooth and $I$-complete simplicial commutative $A$-algebra $R$ is the derived $I$-completion of a smooth $A$-algebra. Using \cite[Tag 07M8]{Stacks}, choose a smooth $A$-algebra $R'$ lifting the smooth $A/I$-algebra $R \otimes_A^L A/I$. We shall show that the derived $I$-completion $S$ of $R'$ is isomorphic to $R$. Choose $f_1,...,f_r \in I$ generators, and write $A_n = \mathrm{Kos}(A;f_1^n,....,f_r^n)$, so $\widehat{A} \simeq \lim_n A_n$ is the derived $I$-completion of $A$ and $A/I = \pi_0(A_1)$. Each $A_n$ is a simplicial commutative ring, and (by reduction to $A=\mathbf{Z}[f_1,...,f_r]$) each map $A_{n+1} \to A_n$ is also a finite composition of square-zero extension where the ``ideals'' of the extension lie in $D^{b,\leq 0}$ at each step. Moreover, via the canonical filtration,  the map $A_1 \to \pi_0(A_1) = A/I$ is  a finite composition of square-zero extensions where the ``ideals'' of the extension lie in $D^{b,< 0}$ at each step. The $A$-algebras $R$ and $S$ are isomorphic and smooth after derived base change to $A/I$ by construction. By derived deformation theory, any such isomorphism can then be lifted successively to yield a compatible system $\{S \otimes_A^L A_n \simeq R \otimes_A^L A_n\}_{n \geq 1}$ of isomorphisms of smooth algebras over $\{A_n\}_{n \geq 1}$. Indeed, the obstruction to lifting across each square-zero extension encountered lies in a group of the form $\mathrm{Ext}^{1}_B(P,M)$, where $B$ is a simplicial commutative ring, $P$ is a finite projective $B$-module (i.e.,  $P \in D(B)$ with $P \otimes_B^L \pi_0(B)$ being finite projective over $\pi_0(B)$) and $M \in D^{b,\leq 0}(B)$; filtering $M$ by the canonical filtration and using the projectivity of $P$ shows that these groups are $0$, so the obstruction vanishes. Taking limits over $n$ then gives an isomorphism $S \simeq R$, as wanted. Note that the derivedness in this proof can be suppressed if $A$ is classically $I$-complete by simply working with the tower $\{A/I^n\}_{n \geq 1}$.}, $R$ is $I$-completely \'etale (resp.~smooth) if and only if it is the derived $I$-completion of an \'etale (resp.~smooth) $A$-algebra, which is the condition that we have used in \cite{BMS1}, \cite{BMS2}.

We will sometimes use simplicial commutative rings, e.g., to resolve arbitrary algebras by ind-smooth algebras as in the definition of the cotangent complex. (All our rings, including the simplicial ones, are assumed to be commutative.) Following topological terminology, we say that a simplicial ring $A$ is \emph{discrete} if it is concentrated in degree $0$, i.e.~$\pi_i A=0$ for $i>0$; in other words, a discrete simplicial ring is equivalent to the usual ring $\pi_0 A$. Occasionally, we use the same terminology more generally for any object of the derived category: A complex $M$ is {\em discrete} if $H^i(M)=0$ for $i\neq 0$, in which case $M$ is isomorphic to the module $H^0(M)$; likewise, a complex $M$ is {\em connective} (resp. {\em coconnective}) if $H^i(M) = 0$ for $i > 0$ (resp. $i < 0$). Given a simplicial commutative ring $A$ and $M \in D(A)$, recall that one calls $M$ {\em (faithfully) flat} if $M \otimes_A^L \pi_0(A) \in D(\pi_0(A))$ is (faithfully) flat in the usual sense; similarly, one can define the notion of {\em $I$-complete (faithful) flatness} given a finitely generated ideal $I \subset \pi_0(A)$.

Finally, we occasionally use $\infty$-categorical techniques (especially when constructing and using derived prismatic cohomology, and when talking about descent in the derived category). We shall follow the following standard convention in these situations: given a commutative (or $E_\infty$-) ring $A$, we write $\mathcal{D}(A)$ for its derived $\infty$-category. Similarly, given a finitely generated ideal $I \subset \pi_0(A)$, we write $\mathcal{D}_{I-comp}(A)$ (or simply $\mathcal{D}_{comp}(A)$ if $I$ is clear) for the full $\infty$-subcategory of $\mathcal{D}(A)$ spanned by $I$-complete objects. The following variant of faithfully flat descent will be used often: for a commutative ring $A$ with a finitely generated ideal $I \subset A$, the functor $\mathcal{D}_{I-comp}(-)$ on simplicial commutative $A$-algebras is a sheaf for the topology defined by $I$-completely faithfully flat maps\footnote{To see this, fix generators $f_1,...,f_r \in I$. Given a map $B \to C$ of simplicial commutative $A$-algebras with derived Cech nerve $C^*$, one has $\mathcal{D}_{I-comp}(B) \simeq \lim_n \mathcal{D}(\mathrm{Kos}(B;f_1^n,...,f_r^n))$ and similarly for $C^*$; this is proven for the bounded above variant in \cite[Theorem 8.3.4.4]{LurieSAG} and asserted for the full derived $\infty$-categories in \cite[Remark 8.3.4.5]{LurieSAG}; the latter can be proven by the same method as \cite[Lemma 3.5.5]{BSProetale}. Thus, the map $\mathcal{D}_{I-comp}(B) \to \lim \mathcal{D}_{I-comp}(C^*)$ is the inverse limit over $n$ of the maps $\mathcal{D}(\mathrm{Kos}(B;f_1^n,...,f_r^n)) \to \lim \mathcal{D}(\mathrm{Kos}(C^*;f_1^n,...,f_r^n))$; the latter maps are all equivalences by faithfully flat descent for simplicial commutative rings (\cite[Corollary D.6.3.3]{LurieSAG}), so the inverse limit is also an equivalence.}.

\subsection*{Acknowledgements}
We are grateful to Johannes Ansch\"utz, Ben Antieau, Sasha Beilinson, K\k{e}stutis \v{C}esnavi\v{c}ius, Johan de Jong, Vladimir Drinfeld,  H\'el\`ene Esnault, Ofer Gabber, Lars Hesselholt, Mel Hochster, Luc Illusie, Arthur-C\'esar Le Bras, Judith Ludwig, Jacob Lurie, Akhil Mathew, Matthew Morrow, Emanuel Reinecke, Martin Speirs and Nils Wa\ss muth for useful conversations. We thanks K\k{e}stutis \v{C}esnavi\v{c}ius, Lars Hesselholt, and especially Jacob Lurie, for their comments on a draft version of this paper. We are also grateful to Yves Andr\'e, Ko Aoki, Toby Gee, Kiran Kedlaya, Mark Kisin, Dmitry Kubrak, Bernard Le Stum, Zhouhang Mao, Akhil Mathew, Matthew Morrow, Arthur Ogus and the anonymous referees for comments on the initial versions of this paper; thanks especially for Teruhisa Koshikawa for pointing out an error (and a fix) in construction of \v{C}ech-Alexander complexes for prismatic cohomology (now incorporated in Construction~\ref{PrismaticFunctorialCC}).

Work on this project was carried out at a number of institutions, including the University of Bonn, MPIM, University of Michigan, Columbia University and MSRI (supported by the NSF grant \#1440140);  we thank these institutions for their hospitality. In the duration of this project, the first author was partially supported by the NSF grants  \#1501461 and \#1801689, a Packard fellowship and the Simons foundation grant \#622511, while the second author was supported by a DFG Leibniz Grant and the Hausdorff Center for Mathematics (GZ 2047/1, Projekt-ID 390685813).

\newpage

\section{$\delta$-rings}
\label{sec:DeltaRings}

Fix a prime $p$. In this section, we discuss the theory of $\delta$-rings. In \S \ref{ss:DeltaDef} and \S \ref{ss:ExtendDelta}, which are essentially review, we discuss the definitions and basic properties of $\delta$-rings \cite{JoyalDelta}; a good reference for this material is \cite{BorgerCourse}. In \S \ref{ss:DistElt}, we introduce distinguished elements, which are essential to defining prisms. Basic properties of perfect $\delta$-rings are then the subject of \S \ref{ss:PerfDelta}. In \S \ref{ss:DeltaPD}, we establish a relationship between divided power envelopes and $\delta$-structures; this relationship plays an important role in many subsequent computations in this paper, thanks in large part to the nice commutative algebra properties of divided power envelopes of regular sequences, which are reviewed in \S \ref{ss:PDEnvelopeFlat}.

All our rings are $\mathbf{Z}_{(p)}$-algebras. We write $\rad(A)$ for the Jacobson radical of a ring $A$. In the following, note that the expression
\[
\frac {x^p+y^p-(x+y)^p}p\in \mathbf Z[x,y]
\]
can be evaluated in any ring.

\subsection{Definition and basic properties}
\label{ss:DeltaDef}

\begin{definition}\label{def:deltaring}
A {\em $\delta$-ring} is a pair $(R,\delta)$ where $R$ is a commutative ring and $\delta:R \to R$ is a map of sets with $\delta(0) = \delta(1) = 0$, satisfying the following two identities
\[ \delta(xy) = x^p \delta(y) + y^p \delta(x) + p \delta(x) \delta(y) \quad \text{and} \quad \delta(x+y) = \delta(x) + \delta(y) + \frac {x^p+y^p-(x+y)^p}p.\]
There is an evident category of $\delta$-rings. (In the literature, a $\delta$-structure is often called a {\em $p$-derivation}.)
\end{definition}

\begin{remark}[$\delta$-structures give Frobenius lifts]
\label{FrobLiftTheta}
Given a $\delta$-ring $(R,\delta)$, we write $\phi:R \to R$ for the map defined by $\phi(x) = x^p + p \delta(x)$; the identities on $\delta$ ensure that this is a ring homomorphism that lifts Frobenius on $R/p$. In fact, the identities on $\delta$ are reverse engineered from this requirement: if $R$ is a $p$-torsionfree ring, then any lift $\phi:R \to R$ of the Frobenius on $R/p$ comes from a unique $\delta$-structure on $R$, given by the formula $\delta(x) = \frac{\phi(x) - x^p}{p}$.  Note that as the condition of being a lift of Frobenius is vacuous if $p$ is invertible, a $\delta$-ring over $\mathbf{Q}$ is the same thing as a $\mathbf{Q}$-algebra with an endomorphism. In general, given a $\delta$-ring $R$, we shall write $\phi:R \to R$ for its Frobenius endomorphism; if there is potential for confusion, we shall denote this map by $\phi_R$ instead.
\end{remark}

\begin{remark}[$\delta$-structures and $\lambda$-structures] By a theorem of Wilkerson \cite[Proposition 1.2]{Wilkerson}, giving a $\lambda$-ring structure on a flat $\mathbf Z$-algebra $R$ is equivalent to specifying commuting Frobenius lifts $\phi_p$ on $R$ for all primes $p$. Borger has extended this in \cite{BorgerWitt1} in multiple ways (including allowing all $\mathbf{Z}$-algebras). In particular, it follows from his work that a $\delta$-structure on a ring $R$ is the same as a $p$-typical $\lambda$-structure. This motivates the following terminology: an element $x$ in a $\delta$-ring $B$ has {\em rank $1$} if $\delta(x) = 0$; such elements satisfy $\phi(x) = x^p$ (and the converse holds true if $B$ is $p$-torsionfree).
\end{remark}

\begin{remark}[$\delta$-rings via $W_2(-)$, following Rezk \cite{RezkLambda}]
\label{ThetaW2}
For any ring $R$, the ring $W_2(R)$ of $p$-typical length $2$ Witt vectors is defined as follows: we have $W_2(R) = R \times R$ as sets, and addition and multiplication are defined via
\[ (x,y) + (x',y') := (x+x',  y+y' + \frac{x^p + (x')^p - (x+x')^p}{p}) \quad \text{and} \quad (x,y) \cdot (x',y') = (xx', x^py' + x'^py + pyy').\]
Ignoring the second component gives a ring homomorphism $\epsilon:W_2(R) \to R$. It is immediate from the definitions that specifying a $\delta$-structure on $R$ is the same as specifying a ring map $w:R \to W_2(R)$ such that $\epsilon \circ w = \mathrm{id}$: the correspondence attaches the map $w(x) = (x,\delta(x))$ to a $\delta$-structure $\delta:R \to R$ on $R$.
\end{remark}

\begin{remark}[Derived Frobenius lifts give $\delta$-structures]
\label{DerivedFrobLift}
Let $R$ be any $\mathbf{Z}_{(p)}$-algebra. Then specifying a $\delta$-structure on $R$ is the same as specifying a map $\phi:R \to R$ and a path in the space $\mathrm{End}_{SCR_{\mathbf{F}_p}}(R \otimes_{\mathbf{Z}}^L \mathbf{F}_p)$, between the points defined by $\phi$ and the Frobenius; here $SCR_{\mathbf F_p}$ denotes the $\infty$-category of simplicial commutative $\mathbf F_p$-algebras. In other words, giving a $\delta$-structure on $R$ is the same as specifying a lift of Frobenius in the derived sense. To prove this, first note that for a $p$-torsionfree ring $R$, the square
\[\xymatrix{
W_2(R)\ar[r]^F \ar[d] & R\ar[d]\\
R\ar[r]^{\tilde{\phi}} & R/p
}\]
is a pullback square of rings, where the right map i the canonical projection, the lower map is the composite of the projection $R\to R/p$ with the Frobenius of $R/p$, and the top map is the Witt vector Frobenius (defined on Witt co-ordinates by $F(x,y) = x^p + py$). Passing to simplicial resolutions, this implies that for any simplicial ring, there is a functorial pullback square
\[\xymatrix{
W_2(R)\ar[r]^F \ar[d] & R\ar[d]\\
R\ar[r]^{\tilde{\phi}} & R\otimes^L_{\mathbf Z} \mathbf F_p
}\]
of simplicial rings. Using this pullback square, \cref{ThetaW2} translates into the desired description of $\delta$-rings.
\end{remark}

\begin{example}[The initial $\delta$-ring]
\label{InitialDeltaRing}
The identity map on the $p$-torsionfree ring $\mathbf{Z}_{(p)}$ is its unique endomorphism and lifts the Frobenius modulo $p$, so $\mathbf{Z}_{(p)}$ carries a unique $\delta$-structure given by $\delta(x) = \frac{x-x^p}{p}$. In fact, this is the initial object in the category of $\delta$-rings (as we work with $\mathbf{Z}_{(p)}$-algebras). One checks that $\delta$ on $\mathbf{Z}_{(p)}$ lowers the $p$-adic valuation by $1$ for non-units. In particular, $\delta^n(p^n)$ is a unit for all $n$. As $\mathbf{Z}_{(p)}$ is the initial $\delta$-ring, it follows that there is no nonzero $\delta$-ring where $p^n = 0$ for some $n \geq 0$.
\end{example}

\begin{remark}[Limits and colimits of $\delta$-rings, and Witt vectors]
\label{ThetaWitt}
One can show that the category of $\delta$-rings admits all limits and colimits, and these are computed at the level of underlying rings: this is easy for limits, and for colimits it follows via the characterization in \cref{ThetaW2} as there is a natural map $\colim_i W_2(-) \to W_2(\colim_i -)$ of functors. It follows by general nonsense that the forgetful functor from $\delta$-rings to rings has both left and right adjoints, since it commutes with both colimits and limits. The left adjoint provides one with a notion of ``free objects'' and is studied in \cref{freelambdaring} below. The right adjoint is given by the (always $p$-typical) Witt vector functor $W(-)$ by a result of Joyal \cite{JoyalDelta}. Thus, for each $\delta$-ring $R$, we have a natural map $R \xrightarrow{w_R} W(R)$ of $\delta$-rings by adjunction; the first two components of this map (in standard Witt co-ordinates) are given by $x \mapsto (x,\delta(x))$, as in \cref{ThetaW2}.
\end{remark}

\begin{notation}
We shall use the symbols $\{ \}$ and $()_\delta$ to denote the adjoining and killing of elements in the theory of $\delta$-rings. Thus, $\mathbf{Z}_{(p)}\{x\}$ is the free $\delta$-ring on one generator $x$, $\mathbf{Z}_{(p)}\{x,y\}/(f)_\delta$ is defined by a pushout square 
\[ \xymatrix{ \mathbf{Z}_{(p)}\{t\} \ar[r]^-{t \mapsto f} \ar[d]^-{t \mapsto 0} & \mathbf{Z}_{(p)}\{x,y\} \ar[d] \\
		   \mathbf{Z}_{(p)} \ar[r] & \mathbf{Z}_{(p)}\{x,y\}/(f)_\delta,}\]
etcetera. 
\end{notation}

Regarding the quotients that one takes here, we note the following lemma.

\begin{lemma}[Quotients]
Let $A$ be a $\delta$-ring. Let $I \subset A$ be an ideal. Then $I$ is stable under $\delta$ if and only if there exists a (necessarily unique) $\delta$-structure on $A/I$ compatible with the one on $A$. 
\end{lemma}

\begin{proof}
The ``if'' direction is clear. For the ``only if'' direction, we must show that if $a \in A$ and $f \in I$, then $\delta(a) \equiv \delta(a+f) \mod I$; but this follows immediately from the additivity formula.
\end{proof}

\begin{example}[Quotients in $\delta$-rings]
Let $A$ be a $\delta$-ring, and let $I \subset A$ be an ideal. Then the universal $\delta$-$A$-algebra $B$ with $IB = 0$ is given by $B = A/J$, where $J$ is the $\delta$-stabilization of $I$, i.e., the ideal generated by $\cup_n \delta^n(I)$. 
\end{example}

\begin{lemma}[Free $\delta$-rings]
\label{freelambdaring}
The ring $\mathbf{Z}_{(p)}\{x\}$ is a polynomial ring on the set $\{x,\delta(x),\delta^2(x),...\}$ and its Frobenius endomorphism is faithfully flat. The ring $\mathbf{Q}\{x\} = \mathbf{Z}_{(p)}\{x\}[\frac{1}{p}]$ is also a polynomial ring on the set $\{x,\phi(x),\phi^2(x),...\}$. 
\end{lemma}
\begin{proof}
Consider the polynomial ring $A = \mathbf{Z}_{(p)}[x_0,x_1,x_2,...]$ on countably many generators. The assignment $x_i \mapsto x_i^p + p x_{i+1}$ gives an endomorphism $\phi$ of $A$ that lifts the Frobenius on $A/p$. As $A$ is $p$-torsionfree, there is a unique $\delta$-structure on $A$ described by $\delta(x_i) = x_{i+1}$. Given any $\delta$-ring $R$ with an element $f \in R$, thanks to the universal property of the polynomial ring, there is a unique ring homomorphism $\eta_f:A \to R$ defined by $\eta_f(x_i) = \delta^i(f)$. By construction, we have $\eta_f(\delta(x_i)) = \delta(\eta_f(x_i))$, so $\eta_f$ is a map of $\delta$-rings, as the relations imposed on $\delta$ in \cref{def:deltaring} determine the behaviour on the $\mathbf Z_{(p)}$-algebra generated by the $x_i$ from the behaviour on the $x_i$. It is then also clear that $\eta_f$ is uniquely determined as a map of $\delta$-rings by the requirement $\eta_f(x_0) = f$. It follows that setting $\mathbf{Z}_{(p)}\{x\} = A$ with $x = x_0$ gives the free $\delta$-ring on a generator $x$, so this ring is indeed a polynomial ring on $\{x,\delta(x),\delta^2(x),...\}$ as asserted. Inverting $p$ easily gives the desired assertion for $\mathbf{Q}\{x\}$ as well.

We explain the faithful flatness assertion. The map $\phi: A\to A$ can be written as the filtered colimit of the maps
\[
\phi_i: \mathbf Z_{(p)}[x,\delta(x),\ldots,\delta^i(x)]\to \mathbf Z_{(p)}[x,\delta(x),\ldots,\delta^{i+1}(x)]\ ,
\]
so it is enough to prove that each $\phi_i$ is faithfully flat. By the fibrewise criterion for flatness, it suffices to see that $\phi_i[1/p]$ and $\phi_i/p$ are faithfully flat. But $\phi_i[1/p]$ agrees with the map
\[
\mathbf{Q}[x,\phi(x),\ldots,\phi^i(x)]\to \mathbf{Q}[x,\phi(x),\ldots,\phi^{i+1}(x)]
\]
shifting generators, which is evidently faithfully flat. On the other hand, $\phi_i/p$ agrees with the composite of the Frobenius on $\mathbf F_p[x,\delta(x),\ldots,\delta^i(x)]$ with the inclusion
\[
\mathbf F_p[x,\delta(x),\ldots,\delta^i(x)]\hookrightarrow \mathbf F_p[x,\delta(x),\ldots,\delta^{i+1}(x)]\ ,
\]
both of which are faithfully flat.
\end{proof}

\begin{corollary}[Frobenius is fpqc locally surjective]\label{Frobeniusfpqcsurj}
Fix a $\delta$-ring $A$ and an element $x \in A$. Then there exists a faithfully flat map $A \to B$ of $\delta$-rings such that the image of $x$ in $B$ has the form $\phi(y)$ for some $y \in B$. 
\end{corollary}
\begin{proof}
Set $B$ to be the pushout of the diagram $\mathbf{Z}_{(p)}\{s\} \gets \mathbf{Z}_{(p)}\{t\} \to A$ of $\delta$-rings, where the first map sends $t$ to $\phi(s)$ and the second map sends $t$ to $x$. The resulting map $A \to B$ is faithfully flat by Lemma~\ref{freelambdaring} (recalling that pushouts of $\delta$-rings can be computed on the level of underlying rings by \cref{ThetaWitt}), and the image of $x$ equals that of $\phi(s)$ in $B$.
\end{proof}

\begin{remark}[Joyal's $\delta_n$-operations]
\label{JoyalOps}
As $\delta$-rings $R$ admit a natural map $w: R\to W(R)$ to their ring of $p$-typical Witt vectors, there are natural functorial operations $\delta_n:R \to R$ for $n \geq 0$ on any $\delta$-ring $R$ such that
\[
w(x)=(\delta_0(x),\delta_1(x),\delta_2(x),\ldots)\in W(R)
\]
in Witt vector coordinates; in particular $\delta_0(x)=x$ and $\delta_1(x)=\delta(x)$. These operations can be characterised by the following universal identity: for each element $x$ in each $\delta$-ring $R$, we have 
\[ \phi^n(x) = \delta_0(x)^{p^n} + p \delta_1(x)^{p^{n-1}} + ... + p^n \delta_n(x).\]
In general, one can express $\delta_n(-)$ as a monic polynomial of degree $n$ in $\delta$. 
\end{remark}

\begin{remark}[The free $\delta$-ring via Joyal's operations]
\label{JoyalOpsGenerate}
As the operation $\delta_n$ from Remark~\ref{JoyalOps} is a monic degree $n$ polynomial in $\delta$, one can reformulate the first part of Lemma~\ref{freelambdaring} as the following assertion: the ring $\mathbf{Z}_{(p)}\{x\}$ is a polynomial ring on the set $\{\delta_n(x)\}_{n \geq 0}$. 
\end{remark}

\subsection{Extending $\delta$-structures}
\label{ss:ExtendDelta}

\begin{lemma}[Localizations]
\label{ExtendLocalize}
Let $A$ be a $\delta$-ring. Let $S \subset A$ be a multiplicative subset such that $\phi(S) \subset S$. Then the localization $S^{-1} A$ admits a unique $\delta$-structure compatible with the map $A \to S^{-1}(A)$. Moreover, the map $A \to S^{-1} A$ is initial amongst all $\delta$-$A$-algebras $B$ such that each element of $S$ is invertible in $B$. 
\end{lemma}
\begin{proof}
We first explain the argument when $A$ is $p$-torsionfree. In this case, the localization $S^{-1} A$ is also $p$-torsionfree. The map $\phi_A:A \to A$ carries $S$ to itself, and hence induces a map $\phi_{S^{-1}A}:S^{-1} A \to S^{-1} A$. As $A \to S^{-1} A$ is a localization, it is easy to see that $\phi_{S^{-1}A}$ is a lift of Frobenius as $\phi_A$ is so; this gives the first part of the lemma. The second part is clear.

In general, given a pair $(A,S)$ as in the lemma, choose a surjection $\alpha:F \twoheadrightarrow A$ with $F$ being a free $\delta$-ring on some set. Then $F$ is $p$-torsionfree (by \cref{freelambdaring}). The preimage $T := \alpha^{-1}(S) \subset F$ is a multiplicative closed subset of $F$ (as $\alpha$ is multiplicative) that is $\phi$-stable (as $\alpha$ commutes with $\phi$).  The localization $T^{-1} F$ carries a unique $\delta$-structure compatible with the one on $F$ by the preceding paragraph. The formula $S^{-1} A = T^{-1} F \otimes_F A$, and the fact that colimits of $\delta$-rings coincide with those of the underlying rings, then shows that $S^{-1} A$ also carries a unique $\delta$-structure compatible with the one on $A$. The last part is clear.
\end{proof}

\begin{remark}[Localizations in the $p$-local world]\label{ExtendLocalizePLocal}
In a $\delta$-ring $A$ with $p$ in the Jacobson radical $\rad(A)$, the formula $\phi(f) = f^p + p \delta(f)$ shows that if $f$ is a unit, so is $\phi(f)$. Thus, for any $\delta$-ring $A$ and any multiplicative subset $S \subset A$, the $p$-localization $(S^{-1} A)_{(p)}$ of the localization $S^{-1} A$ of $A$ coincides with the $p$-localization of $T^{-1} A$ where $T = \{S,\phi(S),\phi^2(S),...\}$. By \cref{ExtendLocalize}, it follows that $(S^{-1} A)_{(p)}$ carries a unique $\delta$-structure compatible with the one on $A$, and can be characterized as the initial object in the category of all $\delta$-$A$-algebras $B$ with $p\in \rad(B)$ where $S$ becomes invertible.
\end{remark}

In the paper, we will use the preceding remark mostly with $p$-localization replaced with $p$-completion. Regarding completions, we have the following general result.

\begin{lemma}[Completions]
\label{LambdaCompletions}
Let $A$ be a $\delta$-ring, and let $I \subset A$ be a finitely generated ideal containing $p$. Then the map $\delta: A\to A$ is $I$-adically continuous; more precisely, for each $n$ there is some $m$ such that for all $x\in A$, one has $\delta(x+I^m)\subset \delta(x)+I^n$.

Moreover, the classical $I$-adic completion of $A$ acquires a unique $\delta$-structures compatible with the one on $A$.
\end{lemma}

The case of derived completions will be handled by the next lemma.	

\begin{proof} Once we have proved that $\delta$ is $I$-adically continuous, it follows that $\delta$ extends to a continuous map on the classical $I$-adic completion $\widehat{A}$ of $A$, which will by continuity still be a $\delta$-structure. This extension is also unique as the $\delta$-structure on $\widehat{A}$ must be $I\widehat{A}$-adically continuous by the same result applied to $\widehat{A}$.

For continuity, we note that the additivity formula implies that
\[
\delta(x+I^m)-\delta(x)\subset \delta(I^m)+I^m\ ,
\]
so it suffices to see that for any $n$ there is some $m\geq n$ such that $\delta(I^m)\subset I^n$. Note that the product (and addition) formula for $\delta$ imply that for any two ideals $J_1,J_2$,
\[
\delta(J_1J_2)\subset J_1+J_2+p\delta(J_1)\delta(J_2)A\ .
\]
In particular, taking $J_1=J_2=I$, we see that $\delta(I^2)\subset I$ as $p\in I$ by assumption. Taking $J_1=J_2=I^{2^n}$ then shows inductively that $\delta(I^{2^{n+1}})\subset I^{2^n}$, as desired.
\end{proof}

\begin{lemma}[\'Etale maps]\label{ExtendEtale}
Let $A$ be a $\delta$-ring equipped with a finitely generated ideal $I \subset A$ containing $p$. Assume that $B$ is a derived $I$-complete and $I$-completely \'etale $A$-algebra. Then $B$ admits a unique $\delta$-structure compatible with the one on $A$.
\end{lemma}

This lemma implies, in particular, that the $\delta$-structure on $A$ passes uniquely  to its derived $I$-completion for any $I \subset A$ containing a power of $p$.

\begin{proof} We use the characterization of $\delta$-rings in terms of $W_2$ as in Remark~\ref{ThetaW2}. By Elkik's algebraization theorem, we can write $B$ as the derived $I$-completion of some \'etale $A$-algebra $B^\prime$. Then $W_2(A)\to W_2(B^\prime)$ is \'etale by van der Kallen's theorem, cf.~\cite[Theorem 10.4]{BMS1}. We claim that $W_2(B)$ is the derived $I$-completion of $W_2(B^\prime)$, when regarded as $A$-algebras via $A\xrightarrow{w_A} W_2(A)\to W_2(B^\prime)$. For this, note that there is a short exact sequence
\[
0\to \phi_\ast B^\prime\to W_2(B^\prime)\to B^\prime\to 0
\]
of $A$-modules, and derived $I$-completion agrees with derived $\phi(I)$-completion as $p\in I$. In particular $W_2(B)$ is derived $I$-complete and $I$-completely \'etale over $W_2(A)$.

Considering the diagram
\[\xymatrix{
A\ar[r]\ar[d] & W_2(B)\ar[d]\\
B^\prime\ar[r]& B
}\]
and using that $A\to B^\prime$ is \'etale while $W_2(B)\to B$ is a pro-infinitesimal thickening, we see that there is a unique lift $B^\prime\to W_2(B)$ making the diagram commute, which then extends to a unique map $w_B: B\to W_2(B)$ as $W_2(B)$ is derived $I$-complete. This gives the desired unique $\delta$-structure on $B$ compatible with the one on $A$.
\end{proof}

\subsection{Distinguished elements}
\label{ss:DistElt}

The following notion plays a central role in this paper:

\begin{definition}
An element $d$ of a $\delta$-ring $A$ is {\em distinguished} if $\delta(d)$ is a unit.
\end{definition}

Any morphism of $\delta$-rings preserves distinguished elements. 

\begin{example} 
\label{distinguishedElts}
The following examples of distinguished elements are crucial for cohomological purposes.
\begin{enumerate}
\item {\em Crystalline cohomology.} Take $A = \mathbf{Z}_p$ with $d = p$. Indeed, $\delta(p) = 1 - p^{p-1} \in \mathbf{Z}_{p}^*$.

\item {\em $q$-de Rham cohomology.} Take $A = \mathbf{Z}_p\llbracket q-1\rrbracket$, $d = [p]_q := \frac{q^p-1}{q-1} \in A$, with $\delta$-structure determined by $\phi(q) = q^p$. The distinguishedness of $d$ can be seen directly, or by simply observing that $\delta([p]_q) \equiv \delta(p) \mod (q-1)$, so the claim follows from (1) and $(q-1)$-adic completeness.

\item {\em $A_{\inf}$-cohomology.} For a perfectoid field $C/\mathbf{Q}_p$, take $A = A_{\inf}(\mathcal{O}_C)$ with $d=\xi$ being any generator of the kernel of Fontaine's map $A_{\inf} \to \mathcal{O}_C$. The ring $A$ carries a unique $\delta$-structure given by the usual lift of Frobenius. The distinguishedness of $d$ can be seen as in (2) via specialization along the $\delta$-map $A_{\inf} \to W(k)$, where $k$ is the residue field of $C$.

\item {\em Breuil-Kisin cohomology.} Fix a discretely valued extension $K/\mathbf{Q}_p$ with uniformizer $\pi$. Let $W \subset \mathcal{O}_K$ be the maximal unramified subring. Take $A = W \llbracket u \rrbracket$ with $\delta$-structure determined by the canonical one on $W$ and satisfying $\phi(u) = u^p$. There is a $W$-equivariant surjection $A \to \mathcal{O}_K$ determined by $u \mapsto \pi$. Any generator $d \in A$ of the kernel of this map is distinguished; this can be seen as in (2) via specialization along the $\delta$-map $W\llbracket u \rrbracket \xrightarrow{u \mapsto 0} W$.
\end{enumerate}
\end{example}

\begin{example}
Consider the initial object in the category $\delta$-rings equipped with a distinguished element $d$. By \cref{ExtendLocalize} and \cref{freelambdaring}, such an object exists and can be described explicitly as the localization $S^{-1} \mathbf{Z}_{(p)}\{d\}$, where $S = \{\delta(d),\phi(\delta(d)),\phi^2(\delta(d)),\ldots\}$. In particular, the universal distinguished element $d$ is a nonzerodivisor modulo $p$. If no confusion arises, we shall denote this $\delta$-ring by $\mathbf{Z}_{(p)}\{d,\delta(d)^{-1}\}$.

\end{example}

\begin{remark}[Viewing a distinguished element as a deformation of $p$]
\label{rmk:DistEltDefp}
There is a relatively easy way to map any distinguished element to the distinguished element $p$ (up to units). Let $A$ be a derived $p$-complete $\delta$-ring equipped with a distinguished element $d$. Consider the composite 
\[ s:A \xrightarrow{\phi} A \xrightarrow{w} W(A) \xrightarrow{can} W(A/d),\] 
where $w$ comes from \cref{ThetaWitt}. As $s$ is a map of $\delta$-rings, the Witt vector $s(d) \in W(A/d)$ has the form $F(0,\delta(d),\ldots) \in W(A/d)$, where $F$ is the Witt vector Frobenius. Thus, we can write $s(d) = F(V(\delta(d),...)) = p \cdot (\delta(d),....)$. As $\delta(d)$ is a unit and $W(A/d)$ is complete along the kernel of $W(A/d) \to A/d$ (by the derived $p$-completeness of $A/d$ and the analogous assertion for $\mathbf{Z}/p^n$-algebras), the Witt vector $(\delta(d),...) \in W(A/d)$ is a unit as well, so $s(d) = pu$ for a unit $u \in W(A/d)$. Thus, $s$ gives the promised map. However, the passage from $A$ to $W(A/d)$ is fairly drastic, and one often loses control on the algebraic properties of this map. A refinement of $s$ which is better behaved homologically is presented in Construction~\ref{UnivOrientedCrystallize} (under some extra hypotheses on $A$).
\end{remark}

\begin{lemma}
\label{distinguishedUnit}
Let $A$ be a $\delta$-ring. Fix a distinguished element $d \in A$ and a unit $u \in A^*$. If $d,p \in \rad(A)$, then $ud$ is distinguished.
\end{lemma}
\begin{proof}
We want to show that $\delta(ud)$ is a unit. Expanding, we have
\[ \delta(ud) = u^p \delta(d) + d^p \delta(u) + p \delta(u) \delta(d).\]
The first term on the right side is a unit and the other two terms lie in $\rad(A)$, so the whole expression is also a unit.
\end{proof}

\begin{lemma}
\label{distinguishedDivide}
Let $A$ be a $\delta$-ring with a distinguished element $d\in A$. Assume that we can write $d = fh$ for some $f,h \in A$ such that $f,p \in \rad(A)$. Then $f$ is distinguished and $h$ is a unit.
\end{lemma}
\begin{proof}
Applying $\delta$ to $d=fh$ gives 
\[ \delta(d) = f^p \delta(h) + h^p \delta(f) + p \delta(f) \delta(h).\]
The left side is a unit, while the first and last terms of the right side lie in $\rad(A)$, so $h^p \delta(f)$ is a unit, which proves both claims.
\end{proof}

\begin{lemma}
\label{distinguishedIntersect}
Fix a $\delta$-ring $A$ and an element $d\in A$ such that $d,p \in \rad(A)$. Then $d$ is distinguished if and only if $p \in (d,\phi(d))$. In particular, the property ``$d$ is distinguished'' only depends on the ideal $(d)$. 
\end{lemma}

\begin{proof}
Assume first that that $d$ is distinguished, so $\delta(d)$ is a unit. Then the formula $\phi(d) = d^p + p \delta(d)$ immediately shows that $p \in (d,\phi(d))$.

Conversely, assume we can write $p = ad + b\phi(d)$ for some $a,b \in A$. We want to show $\delta(d)$ is invertible. As $d,p \in \rad(A)$, it suffices to show that $\delta(d)$ is invertible modulo $(d,p)$; equivalently, it suffices to show that $A/(p,d,\delta(d)) = 0$. To prove this, we may replace $A$ with its $(p,d,\delta(d))$-adic completion (or just a suitable ind-Zariski localization) to assume that $p, d, \delta(d) \in \rad(A)$. Simplifying the equation $p = ad + b\phi(d)$ using the definition of $\phi$ then yields an equation of the form $p(1 - b\delta(d)) = cd$ for suitable $c \in A$. As $p$ is distinguished and $\delta(d) \in \rad(A)$, the left hand side is distinguished by Lemma~\ref{distinguishedUnit}. But then Lemma~\ref{distinguishedDivide} implies that $d$ is distinguished, as desired.
\end{proof}

\subsection{Perfect $\delta$-rings}
\label{ss:PerfDelta}

\begin{definition}
A $\delta$-ring $A$ is {\em perfect} if $\phi$ is an isomorphism.
\end{definition}

\begin{remark}[Perfection]
 The inclusion of perfect $\delta$-rings into all $\delta$-rings has left and right adjoints given by the perfection functors $A \mapsto A_\perf = \colim_\phi A$ and $A \mapsto A^\perf = \lim_\phi A$ respectively.
 \end{remark}

A pleasant feature of this theory is that Frobenius kills the $p$-torsion.

\begin{lemma}[$p$-torsion in $\delta$-rings]
\label{PerfLambdaTF}
Fix a $\delta$-ring $A$. Then $A$ is $p$-torsionfree if at least one of the following holds true:
\begin{enumerate}
\item The Frobenius $\phi$ is injective (e.g., $A$ is perfect).
\item $A$ is reduced.
\end{enumerate}
\end{lemma}

\begin{proof}
For (1), we shall prove a stronger statement: given $x \in A$ with $px = 0$, we have $\phi(x) = 0$. To see this, apply $\delta$ to $px = 0$ to get 
\[ 0 = p^p \delta(x) + x^p \delta(p) + p \delta(x) \delta(p) = p^p \delta(x) + \phi(x) \delta(p).\]
As $\delta(p)$ is a unit, it is enough to show that $p^p \delta(x) = 0$. But we have
\[ p^p \delta(x) = p^{p-1}(\phi(x) - x^p) = \phi(p^{p-1} x) - p^{p-1} x^p,\]
and this vanishes as $px = 0$ and $p \geq 2$.

For (2), say $A$ is a reduced $\delta$-ring and $x \in A$ with $px = 0$. We have seen above that $\phi(x) = 0$ and that $p^p \delta(x) = 0$. The latter implies that $p\delta(x) = 0$ as $A$ is reduced. On the other hand, since $\phi(x) = 0$, we have $x^p = -p\delta(x)$, so $x^p = 0$, whence $x=0$ by reducedness of $A$. 
\end{proof}

\begin{remark}
A ring $R$ of characteristic $p$ is reduced if and only if its Frobenius endomorphism is injective. Thus, one might wonder if assumptions (1) and (2) in \cref{PerfLambdaTF} are equivalent. In fact, they are mutually non-comparable, i.e., neither implies the other. For example, the reduced ring $A = \mathbf{Z}_p[x]/(x^p-1)$ supports a unique $\delta$-structure $\phi(x) = x^p = 1$; the Frobenius lift $\phi$ on $A$ is not injective since $\phi(x-1) = 0$. Conversely, the non-reduced ring $B = \mathbf{Z}_p[x]/(x^2)$ supports a unique $\delta$-structure with $\phi(x) = px$; since $\phi(a+bx) = a+pbx$ for $a,b \in \mathbf{Z}_p$, the map $\phi$ is injective. 
\end{remark}

\begin{remark}
Lemma~\ref{PerfLambdaTF} (1) can also be conceptually understood using the fact that $\phi$ is a derived Frobenius lift, as in Remark~\ref{DerivedFrobLift}.  To see this, given a $\delta$-ring $(A,\delta)$, write $B$ for the simplicial commutative $\mathbf{F}_p$-algebra $A \otimes_{\mathbf{Z}}^L \mathbf{F}_p$ obtained by base change from $A$. Then the map $\phi_B:B \to B$ induced by base change from $\phi:A \to A$ coincides with the Frobenius endomorphism of $B$. Now it is known that the Frobenius is always zero on $\pi_i(B)$ for $i > 0$, cf.~e.g.~\cite[Proposition 11.6]{BhattScholzeWitt}. In particular, it follows that $\pi_1(\phi_B):\pi_1(B) \to \pi_1(B)$ is zero. But the standard resolution of $\mathbf{F}_p$ over $\mathbf{Z}$ identifies $\pi_1(B)$ with $A[p]$ in a $\phi$-compatible manner, so it follows that $\phi(A[p]) = 0$, as wanted. 
\end{remark}

\begin{corollary}[Perfect $\delta$-rings]
\label{PerfectLambda}
The following categories are equivalent:
\begin{enumerate}
\item The category $\mathcal{C}_1$ of perfect $p$-complete $\delta$-rings.
\item The category $\mathcal{C}_2$ of $p$-adically complete and $p$-torsionfree rings $A$ with $A/p$ being perfect.
\item The category $\mathcal{C}_3$ of perfect $\mathbf{F}_p$-algebras. 
\end{enumerate}
The functor relating (1) and (2) is the forgetful functor; in particular, any ring homomorphism between two perfect $p$-complete $\delta$-rings is automatically a $\delta$-map. The functors relating (2) and (3) are $A \mapsto A/p$ and $R \mapsto W(R)$; in particular, there is only one $\delta$-structure on $W(R)$ for $R$ perfect of characteristic $p$.
\end{corollary}

\begin{proof}
\cref{PerfLambdaTF} ensures that forgetting the $\delta$-structure gives a functor $F_{12}:\mathcal{C}_1 \to \mathcal{C}_2$, which is then tautologically faithful. The functor $F_{23}:\mathcal{C}_2 \to \mathcal{C}_3$ is given by reduction modulo $p$, i.e.~by $A \mapsto A/p$. Finally, the functor $F_{31}:\mathcal{C}_3 \to \mathcal{C}_1$ is given by $R \mapsto W(R)$. It is well-known (e.g., by deformation theory) that the resulting functor $F_{32} := F_{12} \circ F_{31}$ is an inverse to $F_{23}$. As $F_{12}$ is faithful and the equivalence $F_{32}$ factors over $F_{12}$, it formally follows that $F_{12}$ is also an equivalence. 
\end{proof}

\begin{lemma}[Perfect elements have rank $1$]
\label{Rank1pdivisible}
Fix a $\delta$-ring $A$ and some $x \in A$. Then $\delta(x^{p^n}) \in p^n A$ for all $n$. In particular, if $A$ is $p$-adically separated and $y \in A$ admits a $p^n$-th root for all $n \geq 0$, then $\delta(y) = 0$, i.e., $y$ has rank $1$.
\end{lemma}

The conclusion of last part is false if we do not impose some form of $p$-locality on $A$: any endomorphism of the ring $\mathbf{Q}[x^{\frac{1}{p^\infty}}]$ determines a $\delta$-structure, but very few of them make $x$ have rank $1$. An explicit example is given by taking the Frobenius to be the identity.

\begin{proof}
The last assertion is automatic from the first one. For the first one, by reduction to the universal case from \cref{freelambdaring}, we may assume $A$ is $p$-torsionfree. Thus, $\delta(x^{p^n}) \in p^n A$ if and only if $p \delta(x^{p^n}) \in p^{n+1} A$. But showing the latter is equivalent to checking that $\phi(x^{p^n}) \equiv x^{p^{n+1}} \mod p^{n+1} A$. For $n=0$, this is true by definition. In general, given elements $a,b$ of a ring $C$ such that $a \equiv b \mod p^k C$ for some $k \geq 0$, we have $a^p \equiv b^p \mod p^{k+1} C$ using the binomial theorem. Applying this inductively to $\phi(x) \equiv x^p \mod pA$ then gives the desired claim.
\end{proof}

\begin{lemma}[Distinguished elements in perfect $\delta$-rings]
\label{distinguishedPerfect}
Let $A$ be a perfect $p$-complete $\delta$-ring, and fix $d \in A$. Then $d$ is distinguished if and only if the coefficient of $p$ in the Teichm\"uller expansion of $d$ (defined via \cref{PerfectLambda}) is a unit. 
\end{lemma}
\begin{proof}
Say $d = \sum_{i=0}^\infty [a_i] p^i$ is the Teichm\"uller expansion of some $d \in A =W(R)$ for a perfect $\mathbf{F}_p$-algebra $R$. One then has
\[ \delta(d) = \frac{1}{p} \cdot \big(\sum_{i=0}^\infty [a_i^p] p^i - (\sum_{i=0}^\infty [a_i] p^i)^p\big).\]
Reducing modulo $p$, this gives
\[ \delta(d) \equiv a_1^p \mod pA,\]
so, by $p$-completeness of $A$, $d$ is distinguished exactly when $a_1 \in R$ is a unit.
\end{proof}

\begin{lemma}
\label{PrimEltPRop}
 Let $A$ be a $p$-torsionfree and $p$-adically separated $\delta$-ring with $A/p$ reduced. (For example, $A$ could be perfect and $p$-complete.) Fix $d \in A$ that is distinguished. Then 
\begin{enumerate}
\item The element $d\in A$ is a nonzerodivisor.
\item The ring $R = A/d$ has bounded $p^\infty$-torsion; in fact, we have $R[p] = R[p^\infty]$.
\end{enumerate}
\end{lemma}

\begin{proof} For (1), assume $f d = 0$ for some $f \in A$. We must show $f = 0$.  If not, then since $A$ is $p$-torsionfree and $p$-adically separated, we may assume $p \nmid f$ (by dividing $f$ by a suitable power of $p$). Applying $\delta$ to $fd = 0$ gives
\[ f^p \delta(d) + \delta(f) \phi(d) = 0.\]
Multiplying by $\phi(f)$ and using $\phi(fd) = 0$ gives $f^p \phi(f) \delta(d) = 0$, and hence $f^p \phi(f) = 0$ as $\delta(d)$ is a unit. Reducing modulo $p$ yields $f^{2p} = 0 \mod pA$, whence $f = 0 \mod pA$ as $A/pA$ has an injective Frobenius. But we assumed $p \nmid f$, so we get a contradiction.

For (2), it is enough to show that $R[p] = R[p^2]$. Lifting to $A$, we must show the following: given $f,g \in A$ with $p^2 f = gd$, we must have $p \mid g$ and hence $pf \in d A$ since $A$ is $p$-torsionfree.  Applying $\delta$ to the containment $gd \in p^2 A$ gives $\delta(d) g^p + \delta(g) \phi(d)  \in pA$. Multiplying by $\phi(g)$, and using that $\phi(d g) \in p A$, gives $\delta(d) g^p \phi(g) \in p A$. As $\delta(d)$ is a unit, this gives $g^p \phi(g) \in pA$ and hence $g^{2p} \in pA$. Finally, as Frobenius is injective modulo $p$, we conclude that $g \in p A$, as wanted.
\end{proof}

\subsection{Relation to divided power algebras}
\label{ss:DeltaPD}

In a $p$-torsionfree $\mathbf Z_{(p)}$-algebra $A$, we let
\[
\gamma_n(x)=\frac{x^n}{n!}\in A[\tfrac 1p]
\]
be the usual divided powers.

\begin{lemma}
\label{FirstDividedPowerLambda}
Let $A$ be a $p$-torsionfree $\delta$-ring. Fix $z \in A$ with $\gamma_p(z) \in A$. Then $\gamma_n(z) \in A$ for all $n \geq 0$.
\end{lemma}

\begin{proof}
We first explain why $\gamma_{p^2}(z) \in A$. As $\mathrm{val}_p(p^2!) = p+1$, we must check that $\frac{z^{p^2}}{p^{p+1}} \in A$. As $A$ is a $\delta$-ring, we have $\delta(\frac{z^p}{p}) \in A$. Using the definition of $\phi$, we get
\[ \delta(\frac{z^p}{p}) = \frac{1}{p}\Big(\frac{\phi(z)^p}{p} - \frac{z^{p^2}}{p^p}\Big) = \frac{(z^p + p \delta(z))^p}{p^2} - \frac{z^{p^2}}{p^{p+1}} \in A.\]
Now $\frac{z^p + p \delta(z)}{p} \in A$ by assumption, which gives $\frac{(z^p + p\delta(z))^p}{p^2} = p^{p-2} \cdot \Big(\frac{z^p + p\delta(z)}{p}\Big)^p \in A$ as well (as $p\geq 2$), so the above formula for $\delta(\frac{z^p}{p})$ shows that $\frac{z^{p^2}}{p^{p+1}} \in A$.

We now prove the lemma. For the purposes of solely this lemma, call an element $z \in A$ {\em admissible} if $\gamma_p(z) \in A$. We shall prove by induction on $n$ that for any admissible $z \in A$, we have $\gamma_n(z) \in A$. The $n=1$ case is clear. Assume inductively that for all admissible $y \in A$, we have $\gamma_m(y) \in A$ for $m < n$. Fix admissible $z \in A$. We must check $\gamma_n(z) \in A$. The statement is clear by induction if $n$ is not divisible by $p$, so we may assume $n=kp$, whence $k < n$. But one checks (by induction on $k$) that the general identity $\gamma_{kp}(-) = u \cdot \gamma_k(\gamma_p(-))$ for a unit $u \in \mathbf{Z}_{(p)}$. Using this identity for $k=p$, the previous paragraph shows that $\gamma_p(z)$ is admissible. As $k < n$, induction shows that $\gamma_k(\gamma_p(z)) \in A$, so the claim follows.
\end{proof}

\begin{lemma}
\label{PDalgLambda}
The ring $C := \mathbf{Z}_{(p)}\{x,\frac{\phi(x)}{p}\}$ defined by the pushout square
\[ \xymatrix{ \mathbf{Z}_{(p)}\{y\} \ar[r]^-{y \mapsto pz} \ar[d]_-{y \mapsto \phi(x)} & \mathbf{Z}_{(p)}\{z\} \ar[d] \\
		  \mathbf{Z}_{(p)}\{x\} \ar[r] & \mathbf{Z}_{(p)}\{x,z\}/(\phi(x) - pz)_{\delta} =: C }\]
of $\delta$-rings identifies with the pd-envelope $D := D_{(x)}(\mathbf{Z}_{(p)}\{x\})$.
\end{lemma}

\begin{proof}
The left vertical map in the pushout square defining $C$ identifies abstractly with the map $\phi$ on $\mathbf{Z}_{(p)}\{x\}$, and is thus faithfully flat by \cref{freelambdaring}. The right vertical map is then also faithfully flat, so $C$ has no $p$-torsion. Moreover, as the top horizontal map is an isomorphism after inverting $p$, the same is true for the bottom horizontal map. We may thus view $C$ as the smallest $\delta$-subring of $\mathbf{Z}_{(p)}\{x\}[\frac{1}{p}]$ that contains $\mathbf{Z}_{(p)}\{x\}$ and $\frac{\phi(x)}{p}$. Equivalently, as $\frac{\phi(x)}{p} = \frac{x^p}{p} + \delta(x)$, we can also view $C$ as the smallest $\delta$-subring of $\mathbf{Z}_{(p)}\{x\}[\frac{1}{p}]$ containing $\mathbf{Z}_{(p)}\{x\}$ and $\frac{x^p}{p}$.

Now consider the pd-envelope $D$ of $\mathbf{Z}_{(p)}\{x\}$ along $(x)$. As $x$ is a free variable, it is standard that $D$ is $p$-torsionfree, and may thus be viewed as the smallest subring of $\mathbf{Z}_{(p)}\{x\}[\frac{1}{p}]$ that contains $\mathbf{Z}_{(p)}\{x\}$ and $\frac{x^n}{n!}$ for all $n \geq 1$. We shall check $D = C$ as subrings of $\mathbf{Z}_{(p)}\{x\}[\frac{1}{p}]$.

The containment $D \subset C$ is immediate from \cref{FirstDividedPowerLambda}. To show $C \subset D$, since $\frac{x^p}{p} \in D$, it is enough to show that $\phi$ preserves $D$ and that the resulting endomorphism of $D$ gives a $\delta$-structure (or, equivalently, that $\phi$ restricts to a Frobenius lift on $D$). Observe that
\[ \phi(\frac{x^n}{n!}) = \frac{(x^p + p \delta(x))^n}{n!} = \frac{\sum_{i=0}^n \binom{n}{i} x^{pi} p^{n-i} \delta(x)^{n-i}}{n!} = \frac{\sum_{i=0}^n \binom{n}{i} \cdot (pi)! \cdot p^{n-i} \cdot \frac{x^{pi}}{(pi)!} \cdot \delta(x)^{n-i}}{n!}.   \]
To prove $\phi$ preserves $D$, it suffices to prove that the coefficient 
\[ \frac{\binom{n}{i} \cdot (pi)! \cdot p^{n-i}}{n!}\] 
is $p$-integral. But we have 
\[ \frac{\binom{n}{i} \cdot (pi)! \cdot p^{n-i}}{n!} = \frac{(pi)!}{i!} \cdot \frac{p^{n-i}}{(n-i)!}.\]
Both terms lie in $\mathbf Z_{(p)}$, and the first factor is divisible by $p$ for $i>0$, while the second factor is divisible by $p$ if $i<n$ (as it is a divided power of $p$). This proves that $\phi$ preserves $D$, and that 
\[ \phi(\frac{x^n}{n!}) \equiv 0 \mod pD\] 
for $n > 0$. We also have
\[ \Big(\frac{x^n}{n!}\Big)^p = \gamma_p(\frac{x^n}{n!}) \cdot p! = 0 \mod pD\]
for $n > 0$. In particular, the endomorphism $\phi$ of $D$ reduces modulo $p$ to Frobenius on all generators, and hence all elements, of $D$, as wanted.
\end{proof}

\begin{remark}
\label{PDalgLambdaDerived}
The proof of \cref{PDalgLambda} shows also that the simplicial commutative ring obtained by freely adjoining $\frac{\phi(x)}{p}$ to $\mathbf{Z}_{(p)}\{x\}$ is discrete (i.e., that the pushout square in \cref{PDalgLambda} is also a pushout square in the $\infty$-category of simplicial commutative rings) and thus coincides with the ring $\mathbf{Z}_{(p)}\{x,\frac{\phi(x)}{p}\} \simeq D_{(x)}(\mathbf{Z}_{(p)})$ above.
\end{remark}

\begin{lemma}[Divided power envelopes for regular sequences]
\label{PDEnvBCbasic}
Let $B$ be a $p$-torsionfree ring. Fix $f_1,...,f_r \in B$ that give a Koszul-regular sequence on $B/p$. Consider simplicial commutative ring $D'$ defined by the pushout square
\[ \xymatrix{ \mathbf{Z}_{(p)}[x_1,...,x_r] \ar[r]^-{x_i \mapsto f_i} \ar[d] & B \ar[d] \\
D_{(x_1,...,x_r)}(\mathbf{Z}_{(p)}[x_1,...,x_r]) \ar[r] & D' } \]
in the $\infty$-category of simplicial commutative rings. Then $D'$ is discrete, $p$-torsionfree, and identified with the PD-envelope $D_{(f_1,...,f_r)}(B)$. In particular, the latter is also $p$-torsionfree, and its formation from $B$ commutes with base change along maps $B \to B'$ of $p$-torsionfree rings with the property that $f_1,...,f_r$ give a Koszul-regular sequence on $B'/p$.
\end{lemma}

\begin{proof}
We first prove that $D'$ is discrete and $p$-torsionfree. As the left vertical map is an isomorphism after inverting $p$, the same is true for the right one, so $B[1/p] = D'[1/p]$. To show that $D'$ is discrete and $p$-torsionfree, it then suffices to show that $D' \otimes_{\mathbf{Z}}^L \mathbf{F}_p$ is discrete. Now recall the standard calculation that the map
 $\mathbf{F}_p[x_1,..,x_r] \to D_{(x_1,...,x_r)}(\mathbf{F}_{(p)}[x_1,...,x_r])$ expresses the target as a free module over $\mathbf{F}_p[x_1,...,x_r]/(x_1^p,...,x_r^p) = \mathrm{Kos}(\mathbf{F}_p[x_1,...,x_r]; x_1^p,...,x_r^p)$. Applying $- \otimes_{\mathbf{Z}}^L \mathbf{F}_p$ to the above square, we learn that $D' \otimes_{\mathbf{Z}}^L \mathbf{F}_p$ is a free module over $\mathrm{Kos}(B/p;f_1^p,...,f_r^p)$. Our assumption on the $f_i$'s shows that $\mathrm{Kos}(B/p;f_1,...,f_r)$ is concentrated in degree $0$. By induction on $\sum_i n_i$, it then follows that $\mathrm{Kos}(B/p;f_1^{n_1},...,f_r^{n_r})$ is concentrated in degree $0$ for all $r$-tuples $(n_1,...,n_r)$ of positive integers. In particular,  $\mathrm{Kos}(B/p;f_1^p,...,f_r^p)$ is concentrated in degree $0$, and thus the free module $D' \otimes_{\mathbf{Z}}^L \mathbf{F}_p$  is also concentrated in degree $0$. 

In the $p$-torsionfree $B$-algebra $D'$, the images of the $f_i$'s admit divided powers (as the images of the divided powers of the $x_i$'s on the bottom left in the square above). The universal property then gives a map $a:D_{(f_1,...,f_r)}(B) \to D'$ of $B$-algebras. On the other hand, the description of $D'$ as a pushout and functoriality of divided power envelopes also gives a natural map $b:D' \to D_{(f_1,...,f_r)}(B)$ of $B$-algebras. The composition $a \circ b:D' \to D'$ must be the identity: it induces the identity on $D'[1/p] = B[1/p]$ by virtue of being a $B$-algebra map. On the other hand, the composition $b \circ a:D_{(f_1,...,f_r)}(B) \to D_{(f_1,...,f_r)}(B)$ is a $B$-algebra map that sends each divided power $\gamma_n(f_i) \in  D_{(f_1,...,f_r)}(B)$ to the image under the functoriality map $D_{(x_1,...,x_r)}(\mathbf{Z}_{(p)}[x_1,...,x_r]) \to D'  \xrightarrow{b}  D_{(f_1,...,f_r)}(B)$ of $\gamma_n(x_i) \in D_{(x_1,...,x_r)}(\mathbf{Z}_{(p)}[x_1,...,x_r])$. As this image is exactly $\gamma_n(f_i)$, it follows that $b \circ a$ is the identity on all $B$-algebra generators of $ D_{(f_1,...,f_r)}(B)$, and must thus be the identity. Thus, $a$ and $b$ give mutually inverse isomorphisms between $D'$ and $ D_{(f_1,...,f_r)}(B)$.

The last part is clear as the formation of $D'$ from $B$ commutes with all base changes.
\end{proof}

\begin{corollary}
\label{PDenvRegSeqLambda}
Let $A$ be a $p$-torsionfree $\delta$-ring. Fix $f_1,...,f_r \in A$ that define a regular sequence in $A/p$. Consider the simplicial commutative $\delta$-ring $A\{\frac{\phi(f_1)}{p},...,\frac{\phi(f_r)}{p}\}$ obtained by freely adjoining $\frac{\phi(f_i)}{p}$ to $A$ for $1 \leq i \leq r$, i.e., the simplicial commutative $\delta$-ring defined by the derived pushout square
\[ \xymatrix{ \mathbf{Z}_{(p)}\{x_1,....,x_r\} \ar[r]^-{x_i \mapsto py_i} \ar[d]^-{x_i \mapsto \phi(f_i)} & \mathbf{Z}_{(p)}\{y_1,...y_r\} \ar[d] \\
A \ar[r] & A\{\frac{\phi(f_1)}{p},...,\frac{\phi(f_r)}{p}\} }\]
in the $\infty$-category of simplicial commutative $\delta$-rings. Then $A\{\frac{\phi(f_1)}{p},...,\frac{\phi(f_r)}{p}\}$ is discrete, $p$-torsionfree, and identifies with the pd-envelope $D_I(A)$ of $I = (f_1,...,f_r) \subset A$. In particular, $D_I(A)$ is a $\delta$-ring, and is finitely presented (over $A$) as such.
\end{corollary}
\begin{proof} 
We can write the pushout square in the statement as a composition of two pushout squares in the $\infty$-category of simplicial commutative $\delta$-rings:
\[ \xymatrix{ \mathbf{Z}_{(p)}\{x_1,....,x_r\} \ar[r]^-{x_i \mapsto py_i} \ar[d]^-{\phi} &  \mathbf{Z}_{(p)}\{y_1,...y_r\} \ar[d]  \\
 \mathbf{Z}_{(p)}\{x_1,....,x_r\} \ar[r] \ar[d]^{x_i \mapsto f_i} & \mathbf{Z}_{(p)}\{x_1,...,x_r,\frac{\phi(x_1)}{p},...,\frac{\phi(x_r)}{p}\} \ar[d] \\
A \ar[r] & A\{\frac{\phi(f_1)}{p},...,\frac{\phi(f_r)}{p}\}, }\]
where the second and third vertex of the second column are defined by virtue of the pushout property. Using Lemma~\ref{PDalgLambda} and Remark~\ref{PDalgLambda}, the middle row can be identified with the map  $\mathbf{Z}_{(p)}\{x_1,....,x_r\}  \to D_{(x_1,...,x_r)}(\mathbf{Z}_{(p)}\{x_1,....,x_r\} )$ that formally adjoins divided powers of the $x_i$'s. Lemma~\ref{PDEnvBCbasic} (and the fact that the formation of PD-envelopes commutes with flat base change along $\mathbf{Z}_{(p)}[x_1,...,x_r] \to \mathbf{Z}_{(p)}\{x_1,...,x_r\}$) then imply the bottom right vertex of the above square is indeed discrete, $p$-torsionfree, and identified with $D_I(A)$, as wanted.
\end{proof}

\begin{remark}
\label{PDEnvFDescent}
\cref{PDenvRegSeqLambda} shows that the pd-envelope $D_I(A)$ is  the $\phi$-pullback of $A\{\frac{f_1}{p},...,\frac{f_r}{p}\}$. This last ring only depends on the ideal $I = (f_1,..,f_r)$, and not on the specific sequence chosen: it can be characterized as the universal $p$-torsionfree $\delta$-$A$-algebra $B$ where $p \mid I$. In particular, the pd-envelopes that arise in this fashion have a canonical $\phi$-descent.
\end{remark}

\begin{warning}
\cref{PDenvRegSeqLambda} shows that the subring $A\{\frac{\phi(f_1)}{p},...,\frac{\phi(f_r)}{p}\} \subset A[\frac{1}{p}]$, which also equals $A\{\frac{f_1^p}{p},...,\frac{f_r^p}{p}\}$,  is independent of the $\delta$-structure on $A$. The Frobenius appearing in the preceding sentence is crucial, and the claim does not hold without the $\phi$-twist. For example, consider the ring $A = \mathbf{Z}_{(p)}[x]$ with $f = x$. If we use the Frobenius lift $x \mapsto x^p$ to define a $\delta$-structure $\delta_1$, then $R_1 := A\{\frac{x}{p}\} \subset A[\frac{1}{p}]$ contains $\delta_1(\frac{x}{p}) = \frac{x^p}{p^{p+1}} \cdot (p^{p-1} - 1)$. On the other hand, if we use the Frobenius lift $x \mapsto x^p + p$ to define a $\delta$-structure $\delta_2$, then $R_2 := A\{\frac{x}{p}\} \subset A[\frac{1}{p}]$ contains $\delta_2(\frac{x}{p}) = \frac{x^p}{p^{p+1}} \cdot (p^{p-1} - 1) + \frac{1}{p}$. Any subring of $A[\frac{1}{p}]$ that contains both $\delta_1(\frac{x}{p})$ and $\delta_2(\frac{x}{p})$ must then also contain $\frac{1}{p}$. On the other hand, it is easy to see using \cref{PDenvRegSeqLambda} that neither $R_1$ nor $R_2$ contain $\frac{1}{p}$. In particular, we have $R_1 \not\subset R_2$ and $R_2 \not\subset R_1$.
\end{warning}

\subsection{Flatness for certain pd-envelopes and prismatic-envelopes}
\label{ss:PDEnvelopeFlat}

For future reference, we record some well-known facts about divided power envelopes as well as their consequences for prismatic envelopms. We formulate these in the generality of simplicial commutative rings as it will be convenient for applications. 

%Recall that a map $A \to B$ of simplicial commutative rings is flat exactly when $\pi_0(B)$ is flat over $\pi_0(A)$ and $\pi_i(A) \otimes_{\pi_0(A)} \pi_0(B) \to \pi_i(B)$ is an isomorphism for all $i$; this condition can be detected after derived base change along $A \to \pi_0(A)$.

Fix a simplicial commutative ring $A$ and a finitely generated ideal $I = (f_1,...,f_n) \subset \pi_0(A)$. Let $B$ be a derived $I$-complete simplicial commutative $A$-algebra. Then it follows from the definitions that $B$ is $I$-completely flat over $A$ if and only if the map $\mathrm{Kos}(A; f_1,..,f_n) \to \mathrm{Kos}(B; f_1,...,f_n)$ is a flat map of simplicial commutative rings. Here the Koszul complexes are defined by
\[
\mathrm{Kos}(A;f_1,\ldots,f_n) = A\otimes^L_{\mathbf Z[f_1,\ldots,f_n]} \mathbf Z[f_1,\ldots,f_n]/(f_1,\ldots,f_n)\ .
\]
We shall need a variant.

\begin{definition}
\label{CompleteRegularity}
Fix a simplicial commutative ring $A$ and a finitely generated ideal $I = (f_1,...,f_n) \subset \pi_0(A)$. Let $B$ be a derived $I$-complete simplicial commutative $A$-algebra. A sequence $x_1,...,x_r \in \pi_0(B)$ is {\em $I$-completely regular relative to $A$} if the map $\mathrm{Kos}(A; f_1,..,f_n) \to \mathrm{Kos}(B; f_1,...,f_n, x_1,...,x_r)$ is a flat map of simplicial commutative rings.
\end{definition}

In the above setup, it is immediate that these notions are stable under $I$-complete base changes on $A$. Moreover, the condition of $I$-complete regularity on a sequence $x_1,...,x_r$ is insensitive to perturbing each $x_i$ by an element of $(f_1,...,f_r) \pi_0(B)$.  Using notion, we obtain a relative variant of a slight strengthening of Lemma~\ref{PDEnvBCbasic} with a similar proof:

\begin{lemma}
\label{PDEnvelopeFlat}
Let $A$ be a simplicial commutative ring. Let $B$ be a $p$-complete simplicial commutative $A$-algebra equipped with a sequence $x_1,...,x_r \in \pi_0(B)$ such that the sequence $x_1^p,...,x_r^p \in \pi_0(B)$ is $p$-completely regular relative to $A$. Let $D$ be the simplicial commutative ring obtained defined by the pushout square 
\[ \xymatrix{ \mathbf{Z}_p[y_1,...,y_r]^{\wedge}_p \ar[r] \ar[d]^{\mathrm{can}} & B \ar[d] \\
D_{(y_1,...,y_r)}(\mathbf{Z}_p[y_1,...,y_r])^{\wedge}_p \ar[r] & D }\]
in the $\infty$-category of derived $p$-complete simplicial commutative rings, where the top horizontal map is any map of simplicial commutative rings sending $y_i$ to $x_i$ on $\pi_0$; there is a unique such map up to homotopy.   Then $A \to D$ is $p$-completely flat. 
\end{lemma}

The condition on the $x_i$'s is implied if we assume that $x_1,...,x_r$ is itself $p$-completely regular relative to $A$ by filtering the Koszul complex on $x_1^p,...,x_r^p$ in terms of that on $x_1,...,x_r$ (as in the proof of Lemma~\ref{PDEnvBCbasic}).

\begin{proof}
To show that $A \to D$ is $p$-completely flat, it suffices to show that the base change $\pi_0(A/p) \to \pi_0(A/p) \otimes_A^L D$ is flat. As the formation of $D$ from $B$ commutes with base change, and because our hypotheses are stable under base change, we may base change along $A \to \pi_0(A/p)$ to reduce to the case where $A$ is a discrete $\mathbf{F}_p$-algebra. In this case, $B$ is a simplicial commutative $A$-algebra, and $x_1,...,x_r \in \pi_0(B)$ is a sequence such that $\mathrm{Kos}(B; x_1^p,...,x_r^p)$ is $A$-flat.  As we are working over $\mathbf{F}_p$, one learns from the square defining $D$ that $D$ is a free $\mathrm{Kos}(B; x_1^p,...,x_r^p)$-module: this reduces to the analogous statement for the mod $p$ reduction of the left vertical map, which is classical. Since we already assume that $\mathrm{Kos}(B; x_1^p,...,x_r^p)$ is $A$-flat, the claim follows.
\end{proof}

\begin{corollary}
\label{PDFlatp}
Let $A$ be a $p$-complete simplicial commutative $\delta$-ring. Let $B$ be a $p$-complete simplicial commutative $\delta$-$A$-algebra. Fix a sequence $x_1,...,x_r \in \pi_0(B)$ that is $p$-completely regular relative to $A$. Consider the $p$-complete simplicial commutative $\delta$-$A$-algebra $C := B\{\frac{x_i}{p}\}^\wedge$ obtained by freely adjoining $\frac{x_i}{p}$ to $B$, i.e., $C$ is described by a pushout diagram
\[ \xymatrix{ \mathbf{Z}_p\{x_1,...,x_r\}^{\wedge}_p \ar[rr]^-{x_i \mapsto x_i} \ar[d]^{x_i \mapsto pz_i} && B \ar[d] \\
  \mathbf{Z}_p\{z_1,...,z_r\}^{\wedge} \ar[rr] && C }\]
in the $\infty$-category of simplicial commutative $\delta$-rings. Then $C$ is $p$-completely flat over $A$.
\end{corollary}

\begin{proof}
Consider the commutative diagram
\[ \xymatrix{ A \ar[r] \ar@{=}[d] & A \{x_1,...,x_r\}^\wedge \ar[r] \ar[d]^-{\psi} & B \ar[r] \ar[d]^-{\psi_B} & B\{\frac{x_1}{p},...,\frac{x_r}{p}\}^\wedge =: C \ar[d]^-{\psi_C} \\
		  A \ar[r] & A\{y_1,...,y_r\}^\wedge \ar[r] & B' \ar[r] & B'\{\frac{\phi(y_1)}{p},...,\frac{\phi(y_r)}{p}\}^\wedge =: C' }\]
Here the second vertical map $\psi$ is the unique $\delta$-$A$-algebra map defined by $x_i \mapsto \phi(y_i)$, and $B'$ and $C'$ are simplicial commutative $\delta$-rings defined by requiring the middle and right squares above to be derived pushout squares in $p$-complete simplicial commutative rings. The map $\psi$ is $p$-completely faithfully flat, and hence the same holds for $\psi_B$ and $\psi_C$. It is thus enough to show that $C'$ is $p$-completely flat over $A$.  Now the images $\phi(y_1),...,\phi(y_r) \in \pi_0(B')$  of $x_1,...,x_r \in \pi_0(B)$ under $\psi_B$ form a $p$-completely regular sequence on $B'$ relative to $A$ by the assumption on $x_1,...,x_r \in \pi_0(B)$ and the $p$-complete flatness of $\psi_B$. When checking $p$-complete regularity, we may perturb the sequence by multiples of $p$, so we learn that the sequence $y_1^p,....,y_r^p \in \pi_0(B')$ is also $p$-completely regular relative to $A$. The $p$-complete flatness of $C'$ over $A$ now follows from \cref{PDEnvelopeFlat} and \cref{PDenvRegSeqLambda}.
\end{proof}

\newpage

\section{Prisms}
\label{sec:Prisms}

In this section, we define prisms (\cref{DefPrismCat}) and prove some basic properties; notably, we establish the equivalence between perfectoid rings and perfect prisms (\cref{PrismPerfection}), and prove a flatness result for certain ``prismatic envelopes'' (\cref{PrismaticEnvSmooth}) that plays an important role later.

Before defining prisms, let us record the following criterion for when an ideal in a $\delta$-ring can be generated by a distinguished element.

\begin{lemma}
\label{prismcrit}
Fix a pair $(A,I)$ where $A$ is a $\delta$-ring and $I \subset A$ is an ideal that is locally principal. Assume that $p$ and $I$ lie in $\rad(A)$. The following are equivalent:
\begin{enumerate}
\item We have $p \in I^p + \phi(I)A$.
\item We have $p \in I + \phi(I)A$.
\item There exists a faithfully flat map $A \to A'$ of $\delta$-rings that is an ind-(Zariski localization)\footnote{By definition, an $A$-algebra $B$ is called an ind-(Zariski localization) if it can be presented as a filtered colimit of maps of $A$-algebras of the form $B_j := \prod_{i=1}^{n_j} A_{f_i}$ for $f_1,...,f_{n_j} \in A$. Note that such an $A$-algebra $B$ is faithfully flat over $A$ exactly when each tuple $(f_1,...,f_{n_j})$ generates the unit ideal in some (or equivalently any) presentation of $B$ as such a filtered colimit.}  such that $IA'$ is generated by a distinguished element $d$ and $d,p \in \rad(A')$.
\end{enumerate}
\end{lemma}
\begin{proof}
We trivially have $(1) \Rightarrow (2)$. 

For $(2) \Rightarrow (3)$:
choose a sequence $g_1,...,g_n \in A$ of elements generating the unit ideal of $A$ such that $IA[1/g_i]$ is principal. Write $A' = \prod_{i=1}^n \widetilde{A[1/g_i]}$ where $\widetilde{A[1/g_i]}$ denotes the Zariski localization of $A[1/g_i]$ along $V(p,I)$, so $p$ and $IA'$ lie in $\rad(A')$. By \cref{ExtendLocalizePLocal}, the ring $A'$ admits a unique $\delta$-$A$-algebra structure. As $p$ and $I$ lie in $\rad(A)$, it trivially follows that $A \to A'$ is faithfully flat. By construction, $IA' = (d)$ for some $d\in A'$, and this element is distinguished by \cref{distinguishedIntersect}.

For $(3) \Rightarrow (1)$: under the hypothesis of $(3)$, we must check that $p = 0$ in $A/(I^p + \phi(I)A)$. This can be checked after base change to $A'$, so it is enough to show that $p=0$ in $A'/(d^p, \phi(d))$ where $d\in IA'$ is a distinguished generator. But this follows immediately: $\delta(d)$ is a unit by the distinguishedness of $d$, and we have $\phi(d) = d^p + p\delta(d)$. 
\end{proof}

\begin{definition}[The category of prisms]
\label{DefPrismCat}
Fix a pair $(A,I)$ comprising a $\delta$-ring $A$ and an ideal $I \subset A$; the collection of all such pairs forms a category called the category of {\em $\delta$-pairs}.
\begin{enumerate}
\item The pair $(A,I)$ is a {\em prism} if $I \subset A$ defines a Cartier divisor on $\mathrm{Spec}(A)$ such that $A$ is derived $(p,I)$-complete, and $p \in I + \phi(I)A$. The category of prisms is the corresponding full subcategory of all $\delta$-pairs.
\item A prism $(A,I)$ is called
\begin{itemize}[label={-}]
\item {\em perfect} if $A$ is a perfect $\delta$-ring, i.e., $\phi:A \to A$ is an isomorphism. 
\item {\em bounded} if $A/I$ has bounded $p^\infty$-torsion.
\item {\em orientable} if the ideal $I$ is principal, and the choice of a generator of $I$ is called an {\em orientation}. 
\item {\em crystalline} if $I = (p)$; any such prism is bounded and orientable. 
\end{itemize}

\item A map $(A,I) \to (B,J)$ of prisms is {\em (faithfully) flat} if the map $A \to B$ is $(p,I)$-completely (faithfully) flat.
\end{enumerate}
\end{definition}

\begin{example}
For any $p$-torsionfree and $p$-complete $\delta$-ring $A$, the pair $(A,(p))$ is a crystalline prism, and conversely any crystalline prism is of this form.
\end{example}

\begin{example}
\label{UnivOrientedPrism}
Let $A_0 = \mathbf{Z}_{(p)}\{d, \delta(d)^{-1}\}$ be the displayed localization of the free $\delta$-ring on a variable $d$. Let $A$ be the $(p,d)$-completion of $A_0$, and let $I = (d) \subset A$. Then the pair $(A,I)$ is a bounded prism, and the element $d \in I$ is a distinguished generator (and thus defines an orientation). In fact, this pair is the universal orient\emph{ed} prism (but not the universal orient\emph{able} prism).
For future reference, we note the following: in this prism, the sequence $p,d$ is regular and the Frobenius map $\phi:A/p \to A/p$ is $d$-completely flat.
\end{example}
%
%\begin{remark}
%\label{rmk:PrismDefp}
%Let $(A,I)$ be a prism. Then we have natural $\delta$-map 
%\[s:A \xrightarrow{\phi} A \xrightarrow{w} W(A) \xrightarrow{can} W(A/I),\]
%where the first map is the Frobenius on $A$, the second map is the one from \cref{ThetaWitt}, the third map is the obvious one. The composite takes $I \subset A$ into $(p) \subset W(A/I)$ by effectively the same reasoning used in \cref{rmk:DistEltDefp}: as the canonical map $A \xrightarrow{w} W(A) \xrightarrow{can} W(A/I)$ takes $I$ into $VW(A/I)$, the map $s$, being a $\delta$-map, carries $I$ into $F(VW(A/I)) = (p)$. The kernel of the map $s$ is topologically nilpotent, so this gives a prismatic variant of Remark~\ref{rmk:DistEltDefp}.
%\end{remark}

An important property of prisms is the rigidity of the ideal $I$.

\begin{lemma}[Rigidity of maps]
\label{PrismMapTaut}
If $(A,I) \to (B,J)$ is a map of prisms, then the natural map induces an isomorphism $I \otimes_A B \cong J$. In particular, $IB = J$. 

Conversely, if $A \to B$ is a map of $\delta$-rings with $B$ being derived $(p,I)$-complete, then $(B,IB)$ is a prism exactly when $B[I] = 0$.
\end{lemma}

\begin{proof} We use \cref{prismcrit} to choose faithfully flat maps $A\to A'$ and $B\to B'$ such that $IA' = (d)$ and $JB' = (e)$ are principal and $d,p\in \rad(A')$, $p,e\in \rad(B')$. Note that necessarily $d$ and $e$ are nonzerodivisors, as $I$ and $J$ define Cartier divisors. We may also assume that the map $A\to B$ extends to a map $A'\to B'$ (replacing $B'$ by a localization of $A'\otimes_A B'$ if necessary). To see that $I\otimes_A B\to J$ is an isomorphism, it suffices to see that the similar claim holds for the map of $\delta$-pairs $(A',(d))\to (B',(e))$, by faithfully flat descent. But now $d=ef$ for some $f\in B'$ and $p,e\in \rad(B')$ by assumption, so \cref{distinguishedDivide} says that $f$ is a unit, and thus $(d)=(e)$ in $B'$.

For the second, note that $B[I] = 0$ if and only if $I\otimes_A B \simeq IB$ via the natural map. It then follows from the first part that if $(B,IB)$ is a prism, then $B[I] = 0$. Conversely, if $I \otimes_A B \simeq IB$, then $IB \subset B$ is an invertible $B$-module, so it defines a Cartier divisor on $\mathrm{Spec}(B)$. Also, the containment $p \in IB + \phi(I)B$ comes from the corresponding containment on $A$ and $B$ is derived $(p,I)$-complete by hypothesis, so $(B,IB)$ is indeed a prism.
\end{proof}

One can get close to proving that the ideal $I$ is necessarily principal:

\begin{lemma}
\label{PrismFaceFrob}
Let $(A,I)$ be a prism. Then the ideal $\phi(I)A \subset A$ is principal and any generator is a distinguished element. Moreover, the invertible $A$-modules $\phi^\ast(I)=I\otimes_{A,\phi} A$ and $I^p$ are trivial, i.e.~isomorphic to $A$.
\end{lemma}

\begin{proof}
Once we know one generator of $\phi(I)A$ is distinguished, it follows from Lemma~\ref{distinguishedDivide} that all generators are distinguished. So for the first part it suffices to show that $\phi(I)A $ is generated by a distinguished element. By Lemma~\ref{prismcrit}, we can write $p = a + b$ where $a \in I^p$ and $b \in \phi(I)A$. We claim that $b$ generates $\phi(I)A$, i.e., the map $A \to \phi(I)A$ defined by $1 \mapsto b$ is surjective. Choose a faithfully flat map $A \to B$ as in Proposition~\ref{prismcrit} (3). It is enough to show that the map $B \to \phi(I)B$ induced by $1 \mapsto b$ is surjective and that $b$ is distinguished. We have $IB = (d)$ for a distinguished element $d\in B$. Since $a \in I^p$ and $b \in \phi(I)$, we can write $a = xd^p$ and $b = y\phi(d)$ for suitable $x,y \in B$. Our task is to show that $y$ is a unit. As $d,p \in \rad(B)$, it suffices to show that $y$ is a unit modulo $(d,p)$ or equivalently that $B/(d,p,y) = 0$. If not, by localizing along $\mathrm{Spec}(B/(d,p,y)) \subset \mathrm{Spec}(B)$, we may assume that $d,p,y \in \rad(B)$. The equation $p = a + b = xd^p + y\phi(d)$ simplifies to show
\[ p(1 - y\delta(d)) = d^p(x+y) = d \cdot (d^{p-1} (x+y))\]
Now $1 - y\delta(d)$ is a unit as $y \in \rad(B)$, so the left side is distinguished by Lemma~\ref{distinguishedUnit}.  Lemma~\ref{distinguishedDivide} then implies that $d^{p-1} (x+y)$ is a unit, whence $d$ is a unit, which contradicts $d \in \rad(B)$.

Note that $\phi^\ast(I)$ and $I^p$ become isomorphic over $A/p$ as in characteristic $p$ Frobenius pullback of a line bundle is its $p$-th power. As $p\in \rad(A)$, it follows that $\phi^\ast(I)$ and $I^p$ are isomorphic, so it suffices to prove that $\phi^\ast(I)$ is isomorphic to $A$. There is a natural surjection $\phi^\ast(I)\to \phi(I)A\subset A$, and we proved that $\phi(I)A$ is principal. Let $f\in \phi^\ast(I)$ be any element mapping to a generator of $\phi(I)A$. We claim that the induced map $A\to \phi^\ast(I), 1\mapsto f$ is an isomorphism. As it is a map of invertible modules, it suffices to prove that it induces a surjection on fibres at all closed points of $\Spec(A)$ (all of which lie in $\Spec(A/p)$ as $p\in \rad(A)$). In particular, letting $B$ be the $p$-adic completion of $A_\perf$, it suffices to see that the map $B\to \phi^\ast(I)\otimes_A B, 1\mapsto f$ is surjective. But now the natural map $\phi^\ast(I)\otimes_A B\to \phi(I)B$ is an isomorphism (as any distinguished element in a perfect $\delta$-algebra is a nonzerodivisor by \cref{PrimEltPRop}) and evidently $B\to \phi(I)B, 1\mapsto f$ is surjective as we chose $f$ to be a generator of $\phi(I)A$.
\end{proof}

\begin{lemma}[Properties of bounded prisms]
\label{BoundedPrismProp}
Let $(A,I)$ be a bounded prism. 
\begin{enumerate}
\item The ring $A$ is classically $(p,I)$-complete.
\item Fix a $(p,I)$-completely flat $A$-complex $M \in D(A)$. Then $M$ is discrete and classically $(p,I)$-complete. For any $n \geq 0$, we have $M[I^n] = 0$ and $M/I^nM$ has bounded $p^\infty$-torsion.
\item The category of (faithfully) flat maps $(A,I) \to (B,J)$ of prisms identifies with the category of $(p,I)$-completely (faithfully) flat $\delta$-$A$-algebras $B$ by the functor sending such an $A$-algebra $B$ to $(B,IB)$. 
\item (Bounded prisms are locally orientable) There exists a $(p,I)$-completely faithfully flat map $A \to B$ of $\delta$-rings such that $IB = (d)$ for a distinguished element $d\in B$ that is a nonzerodivisor. In fact, we may choose $A \to B$ to be the derived $(p,I)$-completion of an ind-Zariski localization of $A$. In particular, $(A,I) \to (B,(d))$ is a faithfully flat map of bounded prisms.
\end{enumerate}
\end{lemma}

Observe that, by $(p,I)$-completely faithfully flat descent, the functor that carries $(A,I)$ to the full subcategory of $\mathcal{D}(A)$ spanned by derived $(p,I)$-complete $(p,I)$-completely flat $A$-complexes is a sheaf on the category of all prisms. Part (2) shows that the restriction of this functor to the category of bounded prisms coincides with the assignment carrying $(A,I)$ to derived $(p,I)$-complete $(p,I)$-completely flat $A$-modules; in particular, the latter assignment is functorial under derived $(p,I)$-completed derived base changes, and forms a sheaf.

\begin{proof}
For (1), we must show that $A \simeq \lim_k A/(I^k, p^k)$. Given finitely many invertible $A$-modules $J_1,...,J_m$ equipped with maps $a_i:J_i \to A$, we write $\mathrm{Kos}(A;J_1,...,J_m)$ for the Koszul complex attached to the map $\oplus_{i=1}^m J_i \xrightarrow{\sum a_i} A$ (see \cite[Tag 0622]{Stacks}; note that if $J_i = A$ for all $i$, then this agrees with $\mathrm{Kos}(A; a_1(1),...,a_m(1))$. We now observe that one has the following sequence of isomorphisms
\begin{align*}
A &\simeq R\lim_n R\lim_m \mathrm{Kos}(A; I^n, p^m) \\
&\simeq R\lim_n R\lim_m \mathrm{Kos}(A/I^n; p^m) \\
&\simeq R\lim_n R\lim_m A/(I^n, p^m) \\
&\simeq \lim_k A/(I^k,p^k), 
\end{align*}
where the first equality comes from derived $(p,I)$-completeness of $A$, the second from the fact that $I^n \subset A$ is locally generated by a nonzerodivisor, the third by virtue of $A/I^n$ having bounded $p^\infty$-torsion (by devissage and the fact that $A/I$ has this property by assumption), and the last by a simple cofinality argument.

For (2), note that $M \otimes_A^L A/I^n$ is a $p$-completely flat $A/I^n$-complex for any $n \geq 0$. Since $A/I$ has bounded $p^\infty$-torsion, the same holds true for $A/I^n$ for any $n \geq 0$. But then \cite[Lemma 4.7]{BMS2} implies that $M \otimes_A^L A/I^n \simeq M/I^nM$ is a discrete $A/I^n$-module with bounded $p^\infty$-torsion; in particular, $M[I^n] = 0$ for all $n \geq 0$. The discreteness of $M$ now follows from the derived $I$-completeness of $M$ and the discreteness of $M \otimes_A^L A/I^n$. Finally, to prove that $M$ is classically $(p,I)$-complete, we may proceed as in (1).

For (3), note that if $(A,I) \to (B,J)$ is a (faithfully) flat map of prisms, then $B$ is a $(p,I)$-completely flat $A$-algebra by hypothesis, and $J =IB$ by Lemma~\ref{PrismMapTaut}. Conversely, say $B$ is a $(p,I)$-completely (faithfully) flat $A$-algebra. By \cref{PrismMapTaut}, it suffices to see that $B[I]=0$, which follows from part (2).

For (4), take $B$ to be the derived $(p,I)$-completion of the ring $A'$ from Lemma~\ref{prismcrit} (3), so $IB = (d)$ for a distinguished element $d\in B$. By (2), the ring $B$ is discrete, $B/IB$ has bounded $p^\infty$-torsion, and $B[I] = 0$, so $(B,IB)$ defines a prism by \cref{PrismMapTaut}.
\end{proof}

\begin{lemma}[Properties of perfect prisms]
\label{PerfectPrisms}
Let $(A,I)$ be a perfect prism.
\begin{enumerate}
\item The ideal $I$ is principal and any generator is a distinguished element.
\item The prism $(A,I)$ is bounded. In particular, the ring $A$ is classically $(p,I)$-complete.
\end{enumerate}
\end{lemma}

\begin{proof}
(1) follows from Lemma~\ref{PrismFaceFrob} as $\phi:A \to A$ is an isomorphism by hypothesis. For (2), choose a generator $d\in I$. By Lemma~\ref{BoundedPrismProp}, it is enough to check that $A/d$ has bounded $p^\infty$-torsion. Now $A$ is $p$-torsionfree by Lemma~\ref{PerfLambdaTF}. As $A/p$ is perfect, it follows from Lemma~\ref{PrimEltPRop} that $A/d$ has bounded $p^\infty$-torsion.
\end{proof}

\begin{lemma}[Perfection of a prism]
\label{PrismPerfection}
Let $(A,I)$ be a prism. Let $A_\perf = \colim_\phi A$ be the perfection of $A$. Then $IA_\perf = (d)$ is generated by a distinguished element, $d$ and $p$ are nonzerodivisors in $A_\perf$, and $A_\perf/d[p^\infty]=A_\perf/d[p]$. In particular, the derived $(p,I)$-completion $A_\infty$ of $A_\perf$ agrees with its classical $(p,I)$-adic completion, and $(A,I)_\perf := (A_\infty,IA_\infty)$ is the universal perfect prism under $(A,I)$.
\end{lemma}

\begin{proof} As $A\to A_\perf$ factors over $\phi: A\to A$, it follows that from \cref{PrismFaceFrob} that $IA_\perf$ is generated by a distinguished element $d\in A_\perf$. As $A_\perf$ is perfect, $p$ is a nonzerodivisor, and by \cref{PrimEltPRop}, $d$ is a nonzerodivisor and $A/d[p^\infty] = A/d[p]$. Passing to $p$-adic completions first, the derived and classical completions agree by $p$-torsionfreeness, and $d$ is still a nonzerodivisor by \cref{PrimEltPRop}. Note that as $A_\perf$ is $p$-adically complete, any quotient $A_\perf/d^n$ is derived $p$-complete, and it also has bounded $p^\infty$-torsion, so that it is in fact classically $p$-complete. This implies that we may now also pass to the $d$-adic completion, where again the derived and classical completions agree, and it is classically $(p,d)$-adically complete. The universal property of $(A_\infty,IA_\infty)$ is clear.
\end{proof}

\begin{theorem}[Perfectoid rings $=$ perfect prisms]
\label{PerfdPrism}
The following two categories are equivalent:
\begin{itemize}
\item The category of perfectoid rings $R$ (in the sense of \cite[\S 3]{BMS1}).
\item The category of perfect prisms $(A,I)$.
\end{itemize}
The functors are $R \mapsto (A_{\inf}(R),\ker(\theta))$ and $(A,I) \mapsto A/I$ respectively. 
\end{theorem}

\begin{proof}
For a perfectoid ring $R$ in the sense of \cite[\S 3]{BMS1}, the ring $A_{\inf}(R)$ is a perfect $\delta$-ring (clear) and $\ker(\theta)$ is generated by a distinguished element that is a nonzerodivisor by \cite[Remark 3.11]{BMS1}. Moreover, $A_{\inf}(R)$ is classically $(p,\ker(\theta))$-complete by construction, and hence also derived $(p,\ker(\theta))$-complete. This gives a functor $F$ in the forward direction.

Conversely, fix a perfect prism $(A,I)$. Lemma~\ref{PerfectPrisms} shows that $I = (d)$ for a distinguished element $d \in A$ that is a nonzerodivisor, and that $A$ is classically $(p,d)$-complete. To get a functor  $G$ in the reverse direction, we shall show that $R = A/d$ is perfectoid. The perfectness of $A$ ensures that the Frobenius on $R/p$ is surjective. As $A$ is perfect, we have $A \cong W(S)$ for a perfect $\mathbf{F}_p$-algebra $S$ by \cref{PerfectLambda}. Via this isomorphism, we have $d = [a_0] + p u$ for a unit $u \in A$ thanks to \cref{distinguishedPerfect}. If we let $\pi \in R$ be the image of $[a_0^{1/p}]$, then $\pi^p \mid p$ in $R$. It remains to show that $R$ is classically $p$-adically complete. Note that $R$ is derived $p$-complete as it is the cokernel of the map $A \xrightarrow{d} A$ of derived $p$-complete modules. As it has bounded torsion by Lemma~\ref{PrimEltPRop}, we conclude that $R$ is also classically $p$-adically complete. This proves that $R$ is perfectoid.

It is immediate from the construction that $GF \simeq \mathrm{id}$. To prove $FG \simeq \mathrm{id}$, we must check that if $(A,I)$ is a perfect prism, then $(A,I) \simeq (A_{\inf}(A/I), \ker(\theta))$. As $A_{\inf}(A/I) \xrightarrow{\theta} A/I$ is the universal pro-infinitesimal thickening of $A/I$, there is a unique map $A_{\inf}(A/I) \to A$ factoring the map down to $A/I$ on either side; in particular, we have $\ker(\theta)A \subset IA$. By \cref{PerfectLambda}, it is automatically a map of $\delta$-rings. As $\ker(\theta) \subset A_{\inf}(A/I)$ and $I \subset A$ are both generated by distinguished elements, it follows from Lemma~\ref{distinguishedDivide} that if $\ker(\theta) = d A_{\inf}(A/I)$, then $I = d A$. Moreover, $d$ is a nonzerodivisor on both rings by Lemma~\ref{PrimEltPRop} and both rings are $d$-adically complete. It remains to simply observe that $A_{\inf}(A/I)/d A_{\inf}(A/I) \to A/d$ is an isomorphism as both sides identify with $A/I$.
\end{proof}

\begin{remark}[Perfectoid covers of regular local rings via prisms]
\label{RegularLocalPrism}
Let $R$ be a complete noetherian regular local ring of dimension $d$ with residue field $k$ of characteristic $p > 0$. We explain how the formalism of prisms helps understand the construction of a faithfully flat map $R \to R_\infty$ with $R_\infty$ perfectoid that appeared in \cite[Example 3.4.6 (3)]{AndreAbhyankar} and \cite[Proposition 5.2]{BhattDSC}. We focus on the essential case of mixed characteristic, i.e., we assume $R$ is $p$-torsionfree.

Choose a Cohen ring $W$ for $k$ with a Frobenius lift $\phi_W$. Thanks to the Cohen structure theorem, we may write $R = A/(f)$, where $A := W\llbracket x_1,...,x_d\rrbracket$ and $f \in \mathfrak{m} - \mathfrak{m}^2$, where $\mathfrak{m} = (p,x_1,...,x_d) \subset A$ is the maximal ideal. We first explain how to endow $A$ with the structure of a $\delta$-ring that makes $f$ a distinguished element (and thus $(A,(f))$ becomes a prism). In the $(d+1)$-dimensional $k$-vector space $\mathfrak{m}/\mathfrak{m}^2$, we can write $f = a_0 p + \sum_{i=1}^d a_i x_i$, with $a_0 \in k$. There are two cases to consider, depending on whether $a_0 = 0$ or not. 

If $a_0 \neq 0$, then constant coefficient $f(0)$ has $p$-adic valuation $1$. Consider the Frobenius lift $\phi_A$ on $A$ extending $\phi_W$ and satisfying $\phi_A(x_i) = x_i^p$ for all $i$. We claim that the resulting $\delta$-structure on $A$ makes $f$ distinguished. Indeed, the ``evaluating at $0$'' map $A \xrightarrow{x_i \mapsto 0} W$ is a $\delta$-map with topologically nilpotent kernel that carries $f$ to a distinguished element, so $f$ must be distinguished. 

If $a_0 = 0$, then $f$ can be regarded a nonzero $W$-linear form in the $x_i$'s modulo $\mathfrak{m}^2$. We can then change variables to assume that $f = x_1-p$. In this case, the same $\delta$-structure used in the previous paragraph makes $f$ distinguished. 

Thus, we have written $R = A/(f)$, where $(A,(f))$ is a prism. As $A$ is a $p$-torsionfree $\delta$-ring whose mod $p$ reduction $A/p = k\llbracket x_1,...,x_d\rrbracket$ is regular, the Frobenius $\phi_A:A \to A$ is $p$-completely faithfully flat by inspection (or by the easy half of Kunz's theorem). Write $(A_\infty,(f))$ for the perfection of $(A,(f))$ (Proposition~\ref{PrismPerfection}), and $R_\infty = A_\infty/(f)$ for the corresponding perfectoid ring (Theorem~\ref{PerfdPrism}). As $(A,(f)) \to (A_\infty, (f))$ is a faithfully flat map of prisms with perfect target, the map $R \to R_\infty$ is then $p$-completely faithfully flat with $R_\infty$ perfectoid. As $R$ is noetherian, one checks that $R \to R_\infty$ is actually faithfully flat \cite[Proposition 5.1]{BhattDSC}, thus providing the requisite perfectoid faithfully flat cover of $R$.

\end{remark}

\begin{corollary}[The site of all prisms]
\label{BoundedPrismSite}
The category opposite to that of all bounded prisms $(A,I)$, endowed with the topology where covers are determined by faithfully flat maps of prisms, forms a site. The functor that carries $(A,I)$ to $A$ (resp. $A/I$) forms a sheaf for this topology with vanishing higher cohomology on any $(A,I)$.
\end{corollary}

\begin{proof}
To check the site axioms, we must show the following: $(a)$ isomorphisms are covers, $(b)$ a composition of covers is a cover, and $(c)$ the pushout of a cover along an arbitrary map is a cover. The first two are automatic. For the last one, using the description from Lemma~\ref{PrismMapTaut}, we are reduced to showing the following: given a diagram $(C,IC) \xleftarrow{c} (A,IA) \xrightarrow{b} (B,IB)$ of maps of bounded prisms with $b$ faithfully flat, the map $b$ admits a pushout along $c$ which is also faithfully flat. For this, we simply take $D$ to be the derived $(p,I)$-completion of $B \otimes^L_A C$. By standard properties, $D$ is $(p,I)$-completely faithfully flat over $C$. By Lemma~\ref{BoundedPrismProp}, it follows that $(C,IC) \to (D,ID)$ is a faithfully flat map of bounded prisms; one checks that it serves as a pushout of $b$ along $c$.

Let us check the sheaf axiom now. Say $(A,I) \to (B,IB)$ is a faithfully flat map of prisms. Let $(B^\bullet,IB^\bullet)$ be the \v{C}ech nerve of this map in the category of bounded prisms, regarded as a cosimplicial bounded prism over $(A,I)$. By the previous paragraph, this means that $B^\bullet$ is the derived $(p,I)$-completion of the derived \v{C}ech nerve of $A \to B$, regarded as a cosimplicial object in $\mathcal{D}(A)$. In particular, $\mathrm{Kos}(A; p^n,I^n) \to \mathrm{Kos}(B^\bullet;p^n, I^n)$ is a limit diagram in $\mathcal{D}(A)$ for all $n$ by faithfully flat descent for maps of simplicial commutative rings. Taking an inverse limit over $n$ then shows that $A \to B^\bullet$ is also a limit diagram in $\mathcal{D}(A)$; this implies the sheaf axiom as well as the vanishing of higher \v{C}ech cohomology. The case of $(A,I)\mapsto A/I$ follows by tensoring the equivalence $A \simeq \lim B^\bullet$ in $\mathcal{D}(A)$ with the perfect complex $A/I$.
\end{proof}

\begin{proposition}[Existence of prismatic envelopes for regular ideals]
\label{qPDEnvRegular}
\label{PrismaticEnvSmooth}
Let $(A,I)$ be a bounded prism, and let $B$ be a $(p,I)$-completely flat $\delta$-$A$-algebra. Let $J \subset B$ be an ideal containing $IB$, so there is an induced map $(A,I) \to (B,J)$ of $\delta$-pairs. Assume that, Zariski locally on $\mathrm{Spf}(B)$, we can write $J = (I,x_1,...,x_r)$ for a sequence $x_1,...,x_r \in B$ that is $(p,I)$-completely regular relative to $A$ (in the sense of Definition~\ref{CompleteRegularity}). 

Then there is a universal map $(B,J) \to (C,IC)$ of $\delta$-pairs to a prism $(C,IC)$ over $(A,I)$. If there is no potential for confusion, we shall write $C := B\{\frac{J}{I}\}^\wedge$.  Moreover, this universal object has the following properties:
\begin{enumerate}
\item $(B\{\frac{J}{I}\}^\wedge,IB\{\frac{J}{I}\}^\wedge)$ is $(p,I)$-completely flat over $(A,I)$.

\item The construction $(B,J) \mapsto (B\{\frac{J}{I}\}^\wedge,IB\{\frac{J}{I}\}^\wedge)$ commutes with base change along arbitrary maps $(A,I) \to (A',IA')$ of bounded prisms.

\item The construction $(B,J) \mapsto (B\{\frac{J}{I}\}^\wedge,IB\{\frac{J}{I}\}^\wedge)$ is compatible with flat localization on $B$: if $B \to B'$ is a $(p,I)$-completely flat map of $\delta$-$A$-algebras, then $B\{\frac{J}{I}\}^{\wedge} \widehat{\otimes}_B^L B' \cong B'\{\frac{J'}{I}\}^{\wedge}$, where $J' = JB' \subset B'$. 
\end{enumerate}

\end{proposition}

For evident reasons, we shall refer to the map $(B,J) \to (B\{\frac{J}{I}\}^\wedge,IB\{\frac{J}{I}\}^\wedge)$ as the {\em prismatic envelope} of the $\delta$-pair $(B,J)$ over the bounded prism $(A,I)$.

\begin{proof}
We first observe that it suffices to prove the proposition when the prism $(A,I)$ is orientable, i.e., when $I$ is principal. Indeed, the base change compatibility in part (2) of the proposition in the orientable case coupled with $(p,I)$-complete faithfully flat descent for derived $(p,I)$-complete $(p,I)$-completely flat $A$-modules (see remark following statement of Lemma~\ref{BoundedPrismProp}) then solve the problem in general. Thus, assume $I=(d)$ is principal with distinguished generator $d \in I$. 

Next, we observe that it suffices to prove the proposition when $J$ globally has the form $J=(I,x_1,...x_r)$ with $x_1,...,x_r \in B$ being $(p,I)$-completely regular relative to $A$. This follows by a similar argument as in the previous paragraph using part (3) instead of part (2), and using the sheaf property for $\mathcal{D}_{(p,I)-comp}(-)$ instead of its restriction to the $(p,I)$-completely flat objects (as $B\{\frac{J}{I}\}^\wedge$ is typically not $(p,I)$-completely flat over $B$). Thus, we may and do assume that $J = (I,x_1,...,x_r)$ has this form. 

Consider the simplicial commutative $\delta$-$A$-algebra $C := B\{\frac{x_1}{d},...,\frac{x_r}{d}\}^\wedge$ obtained by freely adjoining $\frac{x_i}{d}$ to $B$ in the category of derived $(p,I)$-complete simplicial $\delta$-$A$-algebras (i.e., as an appropriate pushout, as in Corollary~\ref{PDFlatp}). We shall check that $C$ is $(p,I)$-completely flat over $A$. Granting this, the proposition follows. Indeed, the $(p,I)$-complete flatness of $C$ forces $C$ to be discrete and $d$-torsionfree by Lemma~\ref{BoundedPrismProp}, so $(C,(d))$ is a bounded prism over $(A,(d))$. Moreover, there is a natural map $(B,J) \to (C,(d))$ over $(A,(d))$ by construction, and its universality is clear: to specify a map from $C$ to any $d$-torsionfree $(p,d)$-complete $\delta$-ring $D$, we must specify a map $B \to D$ that carries all the $x_i$, and hence the ideal $J$, into $d D$.  The compatibility with base change is also clear from the definition of $C$ as the free derived $(p,I)$-complete $\delta$-$B$-algebra on $\{\frac{x_i}{d}\}_{i=1,...,r}$. Similarly, the last assertion of the proposition --- the flat localization property --- is clear either from the universal property, or from the definition of $C$ (and the fact that the sequence $x_1,...,x_r \in B'$ is $(p,I)$-completely regular relative to $A$ thanks to $(p,I)$-complete flatness of $B \to B'$).

We have reduced to checking that $C$ is $(p,d)$-completely flat over $A$. Consider the following commutative diagram of derived $(p,d)$-complete simplicial commutative $\delta$-rings:
\[ \xymatrix{ \mathbf{Z}_p \{z\}^\wedge \ar[d]_\psi^-{z \mapsto \phi(y)} \ar[r]^{z \mapsto d} & A \ar[d] \ar[r] & B \ar[r] \ar[d] & B\{\frac{x_1}{d},...,\frac{x_r}{d}\}^\wedge =: C \ar[d] \\
		   \mathbf{Z}_p \{y\}^\wedge \ar[r] & A' \ar[r] \ar[d] & B' \ar[r] \ar[d] & B'\{\frac{x_1}{\phi(y)},...,\frac{x_r}{\phi(y)}\}^\wedge =: C' \ar[d] \\
		   	& D := A'\{\frac{\phi(y)}{p}\}^\wedge \ar[r] & B'' \ar[r] & B''\{\frac{x_1}{p},...,\frac{x_r}{p}\}^\wedge =: C'' }\]
Here the first row is our input data, the top left vertical map labelled $\psi$ abstractly identifies with the Frobenius on the free $(p,z)$-complete $\delta$-ring $\mathbf{Z}_p\{z\}^{\wedge}$,  all squares are derived pushout squares of $(p,z)$-complete simplicial commutative rings over the top left vertex, and description of the bottom right vertex follows by observing that $\phi(y) = p \cdot u \in \pi_0(D)$ for a unit $u \in \pi_0(D)$ by Lemma~\ref{distinguishedDivide}.

The map $\psi$ is $(p,z)$-completely faithfully flat, and hence the same holds true for all the vertical maps relating the first and second rows, so it suffices to show that $A' \to C'$ is $(p,z)$-completely flat. The map $D \to C''$ is $(p,z)$-completely flat by Corollary~\ref{PDFlatp} (as our hypotheses are stable under base change). So it is enough to show that $(p,z)$-complete flatness, or equivalently $(p,y)$-complete flatness, over $A'$ can be detected after base change along $A' \to D$. But Corollary~\ref{PDenvRegSeqLambda} ensures that that $D$ identifies with the ring obtained by freely adjoining divided powers of $y$ to $A'$. Since $(p,y)$-complete flatness of an $A'$-algebra is defined as flatness after base change to $A' \to \mathrm{Kos}(A';p,y)$, it suffices to observe that there is an $A'$-algebra map $D \to \mathrm{Kos}(A; p,y)$ (which easily follows, for instance, by reduction to the universal case $A' = \mathbf{Z}_p\{y\}$).
\end{proof}

\begin{example}
\label{ex:PrismativEnvSmooth}
Let us describe the typical application of Proposition~\ref{qPDEnvRegular} for the purposes of this paper. Given a $p$-completely smooth $A/I$-algebra $R$, a surjection $B_0 \to R$ of derived $(p,I)$-complete $A$-algebras with $B_0$ being $(p,I)$-completely smooth over $A$, and a $(p,I)$-completely flat map $B_0 \to B$ with $B$ being $\delta$-$A$-algebra, let $J \subset B$ be the derived $(p,I)$-complete ideal generated by the image of $\ker(B_0 \to R)$; equivalently, $J$ is the kernel of $B \to R \widehat{\otimes}^L_{B_0} B$. Then the pair $(B,J)$ satisfies the hypothesis of Proposition~\ref{qPDEnvRegular}. Consequently, one can form the prismatic envelope $(B,J) \to (B\{\frac{J}{I}\}^\wedge,IB\{\frac{J}{I}\}^\wedge)$ satisfying (1), (2) and (3) in Proposition~\ref{qPDEnvRegular}.
\end{example}

\newpage

\section{The prismatic site and formulation of the Hodge-Tate comparison}
\label{sec:PrismaticSite}

Fix a bounded prism $(A,I)$. Unless otherwise specified, all formal schemes over $A$ are assumed to have the $(p,I)$-adic topology, and thus formal schemes over $A/I$ have the $p$-adic topology. Fix a smooth $p$-adic formal $A/I$-scheme $X$. In \S \ref{ss:PrismaticSite}, we introduce the prismatic site $(X/A)_\Prism$ and discuss its relation to the \'etale topology of $X$. A key comparison theorem for prismatic cohomology --- the Hodge-Tate comparison --- is then formulated in \S \ref{ss:HT}. We end in \S \ref{ss:ComputePrismatic} by discussing the \v{C}ech-Alexander approach to computing prismatic cohomology, drawing consequences from the flatness results for prismatic envelopes established earlier. 

\subsection{The prismatic site}
\label{ss:PrismaticSite}

\begin{definition}[The prismatic site and its structure sheaf]
\label{DefPrismaticSite}
Let $X$ be a $p$-adic formal scheme over $A/I$. Let $(X/A)_\Prism$ be the category of maps $(A,I) \to (B,IB)$ of bounded prisms together with a map $\mathrm{Spf}(B/IB) \to X$ over $A/I$; the notion of morphism is the obvious one. We shall often denote such an object by 
\[ (\mathrm{Spf}(B) \gets \mathrm{Spf}(B/IB) \to X) \in (X/A)_\Prism\]
if no confusion arises. A map $(\mathrm{Spf}(C) \gets \mathrm{Spf}(C/IC) \to X) \to (\mathrm{Spf}(B) \gets \mathrm{Spf}(B/IB) \to X)$ in $(X/A)_\Prism$ is a {\em flat cover} if $(B,IB) \to (C,IC)$ is a faithfully flat map of prisms, i.e., $C$ is $(p,IB)$-completely flat over $B$.  The category $(X/A)_\Prism$ with the topology defined by flat covers is called the {\em prismatic site} of $X/A$.  

The assignment $(\mathrm{Spf}(B) \gets \mathrm{Spf}(B/IB) \to X) \mapsto B$ (resp. $B/IB$) defines a presheaf $\mathcal{O}_\Prism$ (resp. $\overline{\mathcal{O}}_\Prism$) of commutative $A$-algebras (resp. $\mathcal{O}(X)$-algebras) on $(X/A)_\Prism$. It follows from Corollary~\ref{BoundedPrismSite} that $(X/A)_\Prism$ with topology as defined above is indeed a site, and that $\mathcal{O}_\Prism$ and $\overline{\mathcal{O}}_\Prism$ are sheaves. We refer to $\mathcal{O}_\Prism$ as the {\em structure sheaf} on $(X/A)_\Prism$. 
\end{definition}

Thus, given a formal scheme $X$ as above, the data of an object of the prismatic site can be summarized by the following diagram:

\[ \xymatrix{ \mathrm{Spf}(B/IB) \ar[d]^{f} \ar[r] & \mathrm{Spf}(B) \ar[dd]^g \\
			X \ar[d] & \\
			\mathrm{Spf}(A/I) \ar[r] & \mathrm{Spf}(A), }\]
where $f$ and $g$ are the maps that must be specified (together with the $\delta$-$A$-algebra $B$).

\begin{remark}
In this paper, we shall almost exclusively use the prismatic site $(X/A)_\Prism$ as defined above under the additional assumption that $X$ is a smooth $p$-adic formal scheme over $A/I$: the main theorems (such as the Hodge-Tate comparison) break down without this assumption. We shall later extend prismatic cohomology to all $p$-adic formal $A/I$-schemes via left Kan extension from the smooth case; though we do not do so here, it is possible to recover the same definition using a {\em derived} prismatic site, i.e.~by working with simplicial rings throughout.
\end{remark}

Unless otherwise specified, assume for the rest of this subsection that $X$ is  a smooth $p$-adic formal scheme over $A/I$.

\begin{remark}[The prismatic topos as a slice topos and its functoriality]
\label{PrismaticSiteSlice}
Let $(\ast/A)_\Prism$ be the category of all bounded prisms over $(A,I)$, made into a site as in Corollary~\ref{BoundedPrismSite}; this also coincides with the site $(\mathrm{Spf}(A/I)/A)_\Prism$ constructed above. The smooth $p$-adic formal scheme $X$ defines a presheaf  $h_X$ on $(\ast/A)_\Prism$ by the formula $h_X(B,IB) = \mathrm{Hom}_{A/I}(\mathrm{Spf}(B/IB), X)$. A variant of the argument used to prove Corollary~\ref{BoundedPrismSite} shows that $h_X$ is a sheaf on $(\ast/A)_\Prism$. Unwinding definitions, the prismatic site $(X/A)_\Prism$ introduced above coincides with the category of pairs $((B,IB) \in (\ast/A)_\Prism,\eta \in h_X(B,IB))$, with topology inherited from $(\ast/A)_\Prism$. In particular, we may view $\Shv((X/A)_\Prism)$ as the slice topos $\Shv((\ast/A)_\Prism)_{/h_X}$. 

The slice topos perspective makes the functoriality of the prismatic topos clear: if $f:X \to Y$ is a map of smooth $p$-adic formal schemes over $A/I$, there is an induced natural transformation $h_X \to h_Y$ of sheaves on $(\ast/A)_\Prism$, which then yields a morphism $\Shv((X/A)_\Prism) \to \Shv((Y/A)_\Prism)$ of slice topoi.
\end{remark}

\begin{construction}[Relating prismatic sheaves to \'etale sheaves]
\label{PrismtoEtale}
Write $f\Sch_{/X}$ for the category of $p$-adic formal schemes $U$ equipped with an adic map $U \to X$; we endow this category with the \'etale topology. There is a natural functor $\mu:(X/A)_\Prism \to f\Sch_{/X}$ given by sending $(B,IB)$ to $\mathrm{Spf}(B/IB) \to X$. This functor is cocontinuous by Lemma~\ref{ExtendEtale}: given a bounded prism $(B,IB)$ and a $p$-completely \'etale map (resp. covering) $B/IB \to \overline{C}$, there is a unique $(p,I)$-completely \'etale map (resp. covering) $B \to C$ of $\delta$-rings lifting the previous map modulo $IB$. Consequently,  this construction defines a morphism
\[ \mu_X:\mathrm{Shv}( (X/A)_\Prism) \to \mathrm{Shv}(f\Sch/X)\]
of topoi by \cite[Tag 00XO]{Stacks}; we shall simply write $\mu$ instead of $\mu_X$ if no confusion arises. There is also a natural map $\Shv(f\Sch_{/X}) \to \Shv(X_{\et})$ of topoi defined by restriction. Composing with $\mu_X$ gives a map
\[ \nu_X:\Shv( (X/A)_\Prism) \to \Shv(X_{\et});\]
again, we write $\nu$ instead of $\nu_X$ if no confusion arises. We shall be interested in the pushforward $\nu_*$ along this map, which is described as follows. For any \'etale map $U \to X$, the category $(U/A)_\Prism$ is naturally a slice of $(X/A)_\Prism$ in a manner that preserves the notion of coverings. Moreover, restriction along this functor carries sheaves to sheaves. Unwinding definitions then gives that 
\[ (\nu_\ast F)(U \to X) = H^0( (U/A)_\Prism, F|_{(U/A)_{\Prism}}).\]
In other words, the functor $\nu$ allows us to localize the prismatic cohomology on $X_{\et}$. 
\end{construction}

\begin{remark}[Change of topology]
In Definition~\ref{DefPrismaticSite}, we used the flat topology on the category of prisms to define coverings. Replacing the flat topology with the Zariski, Nisnevich or \'etale topologies also leads to useful site structures on the same underlying category. Moreover, one has natural comparison maps 
\[ \Shv( (X/A)_\Prism) \to \Shv((X/A)_{\Prism,et}) \to \Shv((X/A)_{\Prism,Nis}) \to \Shv((X/A)_{\Prism,Zar})\]
relating the corresponding topoi. By Corollary~\ref{BoundedPrismSite} as well as the perspective adopted in Remark~\ref{PrismaticSiteSlice}, derived pushforward along every map in the above composition carries $\mathcal{O}_\Prism$ (resp. $\overline{\mathcal{O}_\Prism}$) to itself. In particular, we can use any of these topologies to compute the cohomology of $X$ with coefficients in $\mathcal{O}_\Prism$ or $\mathcal{O}_{\overline{\Prism}}$. More generally, using $(p,I)$-completely faithfully flat descent for $D_{(p,I)-\text{comp}}(-)$, a similar remark applies to cohomology with coefficients in a crystal $M$ in $(p,I)$-complete complexes, i.e., an assignment 
\[ \eta := (\mathrm{Spf}(B) \gets \mathrm{Spf}(B/IB) \to X)  \in (X/A)_\Prism \mapsto M(\eta)  \in \mathcal{D}_{(p,I)-comp}(B)\]
that is compatible with base change. More precisely, one defines such crystals as objects in the full subcategory of $D(\mathrm{PShv}( (X/A)_\Prism, \mathcal{O}_\Prism))$ spanned by objects $M$ satisfying these properties. (Note that the notion of crystal only depends on the category $(X/A)_\Prism$ equipped with the presheaf $\mathcal{O}_\Prism$ of $A$-algebras; it does not depend on the topology chosen.)

Given the insensitivity of the notion of prismatic cohomology and crystals to the choice of topology (as explained in the  previous paragraph), it seems worth explaining why we choose the flat topology (instead of, say, the \'etale topology) in our definition of the prismatic site. Besides naturality, the primary reason is to leave open the possibility of making arguments via quasisyntomic descent from ``perfectoid'' situations; see Examples~\ref{PerfectoidCoverFinalObject} and \ref{PerfectoidCoverFinalObjectAbs} for explicit examples of flat covers that are useful in practice (e.g., such covers are used in \cite{BhattScholzeFCrys} to understand $F$-crystals on the absolute prismatic site of $\mathcal{O}_K$ for finite extension $K/\mathbf{Q}_p$).
\end{remark}

\begin{remark}[The perfect prismatic site] 
\label{PerfectPrismatic}
For any $p$-adic formal $A/I$-scheme $Y$, one can define a version of the prismatic site (Definition~\ref{DefPrismaticSite}) by restricting to the full subcategory of $(Y/A)_\Prism^\perf \subset (Y/A)_\Prism$ of those objects $(\Spf (B) \gets \Spf (B/IB) \to Y) \in (Y/A)_\Prism$ for which $B$ is perfect. Equivalently, this is the category of perfectoid rings $R$ over $A/I$ together with a map $\Spf (R) \to Y$, and so this admits a description that does not use the notion of prisms. The cohomology of the perfect prismatic site has a different flavour than that of the prismatic site, and we will analyze it in Section~\ref{generalperfoidization}.
\end{remark}

\begin{remark}[The absolute prismatic site]
\label{AbsPrismaticSite}
For any $p$-adic formal scheme $X$, we can also consider the absolute prismatic site $(X)_\Prism$; its objects are bounded prisms $(B,J)$ equipped with a map $\mathrm{Spf}(B/J) \to X$ of $p$-adic formal schemes. We do not seriously explore this notion in this paper. However, the following observation shall be useful later:

\begin{lemma}\label{PerfectPrismInitial}
Let $(A,I)$ be a perfect prism corresponding to a perfectoid ring $R = A/I$. Then for any prism $(B,J)$, any map $A/I\to B/J$ of commutative rings lifts uniquely to a map $(A,I)\to (B,J)$ of prisms.
\end{lemma}

Using suggestive notation, this says that $A\in (R)_\Prism$ is initial (or rather final, as we pass to opposite categories in order to get a site), so that $(R)_\Prism = (R/A)_\Prism$.

\begin{proof}
Fix a prism $(B,J)$. We must show that any map $A/I \to B/J$ of commutative rings lifts uniquely to map $(A,I) \to (B,J)$ of prisms.  Now $L_{A/\mathbf{Z}_p}$ vanishes after derived $p$-completion as $A$ is ring of Witt vectors of a perfect ring of characteristic $p$, while both $B$ and $B/J$ are derived $p$-complete. It follows by deformation theory that $\mathrm{Hom}(A,B) \simeq \mathrm{Hom}(A, B/J) \simeq \mathrm{Hom}(A, B/(J,p))$, so the composite map $A \to A/I \to B/J$ lifts uniquely to a map $\alpha:A \to B$. This argument also proves that $\alpha$ is compatible with the Frobenius lifts on $A$ and $B$. Also, by construction, we have $\alpha(I) \subset J$. Thus, once we know $\alpha$ is a $\delta$-map (which is not automatic, since $B$ might have $p$-torsion), we will have found a unique map $(A,I) \to (B,J)$ of prisms lifting the given map $A/I \to B/J$, proving the lemma. To check $\alpha$ is a $\delta$-map, note that $\alpha$ factors uniquely over $B^\perf := \lim_{\phi} B \to B$ by the bijectivity of $\phi$ on $A$. As $B^\perf$ is a perfect $\delta$-ring, it is $p$-torsionfree. Thus, the $\phi$-equivariant map $A \to B^\perf$ is automatically a $\delta$-map. Moreover, the map $B^\perf \to B$ is a $\delta$-map by construction, so the composite $\alpha:A \to B$ is indeed a $\delta$-map, as wanted. 
\end{proof}

The ``dual'' to the above assertion is false: given a prism $(A,I)$ and a perfect prism $(B,J)$, there might be multiple maps $(A,I) \to (B,J)$ inducing the same map $A/I \to B/J$. 
\end{remark}

\subsection{The Hodge-Tate comparison theorem: formulation}
\label{ss:HT}

In this subsection, we fix a smooth $p$-adic formal scheme $X$ over $A/I$.

\begin{construction}[The Hodge-Tate comparison map]
\label{CartierPrismatic}
Using Construction~\ref{PrismtoEtale}, we can produce some complexes on $X_{\et}$ that will be fundamental to this paper. Write
\[ \Prism_{X/A} := R\nu_* \mathcal{O}_{\Prism} \in D(X_{\et}, A) \quad \text{and} \quad \overline{\Prism}_{X/A} := R\nu_* \overline{\mathcal{O}}_{\Prism} \in D(X_{\et}, \mathcal{O}_X).\]
These are commutative algebra objects of the corresponding derived categories, and we have the formula 
\[ \overline{\Prism}_{X/A} \simeq \Prism_{X/A} \otimes_A^L A/I.\]
In particular, setting $M\{i\} := M \otimes_{A/I} I^i/I^{i+1}$ for an $A/I$-module $M$, we have a Bockstein differential
\[ \beta_I: H^i(\overline{\Prism}_{X/A})\{i\} \to H^{i+1}(\overline{\Prism}_{X/A})\{i+1\}\]
from the triangle
\[
\overline{\Prism}_{X/A}\{i+1\}\to \Prism_{X/A}\otimes^L_A I^i/I^{i+2}\to \overline{\Prism}_{X/A}\{i\}\ .
\]
As $i$ varies, these maps assemble to endow the graded ring $H^*(\overline{\Prism}_{X/A})\{\ast\}$ with the structure of a differential graded $A/I$-algebra whose $0$-th term is an $\mathcal{O}_X$-algebra; write $\eta_X^0:\mathcal{O}_X \to H^0(\overline{\Prism}_{X/A})$ for the structure map. By the universal property of K\"{a}hler differentials, we have an $\mathcal{O}_X$-module map $\eta_X^1:\Omega^1_{X/(A/I)} \to H^1(\overline{\Prism}_{X/A})\{1\}$ given by sending a local section $fdg$ to $f\beta_I(g)$. To extend this to larger degrees, note that the differential graded algebra $H^*(\overline{\Prism}_{X/A})\{\ast\}$ is graded commutative by generalities on cohomology of commutative ring objects in a topos. Moreover, we have the following weak form of strict graded commutativity

\begin{lemma}
\label{WeakGradedCommutativity}
For any local section $f \in \mathcal{O}_X(U)$, the differential $\beta_I(f) \in H^1(\overline{\Prism}_{X/A})\{1\}(U)$ of the image of $f$ in $H^0(\overline{\Prism}_{X/A})(U)$ squares to zero. 
\end{lemma}

This lemma is only non-obvious when $p=2$. We do not prove it here; it follows from Proposition~\ref{PropHTCAffineLine}, which we shall prove later. Assuming Lemma~\ref{WeakGradedCommutativity}, we can use the universal property of the de Rham complex to extend the graded $\mathcal{O}_X$-module map 
\[ \eta_X^0 \oplus \eta_X^1: \mathcal{O}_X \oplus \Omega^1_{X/(A/I)} \to H^0(\overline{\Prism}_{X/A}) \oplus H^1(\overline{\Prism}_{X/A})\{1\}\]
(that is compatible with the differential) as defined above to a map 
\[ \eta^*_X:\Omega^*_{X/(A/I)} \to H^*(\overline{\Prism}_{X/A})\{\ast\}\]
of commutative differential graded algebras compatibly with the $\mathcal{O}_X$-module structure on the terms. When the prism $(A,I)$ is oriented via the choice of a distinguished elemnt $d \in I$, we have $A/I \simeq I^i/I^{i+1}$ via $d^i$, so we may drop the twists appearing above; we typically do this when $d = p$. For future reference, we remark that the construction of the map $\eta^*_X$ depends functorially on the commutative algebra object $\Prism_{X/A} \in D(X_\et, A)$ together with the $\mathcal O_X$-algebra structure on $\Prism_{X/A}\otimes^L_A A/I$.
\end{construction}

This map forms our main tool for controlling $\Prism_{X/A}$. 

\begin{theorem}[Hodge-Tate Comparison]
\label{HTComp}
The map 
\[ \eta^*_X:\Omega^*_{X/(A/I)} \to H^*(\overline{\Prism}_{X/A})\{\ast\}\]
constructed above is an isomorphism of differential graded $A/I$-algebras. In particular, $\overline{\Prism}_{X/A} \in D(X_{\et}, \mathcal{O}_X)$ is a perfect complex with $H^i(\overline{\Prism}_{X/A}) \cong \Omega^i_{X/(A/I)}\{-i\}$. 
\end{theorem}

Theorem~\ref{HTComp} will be proven in \S \ref{ss:HTProof} in characteristic $p$, and then in \S \ref{sec:HTCompGeneral} in general. In the rest of this subsection, we record some consequences.

\begin{corollary}\label{BaseChangePrismCoh}
The formation of $\Prism_{X/A} \in D(X_{\et}, A)$ commutes with base change along a map $(A,I) \to (B,IB)$ of bounded prisms, i.e., if $g:X_B := X \times_{\mathrm{Spf}(A/I)} \mathrm{Spf}(B/IB) \to X$ denotes the projection, then $(g^* \Prism_{X/A})^\wedge \cong \Prism_{X_B/B}$, where $(-)^\wedge$ denotes derived $(p,I)$-completion.
\end{corollary}

\begin{proof} By derived Nakayama \cite[Tag 0G1U]{Stacks}, it suffices to prove the base change modulo $I$, i.e.~that
\[
(g^* \overline{\Prism}_{X/A})^\wedge\cong \overline{\Prism}_{X_B/B}\ ,
\]
where the completion is the derived $p$-completion. This follows from the Hodge-Tate comparison as differentials $\Omega^\ast_{X/A}$ commute with base change.
\end{proof}

\begin{remark}[A Higgs specialization]
\label{Higgs}
Assume $A$ is $p$-torsionfree. If $X$ lifts to a $(p,I)$-completely flat $(p,I)$-adic formal scheme $\mathfrak{X}/A$ equipped with a compatible lift of Frobenius, then the affines in $X_{\et} \simeq \mathfrak{X}_{\et}$ naturally yield objects of $(X/A)_\Prism$: for each such affine $\mathfrak{U} \to \mathfrak{X}$, the diagram $(\mathfrak{U} \gets \mathfrak{U} \times_A A/I \to X)$ is an object of $(X/A)_\Prism$. Passing to global sections, this gives a $\phi$-equivariant map $\Prism_{X/A} \to \mathcal{O}_{\mathfrak{X}}$ of $E_\infty$-$A$-algebras on $X_{\et}$. Reducing modulo $I$, this determines an $E_\infty$-$A/I$-algebra map $\overline{\Prism}_{X/A} \to \mathcal{O}_X$ that splits the structure map $\mathcal{O}_X \to \overline{\Prism}_{X/A}$; note that this map, which we call a Higgs specialziation\footnote{The reason for this choice of terminology is that one can show the following: the $\infty$-category of crystals $M$ of derived $p$-complete $\overline{\mathcal{O}}_\Prism$-complexes on $(X/A)_\Prism$ is equivalent (via sheafified global sections) to $D_{p-\text{comp}}(X, \overline{\Prism}_{X/A})$, which in turn is equivalent, via base change along a chosen Higgs specialization, to the $\infty$-category of quasi-coherent sheaves on the formal completion $\widehat{T^* X}$ of the relative cotangent bundle of $X/(A/I)$ at its $0$-section. Thus, via base change along Higgs specializations, one can attach topologically nilpotent Higgs fields on $X$ (relative to $A/I$) to crystals of $\overline{\mathcal{O}}_\Prism$-modules on $(X/A)_\Prism$. Giving more precise statements and proofs is beyond the scope of this article as we have not systematically treated prismatic crystals; related results have been established elsewhere, e.g., \cite{TianFCrys}, \cite{MorrowTsuji}, Ogus (to appear) and \cite{BhattLurieChern}.}, depends on the choice of $\mathfrak{X}$  (including its $\delta$-structure).

\end{remark}

\begin{remark}[Obstruction to lifting modulo $I^2$]
One can show that the map $\mathcal{O}_X \to \tau^{\leq 1} \overline{\Prism}_{X/A}$ splits as an $\mathcal{O}_X$-module map if and only if $X$ lifts to $A/I^2$. To see this, we use the cotangent complex; in this remark, all cotangent complexes are interpreted in the $p$-complete sense. By general properties of the cotangent complex of the quotient of a ring by a nonzerodivisor, lifting $X$ to $A/I^2$ is equivalent to splitting the map $L_{X/A} \to L_{X/(A/I)}$. The desired claim now follows from:

\begin{proposition}
\label{AbsCCPrismatic}
There is a natural isomorphism $L_{X/A} \simeq \big(\tau^{\leq 1} \overline{\Prism}_{X/A}\big)\{1\}[1]$.
\end{proposition}
\begin{proof}[Proof sketch]
When $X := \mathrm{Spf}(R)$ is affine, we shall construct a natural map $\eta_R:L_{R/A} \to \overline{\Prism}_{R/A}\{1\}[1]$ in the $\infty$-category $\mathcal{D}(R)$ (where we follow Notation~\ref{NotAffinePrism} below in switching from schemes to rings). As both objects appearing in the statement of the proposition lie in $\mathcal{D}_{perf}(X)$ (by Theorem~\ref{HTComp} for the right side), globalizing (i.e., taking limits over all affine opens) yields a comparison map $L_{X/A} \to \overline{\Prism}_{X/A}\{1\}[1]$ in $\mathcal{D}_{perf}(X)$ and thus (as $L_{X/A} \in D^{\leq 0}$) a map $L_{X/A} \to \tau^{\leq 0} \left(\overline{\Prism}_{X/A}\{1\}[1]\right) \simeq \big(\tau^{\leq 1} \overline{\Prism}_{X/A}\big)\{1\}[1]$ in $\mathcal{D}_{perf}(X)$, as wanted in the proposition. To construct $\eta_R$, observe that given any object $(R \to B/IB \gets B) \in (R/A)_\Prism$, we have a natural map
\[ L_{R/A} \to L_{(B/IB)/B} \simeq I/I^2 \otimes_{A/I}^L B/IB[1] \simeq B/IB\{1\}[1].\]
Taking derived global sections over $(R/A)_\Prism$ then yields the desired comparison map. 

It remains to show that the map $\mu:L_{X/A} \to \big(\tau^{\leq 1} \overline{\Prism}_{X/A}\big)\{1\}[1]$ constructed above is an isomorphism. Note that the cohomology sheaves of both sides are abstractly isomorphic: we get $\mathcal{O}_X\{1\}$ in degree $-1$ and $\Omega^1_{X/A}$ in degree $0$ (by the transitivity triangle for the left side, and by Theorem~\ref{HTComp} for the right side). Moreover, the formation of both sides and the map above commutes with $p$-completed base change (using Corollary~\ref{BaseChangePrismCoh} for the right side). Using this property, we reduce to checking the statement when $I = (p)$. Using Theorem~\ref{HTComp}, it is enough to show that the maps $H^i(\mu)$ for $i=0,1$ agree with the maps coming from Construction~\ref{CartierPrismatic} used to produce the maps in Theorem~\ref{HTComp}. For this compatibility, one reduces to the case where $X = \mathbf{A}^1$. By the base change compatibility of the construction, we may then even assume $A = \mathbf{Z}_p$. In this case, we leave it to the reader to check the desired compatibility. Alternately, a complete proof can be found in \cite[Proposition 3.2.1]{AnschutzLeBrasPrismaticDieudonne}.
\end{proof}

The  $I=(p)$ case of the preceding discussion is closely related to \cite[Theorem 3.5]{DeligneIllusie} via the de Rham comparison isomorphism (Theorem~\ref{generaldeRham}); see also \cite[\S 8.2]{BMS1} for a variant in the $A_{\inf}$-cohomology theory, Remark~\ref{AbsCCNygaard} for a variant in the singular case, and \cite{IllusiePartialDeg,IllusieSaitoConfTalk} for a more general statement with applications in the crystalline case.
\end{remark}

\subsection{Generalities on computing prismatic cohomology}
\label{ss:ComputePrismatic}

\begin{notation}
\label{NotAffinePrism}
Let $(A,I)$ be a bounded prism, and let $X$ be a smooth $p$-adic formal $A/I$-scheme. When $X := \mathrm{Spf}(R)$ is affine, we also write $(R/A)_\Prism$ for $(\mathrm{Spf}(R)/A)_\Prism$; a typical object here is seen as a diagram $(B \to B/IB \gets R)$ where $(B,IB)$ is a prism over $(A,I)$. When we view our objects via rings rather than schemes, we use the algebraic language to describe categorical properties of objects in $(R/A)_\Prism$; for instance, if $A=R/I$, then we refer to $(A \to A/I \gets A/I=R) \in (R/A)_\Prism$ as an {\em initial} object rather than a final object. In this affine situation, set $\Prism_{R/A} := R\Gamma(\mathrm{Spec}(R), \Prism_{\mathrm{Spec}(R)/A})$ and $\overline{\Prism}_{R/A} := \Prism_{R/A} \otimes_A^L A/I$. 

\end{notation}

\begin{construction}[Weakly initial objects]
\label{cons:weakinitialsite}
Let $R$ be a $p$-completely smooth $A/I$-algebra. Let $B_0 \to R$ be a surjection of $A$-algebras with $B_0$ being the derived $(p,I)$-completion of a polynomial $A$-algebra, and let $B = \mathrm{Free}_{\delta,A}(B_0)^{\wedge}$ be the derived $(p,I)$-completion of the free $\delta$-$A$-algebra $\mathrm{Free}_{\delta,A}(B_0)$ on $B_0$. As $A$ is bounded, one can also use the classical completion to define $B$ by \cref{BoundedPrismProp}.  Let $J_0 = \ker(B_0 \to R)$; as $B_0 \to B$ expresses the latter as the derived $(p,I)$-completion of a polynomial $B_0$-algebra, the same holds true for $R \to R' := (B/J_0B)^{\wedge}$, where the completion is now $p$-adic. In particular, $R'$ is $p$-completely ind-smooth over $A/I$, and $\ker(B \to R')$ is the derived $(p,I)$-complete ideal of $B$ generated by $J_0$. Applying \cref{PrismaticEnvSmooth} (see also Example~\ref{ex:PrismativEnvSmooth}) to the pair $(B,\ker(B \to R'))$ yields an object $(C \to C/IC \gets R) \in (R/A)_\Prism$, and the construction fits into the following commutative diagram
\[ \xymatrix{ A \ar[r] \ar[d] & B_0 \ar[r] \ar[d] & B = \mathrm{Free}_{\delta,A}(B_0)^{\wedge} \ar[d] \ar[r] & C \ar[d] \\
A/I \ar[r] & R \ar[r] & R' \ar[r] & C/IC, }\]
with the square in the middle being a pushut square of derived $(p,I)$-complete $A$-algebras. We claim that the resulting object  $(C \to C/IC \gets R) \in (R/A)_\Prism$ is weakly initial, i.e., it maps to any other object of $(R/A)_\Prism$. To see this, fix some $(D \to D/ID \gets R) \in (R/A)_\Prism$. As $B_0$ is the derived $(p,I)$-completion of a polynomial $A$-algebra and $D$ is derived $(p,I)$-complete, one can choose an $A$-algebra map $B_0 \to D$ lifting the $A/I$-algebra map $R \to D/ID$. As $D$ is a $\delta$-$A$-algebra, this choice extends uniquely to a $\delta$-$A$-algebra map $B \to D$. By construction, this extension carries $J_0B$ into $ID$. By the universal property of $C$, there is a unique $\delta$-$A$-algebra factorization $B \to C \to D$. To finish, we must check that the resulting map $R \to C/IC \to D/ID$ agrees with the given map $R \to D/ID$, but this is immediate as the composite $B_0 \to B \to C \to D$ was the map chosen as the start of the construction and thus compatible with the given map $R \to D/ID$ under the natural quotient maps  $B_0 \to R$ and $D \to D/ID$.

For future use, we remark that the construction above carrying the surjection $(B_0 \to R)$ to the weakly final object $(C \to C/IC \gets R) \in (R/A)_\Prism$ is functorial in $B_0$ in the obvious sense. Moreover, using the universal property of the construction of $B$ from $B_0$ and $C$ from $B$, one checks that the construction carries tensor products in the $B_0$ argument to coproducts in $(R/A)_\Prism$. 
\end{construction}

\begin{construction}[\v{C}ech-Alexander complexes for prismatic cohomology]
\label{PrismaticFunctorialCC}
Let $R$ be a $p$-completely smooth $A/I$-algebra. Let $B_0 \to R$ be a surjection of $A$-algebras with $B_0$ being the derived $(p,I)$-completion of a polynomial $A$-algebra. Write $B_0^\bullet$ for the derived $(p,I)$-completed Cech nerve of $A \to B_0$. Applying Construction~\ref{cons:weakinitialsite} to the cosimplicial object of surjections $(B_0^\bullet \to R$) gives a cosimplicial object $(C^\bullet \to C^\bullet/IC^\bullet \gets R) \in (R/A)_\Prism$. By the last sentence of Construction~\ref{cons:weakinitialsite}, this object is the Cech nerve of $(C^0 \to C^0/IC^0 \gets R) \in (R/A)_\Prism$. As the latter object is weakly initial in $(R/A)_\Prism$, it follows from the vanishing assertion\footnote{Given a topos $\mathcal{X}$, a sheaf $A$ of abelian groups, and a weakly final object $X_0 \in \mathcal{X}$ with Cech nerve $X_0^\bullet$, the natural map gives an isomorphism $R\Gamma(\mathcal{X},A) \simeq R\lim_\bullet R\Gamma(X_0^\bullet,A)$; thus, if $A$ has vanishing higher cohomology on each $X_0^i$, then $R\Gamma(\mathcal{X},A)$ is computed by the cosimplicial object $F(X_0^\bullet)$. To prove this, observe that the map $X_0 \to \ast$ in $\mathcal{X}$ is surjective by the weak finality assumption on $X_0$. The claim now follows from \v{C}ech cohomological descent applied to the surjective map $X_0 \to \ast$; the special case of a presheaf topos is the subject of \cite[Tag 07JM]{Stacks}, and the proof of \v{C}ech descent given there adapts here.} in Corollary~\ref{BoundedPrismSite} that $C^\bullet$ (resp. $C^\bullet/IC^\bullet$) computes $\Prism_{R/A}$ (resp. $\overline{\Prism}_{R/A}$). 

In the preceding construction, we may choose the input surjection $B^0 \to R$ functorially in $R$; for instance, one may simply set $B_0$ to be the $(p,I)$-completion of the polynomial $A$-algebra on the set $R$. Via such a choice, we obtain a strictly functorial chain complex $C^\bullet((R/A)_{\Prism}, \mathcal{O}_{\Prism})$ calculating $\Prism_{R/A}$. This complex, however, does not commute with base change on the prism $(A,I)$ as the choice of $B_0$ is not base change compatible.
\end{construction}

\begin{remark}[A base change compatible \v{C}ech-Alexander complex for the affine line]
\label{CechAlexanderForLine}
Let us explain how to make Construction~\ref{PrismaticFunctorialCC} explicit for $R = A/I \langle X \rangle$ being the $p$-adic completion of a polynomial ring. In this case, for the ring $B_0$ from Construction~\ref{PrismaticFunctorialCC}, we may simply take  the derived $(p,I)$-completion of $A[X]$, with the map $B_0 \to R$ being the obvious map. The resulting cosimplicial $\delta$-$A$-algebra $C^\bullet$ commutes with base change on the prism $(A,I)$.
\end{remark}

Using the complexes from Construction~\ref{PrismaticFunctorialCC}, we prove two stability properties for prismatic cohomology that will be useful for the proof of the Hodge-Tate comparison. We note that both are easy consequences of the Hodge-Tate comparison; in particular, the first is a special case of Corollary~\ref{BaseChangePrismCoh}. 

\begin{lemma}[Base change]
\label{WeakBaseChange}
Let $R$ be a $p$-completely smooth $A/I$-algebra. Let $(A,I) \to (A',I')$ be a map of bounded prisms such that $A \to A'$ has finite $(p,I)$-complete Tor amplitude. If we write $R' := R \widehat{\otimes}_A A'$ for the base change, then the natural map induces an isomorphism $\Prism_{R/A} \widehat{\otimes}_A^L A' \cong \Prism_{R'/A'}$, and similarly for $\overline{\Prism}_{R/A}$.
\end{lemma}

Once the Hodge-Tate comparison is proven, the $(p,I)$-complete Tor amplitude assumption above can be dropped; of course, as Lemma~\ref{WeakBaseChange} is used to prove the Hodge-Tate comparison, we cannot drop the assumption at this point of the exposition.

\begin{proof}
By inspection and the base change property of \cref{PrismaticEnvSmooth} (2), the formation of the complex in Construction~\ref{PrismaticFunctorialCC} commutes with $(p,I)$-completed base change along $A \to A'$ in a weak sense, i.e., the cosimplicial ring  $C^\bullet( (R/A)_\Prism, \mathcal{O}_{\Prism}) \widehat{\otimes}_A A'$ obtained as the {\em termwise} derived $(p,I)$-completed base change of
the cosimplicial ring $C^\bullet( (R/A)_\Prism, \mathcal{O}_{\Prism})$ computes $\Prism_{R'/A'}$. It now remains to observe that since $A \to A'$ has finite $(p,I)$-complete Tor amplitude, the functor of $(p,I)$-completed base change along $A \to A'$ commutes with totalizations of cosimplicial $(p,I)$-complete $A$-modules by Lemma~\ref{TorAmpBC}. The result for $\overline{\Prism}$ follows by reduction modulo $I$.
\end{proof}

Second, it localizes for the \'etale topology.

\begin{lemma}[\'Etale localization]
\label{EtaleLocalizePrismatic}
Let $R \to S$ be a $p$-completely \'etale map of $p$-completely smooth $A/I$-algebras. Then the natural map $\overline{\Prism}_{R/A} \widehat{\otimes}_R^L S \to \overline{\Prism}_{S/A}$ is an isomorphism.
\end{lemma}

\begin{proof}
Let us begin with a general observation. There is an obvious functor $(S/A)_\Prism \to (R/A)_\Prism$ given by restriction of scalars along $R \to S$. This functor has a left\footnote{Recall that we are viewing objects in $(R/A)_\Prism$ as diagrams of rings rather than affine schemes.} adjoint $F:(R/A)_\Prism \to (S/A)_\Prism$ described as follows: given a prism $(B \to B/IB \gets R) \in (R/A)_\Prism$,  the induced $p$-completely \'etale map $B/IB \to B/IB \widehat{\otimes}^L_R S$ deforms uniquely to a $(p,I)$-completely \'etale $\delta$-map $B \to B_S$ by Lemma~\ref{ExtendEtale}, giving an object $F(B \to B/IB \gets R) := (B_S \to B_S/IB_S \simeq B/IB \widehat{\otimes}^L_R S \gets S)$ of $(S/A)_\Prism$. As $F$ is a left adjoint, it preserves coproducts and carries weakly initial objects to weakly initial objects.

Now consider the cosimplicial object $(C^\bullet \to C^\bullet/IC^\bullet \gets R)$ in $(R/A)_\Prism$ given by Construction~\ref{PrismaticFunctorialCC}, so $C^\bullet/IC^\bullet$ computes $\overline{\Prism}_{R/A}$. By construction, this cosimplicial object is the \v{C}ech nerve of its zeroth term $(C^0 \to C^0/IC^0 \gets R)$, and this zeroth term is a weakly initial object. Consider the corresponding cosimplicial object $(D^\bullet \to D^\bullet/ID^\bullet \gets S) := F(C^\bullet \to C^\bullet/IC^\bullet \gets R)$ in $(S/A)_\Prism$. As $F$ preserves both coproducts and weakly initial objects, it follows that $\overline{\Prism}_{S/A}$ is computed by $D^\bullet/ID^\bullet$. Our task is thus to show that the natural map $C^\bullet/IC^\bullet \widehat{\otimes}_R^L S \to D^\bullet/ID^\bullet$ is an isomorphism. This is true at the level of terms by construction of the functor $F$, so the claim follows from Lemma~\ref{TorAmpBC}.
\end{proof}

The following purely algebraic lemma was used above. 

\begin{lemma}
\label{TorAmpBC}
Let $C$ be a commutative ring equipped with a finitely generated ideal $J$. Let $C \to C'$ be a map of commutative rings that has finite $J$-complete Tor amplitude. Then the $J$-completed base change functor $- \widehat{\otimes}^L_C C'$ commutes with totalizations in $D^{\geq 0}$. 
\end{lemma}

\begin{proof}
Let $M^\bullet$ be a cosimplicial object in $D^{\geq 0}$ with totalization $M$. We must show that
\[ \mathrm{Tot}(M^\bullet) \widehat{\otimes}^L_C C' \simeq \mathrm{Tot}(M^\bullet \widehat{\otimes}_C^L C')\] 
via the natural map. If instead of $M^\bullet$ one uses the $n$-skeleton $M^{\leq n}$ of $M$, then the corresponding assertion is clear. We are thus reduced to showing the following: there exists some $c \geq 0$ such that if $K \in D^{\geq n}(C)$, then $K \widehat{\otimes}_C^L C' \in D^{\geq n-c}(C')$ for all $n \geq 0$. Choose generators $x_1,...,x_r \in J$. It is enough to find some $c \geq 0$ such that the functor 
\[ K \mapsto (K \otimes_C^L C') \otimes_{C'}^L \mathrm{Kos}(C';x_1,...,x_r)\] 
takes $D^{\geq n}$ to $D^{\geq n-c}$: indeed, this would imply that for the same constant $c$ and any integer $m \geq 1$, the functor
\[ K \mapsto (K \otimes_C^L C') \otimes_{C'}^L \mathrm{Kos}(C';x_1^m,...,x_r^m)\] 
also carries $D^{\geq n}$ to $D^{\geq n-c}$, which then implies the desired claim by taking an inverse limit of $m$. Now we can write 
\[ (K \otimes_C^L C') \otimes_{C'}^L \mathrm{Kos}(C';x_1,...,x_r) \simeq (K \otimes_C^L \mathrm{Kos}(C;x_1,...,x_r)) \otimes_{\mathrm{Kos}(C;x_1,...,x_r)}^L \mathrm{Kos}(C';x_1,...,x_r).\]
Now the operation $K \mapsto K \otimes_C^L \mathrm{Kos}(C;x_1,...,x_r)$ takes $D^{\geq n}$ to $D^{\geq n-r}$, while the operation $N \mapsto N \otimes_{\mathrm{Kos}(C;x_1,...,x_r)}^L \mathrm{Kos}(C';x_1,...,x_r)$ takes $D^{\geq m}$ to $D^{\geq m-a}$ for some fixed $a$ by assumption on the Tor amplitude. It follows that taking $c = a+r$ solves the problem.
\end{proof}

\newpage

\section{The characteristic $p$ case: the Hodge-Tate and crystalline comparisons}
\label{ss:HTProof}

In this section, we prove the crystalline comparison for prismatic cohomology; the key tool is the relationship between divided power envelopes and $\delta$-structures (\cref{PDenvRegSeqLambda}). Using the Cartier isomorphism, we deduce the Hodge-Tate comparison over crystalline prisms (\cref{HodgeTateCharp}).

\begin{notation}
Let $(A,(p))$ be a crystalline prism and let $I\subset A$ be a pd-ideal with $p\in I$. In particular the Frobenius $A/p\to A/p$ factors over the quotient $A/p \to A/I$, inducing a map $\psi: A/I\to A/p$. Let $R$ be a smooth $A/I$-algebra and let $R^{(1)} = R\otimes_{A/I,\psi} A/p$.
\end{notation}

\begin{theorem}[The crystalline comparison in characteristic $p$]\label{CrysComp} Write $R\Gamma_\crys(R/A)$ for the crystalline cohomology of $\mathrm{Spf}(R)$ relative to the pd-thickening $\mathrm{Spec}(A/I) \subset \mathrm{Spf}(A)$. Then there is a canonical isomorphism $\Prism_{R^{(1)}/A} \simeq R\Gamma_\crys(R/A)$ of $E_\infty$-$A$-algebras compatibly with the Frobenius.
\end{theorem}

\begin{remark} We will often apply the theorem when $R=\tilde{R}\otimes_{A/p} A/I$ for some smooth $A/p$-algebra $\tilde{R}$. In that case $R^{(1)}$ is the pullback of $\tilde{R}$ under $\phi$, and the base change compatibility of prismatic cohomology (Corollary~\ref{BaseChangePrismCoh}, reproduced as part of Corollary~\ref{HodgeTateCharp} below) translates the theorem into
\[
\phi_A^\ast \Prism_{\tilde{R}/A}\cong R\Gamma_\crys(\tilde{R}/A)\ .
\]
\end{remark}

\begin{proof}
We first give the proof when $I=(p)$, and then reduce to this case. \\

{\em The $I=(p)$ case.} For $I=(p)$, let us first recall a Cech-Alexander style construction of crystalline cohomology that closely mirrors by Construction~\ref{PrismaticFunctorialCC}. Consider an $A$-algebra surjection $B_0 \to R$ where $B_0$ is the $p$-completion of a polynomial $A$-algebra. Let $J_0 = \ker(B_0 \to R)$, and let $B$ be the free $p$-complete $\delta$-$A$-algebra on $B_0$. Let $R' = B/J_0B$, so $R \to R'$ is the ind-smooth map of $A/p$-algebras obtained via base change from the $p$-completely ind-smooth map $B_0 \to B$ of $p$-complete $A$-algebras. Write $D_0$ (resp. $D$) for the $p$-completed PD-envelope of the ideal $J_0 \subset B_0$ (resp. $J_0B \subset B$). Thus, we have a diagram
\[ \xymatrix{ 
A \ar[r] \ar[dd] & B_0 \ar[r] \ar[d] & B \ar[d] \\
& D_0 \ar[r] \ar[d] & D \ar[d] \\
A/p \ar[r] & R \ar[r] & R'. }\]
Both the small squares on the right are pushout squares whose horizontal maps are $p$-completely flat, the vertical maps on the small square on the bottom right are PD-thickenings, and both $D_0$ and $D$ are $p$-torsionfree (Lemma~\ref{PDEnvBCbasic}). Moreover, for each $n \geq 1$, the rings $D_0/p^n$ and $D/p^n$ identify with the PD-envelopes of $B_0/p^n \to R$ and $B/p^n \to R'$ (compatibly with the standard divided power structure on $(p)$). In particular, for each $n \geq 1$, the object $(D/p^n \to R' \gets R)$ gives an object of the big crystalline site $(R/A)_{CRYS}$ as in \cite[\S III.4]{BerthelotCrisCoh}, where the base $A$ is equipped with the PD-ideal $(p)$ with its unique PD-structure. As $n$ varies, we obtain a pro-object   $\{(D/p^n \to R' \gets R)\}_{n \geq 1}$ of $(R/A)_{CRYS}$. It is easy to see that this object is weakly initial (i.e., that each object of $(R/A)_{CRYS}$ receives a map from some term of the pro-object) and that the construction carrying the initial surjection $(B_0 \to R)$ to the pro-object $\{(D/p^n \to R' \gets R)\}_{n \geq 1}$ commutes with coproducts. Consequently, the \v{C}ech nerve of this object has the form $\{(D^\bullet/p^n \to (R')^\bullet \gets R)\}_{n \geq 1}$, where $D^\bullet$ is obtained by applying the initial construction $B_0 \mapsto D$ to the \v{C}ech nerve of $A \to B_0$. Thus, $R\Gamma_\crys(R/A)$ is computed by the cosimplicial ring $D^\bullet$; here we implicitly use \cite[Proposition III.4.1.4]{BerthelotCrisCoh} to identify crystalline cohomology in the big and small crystalline sites. Moreover, as the collection of surjections $(B_0 \to R)$ is sifted, this recipe is independent of the choice of $(B_0 \to R)$ up to canonical isomorphism.

%In particular, the object $(D \to R' \gets R)$ gives an object of the big crystalline site of $(R/A)_{CRYS}$ (as in \cite[\S III.4]{BerthelotCrisCoh}, where the base $A$ is equipped with the PD-ideal $(p)$ with its unique PD-structure). It is easy to see that this object is weakly initial and that the construction carrying the initial surjection $(B_0 \to R)$ to $(D \to R' \gets R) \in (R/A)_{CRYS}$ commutes with coproducts. Consequently, if $(D^\bullet \to (R')^\bullet \gets R)$ denotes the Cech nerve of this object, then $R\Gamma_\crys(R/A)$ is computed by the cosimplicial ring $D^\bullet$; here we implicitly use \cite[Proposition III.4.1.4]{BerthelotCrisCoh} to identify crystalline cohomology in the big and small crystalline sites. Moreover, as the collection of surjections $(B_0 \to R)$ is sifted, this recipe is independent of the choice of $(B_0 \to R)$ up to canonical isomorphism. 

Next, running Construction~\ref{PrismaticFunctorialCC} with the same input data $(B_0 \to R)$ as in the previous paragraph, we obtain a cosimplicial ring $C^\bullet$ computing $\Prism_{R/A}$. By base change stability of Cech-Alexander complexes, the cosimplicial ring $\phi_A^* C^\bullet$ (obtained by termwise base change along the Frobenius on $A$) computes $\Prism_{R^{(1)}/A}$. (Note that we are not claiming yet that $\phi_A^*$ commutes with totalizing $C^\bullet$.) To prove the isomorphism in the theorem, we shall exhibit a canonical map $\phi_A^* C^\bullet \to D^\bullet$ of cosimplicial $A$-algebras that gives a quasi-isomorphism on underlying complexes. We shall do so by explicitly describing both sides as suitable prismatic envelopes. Write $B_0^\bullet$ and $B^\bullet$ for the Cech nerves of $A \to B_0$ and $A \to B$ computed in $p$-complete $A$-algebras. If $J_0^\bullet \subset B_0^\bullet$ denotes the kernel of $B_0^\bullet \to R$, then $C^\bullet$ identifies with the $p$-completed prismatic envelope $B^\bullet\{\frac{J_0^\bullet B^\bullet}{p}\}^{\wedge}$, while $D^\bullet$ identifies with the PD-envelope $D_{J_0^\bullet B^\bullet}(B^\bullet)$ of $J_0^\bullet$ inside $B^\bullet$. Via \cref{PDenvRegSeqLambda}, we can then also describe $D^\bullet$ as the prismatic envelope $B^\bullet\{\frac{\phi(J_0^\bullet)B^\bullet}{p}\}^{\wedge}$. In particular, the relative Frobenius for $A \to B^\bullet$ induces a canonical map
\[ \eta:\phi_A^* C^\bullet := \phi_A^* B^\bullet\{\frac{J_0^\bullet B^\bullet}{p}\}^{\wedge} \to B^\bullet\{\frac{\phi(J_0^\bullet)B^\bullet}{p}\}^{\wedge} =: D^\bullet,\]
of cosimplicial $A$-algebras that yields the desired comparison map of $E_\infty$-$A$-algebras on passage to the derived category. This description also shows that $D^\bullet$ identifies with the pullback of the cosimplicial $B^{\bullet,(1)}$-module $\phi_A^* C^\bullet$ under the relative Frobenius $B^{\bullet,(1)} \to B^\bullet$ as the latter map is flat. Moreover, both source and target of $\eta$ are $p$-torsionfree: this is standard for PD envelopes and then follows for prismatic envelopes (see, e.g., Lemma~\ref{PDEnvelopeFlat} and Corollary~\ref{PDFlatp} applied with $A=\mathbf{Z}_p$). We may then apply Lemma~\ref{CosimpHT} below (with $N = \phi_A^* C^\bullet/p$ and $B^\bullet$ replaced by its mod $p$ reduction) to conclude that the derived mod $p$ reduction of $\eta$ gives a quasi-isomorphism on underlying complexes; derived Nakayama (\cite[Tag 0G1U]{Stacks}) then implies that $\eta$ itself is a quasi-isomorphism, proving the isomorphism in the theorem in the $I=(p)$ case.

To finish the $I=(p)$ case, it remains to prove that the comparison constructed above is compatible with the Frobenius. Note that the source and target of the map $\eta$ constructed above are naturally cosimplicial $\delta$-$A$-algebras and $\eta$ respects this structure. To finish, it will suffice to show that the Frobenius lift on $D^\bullet$ coming from its $\delta$-structure (ultimately induced by the $\delta$-structure on $B^\bullet$) coincides with the Frobenius on crystalline cohomology $R\Gamma_\crys(R/A)$ under the natural identification $R\Gamma_\crys(R/A) \simeq \lim D^\bullet$. This is a statement entirely on the crystalline side, but we do not know a reference, so we sketch a proof. Recall the following functoriality statement, immediate from the definition in crystalline cohomology:  given any map $\alpha:A \to A'$ of $p$-torsionfree rings, a smooth $A/p$-scheme $Y$ and a smooth $A'/p$-scheme $X$, an $\alpha$-equivariant map $\pi:X \to Y$, simplicial objects $U_\bullet \in (X/A')_{CRYS}$ and $V_\bullet \in (Y/A)_{CRYS}$, and a map $f_\bullet:U_\bullet \to V_\bullet$ of PD-schemes lying over $\pi$ (and thus over $\alpha$), one has an induced commutative diagram
\[ \xymatrix{ R\Gamma_{\crys}(Y/A) \ar[r] \ar[d]^{\pi^*} & R\Gamma_{\crys}(X/A') \ar[d] \\
\lim \mathcal{O}_{CRYS}(V_\bullet) \ar[r]^{f_\bullet^*} & \lim \mathcal{O}_{CRYS}(U_\bullet)}\]
of $E_\infty$-$A$-algebras (where the objects on the right are regarded as $A$-algebras via $\alpha$). Apply this observation with $A=A'$, $X=Y=\mathrm{Spec}(R)$, $U_\bullet=V_\bullet$ given by the PD-thickening $(D^\bullet/p^n \to (R')^\bullet \gets R)$ used above, and $\alpha$, $\pi$ and $f_\bullet$ being the Frobenius maps to conclude that the comparison map $R\Gamma_\crys(R/A) \to \lim_\bullet D^\bullet/p^n$ is $\phi$-equivariant; a similar argument with inductive systems of simplicial objects applied to $\{(D^\bullet/p^n \to (R')^\bullet \gets R)\}_{n \geq 1}$ then allows us to take the inverse limit to conclude that the comparison isomorphism $R\Gamma_\crys(R/A) \simeq \lim_\bullet D^\bullet$ is $\phi$-equivariant. \\

%Now  applying the natural Frobenius to each vertex of the diagram $R \to (R')^\bullet \gets D^\bullet$ gives an endomorphism of this diagram that lives over the Frobenius endomorphism of $A$, so the $\phi$-equivariance of the isomorphism $R\Gamma_{\crys}(R/A) \simeq \lim D^\bullet$ now follows from the preceding observation applied to $X=Y=\mathrm{Spec}(R)$, $\pi=\mathrm{Frob}$

{\em The general case.} We now no longer assume $I=(p)$, so we merely have a containment $p \in I$. As $I$ admits divided powers, the map $A/p \to A/I$ is a surjection with a locally nilpotent kernel (each element is annihilated by the Frobenius). By noetherian approximation, we can choose a lift of the smooth $A/I$-algebra $R$ to a smooth $A/p$-algebra $\widetilde{R}$. After making such a choice, the $I=(p)$ case proven above gives the desired isomorphism $\Prism_{R^{(1)}/A} \simeq R\Gamma_{\crys}(R/A)$: the base change of $\widetilde{R}$ along the Frobenius $A/p \to A/p$ coincides with $R^{(1)}$ as this base change factors over the surjection $A/p \to A/I$. To conclude the argument, it therefore suffices to exhibit a canonically defined comparison map $\Prism_{R^{(1)}/A} \to R\Gamma_\crys(R/A)$ that agrees with the map coming from the choice of $\widetilde{R}$. We shall explain the canonical construction of such a map, and leave the compatibility of the two constructions to the reader as an exercise in  relating distinct Cech-Alexander complexes computing crystalline cohomology.

To construct the comparison, choose a surjection $B_0 \to R$ where $B_0$ is a smooth $A$-algebra equipped with a $\delta$-structure; for instance, $B_0$ could be a polynomial $A$-algebra. Forming the PD envelope of the surjection $B_0 \to R$ as well as its Cech nerve yields a $p$-completely flat $p$-complete cosimplicial $\delta$-$A$-algebra $D^\bullet$ equipped with a surjection $D^\bullet \to R$ whose kernel $K^\bullet$ has divided powers compatible with those on $I$ and such that the resulting map $R\Gamma_\crys(R/A) \to \lim D^\bullet$ is an isomorphism. Now observe that the Frobenius on $D^\bullet$ carries $K^\bullet$ into $pD^\bullet$ by the existence of divided powers on $K^\bullet$, thus giving a map 
\[ R \simeq D^\bullet/KD^\bullet \xrightarrow{\phi_{D^\bullet}} D^\bullet/pD^\bullet.\]
of cosimplicial rings which is linear over the corresponding map $A/I \to A/p$ induced by the Frobenius on $A$. Linearizing over $\psi$, we obtain an object $(D^\bullet \to D^\bullet/pD^\bullet \gets R^{(1)}) \in (R^{(1)}/A)_\Prism$, which then gives the desired comparison map $\Prism_{R^{(1)}/A} \to \lim D^\bullet \simeq R\Gamma_\crys(R/A)$. This comparison map, a priori, depends on the choice of the initial surjection $B_0 \to R$. But the collection of all such choices is sifted, so the constructed comparison map $\Prism_{R^{(1)}/A} \to R\Gamma_\crys(R/A)$ is independent of the choice of $B_0 \to R$ up to contractible ambiguity (e.g., one may take a colimit over all choices to obtain a canonical map). 
\end{proof}

The following lemma on cosimplicial modules was used in the above proof.

\begin{lemma}
\label{CosimpHT}
Let $k$ be a commutative ring of characteristic $p$. Let $B$ be a polynomial $k$-algebra. Let $B^\bullet$ be the Cech nerve of $k \to B$, and let $B^{\bullet,(1)}$ be its Frobenius twist, so we have the relative Frobenius $B^{\bullet,(1)} \to B^\bullet$ over $k$. For any cosimplicial $B^{\bullet,(1)}$-module $N$, the natural map
\[ N \to N \otimes_{B^{\bullet,(1)}} B^\bullet\]
gives a quasi-isomorphism on associated unnormalized chain complexes of $k$-modules.
\end{lemma}

The argument given below is somewhat indirect, and relies crucially on the fact that degeneracy maps may be ignored when passing from cosimplicial abelian groups to the corresponding object in the derived category. It is desirable to have a better proof of this comparison. A similar statement can also be found in \cite[Lemma 3.28]{NiziolKLog} (which we learnt about from \cite[Appendix B]{KoshikawaLog1}).

\begin{proof}
By a filtered colimit argument, we reduce to the case where $B$ is a polynomial algebra in finitely many variables. We first explain why $B^\bullet$ admits a ``basis'' over $B^{\bullet,(1)}$ as a semi-cosimplicial module (i.e., ignoring degeneracy maps), and then we use the basis to prove the lemma. For ease of notation, write $B = k[M]$ for the monoid $M = \mathbf{N}^r$ with suitable $r \geq 0$.

For any $[n] \in \Delta$, we have the monoid $M^{\oplus [n]}$. As $[n]$ varies, the assignment $[n] \mapsto M^{\oplus [n]}$ assembles naturally into a cosimplicial object $EM$: the face maps are given by inclusions of finite products of $M$ into larger finite products (and thus only use the knowedge of $M$ as a set pointed by $0 \in M$), while the degeneracies involve the monoid operation. More precisely, $EM$ is the Cech nerve of the map $0 \to M$ of commutative monoids. The construction $M \mapsto EM$ is functorial in $M$, and the map $B^{\bullet,(1)} \to B^\bullet$ in the question identifies with the map $k[E(pM)] \to k[EM]$ on monoid algebras induced by the inclusion $pM \subset M$ by functroriality of $E(-)$.

Let $S \subset M$ be the standard set of coset representatives for $M^{grp}/pM^{grp}$, i.e., writing $M = \mathbf{N}^r$, we take $S = \{0,...,p-1\}^r \subset \mathbf{N}^r$. We regard $S$ as a set pointed by $0 \in S$, so the composition $S \to M \to M^{grp}/pM^{grp}$ is a pointed bijection. We then have
\begin{equation}
\label{monoidcoset}
\bigsqcup_{s \in M} (pM) + s \simeq M.
\end{equation} 
Taking products, we obtain subsets $S^{\oplus [n]} \subset M^{\oplus [n]}$ for each $[n] \in \Delta$. Note that a face map $M^{\oplus [n]} \to M^{\oplus [m]}$ in $EM$ carries $S^{[n]}$ into $S^{[m]}$: up to a permutation, a face map $M^{\oplus [n]} \to M^{\oplus [m]}$ is given by the product of the identity map on the first $\#[n]$ factors of the target and the $0$ map on the remaining factors, so the claim follows  as $0 \in S$. Thus,  we obtain a semi-cosimplicial set $ES$ together with an inclusion $ES \subset EM$ of semi-cosimplicial sets. The decomposition \eqref{monoidcoset} yields a similar decomposition of the semi-cosimplicial set $EM$, i.e., for each $[n] \in \Delta$, the subset $S^{\oplus [n]} \subset M^{\oplus [n]}$ gives a set of coset representatives for $(M^{grp}/pM^{grp})^{\oplus [n]}$,  the natural maps give a decomposition
\begin{equation}
\label{monoidcosetcosimp}
\bigsqcup_{\underline{s} \in S^{\oplus [n]}} (pM^{\oplus [n]}) + \underline{s} \simeq M^{\oplus [n]},
\end{equation}
and this decomposition is compatible with the face maps of the semi-cosimplicial sets $EM$ and $ES$ respectively (as the face maps for $ES$ were defined by restriction from $EM$). 

For future use, we remark that if $\ast$ denotes the constant semi-cosimplicial set with value a point (i.e., the final object in semi-cosimplicial sets), then there is a natural map $\eta_0:\ast \to ES$ of semi-cosimplicial sets determined by the point $0 \in S$; let 
\[ \eta = k[\eta_0]:k \to k[ES]\]
be the induced map on semi-cosimplicial $k$-modules.

Passing to free commutative $k$-algebras, we learn from \eqref{monoidcosetcosimp} that the semi-cosimplicial $k$-algebra $B^\bullet = k[EM]$ is free when regarded as a module over the semi-cosimplicial $k$-algebra $B^{\bullet,(1)} = k[E(pM)]$ with basis $ES$, i.e., we have a natural isomorphism
\begin{equation}
\label{basiscosimp}
 k[E(pM)] \otimes_k k[ES] \simeq k[EM]
 \end{equation}
of semi-cosimplicial $k[E(pM)]$-modules. Moreover, under this isomorphism, the structure map 
\[ k[E(pM)] \to k[EM]\] 
of semi-cosimplicial $k$-modules to the right side of \eqref{basiscosimp} identifies with the map 
\[ \mathrm{id}_{k[E(pM)]} \otimes \eta:k[E(pM)] = k[E(pM)] \otimes_k k \to k[E(pM)] \otimes_k k[ES] \]
to the left side of \eqref{basiscosimp}.

Now we prove the lemma, so let $N$ be any cosimplicial $k[E(pM)]$-module. Our task is to show that the natural map
\[ \alpha:N \to N \otimes_{k[E(pM)]} k[EM]\]
of cosimplicial $k$-modules gives a quasi-isomorphism on associated unnormalized chain complexes of $k$-modules. As the formation of the unnormalized chain complex only depends on underlying semi-cosimplicial objects, we can use the discussion in the previous paragraph to reduce to checking that the natural map of semi-cosimplicial $k$-modules
\[ \beta = \mathrm{id}_N \otimes \eta: N = N \otimes_k k \to N \otimes_k k[ES]\]
gives a quasi-isomorphism on associated unnormalized chain complexes of $k$-modules.  We next observe that, using the pointed bijection $S \simeq M^{grp}/pM^{grp}$, the semi-cosimplicial $k$-module $k[ES]$ can also be identified with the semi-cosimplicial $k$-module underlying the Cech nerve $C^\bullet$ of the $k$-algebra map $k \to C^0 := k[M^{grp}/pM^{grp}]$. Moreover, under this identification, the map $\eta:k \to k[ES]$ identifies with the structure map $\gamma:k \to C^\bullet$. In our context, this means that the map
\[ \delta = \mathrm{id}_N \otimes \gamma:N = N \otimes_k k \to N \otimes_k C^\bullet \]
of cosimplicial $k$-modules recovers the map $\beta$ on underlying semi-cosimplicial $k$-modules, so it suffices to prove $\delta$ is a quasi-isomorphism on underlying unnormalized complexes. But $\gamma$ is a homotopy equivalence of cosimplicial $k$-modules as it is the Cech nerve of the $k$-algebra map $k \to C^0$ that admits a section (determined by the map $M \to 0$ via functoriality).  As homotopy equivalences of cosimplicial $k$-modules are stable under tensor products with arbitrary cosimplicial $k$-modules, the claim follows.
\end{proof}

\begin{corollary}[The Hodge-Tate comparison in characteristic $p$]
\label{HodgeTateCharp}
For any smooth $A/p$-algebra $S$, the comparison map from Theorem~\ref{HTComp} yields an isomorphism
\[
H^i(\overline{\Prism}_{S/A})\cong \Omega^i_{S/(A/p)}\{-i\}.
\]
In particular, the assigment $R \mapsto \Prism_{R/A}$ commutes with arbitrary $p$-complete base change on the prism $(A,(p))$.
\end{corollary}

\begin{proof} By Lemma~\ref{EtaleLocalizePrismatic}, we can assume that $S=A/p[X_1,\ldots,X_n]$ is a polynomial algebra, and then by \cref{WeakBaseChange} we can reduce to the case $A=\mathbf Z_p$. Then $S=R^{(1)}$ for the $\mathbf F_p$-algebra $R=\mathbf F_p[X_1,\ldots,X_n]$, and the result follows from Theorem~\ref{CrysComp} and the Cartier isomorphism.
\end{proof}

\newpage

\section{The mixed characteristic case of the Hodge-Tate comparison}
\label{sec:HTCompGeneral}

To prove the Hodge-Tate comparison in general, we shall need the following construction to pass from general bounded prisms $(A,I)$ to crystalline ones.

\begin{construction}[Mapping the universal oriented prism to a crystalline one]
\label{UnivOrientedCrystallize}
Let $(A,(d))$ be an oriented prism with the following properties: the Frobenius on $A/p$ is flat, and the ring $A/(d)$ is $p$-torsionfree. For example, the universal oriented prism from Example~\ref{UnivOrientedPrism} has these properties. Let $B := A\{\frac{\phi(d)}{p}\}^\wedge$ be the $p$-complete simplicial commutative $\delta$-$A$-algebra obtained by freely adjoining $\phi(d)/p$ to $A$; as $\phi(d) = p u$ for a unit $u \in \pi_0(B)$ by Lemma~\ref{distinguishedDivide}, the $A$-algebra $B$ is $(p,\phi(d))$-complete, and hence (as $d^p \in (p,\phi(d))$) also $(p,d)$-complete Moreover, by Corollary~\ref{PDenvRegSeqLambda}, the $A$-algebra $B$ identifies with the $p$-completed pd-envelope of $(d) \subset A$ via the natural map $A \to B$, and is thus discrete and $p$-torsionfree since $d$ is a nonzerodivisor modulo $p$.  Now consider the map $\alpha:A \to B$ determined as the composition of the natural map $A \to B$ with $\phi:A \to A$. We can regard $\alpha$ as a map of prisms $(A,(d)) \to (B,(p))$ since $\alpha(d) \in pB$. Modulo $p$, the map $\alpha$ factors as
\begin{equation}
\label{CrystallizeTor}
 A/p \xrightarrow{can} A/(p,d) \xrightarrow{\phi} A/(p,d^p) \xrightarrow{can} B/p \simeq D_{(d)}(A/p)
 \end{equation}
where the first map has finite Tor amplitude, the second map is faithfully flat by construction, and the last map is faithfully flat (even free) by the structure of divided power envelopes. Using this description, one obtains the following  properties of the base change functor
\[ \widehat{\alpha^*}:{D}_{comp}(A) \to {D}_{comp}(B)\] 
on derived $(p,d)$-complete objects of the derived category:
\begin{enumerate}
\item The functor $\widehat{\alpha^*}$ is conservative. Indeed, by derived Nakayama (\cite[Tag 0G1U]{Stacks}), it suffices to know that derived $d$-completed base change along the map $A/p \to D_{(d)}(A/p)$ is conservative on derived $d$-complete objects; this follows as the same statement is true individually for each map appearing in the composition \eqref{CrystallizeTor} (by derived Nakayama for the first, and faithful flatness for the rest).

\item The functor $\widehat{\alpha^*}$ has finite $(p,d)$-complete Tor amplitude. Again, this reduces to the statement that derived $d$-completed base change along the composition in \eqref{CrystallizeTor} has finite $d$-complete Tor amplitude, which holds true individually for each map.

\item For any $p$-completely smooth $A/I$-algebra $R$ with $p$-completed base change $R_B$ to $B$, we have a base change isomorphism $\widehat{\alpha^*} \Prism_{R/A} \simeq \Prism_{R_B/B}$ by Lemma~\ref{WeakBaseChange}.

\end{enumerate}
We shall use these properties to deduce general properties of prismatic cohomology from those of crystalline cohomology.
\end{construction}

Before proceeding to the proof of the Hodge-Tate comparison, we need to prove Lemma~\ref{WeakGradedCommutativity}, ensuring that we have well-defined Hodge-Tate comparison maps. For this, we actually prove the Hodge-Tate comparison for the affine line; the proof below is structured in a way that generalizes to higher dimensions.

\begin{proposition}[The Hodge-Tate comparison for the affine line]
\label{PropHTCAffineLine}
If $R$ is the $p$-adic completion $A/I\langle X \rangle$ of the polynomial ring, then the maps $\eta_R^0:R \to H^0(\overline{\Prism}_{R/A})$ and $\eta_R^1\{-1\}:\Omega^1_{R/(A/I)}\{-1\} \to H^1(\overline{\Prism}_{R/A})$ coming from Construction~\ref{CartierPrismatic} (by twisting) are isomorphisms, and $H^i(\overline{\Prism}_{R/A}) = 0$ for $i > 1$.  In particular, Lemma~\ref{WeakGradedCommutativity} holds true in general.
\end{proposition}

\begin{proof}
The last part follows by functoriality. For the rest, choose a map $\eta: R\oplus \Omega^1_{R/(A/d)}\{-1\}[-1] \to \overline{\Prism}_{R/A}$ in $D(R)$ inducing the maps $\eta_R^i$ on cohomology for $i=0,1$; this is possible because $R$ and $\Omega^1_{R/(A/d)}$ are finite  free over $R$. We will show that $\eta$ is an isomorphism.

First, if $I =(p)$, we are done thanks to the Hodge-Tate comparison in characteristic $p$ (Corollary~\ref{HodgeTateCharp}).  Next, if $(A,I)$ is the universal oriented prism $(A_0,(d))$ from Example~\ref{UnivOrientedPrism}, then the claim follows by base change to the $I = (p)$ case using the map $\alpha$ and properties (1) and (3) from Construction~\ref{UnivOrientedCrystallize}.

Next, consider a general bounded oriented prism $(A,(d))$. The choice of the generator $d \in A$ determines a map $A_0 \to A$, where $(A_0,(d))$ is the universal  oriented prism considered in the previous paragraph. If the map $A_0\to A$ had finite $(p,I)$-complete Tor-amplitude (or we knew base change for prismatic cohomology in general already), we would get the claim via \cref{WeakBaseChange}. We can however argue as follows. Letting $a:A_0 \to D_0$ denote the map $\alpha$ from Construction~\ref{UnivOrientedCrystallize}, we obtain a diagram
\[ \xymatrix{ & A_0 \ar[r] \ar[d]^-{a} & A \ar[d]^-{b} \\
		   \mathbf{Z}_p  \ar[r] & D_0 \ar[r] &  E := A \widehat{\otimes}_{A_0}^L D_0 }\]
of $(p,d)$-complete simplicial commutative rings, where the square is a pushout square of $(p,d)$-complete simplicial commutative rings. Note that the bottom left arrow lifts to a map of prisms $(\mathbf{Z}_p,(p)) \to (D_0,(a(d)))$ as $(p) = (a(d))$ by construction.  Write $c$ for the composite map $\mathbf{Z}_p \to E$. The base change functor $\widehat{b^*}$ is conservative on $(p,d)$-complete objects in the derived category (since the same holds true for $\widehat{a^*}$, see property (1) in Construction~\ref{UnivOrientedCrystallize}), so it suffices to show that $\widehat{b^*}(\eta)$ is an isomorphism. We claim that $\widehat{b^*} \overline{\Prism}_{R/A} \simeq \widehat{c^*} \overline{\Prism}_{\mathbf{F}_p[X]/\mathbf{Z}_p}$: this is clear at the level of explicit complexes if one uses the complexes from Remark~\ref{CechAlexanderForLine}, and the rest follows from Lemma~\ref{TorAmpBC} together with the observation that $b$ (resp. $c$) has finite $(p,d)$-complete (resp. $p$-complete) Tor amplitude. We leave it to the reader to check that the resulting identification $\widehat{b^*} \overline{\Prism}_{R/A} \simeq \widehat{c^*} \overline{\Prism}_{\mathbf{F}_p[X]/\mathbf{Z}_p}$ carries the map $\widehat{b^*} \eta$ to a map realizing the Hodge-Tate comparison map on cohomology for the target. In particular, our claim now reduces to the Hodge-Tate comparison for $\overline{\Prism}_{\mathbf{F}_p[X]/\mathbf{Z}_p}$, which was already explained earlier (Corollary~\ref{HodgeTateCharp}).

Finally, for a general bounded prism $(A,I)$, we reduce to the case treated in the previous paragraph by making a faithfully flat base change to a bounded oriented prism (Lemma~\ref{BoundedPrismProp} (4)).
\end{proof}

Thus, we now have a well-defined Hodge-Tate comparison map from Construction~\ref{CartierPrismatic}.

\begin{theorem}[The Hodge-Tate comparison]
\label{HTCompPrismatic}
The Hodge-Tate comparison maps
\[ \Omega^i_{X/(A/I)} \to H^i(\overline{\Prism}_{X/A})\{i\}\]
of sheaves on $X_{\et}$ from Construction~\ref{CartierPrismatic} are isomorphisms.
\end{theorem}

\begin{proof}
Since the comparison map is globally defined, it is enough to prove the statement for affines. Thus, we shall check that for any $p$-completely smooth $A/I$-algebra $R$, the analogous map
\[ \Omega^i_{R/(A/I)} \to H^i(\overline{\Prism}_{R/A})\{i\}\]
is an isomorphism of $R$-modules. As both differential forms and prismatic cohomology localize in the \'etale topology (Lemma~\ref{EtaleLocalizePrismatic}), it is enough to prove the statement when $R = A/I \langle X_1,...,X_n \rangle$ is the $p$-adic completion of a polynomial ring. In this case, we choose a map 
\[ \eta:\bigoplus_{i=0}^n \Omega^i_{R/(A/I)}\{-i\}[-i] \to \overline{\Prism}_{R/A}\]
 in the derived category $D(R)$ lifting the Hodge-Tate comparison maps (up to twists) on cohomology, and proceed exactly as explained in Lemma~\ref{PropHTCAffineLine} for $n=1$.
\end{proof}

Using base change for prismatic cohomology (proven in Corollary~\ref{BaseChangePrismCoh} using Theorem~\ref{HTCompPrismatic}), we can now also prove the de~Rham comparison, at least under the technical assumption that $W(A/I)$ is $p$-torsion free. This is satisfied, for example,  if $A/I$ is $p$-torsion free or if $I=(p)$ and $A/p$ is reduced. This technical assumption will be removed later (\cref{generaldeRham}).

\begin{theorem}[The de Rham comparison]
\label{dRComp1}
Let $X$ be a smooth formal $A/I$-scheme and assume that $W(A/I)$ is $p$-torsion free. Then there is a natural isomorphism $\Prism_{X/A}\widehat{\otimes}_{A,\phi}^L A/I\simeq \Omega^*_{X/(A/I)}$ of commutative algebras in $D(X_\et, A/I)$, where the completion is $p$-adic.
\end{theorem}

\begin{proof}
As before, it is enough to solve the problem in the affine case. Thus, for a $p$-completely smooth $A/I$-algebra $R$, we must construct a functorial isomorphism $\Prism_{R/A}\widehat{\otimes}_{A,\phi}^L A/I\simeq \Omega^*_{R/(A/I)}$ of
commutative algebras in $D(A/I)$. We shall prove a stronger statement that involves relating both sides to crystalline cohomology. 

To formulate this stronger statement, note that the kernel $J \subset W(A/I)$ of the canonical composite surjection $W(A/I) \to A/I \to A/(I,p)$ is a pd-ideal. By generalities on crystalline cohomology, the commutative algebra object $R\Gamma_{\crys}((R/p)/W(A/I)) \in D(W(A/I))$ lifts the de Rham complex $\Omega^*_{R/(A/I)} \in D(A/I)$. On the other hand, the natural map $A \to A/I$ refines uniquely to a $\delta$-map $A \to W(A/I)$ by Remark~\ref{ThetaWitt}. It is therefore enough to construct a functorial isomorphism
\[ \Prism_{R/A} \widehat{\otimes}^L_{A,\psi} W(A/I) \simeq R\Gamma_{\crys}((R/p)/W(A/I))\]
of commutative algebras in $D(W(A/I))$, where $\psi:A \to W(A/I)$ is obtained by composing the canonical map $A \to W(A/I)$ with the Frobenius. We shall deduce this from the crystalline comparison.

Note that the composite $\psi:A \to W(A/I)$ carries $I$ into $(p)$: the canonical map $A \to W(A/I)$ carries $I$ into $VW(A/I)$, so $\psi$ carries $I$ into $FVW(R) = (p)$. By our $p$-torsionfreeness assumption on $W(A/I)$, the map $\psi$ thus refines to a map $(A,I) \to (W(A/I),(p))$ of prisms. Base change for prismatic cohomology then gives an isomorphism
\[
\Prism_{R/A}\widehat{\otimes}_{A,\psi}^L W(A/I)\cong \Prism_{R^\prime/W(A/I)}\ .
\]
for the base change $R^\prime = R \widehat{\otimes}_{A/I,\psi} W(A/I)/p$. It thus suffices to identify the right side above with $R\Gamma_{\crys}((R/p)/W(A/I))$. But this follows by applying Theorem~\ref{CrysComp} to the pd-ideal $J \subset W(A/I)$, and noting that $R'$ is the base change of $R/p$ along $W(A/I)/J \simeq A/(p,I) \to W(A/I)/p$ induced by the Frobenius on $W(A/I)$.
\end{proof}

\newpage

\section{Semiperfectoid rings and perfection in mixed characteristic}
\label{sec:SemiPerfd}

In this section, we study semiperfectoid rings. In \S \ref{InitialPrismSemiperfd}, we construct an initial object of the absolute prismatic site $(S)_\Prism$ of a semiperfectoid ring $S$ essentially by hand, use it to construct a ``perfectoidization'' $S_\perfd$ of $S$, and formulate the main theorem of this section: the map $S \to S_\perfd$ is surjective (\cref{PerfectoidificationSurj}). To prove this theorem, we need a better way to access the aforementioned initial object; this is provided by the theory of derived prismatic cohomology, which is the subject of \S \ref{ss:DerivedPrismatic}. In \S \ref{ss:AndreLemma}, these ideas are put to use: we give a new proof of a key flatness result of Andr\'e from \cite{AndreDirectFactor}, and use it to prove the promised surjectivity of $S \to S_\perfd$.

\subsection{Universal prisms for semiperfectoid rings}
\label{InitialPrismSemiperfd}

\begin{notation} Fix a semiperfectoid ring $S$, i.e.~$S$ is a derived $p$-complete ring that can be written as a quotient of a perfectoid ring $R$. We recall that $S$ is called {\em quasiregular} if $S$ has bounded $p^\infty$-torsion, and the cotangent complex $L_{S/\mathbf Z_p}[-1]$ is $p$-completely flat over $S$ (see  \cite[\S 4.4]{BMS2}). In that case $S$ is classically $p$-adically complete.
\end{notation}

\begin{proposition}\label{InitialPrism} The category of prisms $(A,I)$ equipped with a map $S\to A/I$ admits an initial object $(\Prism_S^{\mathrm{init}},I)$ and $I=(d)$ is principal.
\end{proposition}

In terms of Remark~\ref{AbsPrismaticSite}, the category in question is the absolute prismatic site $(\mathrm{Spf}(S))_\Prism$.

\begin{proof} Choose a surjective map $R\to S$ from a perfectoid ring. By \cref{PerfectPrismInitial}, the category of prisms $(A,I)$ with a map $R\to A/I$ admits the initial object $(A_\inf(R),\ker\theta)$, and we recall that $\ker\theta=(d)$ is always principal. Let $J\subset A_\inf(R)$ be the kernel of the surjection $A_\inf(R)\xrightarrow{\theta} R\to S$. Given any prism $(C,K)$ with a map $S\to C/K$, we get a unique map $(A_\inf(R),(d))\to (C,K)$ making $A_\inf(R)/d=R\to S\to C/K$ commute, and $K=dC$ by \cref{PrismMapTaut}. As $d$ is a nonzerodivisor in $C$, there are unique elements $\frac{f}{d}\in C$ for each $f\in J$. Let
\[
B=A_\inf(R)\{x_f|f\in J\}/(dx_f-f|f\in J)_\delta\ .
\]
The preceding arguments show that there is a unique map of $\delta$-rings $B\to C$. As $C$ is $d$-torsionfree, this extends to a unique map $B/B[d^\infty]_\delta \to C$ (where $B[d^\infty]_\delta$ is the smallest $\delta$-stable ideal of $B$ containing $B[d^\infty]$). As $C$ is also a derived $(p,d)$-complete classical ring, this map uniquely extends further to $H^0$ of the derived $(p,d)$-completion of $B/B[d^\infty]_\delta$. The latter $\delta$-ring may however contain new $d$-torsion elements; but after a transfinite induction (using that completion commutes with $\omega_1$-filtered colimits), we arrive at a derived $(p,d)$-complete $\delta$-$A_\inf(R)$-algebra $\Prism_S^{\mathrm{init}}$ such that $d$ is a nonzerodivisor. We still have a unique map $(\Prism_S^{\mathrm{init}},(d))\to (C,K)$. Moreover, $\Prism_S^{\mathrm{init}}$ is a prism as $d$ is distinguished. Finally, the map
\[
R\cong A_\inf(R)/d\to \Prism_S^{\mathrm{init}}/d
\]
factors over the quotient $S$ of $R$ as all elements of $J$ become divisible by $d$ in $\Prism_S^{\mathrm{init}}$. Thus, $(\Prism_S^{\mathrm{init}},(d))$ with the map $S\to \Prism_S^{\mathrm{init}}/d$ is the desired initial object.
\end{proof}

In general, $\Prism_S^{\mathrm{init}}$ is not particularly well-behaved; for example it is not bounded in general. After perfection, this issue disappears.

\begin{corollary}\label{PerfectionMixedChar} There is a universal perfectoid ring $S_\perfd$ equipped with a map $S\to S_\perfd$. In fact, $S_\perfd$ is the $p$-adic completion of $(\Prism_S^{\mathrm{init}})_\perf/I$.
\end{corollary}

\begin{proof} Note that the category of perfectoid rings $S_\perfd$ with a map $S\to S_\perfd$ is equivalent to the category of perfect prisms $(A,I)$ with a map $R\to A/I$. Thus, the claim follows from \cref{PrismPerfection}.
\end{proof}

One goal in this section is prove the following theorem.

\begin{theorem}\label{PerfectoidificationSurj} The map $S\to S_\perfd$ is surjective.
\end{theorem}

\begin{remark}[Zariski closed sets are strongly Zariski closed]
 In \cite[Section II.2]{ScholzeTorsion}, the second author defined notions of ``Zariski closed" and ``strongly Zariski closed" subsets of an affinoid perfectoid space $X=\Spa(R,R^+)$. Here $R$ is a perfectoid Tate ring, in particular there is some topologically nilpotent unit $\varpi\in R$, and $R^+\subset R$ is an open and integrally closed subring of $R^\circ$. Then strongly Zariski closed subsets are in bijection with quotients $R\to R^\prime$ to perfectoid Tate rings $R^\prime$, while Zariski closed subsets are those closed subsets of $X$ that are the vanishing locus of some ideal $I\subset R$. It was claimed in \cite[Remark II.2.4]{ScholzeTorsion} that not any Zariski closed subset is strongly Zariski closed. The theorem implies that this is false, and that the notions in fact agree. Indeed, let $I\subset R$ be any ideal, and let $S$ be the $\varpi$-adic completion of $R^+/I\cap R^+$. Then $S$ is a semiperfectoid ring, and so by the theorem $R^\prime=S_\perfd[1/\varpi]$ is a perfectoid Tate ring that is a quotient of $R$ and such that any map from $R$ to a perfectoid ring that vanishes on $I$ factors uniquely over $R^\prime$. This implies that the vanishing locus of $I$ is exactly the strongly Zariski closed subset $\Spa(R^\prime,R^{\prime +})$, where $R^{\prime+}\subset R^\prime$ is the minimal open and integrally closed subring that contains the image of $R^+$.
\end{remark}

For the proof of Theorem~\ref{PerfectoidificationSurj}, we will need a different definition of $\Prism_S^{\mathrm{init}}$ which makes it possible to use the Hodge-Tate comparison result that we proved for smooth algebras (Theorem~\ref{HTCompPrismatic}) in order to compute $\Prism_S^{\mathrm{init}}$.

\subsection{Derived prismatic cohomology}
\label{ss:DerivedPrismatic}

In this section, we explain how to extend the notion of prismatic cohomology to all $p$-complete rings via Quillen's formalism of non-abelian left derived functors (aka left Kan extensions), as recast in the $\infty$-categorical language in \cite[\S 5.5.9]{LurieHTT} (see also \cite[\S 9.2]{BLMdRW} for an exposition closer to our context). Write $\mathcal{D}(A)$ for the derived $\infty$-category of a commutative ring $A$. 

\begin{construction}[Constructing derived prismatic cohomology]
\label{DerivedPrismatic}
Fix a bounded prism $(A,I)$. Consider the functor $R \mapsto \Prism_{R/A}$ on $p$-completely smooth $A/I$-algebras $R$ valued in the category of commutative algebras in  the $\infty$-category of $(p,I)$-complete objects in $\mathcal D(A)$ equipped with a $\phi_A$-semilinear endomorphism $\phi$. Write $R \mapsto \Prism_{R/A}$ for the left Kan extension of this functor to all derived $p$-complete simplicial commutative $A/I$-algebras, so $\Prism_{R/A}$ is a derived $(p,I)$-complete commutative algebra in $\mathcal D(A)$ equipped with an endomorphism $\phi:\Prism_{R/A} \to \phi_{A,*} \Prism_{R/A}$; we call this theory {\em derived prismatic cohomology}.

The Hodge-Tate comparison shows that $\overline{\Prism}_{R/A} :=\Prism_{R/A} \otimes_A^L A/I$ admits an exhaustive increasing $\mathbf{N}$-indexed filtration with $\mathrm{gr}_i$ given by sending $R$ to the derived $p$-completion of $\wedge^i L_{R/(A/I)}\{-i\}[-i]$; we refer to this as the {\em Hodge-Tate comparison} for $\overline{\Prism}_{R/A}$. Let us record some immediate consequences:
\begin{enumerate}
\item (The value on smooth algebras) If $R$ is $p$-completely smooth, then $\Prism_{R/A}$ is unchanged, so derived prismatic cohomology extends prismatic cohomology to all $p$-complete $A/I$-algebras.
\item (The \'etale sheaf property) The Hodge-Tate comparison for derived prismatic cohomology shows that $R \mapsto \Prism_{R/A}$ is a sheaf for the $p$-completely \'etale topology on the category of derived $p$-complete $A/I$-algebras. In particular, if $X$ is a $p$-adic formal scheme, then we can naturally define its derived prismatic cohomology $\Prism_{X/A}$ as a commutative algebra in the $\infty$-category of derived $(p,I)$-complete objects in $\mathcal D(X_{\et}, A)$ (equipped with a Frobenius).
\item (The quasisyntomic sheaf property) Assume that $(A,I)$ is a perfect prism. Then the assignment $R \mapsto \Prism_{R/A}$ forms a sheaf for the quasisyntomic topology on the category of quasisyntomic $A/I$-algebras (see \cite[Example 5.11]{BMS2} for a variant involving derived de Rham cohomology in characteristic $p$).
\item (Base change behaviour) The formation of $\Prism_{R/A}$ commutes with base change in the sense that for any map of bounded prisms $(A,I)\to (B,J)$, letting $R_B$ be the derived $p$-completion of $R\otimes^L_A B$, one has
\[
\Prism_{R_B/B}\cong \Prism_{R/A}\widehat{\otimes}^L_A B\ ,
\]
where the completion is the derived $(p,J)$-completion.
\item (Colimit preservation) The functor $\Prism_{-/A}$ from $p$-complete simplicial commutative $A/I$-algebras to $(p,I)$-complete $E_\infty$-algebras in $\mathcal{D}(A)$ is symmetric monoidal. In fact, it commutes with all colimits (as the same holds true for the associated graded of the derived Hodge-Tate cohomology $\overline{\Prism}_{-/A}$ functor).
\end{enumerate}
\end{construction}

A key advantage of derived prismatic cohomology is that there is a large supply of rings for which this theory is concentrated in degree $0$ (Example~\ref{DerivedPrismaticPrismatic}). In such situations, this theory is also closely related to an initial object of the prismatic site. 

\begin{lemma}[Derived prismatic cohomology when it is discrete]
\label{PrismaticCohGivesPrisms}
Assume $(A,I)$ is a bounded prism, and $R$ is a derived $p$-complete simplicial $A/I$-algebra such that $\overline{\Prism}_{R/A}$ is concentrated in degree $0$. 
\begin{enumerate}
\item The $\phi_A$-linear Frobenius $\phi_R$ on $\Prism_{R/A}$ naturally lifts to a $\delta$-$A$-structure on the ring $\Prism_{R/A}$. 
\item The pair $(\Prism_{R/A},I\Prism_{R/A})$ gives a prism over $(A,I)$ equipped with a map $R \to \overline{\Prism}_{R/A} = \Prism_{R/A}/I\Prism_{R/A}$.
\item The category of prisms $(B,J)$ over $(A,I)$ equipped with a map $R\to B/J$ has an initial object, and $(\Prism_{R/A},I\Prism_{R/A})$ is weakly initial. In particular, the initial object is the image of some idempotent endomorphism of $(\Prism_{R/A},I\Prism_{R/A})$.
\end{enumerate}
\end{lemma}

In part (3), we conjecture that $(\Prism_{R/A},I\Prism_{R/A})$ is actually the initial object.

\begin{proof}
Using \v{C}ech-Alexander complexes and canonical simplicial resolutions, we obtain a functor $R \mapsto F_A(R)$ from derived $p$-complete $A/I$-algebras to simplicial cosimplicial derived $(p,I)$-complete $\delta$-$A$-algebras computing prismatic cohomology, i.e., we have a functorial identification $\colim_{\Delta^{op}} \lim_{\Delta} F_A(R) \simeq \Prism_{R/A}$. The $\delta$-$A$-algebra structure on $F_A(R)$ is classified by a map $F_A(R) \to W_2(F_A(R))$ of simplicial cosimplicial rings (where the target is defined by pointwise application of $W_2(-)$). This map lies over the map $A \to W_2(A)$ classifying the $\delta$-structure on $A$ and splits the natural restriction map $W_2(F_A(R)) \to F_A(R)$. This defines a map 
\[ \Prism_{R/A} \simeq \colim_{\Delta^{op}} \lim_{\Delta} F_A(R) \to \colim_{\Delta^{op}} \lim_{\Delta} W_2(F_A(R))\]
of $E_\infty$-rings lying over the map $A \to W_2(A)$ and splitting the restriction map
\[ \colim_{\Delta^{op}} \lim_{\Delta} W_2(F_A(R)) \to \colim_{\Delta^{op}} \lim_{\Delta} F_A(R) \simeq \Prism_{R/A}.\]
Using the functorial pullback square
\[\xymatrix{
W_2(B)\ar[r]^F\ar[d] & B\ar[d]\\
B\ar[r]^\phi & B\otimes_{\mathbf Z}^L \mathbf F_p
}\]
of simplicial commutative (and thus of $E_\infty$) rings, as well as the discreteness of $\Prism_{R/A}$, one can identify $\colim_{\Delta^{op}} \lim_{\Delta} W_2(F_A(R))$ as $W_2(\Prism_{R/A})$. Thus, the above map defines a $\delta$-$A$-algebra structure on $\Prism_{R/A}$ refining the $\phi_A$-linear Frobenius endomorphism $\phi_R$, proving part (1). 

Part (2) is then immediate from the derived $(p,I)$-completeness of $\Prism_{R/A}$ and our assumption that $\overline{\Prism}_{R/A} = \Prism_{R/A} \otimes_A^L A/I$ is discrete.

For part (3), we first construct a functorial map $(\Prism_{R/A},I\Prism_{R/A})\to (B,J)$ for any prism $(B,J)$ over $(A,I)$ equipped with a map $R\to B/J$. Fix a resolution $R_\bullet\to R$ by $p$-completely ind-smooth $A/I$-algebras $R_i$. For each $R_i$, let $(R_i/A)_\Prism$ be the category of prisms $(C,K)$ over $(A,I)$ equipped with a map $R_i\to C/K$. Then
\[
\Prism_{R/A} = \colim_{\Delta^{op}} \lim_{C\in (R_\bullet/A)_\Prism} C
\]
in the category of $(p,I)$-complete complexes, as $\Prism_{R_i/A} = \lim_{C\in (R_i/A)_\Prism} C$ by the theory of \v{C}ech-Alexander complexes. But now $(B,J)$, like any object of $(R/A)_\Prism$, defines compatible objects of all $(R_i/A)_\Prism$, and so restricting the limit to this object gives a functorial map
\[
\Prism_{R/A} = \colim_{\Delta^{op}} \lim_{C\in (R_\bullet/A)_\Prism} C \to \colim_{\Delta^{op}} B = B\ .
\]
Using the argument from part (1), one shows that this is in fact a map of $\delta$-rings. Now part (3) follows from the following general categorical lemma.
\end{proof}

\begin{lemma} Let $\mathcal C$ be an idempotent complete category and $X\in \mathcal C$ an object. Let $\mathcal C_{X\backslash}$ be the category of objects $Y\in \mathcal C$ equipped with a map $X\to Y$. Assume that the identity on $\mathcal C$ factors over the projection $\mathcal C_{X\backslash}\to \mathcal C$ via a functor $F: \mathcal C\to \mathcal C_{X\backslash}$. Then $F(X)$ is an idempotent endomorphism of $X$ and the corresponding retract of $X$ is an initial object of $\mathcal C$.
\end{lemma}

\begin{proof} Let $F(X) = (X\xrightarrow{e} X)$. Applying $F$ to the morphism $(X\xrightarrow{e} X)$ of $\mathcal C$ shows that $e$ is an idempotent. Let $X^\prime$ be the corresponding retract of $X$. Then $F(X^\prime)=(X\to X^\prime)$ where $X\to X^\prime$ is an epimorphism. The functor $F$ is faithful, so for any $Y\in \mathcal C$, one has
\[
\Hom_{\mathcal C}(X^\prime,Y)\hookrightarrow \Hom_{\mathcal C_{X\backslash}}(X\to X^\prime,X\to Y)\ .
\]
But $X\to X^\prime$ is an epimorphism, so the right hand side has at most one element, so any $Y$ admits at most one map from $X^\prime$. On the other hand, there is at least one map $X^\prime\to Y$ as there is a map $X^\prime\to X$ and a map $X\to Y$. Thus, $X^\prime$ is initial, as desired.
\end{proof}

\begin{example}[Prismatic cohomology of an lci quotient]
\label{DerivedPrismaticPrismatic}
Fix a bounded prism $(A,I)$. Let $R := A/J$, where $J = (I,f_1,...,f_r)$ with $f_1,...,f_r$ giving a Koszul-regular sequence on $A/I$; assume $R$ has bounded $p$-torsion. Then $\Prism_{R/A}$ is concentrated in degree $0$ and $I$-torsionfree thanks to the Hodge-Tate comparison, so $(\Prism_{R/A},I\Prism_{R/A})$ is an object of $(R/A)_\Prism$. We claim that this is the initial object. This assertion can be checked locally on $\mathrm{Spf}(A)$, so we may assume $I=(d)$ is principal. In this case, consider the derived $(p,I)$-complete simplicial commutative $\delta$-$A$-algebra $B = A\{\frac{f_1}{d},...,\frac{f_r}{d}\}^{\wedge}$ obtained by freely adjoining $\frac{f_i}{d}$. There is a natural map $B \to \Prism_{R/A}$ of simplicial commutative $\delta$-$A$-algebras since $f_i = 0 \in \overline{\Prism}_{R/A}$. We shall check this map is an isomorphism. Granting this property, it follows that $B$ is discrete and $(B,IB)$ gives an object of $(R/A)_\Prism$; the universal property describing $B$ then shows immediately that $(B,IB)$ is the initial object of $(R/A)_\Prism$, whence $(\Prism_{R/A},I\Prism_{R/A})$ is also initial, as wanted.

It remains to check that the map $B = A\{\frac{f_1}{d},...,\frac{f_r}{d}\}^{\wedge} \to \Prism_{R/A}$ from the previous paragraph is an isomorphism. As the formation of both sides commutes with arbitrary base change, it is enough to check this in the universal case where $A=\mathbf{Z}_p\{d,f_1,...,f_r\}[\delta(d)^{-1}]^{\wedge}$, so $A/I$ and $R$ are both $p$-torsionfree. Setting $(A_0,I_0) := (\mathbf{Z}_p\{d\}[\delta(d)^{-1}]^{\wedge}, (d))$, we see that $f_1,...,f_r \in A/I$ is $p$-completely regular relative $A_0/I_0$. By the construction of the prismatic envelope in \cref{qPDEnvRegular}, it follows that $B$ identifies with the prismatic envelope $A\{\frac{J}{I}\}^{\wedge}$; this implies that $B$ is concentrated in degree $0$ and $(B,IB)$ is the initial object of $(R/A)_\Prism$. The map  $B \to \Prism_{R/A}$ then admits a retraction by Lemma~\ref{PrismaticCohGivesPrisms}; the map $B/IB \to \overline{\Prism}_{R/A}$ then also admits a retraction, and thus has a $p$-torsionfree cokernel as $\overline{\Prism}_{R/A}$ is $p$-torsionfree since $R$ is so. To prove this map is an isomorphism, it is then enough to see that the map $B/IB[\frac 1p]\to \overline{\Prism}_{R/A}[\frac 1p]$ is surjective. By the Hodge-Tate comparison, $\overline{\Prism}_{R/A}[\frac 1p]$ is generated as a Banach $R[\frac 1p]$-algebra by $\mathrm{gr}_1 = L_{R/(A/I)}^\wedge\{-1\}[-1] = J/(J^2+I)\{-1\}$. It thus suffices to see that the map $B=A\{\frac JI\}\to \Prism_{R/A}$ maps $\frac JI$ surjectively onto $\mathrm{gr}_1 = J/(J^2+I)\{-1\}$. This is a standard computation using a free resolution for $R$ over $A/I$; see \cite[Claim 3.30]{Bhattpadicddr} for a variant of this calculation in crystalline cohomology.
%
%Fix a bounded prism $(A,I)$. Let $R := A/J$, where $J = (I,f_1,...,f_r)$ with $f_1,...,f_r$ giving a $p$-completely regular sequence on $A/I$. Then $\Prism_{R/A}$ identifies with the prismatic envelope $B := A\{\frac{J}{I}\}$ (which identifies with $A\{\frac{f_1}{d},...,\frac{f_r}{d}\}^{\wedge}$ if $d \in I$ is a generator). To see this, note first that $\Prism_{R/A}$ is concentrated in degree $0$ and $I$-torsionfree thanks to the Hodge-Tate comparison. Thus, Lemma~\ref{PrismaticCohGivesPrisms} applies. By \cref{qPDEnvRegular}, $B$ is initial in the category of prisms over $A$ equipped with a map $R\to B/IB$. Thus, it remains to see that $\Prism_{R/A}$ is actually also the initial object in this situation.
%
%By base change compatibility of $\Prism_{R/A}$ and prismatic envelopes, we can assume that $I=(d)$ is oriented. Passing to the universal case, we may also assume that $A/d$ is $p$-torsionfree. It is enough to see that the map $B/IB[\frac 1p]\to \overline{\Prism}_{R/A}[\frac 1p]$ is surjective, as this then implies that the idempotent endomorphism of $\overline{\Prism}_{R/A}$ is the identity (using that it is torsion-free), and so the idempotent endomorphism of $\Prism_{R/A}$ is an isomorphism, whence $B\cong \Prism_{R/A}$. By the Hodge-Tate comparison, $\overline{\Prism}_{R/A}[\frac 1p]$ is generated as an $R[\frac 1p]$-algebra by $\mathrm{gr}_1 = L_{R/(A/I)}^\wedge\{-1\}[-1] = J/(J^2+I)\{-1\}$. It thus suffices to see that the map $B=A\{\frac JI\}\to \Prism_{R/A}$ maps $\frac JI$ surjectively onto $\mathrm{gr}_1 = J/(J^2+I)\{-1\}$. This is a standard computation.
\end{example}

\begin{proposition}\label{QRSPPrism} Let $S$ be a quasiregular semiperfectoid ring, let $R$ be any perfectoid ring with a map $R\to S$, and let $(A,I) = (A_\inf(R),\ker\theta)$. Then $\overline{\Prism}_{S/A}$ is discrete, and $\Prism_{S/A}$ is the initial prism $\Prism_S^{\mathrm{init}}$ with a map $S\to \Prism_S^{\mathrm{init}}/I$. In particular, $\Prism_{S/A}=\Prism_S^{\mathrm{init}}$ does not depend on $R$. Moreover, the map $S\to \overline{\Prism}_{S/A}$ is $p$-completely faithfully flat.
\end{proposition}

In the following, we simply write $\Prism_S = \Prism_S^{\mathrm{init}} = \Prism_{S/A}$ for quasiregular semiperfectoid $S$.

\begin{proof} If $S$ is quasiregular, then $L_{S/R}[-1]$ is $p$-completely flat, and thus the derived $p$-completions of all $\wedge^i L_{S/R}[-i]$ are $p$-completely flat $S$-modules and in particular discrete. Thus the Hodge-Tate comparison implies that $\overline{\Prism}_{S/A}$ is discrete and $p$-completely faithfully flat over $S$. It remains to see that $\Prism_{S/A}$ is initial, or equivalently that the idempotent endomorphism from \cref{PrismaticCohGivesPrisms}~(3) is the identity.

Replacing $R$ by $R\langle Y_i^{1/p^\infty}|i\in I\rangle$ for some set $I$ and noting that this does not change the cotangent complexes, and thus not $\Prism_{S/R}$, we may assume that $R\to S$ is surjective. Choose a generating set $\{f_j\}_{j \in J}$ of the kernel of $R \to S$. Then the resulting map
\[ S^\prime = R \langle X_j^{1/p^\infty} | j \in J \rangle/(X_j - f_j | j \in J)^{\wedge} \to S\ : \ X_j \mapsto 0 \]
is surjective and also induces a surjection on cotangent complexes. Thus, by the Hodge-Tate comparison, $\Prism_{S^\prime/R}\to \Prism_{S/R}$ is surjective, so it suffices to show that the idempotent endomorphism of $\Prism_{S'/A}$ is the identity. But, by truncating the set $J$ of generators, we can express $S^\prime$ as a filtered colimit of regular semiperfectoid rings of the form
\[ 
T = R\langle X_1^{1/p^\infty},\ldots,X_r^{1/p^\infty}\rangle/(X_1-f_1,\ldots,X_r-f_r),
\]
so it suffices to prove that the idempotent endomorphism of $\Prism_{T/A}$ is the identity for each such $T$. Now if  $A' = A_{\inf}(R\langle X_1^{1/p^\infty},\ldots,X_r^{1/p^\infty} \rangle)$, then we also have $\Prism_{T/A} \simeq \Prism_{T/A'}$ compatibly with the idempotent endomorphisms.  Example~\ref{DerivedPrismaticPrismatic} then implies that the idempotent endomorphism $\Prism_{T/A'}$ is the identity, proving the claim.
\end{proof}

\subsection{Andr\'e's lemma}
\label{ss:AndreLemma}

\begin{proposition}[Lifting quasisyntomic covers to the prismatic site]
\label{PrismaticRefineQSyn}
Let $(A,I)$ be a bounded prism. Let $R$ be a quasisyntomic $A/I$-algebra.  Then there exists an object $(B \to B/IB \gets R) \in (R/A)_\Prism$ with $R \to B/IB$ being $p$-completely faithfully flat. Consequently:
\begin{enumerate}
\item The map $(A,I) \to (B,IB)$ is a flat map of prisms (and faithfully flat if $A/I \to R$ is $p$-completely faithfully flat). 
\item If $A$ is perfect, then the map $(A,I) \to (B,IB)_{\perf}$ is  flat as well (and faithfully flat if $A/I \to R$ is $p$-completely faithfully flat).
\end{enumerate}
\end{proposition}

\begin{proof}
Let us first construct a quasisyntomic cover $R \to S$ with $S/p$ relatively semiperfect over $A/(I,p)$, i.e., $S/p$ can be written as a quotient of a flat $A/(I,p)$-algebra with a bijective relative Frobenius; any such $S$ has $\Omega^1_{S/(A/I)}/p = 0$. Choose a surjection $A/I\langle x_j|j \in J\rangle\to R$, where $J$ is some index set, and let $S$ be the $p$-complete $R$-algebra obtained by formally extracting all $p$-power roots of the $x_j$'s, i.e., 
\[ S := A/I\langle x_j^{1/p^\infty}|j \in J\rangle\widehat{\otimes}^L_{A/I\langle x_j|j \in J\rangle} R.\]
It is easy to see that $R \to S$ is a quasisyntomic cover (so $S$ is quasisyntomic over $A/I$ as well), and that $A/(I,p) \to S/p$ is relatively semiperfect. In particular, the derived $p$-completion of the cotangent complex $L_{S/(A/I)}^{\wedge}[-1]$ is a $p$-completely flat $S$-module. The argument in the first half of \cref{QRSPPrism} repeats verbatim shows that $B = \Prism_{S/A}$ gives the required prism (using \cref{PrismaticCohGivesPrisms} to get the $\delta$-structure on $B$). The consequence in (1) follows because each map in the composition $A/I \to R \to S \to B/IB$ is $p$-completely flat (and $p$-completely faithfully flat provided $A/I \to R$ is so), and the consequence in (2) follows by passage to filtered colimits.
\end{proof}

Using the previous lemma, one can often construct covers of the final object of the prismatic site (for nice enough rings) via ``perfectoid'' constructions. We give two examples. 

\begin{example}[Covers of the final object of $(R/A)_\Prism$ via relative perfectoids for $R$ smooth]
\label{PerfectoidCoverFinalObject}
Let $(A,I)$ be a bounded prism. Assume $R$ is a $p$-completely smooth $A/I$-algebra. Let $R \to R_\infty$ be a quasisyntomic cover with $R_\infty$ being formally $p$-completely \'etale over $A/I$, i.e., $L_{R_\infty/(A/I)}$ vanishes after derived $p$-completion\footnote{Two examples of such covers are given as follows. First, if $(A,I)$ is a perfect crystalline prism, then there is an obvious choice for $R_\infty$: take $R_\infty$ to be the perfection of $R$. Secondly, given a choice $x_1,...,x_n \in R$ of \'etale co-ordinates, we may set $R_\infty = R[y_1^{1/p^\infty},...,y_n^{1/p^\infty}]^{\wedge}_p/(y_1-x_1,...,y_n-x_n)$ to be the derived $p$-complete $R$-algebra obtained by freely adjoining $p$-power roots of the $x_i$'s.}. Then $B := \Prism_{R_\infty/A}$ is a $(p,I)$-completely flat and relatively perfect $\delta$-$A$-algebra with the natural map inducing $R_\infty \simeq \overline{\Prism}_{R_\infty/A} = B/IB$. In particular, $(B,IB)$ gives an object of $(R/A)_\Prism$. We claim this object covers the final object. To see this, let $(C \to C/IC \gets R) \in (R/A)_\Prism$ be a test object. We must find a $(p,I)$-completely faithfully flat cover $(C,IC) \to (D,ID)$ together with a map $(B,IB) \to (D,ID)$ in $(R/A)_\Prism$. To find such a $D$, consider the base change $C/IC \to C/IC \widehat{\otimes}^L_R R_\infty$ of $R \to R_\infty$. This map is a quasisyntomic cover, so Proposition~\ref{PrismaticRefineQSyn} yields a faithfully flat map $(C,IC) \to (D,ID)$ of prisms together with a factorization $C/IC \to C/IC \widehat{\otimes}^L_R R_\infty \to D/IC$. In particular, the natural map $A/I to D/IC$ factors over the $p$-completely \'etale $A/I$-algebra $R_\infty$. By deformation theory, there is a unique $\delta$-$A$-algebra map $B \to D$ lifting $B/IB \simeq R_\infty \to D/IC$. The resulting map $R \to R_\infty \to D/ID$ (which equals $R \to C/IC \to D/ID$) then gives an object $(R \to D/ID \gets D) \in (R/A)_\Prism$. By construction, one has a faithfully flat map $(C,IC) \to (D,ID)$ of prisms. Moreover, the induced map $C/IC \to D/ID$ can be checked to be an $R$-algebra map by a diagram chase. Thus, $(D,ID) \in (R/A)_\Prism$ provides the desired object.
\end{example}

\begin{example}[Covers of the final object of the absolute prismatic site of a regular ring]
\label{PerfectoidCoverFinalObjectAbs}
Let $(A,I)$ be a bounded prism with $\phi:A \to A$ being $(p,I)$-completely flat. Let $(B,IB)$ be the perfection of $(A,I)$ in the sense of Lemma~\ref{PrismPerfection}. Then $B/IB$ is a perfectoid ring, and $(B,IB)$ provides an object of the absolute prismatic site $(A/I)_\Prism$ from Remark~\ref{AbsPrismaticSite}. An argument similar to that employed in Example~\ref{PerfectoidCoverFinalObject} shows that this object covers the final object. A natural example of such a prism $(A,I)$ is any prism $(A,I)$ with $A/I$ being a regular ring.
\end{example}

We give an application.

\begin{theorem}[Andr\'e's flatness lemma]
\label{AndreFlatness}
Let $R$ be a perfectoid ring. Then there exists a $p$-completely faithfully flat map $R \to S$ of perfectoid rings such that $S$ is absolutely integrally closed (i.e., every monic polynomial has a solution). In particular, every element of $S$ admits a compatible system of $p$-power roots.
\end{theorem}

In fact, the map $R \to S$ can also be chosen to be ind-syntomic modulo $p$, see Remark~\ref{AndreCoverfppf}.

\begin{proof}
Write $R = A/I$ for a perfect prism $(A,I)$. Let $\tilde{R}$ be the $p$-complete $R$-algebra obtained by formally adjoining roots of all monic polynomials over $R$. Then $\tilde{R}$ is a quasisyntomic cover of $R$, and each monic polynomial over $R$ has a solution in $\tilde{R}$. Proposition~\ref{PrismaticRefineQSyn} (2) then gives a perfect prism $(B,IB)$ over $(A,I)$ equipped with an $A/I$-algebra map $\tilde{R} \to B/IB =: R_1$ such that the composition $A/I=R \to \tilde{R} \to R_1$ is $p$-completely faithfully flat. In particular, $R_1$ is a perfectoid ring that is $p$-completely faithfully flat over $R$ such that every monic polynomial over $R$ admits a solution in $R_1$. Transfinitely iterating the construction $R\mapsto R_1$ yields a quasisyntomic cover $R \to R'$ with $R'$ being perfectoid and absolutely integrally closed.
\end{proof}

\begin{remark}[Refined form of Andr\'e's flatness lemma]
\label{AndreCoverfppf}
The map $R \to S$ constructed in Theorem~\ref{AndreFlatness} can be shown to be an ind-syntomic cover modulo $p$, as we now briefly sketch. Unwinding our construction, it suffices to show the following: for a monic equation $f(Y) \in R[Y]$, if we set $S := R\langle Y^{1/p^\infty}\rangle/(f(Y))$ over the prism $(A,I)$, then the resulting map $S\to T := \overline{\Prism}_S$ is an ind-syntomic cover modulo $p$. The ``cover'' part follows from \cref{QRSPPrism}. For ind-syntomicity, by base change, we may assume $R$ is $p$-torsionfree. Choose $\varpi \in R$ such that $(\varpi^p) = (p)$. By the definition of $T$ as a prismatic envelope and \cref{PDenvRegSeqLambda}, $\phi_A^* (T/\varpi)$ is the pd-envelope of $f(Y)$ in $R/\varpi[Y^{1/p^\infty}]$. The claim now follows by combining the following two facts:
\begin{enumerate}
\item If $A$ is an $\mathbf{F}_p$-algebra and $f \in A$ is a nonzerodivisor, then the natural map $A/f^p \to D_{(f)}(A)$ is identified with the natural map $A/f^p \to A/f^p[g_1,g_2,...]/(g_1^p, g_2^p,...)$ via $g_i \mapsto \gamma_{ip}(f)$; in particular, this map is ind-syntomic. This is a well-known fact about divided power algebras in characteristic $p$, and can be proven by reduction to the universal case where $A = \mathbf{F}_p[f]$.

\item If $A \to B = A/I$ is a surjection of $\mathbf{F}_p$-algebras with nilpotent kernel $I$, and $D$ is a flat $A$-algebra with $D/ID \simeq B[g_1,g_2,...]/(g_1^p,g_2^p,...)$, then $D$ is ind-syntomic over $A$. To see this, pick $\bar{h}_i \in D$ lifting $g_i$. Then we get an obvious surjective map $A[h_1,h_2,...] \to D$ by $h_i \mapsto \bar{h}_i$. As $\bar{h}_i^p \in ID$, we can find $\epsilon_i \in IA[h_1,h_2,....]$ such that $h_i^p - \epsilon_i$ maps to $0$ in $D$ for each $i$.  Setting $D_n = A[h_1,h_2,...]/(h_1^p - {\epsilon}_1, h_2^p - {\epsilon}_2,...,h_n^p - {\epsilon}_n)$, we get a map $\colim_n D_n \to D$; this map is an isomorphism as it is a surjection of flat $A$-algebras that is an isomorphism modulo the nilpotent ideal $I$. The claim follows as each $D_n$ is ind-syntomic over $A$.
\end{enumerate}
The observation that $R \to S$ can be chosen to be ind-syntomic modulo $p$ also follows from \cite[Theorem 16.9.17]{GabberRameroFART}.
\end{remark}

Finally, we can prove Theorem~\ref{PerfectoidificationSurj}.

\begin{proof} Recall that we want to prove that for any semiperfectoid ring $S$, the map $S\to S_\perfd$ is surjective, where $S_\perfd$ is the universal perfectoid ring with a map from $S$. Write $S=R/J$ for some perfectoid ring $R$. Using filtered colimits, we may assume that $J$ is finitely generated, and then by induction we may actually assume that $I=(f)$ is principal. By \cref{AndreFlatness}, we can assume that $f$ admits compatible $p$-power roots $f^{1/p^n}\in R$. But as any perfectoid ring $T$ is reduced, any map $S=R/f\to T$ automatically factors over the $p$-completion of $R/(f^{1/p^\infty})$, which is itself perfectoid. Thus, $S_\perfd$ is the $p$-completion of $R/(f^{1/p^\infty})$, and in particular the maps $R\to S\to S_\perfd$ are surjective.
\end{proof}

\newpage
\section{Perfections in mixed characteristic: General case}
\label{generalperfoidization}

The goal of this section is to study a notion of ``perfectoidiziation'' for any algebra over a perfectoid ring. In \S \ref{ss:PerfPrismatic}, we define this notion using derived prismatic cohomology, and then describe it via the perfect prismatic site. These results are then used in \S \ref{ss:arc} to reinterpret the perfectoidization as the sheafification of the structure sheaf for the \arc-topology (on $p$-adic formal schemes over a perfectoid ring), which leads to consequences such as a version of excision (\cref{PerfdBlowup}).

\subsection{Perfections via the perfect prismatic site}
\label{ss:PerfPrismatic}

\begin{notation}
Fix a perfectoid ring $R$ corresponding to a perfect prism $(A,I)$. For notational convenience, we fix a distinguished element $d \in I$ (which is automatically a generator).
\end{notation}

\begin{definition}[Perfectoidizations]
\label{def:perfectoidization}
Fix a derived $p$-complete simplicial $R$-algebra $S$.
\begin{enumerate}
\item The perfection $\Prism_{S/A,\perf}$ of $\Prism_{S/A}$ is defined as
\[ \Prism_{S/A,\perf} := \colim_{\phi_S} \big(\Prism_{S/A} \to \phi_* \Prism_{S/A} \to \phi^2_* \Prism_{S/A} \to ... \big)^{\wedge} \in D(A).\]
This is a derived $(p,I)$-complete $E_\infty$-$A$-algebra equipped with a $\phi_A$-semilinear automorphism induced by $\phi_S$. 

\item The perfectoidization $S_\perfd$ of $S$ is defined as 
\[ S_{\perfd} := \Prism_{S/A,\perf} \otimes_A^L A/I \in D(R).\]
This is a $p$-complete $E_\infty$-$S$-algebra.
\end{enumerate}
\end{definition}

A priori, $S_\perfd$ depends on the choice of perfectoid base ring $A$, but it follows from \cref{PerfdPerfectSite} below that it is actually independent of this choice. The same proposition will also show that Definition~\ref{def:perfectoidization} (2) is compatible with \cref{PerfectionMixedChar}.

\begin{example}[Perfectoidizations of characteristic $p$ rings]
\label{PerfdPerfCharp}
Assume that $S$ is an $R$-algebra where $p=0$. In this case, we claim that $S_\perfd$ coincides with the usual (direct limit) perfection $S_\perf$ of $S$. To see this, as $S_\perfd$ is independent of the choice of the perfectoid base ring, we may simply take $A = \mathbf{Z}_p$ so $R = \mathbf{F}_p$. But then the Frobenius equivariant map $S \to \Prism_{S/A}/p$ coming from the derived Hodge-Tate comparison induces an isomorphism on perfections: one checks (by, e.g., reduction to the case $S = \mathbf{F}_p[x]$) that the endomorphism of $\wedge^i L_{S/\mathbf{F}_p}[-i]$ induced by the Frobenius endomorphism of $\Prism_{S/A}$ via the Hodge-Tate comparison is the one induced by the Frobenius on $S$, and is thus $0$ for $i \neq 0$. Thus, we have $S_\perf \simeq \colim_{\phi_S} \Prism_{S/A}/p =: S_\perfd$, as asserted.
\end{example}

The goal of this section is to describe $S_\perfd$ via maps from $S$ to perfectoid rings.

\begin{lemma}[Coconnectivity of perfectoidizations]
\label{PerfdCoconnective}
For any derived $p$-complete simplicial $R$-algebra $S$, both $\Prism_{S/A,\perf}$ and $S_\perfd$ lie in $D^{\geq 0}$. 
\end{lemma}
\begin{proof}
We first show the claim for $\Prism_{S/A,\perf}$ using the $P^0$ operation for $E_\infty$-$\mathbf{F}_p$-algebras. As the levelwise Frobenius map on a simplicial cosimplicial $\mathbf{F}_p$-algebra induces the $P^0$-operation on its realization, we learn that $\Prism_{S/A,\perf}/p$ is an $E_\infty$-$\mathbf{F}_p$-algebra where $P^0$ acts invertibly. As $P^0$ kills all positive degree homotopy classes \cite[Remark 2.2.7]{LurieRationalpadic}, it follows that $\Prism_{S/A,\perf}/p \in D^{\geq 0}$, so $\Prism_{S/A,\perf} \in D^{\geq 0}$ as well. 

To deduce that $S_\perfd \in D^{\geq 0}$, it is enough to check that $d$ is a nonzerodivisor in $H^0(\Prism_{S/A,\perf})$. As $\Prism_{S/A,\perf}/p \in D^{\geq 0}$, the ring $H^0(\Prism_{S/A,\perf})$ must be $p$-torsionfree. The reasoning used above also shows that the automorphism $\phi_S$ of $H^0(\Prism_{S/A,\perf})$ lifts the Frobenius on $H^0(\Prism_{S/A,\perf})/p \subset H^0(\Prism_{S/A,\perf}/p)$. It follows that $H^0(\Prism_{S/A,\perf})$ is a perfect $p$-complete $\delta$-ring, and $\phi_S$ is the unique Frobenius lift on such a ring. In particular, the map $A \to H^0(\Prism_{S/A,\perf})$ is a map of $\delta$-rings, so it carries $d$ to a distinguished element. The claim now follows from \cref{PrimEltPRop}.
\end{proof}

Fix a derived $p$-complete simplicial $R$-algebra $S$ as above. Our next goal is to describe $\Prism_{S/A,\perf}$ in terms of the perfect prismatic site $(S/A)_\Prism^\perf$ defined (as in Remark~\ref{PerfectPrismatic}) as the category of perfect prisms $(B,IB)$ over $(A,I)$ endowed with an $A/I$-algebra map $S \to B/IB$; thus, $(S/A)_\Prism^{\perf} = (\pi_0(S)/A)_\Prism^{\perf}$. Moreover,  by \cref{PerfectPrismInitial},  the site $(S/A)_\Prism^{\perf}$  is also simply the site of all maps $S \to T$ with $T$ a perfectoid ring. Thus, $(S/A)_\Prism^{\perf}$ is independent of $(A,I)$ and the structure map $A/I \to S$. By \cref{PrismaticCohGivesPrisms}, for $(B,J)\in (S/A)_\Prism^\perf$, we have a functorial map
\[
\Prism_{S/A}\to B\ ,
\]
which after passing to the perfection on the left and the limit over all $B$ on the right induces a comparison map
\[
\Prism_{S/A,\perf}\to R\Gamma((S/A)_\Prism^\perf,\mathcal O_\Prism)\ .
\]

\begin{proposition}\label{PerfdPerfectSite} For any choice of topology on $(S/A)_\Prism^\perf$ between the flat and the chaotic topology, the map
\[
\Prism_{S/A,\perf}\to R\Gamma((S/A)_\Prism^\perf,\mathcal O_\Prism)
\]
is an equivalence. In particular, the perfectoidization
\[
S_\perfd\cong R\Gamma((S/A)_\Prism^\perf,\overline{\mathcal O}_\Prism)\cong \lim_{S\to R^\prime} R^\prime
\]
is the derived limit of $R^\prime$ over all maps from $S$ to perfectoid rings $R^\prime$, and does not depend on the choice of $A$.
\end{proposition}

Proposition~\ref{PerfdPerfectSite} implies that when $S$ is semiperfectoid, $S_\perfd$ agrees with the perfectoidization from Corollary~\ref{PerfectionMixedChar}, and that in general $\Prism_{S/A,\perf}$ and $S_\perfd$ depend only on $\pi_0 S$ (and not on $A$).

\begin{proof} First we check independence of the choice of topology. We claim first that the site $(S/A)_\Prism^\perf$ admits a weakly final object. We can evidently assume that $S$ is discrete. It suffices to prove that there is a map $S\to \tilde{S}$ to a semiperfectoid ring $\tilde{S}$ such that any map from $S$ to a perfectoid ring $R^\prime$ extends to $\tilde{S}$, for then $\tilde{S}_\perfd$ is the desired weakly initial object. For this, we inductively take any element $x\in S$ and pass to $S_x=S\langle X,Y\rangle/(x-X^p-pY)$. This adjoins a $p$-th root of $x$ modulo $p$, and any map $S\to R^\prime$ to a perfectoid ring extends to $S_x$. Repeating this transfinitely gives the desired ring $\tilde{S}$; here we implicitly use the fact (from \cite[Remark 4.22]{BMS2}) that any $p$-complete $R$-algebra $T$ with $T/pT$ semiperfect is itself semiperfectoid as $T$ is a  quotient of the perfectoid ring $R \widehat{\otimes}_{\mathbf{Z}_p} W(T^\flat)$. Then by \v{C}ech-Alexander theory and flat descent, $R\Gamma((S/A)_\Prism^\perf,\mathcal O_\Prism)$ is computed by the cosimplicial perfectoid $S$-algebra $(\tilde{S}^\bullet)_\perfd$, where $\tilde{S}^i = (\tilde{S}^{\otimes_S (i+1)})^\wedge$ is the $p$-completed tensor product of $i+1$ copies of $\tilde{S}$ over $S$, independently of the choice of topology, giving the desired independence.

We now work with the flat topology. First, let us restrict to the case where $\pi_0(S)$ is topologically finitely generated over $A/I$. For this, it is enough to show that both functors $S \mapsto S_\perfd$ and $S \mapsto R\Gamma((S/A)_\Prism^\perf,\mathcal O_\Prism)$ commute with filtered colimits (when viewed as valued in $(p,I)$-complete objects in $\mathcal{D}(A)$). This is clear for $S \mapsto S_\perfd$ as the analogous property holds true for derived prismatic cohomology. For the functor $S \mapsto R\Gamma((S/A)_\Prism^\perf,\mathcal O_\Prism)$, we may assume $S$ is discrete (since $(S/A)_\Prism^\perf = (\pi_0(S)/A)_\Prism^\perf$).  In that case, we can also compute the cohomology by passing to any flat semiperfectoid cover $S\to \tilde{S}$ and taking the similar \v{C}ech-Alexander complex. In particular, we can choose $\tilde{S}$ functorially by extracting all $p$-power roots of all elements of $S$. Using such functorial \v{C}ech-Alexander complexes, one verifies that $R\Gamma((S/A)_\Prism^\perf,\mathcal O_\Prism)$ commutes with filtered colimits in $S$. Thus, we may assume that $\pi_0 S$ is topologically finitely generated over $A/I$.

Now choose a map $R\langle X_1,\ldots,X_r\rangle\to S$ that is surjective on $\pi_0$. Let $\tilde{S}$ be the $p$-completion of $R[X_1^{1/p^\infty},\ldots,X_r^{1/p^\infty}] \otimes_{R[X_1,\ldots,X_r]} S$; then $\pi_0 \tilde{S}$ is semiperfectoid. Let $\tilde{S}^\bullet$ be the derived $p$-completion of $\tilde{S}^{\otimes_S (i+1)}$. By Lemma~\ref{DescentLemma} below, $\Prism_{S/A}$ can be computed as the limit of $\Prism_{\tilde{S}^\bullet/A}$. By the similar descent for the right-hand side, we can reduce to the case that $\pi_0 S$ is semiperfectoid.

Now assume that $\pi_0 S$ is semiperfectoid. Let $R^\prime$ be the perfectoidization of $\pi_0 S$ in the sense of the last section, i.e.~the universal perfectoid ring to which $S$ maps. We need to see that
\[
\Prism_{S/A,\perf}\otimes_A^L A/I\to R^\prime
\]
is an isomorphism. Note that the left-hand side is always coconnective by \cref{PerfdCoconnective}. On the other hand, as $S$ is semiperfectoid and so $L_{S/\mathbf Z_p}[-1]$ and thus all $\wedge^i L_{S/\mathbf Z_p}[-i]$ are connective, it follows from the Hodge-Tate comparison that $\Prism_{S/A}$ and thus also $\Prism_{S/A,\perf}$ are connective. Thus, $\Prism_{S/A,\perf}$ is in fact a perfect $\delta$-ring, and so $S_\perfd = \Prism_{S/A,\perf}\otimes_A^L A/I$ is a perfectoid ring. We have maps
\[
S\to S_\perfd\to R^\prime
\]
where the composite is surjective by \cref{PerfectoidificationSurj}. By the universal property of $R^\prime$, the second map is split surjective. Thus, it suffices to show that $S\to S_\perfd$ is surjective. 

As $S/p$ is semiperfect, we may replace $R$ with $R \langle Y_k^{1/p^\infty} \mid k \in K \rangle$ (for suitable index set $K$), we may assume $R \to S$ is surjective. Let $\{f_j\}_{j \in J}$ be a set of generators of $R \to S$. As the formation of $S \mapsto S_\perfd$ commutes with $p$-completely flat base change on $R$, we may assume by \cref{AndreFlatness} that each $f_j$ admits a compatible system $\{f_j^{1/p^n}\}_{n \geq 0}$ of $p$-power roots in $R$. Fixing such a system for each $f_j$ gives a map 
\[ S' := R \langle X_j^{1/p^\infty} \mid j \in J \rangle/(X_j \mid j \in J)^{\wedge} \to S \ : X_j^{1/p^n} \mapsto f_j^{1/p^n} \]
that is surjective on $\pi_0$ and  induces a surjection on $\pi_1 (L_{-/\mathbf Z_p}/p)$. By the Hodge-Tate comparison theorem, the map $S'_\perfd \to S_\perfd$ is also surjective, so it suffices to prove that $S \to S_\perfd$ is surjective when $S = S'$. By passage to filtered colimits, we may also assume $J$ is finite.  But then \cref{QRSPPrism} and the construction of $R^\prime$ show that indeed $S_\perfd=R^\prime$, as desired.
\end{proof}

We used the following lemma, which relies on Mathew's notion of descendability from \cite{MathewGalois}; see also \cite[\S 11.2]{BhattScholzeWitt} for a quick review.

\begin{lemma}[Strong descent for passage to the perfection]
\label{DescentLemma}
Let $T = R[X_1,...,X_n]^{\wedge}$, and let $T_\infty = R[X_1^{1/p^\infty},...,X_n^{1/p^\infty}]^{\wedge}$. The natural maps $\Prism_{T/A} \to \Prism_{T_\infty/A}$ and $\Prism_{T/A,\perf} \to \Prism_{T_\infty/A,\perf}$ are descendable as maps of commutative algebras in $\mathcal{D}_{comp}(A)$.
\end{lemma}

Recall that $\mathcal{D}_{comp}(A)$ is the full subcategory of $\mathcal{D}(A)$ spanned by derived $(p,I)$-complete objects.

\begin{proof}
The descendability of $\Prism_{T/A,\perf} \to \Prism_{T_\infty/A,\perf}$ follows from that of $\Prism_{T/A} \to \Prism_{T_\infty/A}$ and a general categorical fact: if $B \to C$ is a descendable map of commutative algebras of index $\leq n$ and both $B$ and $C$ come equipped with compatible endomorphisms $\phi$, then the $\colim_\phi B \to \colim_\phi C$ is  descendable of index $\leq 2n$ (see \cite[Lemma 11.22]{BhattScholzeWitt}).

We now prove the descendability of $\Prism_{T/A} \to \Prism_{T_\infty/A}$. As the functor $\Prism_{-/A}$ from $p$-complete simplicial commutative $R$-algebras to commutative algebras in $\mathcal{D}_{comp}(A)$ commutes with coproducts, we may assume $n=1$. Let $F \in \mathcal{D}(\Prism_{T/A})$ denote the fibre of $\Prism_{T/A} \to \Prism_{T_\infty/A}$. It suffices to show all maps $F^{\otimes 2} \to \Prism_{T/A}$ in $\mathcal{D}(\Prism_{T/A})$ are null-homotopic (where the tensor product is in $\mathcal{D}(\Prism_{T/A})$). For this, it is enough to show that 
\[ \mathrm{RHom}_{\Prism_{T/A}}(F^{\otimes 2}, \Prism_{T/A}) \in D^{< 0}.\]
By completeness, it suffices to check the same assertion after base change along $A \to A/(p,\phi^{-1}(d))$. In particular, we may assume $R$ is a perfect ring of characteristic $p$, so $T = R[X]$ and $T_\infty = R[X^{1/p^\infty}]$ is its perfection. Using the de Rham comparison and the Cartier isomorphism, we are reduced to checking the following:

\begin{itemize}
\item[$(\ast)$] Let $\Omega := T \oplus \Omega^1_{T/R}[-1]$, regarded as a $T$-cdga that is a split square zero extension of $T$ by $\Omega^1_{T/R}[-1]$. Let $F$ be the fibre of the canonical map $\Omega \to T \to T_\infty$. Then
\[ \mathrm{RHom}_{\Omega}(F^{\otimes 2}, \Omega) \in D^{< 0},\]
where the tensor product is in $\mathcal{D}(\Omega)$.
\end{itemize}

The fibre of $\Omega \to T$ is $\Omega^1_{T/R}[-1]$, which is isomorphic to $T[-1]$ as a dg-module over $\Omega$ via the choice of the generator $dX \in \Omega^1_{T/R}$. Also, $T_\infty$ is free when regarded as a $T$-module with one of the generators being $1 \in T_\infty$, so the fibre of $T \to T_\infty$ identifies with $T^{\oplus I}[-1]$ where $I$ a set. Combining, we find that the fibre $F$ of $\Omega \to T_\infty$ is isomorphic to an $\Omega$-module of the the form $T^{\oplus J}[-1]$ for some set $J$. As $\mathrm{RHom}(-,-)$ converts coproducts in the first factor into products, it is enough to check that $\mathrm{RHom}_{\Omega}( (T[-1])^{\otimes 2},\Omega) \in D^{< 0}$ (where the tensor product is in $\mathcal{D}(\Omega)$). This is a standard calculation using the resolution of the dg-module $T[-1]$ over the dg-algebra $\Omega$ given by
\[ \Big(\ldots\to \Omega[-3] \to \Omega[-2] \to \Omega[-1]\Big) \xrightarrow{\sim} T[-1],\]
where all transition maps are determined by $dX \in \Omega^1_{T/R}$.
\end{proof}

\subsection{\arc-descent}
\label{ss:arc}

In this section, we use Proposition~\ref{PerfdPerfectSite} to deduce some very strong descent results for the association $S\mapsto S_\perfd$, formulated in terms of the \arc-topology from \cite{BhattMathew}\footnote{We are trying to achieve optimal results here and thus use the \arc-topology. Getting similar results for the slightly weaker $v$-topology would be enough for the applications below, and could avoid reference to \cite{BhattMathew}.}.

\begin{definition}[The \arc-topology of $p$-adic formal schemes] 
Consider the category $\mathrm{fSch}$ of $p$-adic formal schemes $X$. Endow $\mathrm{fSch}$ with the structure of a site by declaring a map $Y\to X$ of qcqs $p$-adic formal schemes to be an \arc-cover if for all $p$-adically complete valuation rings $V$ of rank $1$ with a map $\Spf V\to X$ there is a faithfully flat extension $V\subset W$ of $p$-adically complete valuation rings of rank $1$ and a lift to a map $\Spf W\to Y$, and taking the topology generated by \arc-covers and open covers. We will refer to this as the {\em \arc-site} on $\mathrm{fSch}$. 

Let $\mathrm{Perfd}\subset \mathrm{fSch}$ be the full subcategory of $p$-adic formal schemes that are locally of the form $\Spf R$ for some perfectoid ring $R$, with the induced topology. 
\end{definition}

Perfectoids form a basis for the \arc-topology:

\begin{lemma}
\label{affperfdbasis}
Any $X\in \mathrm{fSch}$ admits a \arc-cover by some $Y\in \mathrm{Perfd}$.
\end{lemma}
\begin{proof} 
We may assume $X=\Spf R$ is affine. Choose a $p$-completely faithfully flat map $R\to \tilde{R}$ such that $\tilde{R}$ is semiperfectoid, so $\Spf \tilde{R} \to \Spf R$ is an \arc-cover. Moreover, the map $\Spf \tilde{R}_\perfd\to \Spf \tilde{R}$ is also an \arc-cover as any map $\tilde{R}\to V$ to a $p$-adically complete valuation ring factors over $\tilde{R}_\perfd$: one can assume that $V$ is perfectoid by replacing it by the $p$-completion of an absolute integral closure, and the universal property of $\tilde{R} \to \tilde{R}_\perfd$ gives the desired extension. The composite $\Spf \tilde{R}_{\perfd} \to \Spf R$ is then the desired $\arc$-cover.
\end{proof}

\begin{remark}
\label{ArcCoverAICVR}
In fact, one can find an even simpler basis that that in Lemma~\ref{affperfdbasis}: any $p$-complete ring $R$ admits a $p$-complete \arc-cover $R \to S$ where $S$ is a product of $p$-complete rank $1$ valuation rings with algebraically closed fraction field (and thus $S$ is perfectoid). Indeed, choose a set $X_R$ of representatives of rank $1$ valuations on $R$,  and take $S = \prod_{x \in X_R} \widehat{V_x^+}$, where $V_x \to V_x^+$ is an absolute integral closure of the valuation ring $V_x \subset \kappa(\ker(x))$ attached to $x \in X_R$ and the completion is $p$-adic. Such covers shall be used in \S \ref{sec:EtaleComp}.
\end{remark}

The structure sheaf on $\mathrm{fSch}$ is obviously not an \arc-sheaf: nilpotent elements of a $p$-complete ring are not detected by mapping into valuation rings. On the other hand, it turns out that such problems disappear on the basis $\mathrm{Perfd} \subset \mathrm{fSch}$:

\begin{proposition}
\label{arcvanishingaffperfd}
 The presheaf $\mathcal O$ on $\mathrm{Perfd}$ is a sheaf, and if $X=\Spf R$ is affine with $R$ perfectoid, then $H^i_\arc(X,\mathcal O)=0$ for $i>0$.

For any $p$-adic formal scheme $T$, the functor $X\mapsto \Hom(X,T)$ is an \arc-sheaf on $\mathrm{Perfd}$.
\end{proposition}

\begin{proof} By Zariski descent, one can reduce to case of affine $X$. For the first statement, it suffices to prove that if $R\to S$ is a map of perfectoid rings such that $\Spf S\to \Spf R$ is an \arc-cover, the complex
\[
0\to R\to S\to S\widehat{\otimes}_R S\to \ldots
\]
is acyclic. It is enough to prove that the complex
\[
0\to R^\flat\to S^\flat\to S^\flat\widehat{\otimes}_{R^\flat} S^\flat\to \ldots
\]
is acyclic where the completion is now $d$-adic, where $d$ generates $\ker(A_\inf(R)\to R)$. Indeed, this implies that the same is true after applying Witt vectors to this sequence, and then also for the quotient modulo $d$.

Now we claim that $\Spec(S)^\flat\sqcup \Spec(R)^\flat[\frac 1d]\to \Spec(R)^\flat$ is an \arc-cover in the sense of \cite[Definition 1.2]{BhattMathew}. Take any map $R^\flat\to V$ to a rank $1$-valuation ring; we can assume that $V$ is perfect. If the image of $d$ in $V$ is invertible, then it lifts to a map $R^\flat[\frac 1d]\to V$. Otherwise, we may replace $V$ by its $d$-adic completion. Then the map $R^\flat\to V$ is equivalent to the map $R\to V^\sharp$ to a $p$-adically complete valuation ring, which by assumption can be lifted to a map $S\to W$ for some $p$-adically complete extension $W$ of $V^\sharp$ which we may assume to be perfectoid. Then $S^\flat\to W^\flat$ gives the desired lift.

By $v$-descent, \cite[Theorem 4.1]{BhattScholzeWitt}, and the characterization \cite[Theorem 1.6]{BhattMathew} of \arc-descent in terms of $v$-descent and a property for valuation rings that in the present case is given by \cite[Lemma 6.3]{BhattScholzeWitt}, for any \arc-cover $\Spec(B)\to \Spec(A)$ of perfect schemes, the complex
\[
0\to A\to B\to B\otimes_A B\to \ldots
\]
is exact. Applying this to $A=R^\flat$, $B=S^\flat\times R^\flat[\frac 1d]$, and passing to (derived) $d$-adic completions, all terms involving $R^\flat[\frac 1d]$ disappear, and we see that indeed
\[
0\to R^\flat\to S^\flat\to S^\flat\widehat{\otimes}_{R^\flat} S^\flat\to \ldots
\]
is exact.

If $T=\Spf S$ is affine, it is immediate from the preceding that $\Hom(X,T)$ is a sheaf on $\mathrm{Perfd}$. In general, assume that $Y\to X$ is an \arc-cover and we have given a map $Y\to T$ such that the two induced maps $Y\times_X Y\to T$ agree. It suffices to see that $|Y|\to |T|$ factors over a continuous map $|X|\to |T|$, as then one can localize on $T$ and assume that $T$ is affine. As the topological space in question depends only on the special fibre, we can assume that $T$ is of characteristic $p$, and then $X$ and $Y$ are perfect schemes. By Zariski descent, it suffices to solve the problem when $X$ and $Y$ are affine. We may then assume that $T$ is quasicompact, and by an argument of Gabber, cf.~\cite[Remark 4.6]{BhattTannaka}, \cite[Tag 03K0]{Stacks}, we can assume that $T$ is also quasiseparated. Then $T$ can be written as an inverse limit of finite type $\mathbf F_p$-schemes along affine transition maps by \cite[Appendix C]{ThomasonTrobaugh}, so we can assume that $T$ is of finite type. In that case we can also assume that $X$ and $Y$ are perfections of schemes of finite type over $\mathbf F_p$. But then $X\to Y$ is an \arc-cover only if it is an h-cover, in which case $|X|\to |Y|$ is a quotient map, giving the desired continuous map $|Y|\to |T|$.
\end{proof}

\begin{corollary}\label{PerfdVCohom} For any $p$-complete ring $S$, one has
\[
S_\perfd = R\Gamma_\arc(\Spf S,\mathcal O)\ .
\]
In particular, the association $S\mapsto S_\perfd$ satisfies descent for the \arc-topology.
\end{corollary}

\begin{proof} Let $\mathrm{affPerfd}\subset \mathrm{Perfd}$ be the subcategory of $\Spf R$ with $R$ perfectoid. As $\mathrm{affPerfd}\subset \mathrm{fSch}$ is a basis (by proof of \cref{affperfdbasis}), one has
\[
R\Gamma_\arc(\Spf S,\mathcal O) = R\Gamma_\arc(\mathrm{affPerfd}_{/\Spf S},\mathcal O)\ .
\]
On the other hand, the projection from $\mathrm{affPerfd}$ with its \arc-topology to $\mathrm{affPerfd}$ with the chaotic topology has trivial higher direct images for $\mathcal O$ by \cref{arcvanishingaffperfd}, giving
\[
R\Gamma_\arc(\Spf S,\mathcal O) = \lim_{S\to R^\prime} R^\prime = S_\perfd\ ,
\]
as desired.
\end{proof}

\begin{corollary}\label{PerfdBlowup} Let $S\to S^\prime$ be a map of derived $p$-complete rings and assume that for some derived $p$-complete ideal $I\subset S$, the map $S\to S^\prime$ induces an isomorphism of \arc-sheaves outside of $I$ in the sense that for any $p$-complete valuation ring $V$ of rank $1$ with a map $S\to V$ that does not kill $I$, there is a unique extension to a map $S^\prime\to V$.\footnote{This holds, for example, if $S\to S^\prime$ is an integral map and $\Spec(S^\prime)\to \Spec(S)$ is an isomorphism outside $V(I)$.} Then
\[\xymatrix{
S_\perfd\ar[r]\ar[d] & S^\prime_\perfd\ar[d]\\
(S/I)_\perfd\ar[r] & (S^\prime/I)_\perfd
}\]
is a pullback square of $E_\infty$-$S$-algebras (in particular, in the derived category of $S$-modules).
\end{corollary}

We note that even when $S$ and $S^\prime$ are perfectoid, the result is interesting, and in particular ensures that the kernel and cokernel of $S\to S^\prime$ are killed by $\ker(S\to (S/I)_\perfd)$.

\begin{proof} Replacing $S^\prime$ by $S^\prime\times S/I$, we may assume that $\Spf S^\prime\to \Spf S$ is an \arc-cover. Now we follow the proof of \cite[Theorem 2.9]{BhattScholzeWitt} (cf.~\cite[Lemma 3.6]{VoevodskyHTop}). Let $\mathcal F_T$ be the sheaf $X\mapsto \Hom(X,\Spf T)$ on $\mathrm{Perfd}$ for any derived $p$-complete ring $T$. By Corollary~\ref{PerfdVCohom}, it suffices to see that
\[
\mathcal F_S = \mathcal F_{S^\prime}\sqcup_{\mathcal F_{S^\prime/I}} \mathcal F_{S/I}
\]
as sheaves of spaces on the site $\mathrm{Perfd}$. But as $\mathcal F_{S^\prime/I}\hookrightarrow \mathcal F_{S^\prime}$, this pushout is still discrete, so we can compute it in the category of sheaves of sets. In other words, we need to see that for any sheaf of sets $F$ on $\mathrm{Perfd}$, we have
\[
F(S) = F(S^\prime)\times_{F(S^\prime/I)} F(S)\ ,
\]
where $F$ is extended to $\mathrm{fSch}$ by descent. For this, we use descent along $\Spf S^\prime\to \Spf S$. Note that
\[
\Spf S^\prime\sqcup \Spf (S^\prime/I)\times_{\Spf (S/I)} \Spf (S^\prime/I)\to \Spf S^\prime\times_{\Spf S}\Spf S^\prime
\]
is an \arc-cover by the assumption on $I$. Unraveling, this gives the result.
\end{proof}

On the other hand, a main advantage of the definition of $S_\perfd$ in terms of derived prismatic cohomology is that it allows us to prove the following result.

\begin{proposition}\label{PerfdSymmMon} The functor $S\mapsto S_\perfd$ from derived $p$-complete simplicial rings admitting a map from a perfectoid ring, to derived $p$-complete $E_\infty$-rings is symmetric monoidal, i.e.~for any diagram $S_1\gets S_3\to S_2$, the induced map
\[
S_{1,\perfd}\widehat{\otimes}^L_{S_{3,\perfd}} S_{2,\perfd}\to (S_1\widehat{\otimes}^L_{S_3} S_2)_\perfd
\]
is an equivalence.
\end{proposition}

\begin{proof} Pick a perfectoid ring $A$ mapping to $S_3$. Then derived prismatic cohomology $\Prism_{-/A}$ is symmetric monoidal by the Hodge-Tate comparison (see property (5) in Construction~\ref{DerivedPrismatic} for a more general statement), and this passes to the perfection, and then to the quotient modulo $I$.
\end{proof}

\begin{corollary}
\label{PerfdDiscreteUniversal}
Say $S$ is a derived $p$-complete simplicial ring admitting a map from a perfectoid ring. Assume that $S_\perfd$ is connective (or, equivalently by \cref{PerfdCoconnective}, concentrated in degree $0$). Then $S_\perfd$ is a perfectoid ring, and $S \to S_\perfd$ is the universal map from $S$ to a perfectoid ring.
\end{corollary}

\begin{proof}
If $S_\perfd$ is discrete, then so is  $\Prism_{S/A,\perf}$ by derived Nakayama. But then $\Prism_{S/A,\perf}$ is a perfect $p$-complete $\delta$-$A$-algebra by the proof of \cref{PrismaticCohGivesPrisms}. The pair $(\Prism_{S/A,\perf},I\Prism_{S/A,\perf})$ is then necessarily a perfect prism (\cref{PrimEltPRop}), so $S_\perfd$ is a perfectoid ring. To prove the universality of $S \to S_\perfd$, we must show that any map $S_\perfd \to S'$ with $S'$ perfectoid is uniquely determined by the composition $S \to S_\perfd \to S'$; equivalently, we must check that any map $\eta:S_\perfd \widehat{\otimes}^L_S S_\perfd \to S'$ with $S'$ perfectoid factors (necessarily uniquely) over the multiplication map $S_\perfd \widehat{\otimes}^L_S S_\perfd \to S_\perfd$. As $S'$ is perfectoid, such a map $\eta$ factors over $S_\perfd \widehat{\otimes}^L_S S_\perfd \to (S_\perfd \widehat{\otimes}^L_S S_\perfd )_\perfd$ by applying the functor $(-)_\perfd$ since $S' \simeq S'_\perfd$. Now $(S_\perfd \widehat{\otimes}^L_S S_\perfd )_\perfd \simeq S_\perfd$  via the multiplication map thanks to \cref{PerfdSymmMon}, so we are done.
\end{proof}

\newpage

\section{The \'etale comparison theorem}
\label{sec:EtaleComp}

In this section, we explain how to recover \'etale cohomology of the generic fibre from prismatic cohomology for arbitrary rings.

\begin{theorem}[The \'etale comparison theorem]
\label{EtaleCompThm}
Fix a perfectoid ring $R$ corresponding to a perfect prism $(A,(d))$, and a $p$-adic formal scheme $X$ over $R$. Write $X_\eta := X \times_{\mathrm{Spf}(R)} \mathrm{Spa}(R[1/p],R)$ for the adic generic fibre of $X$, and let $\mu:X_{\eta,\et} \to X_\et$ denote the ``nearby cycles'' map. There is a canonical identification
\[ R\mu_* \mathbf{Z}/p^n \simeq (\Prism_{X/A}[1/d]/p^n)^{\phi=1}.\]
In particular, if $X = \mathrm{Spf}(S)$ is affine, then there is a canonical identification
\begin{equation}
\label{etalecomp}
 R\Gamma(\mathrm{Spec}(S[1/p]), \mathbf{Z}/p^n) \simeq (\Prism_{S/A}[1/d]/p^n)^{\phi=1}
\end{equation}
for each $n \geq 0$.
\end{theorem}

Note that there are no restrictions on the singularities of $X$ above. We shall apply this in Theorem~\ref{TateTwistPerfd} to relate the $p$-adic Tate twists introduced in \cite{BMS2} to the usual Tate twists on the generic fibre for perfectoid rings. The proof of \cref{EtaleCompThm} will rely on a slightly surprising compatibility of the formation of $\phi$-fixed points with completed colimits.

\begin{lemma}[Commuting fixed points with completed colimits]
\label{PhiFixedColimit}
Let $B$ be an $\mathbf{F}_p$-algebra equipped with an element $t \in B$. Let $\mathcal{D}(B[F])$ be the $\infty$-category of pairs $(M,\phi)$ where $M \in \mathcal{D}(B)$ and $\phi:M \to \phi_* M$ is a map. Let $\mathcal{D}_{comp}(B[F]) \subset \mathcal{D}(B[F])$ be the full subcategory spanned by pairs $(M,\phi)$ with $M$ derived $t$-complete. Then the functors $\mathcal{D}_{comp}(B[F]) \to \mathcal{D}(\mathbf{F}_p)$ given by $M \mapsto M^{\phi=1}$ and $M \mapsto (M[1/t])^{\phi=1}$ commute with colimits.
\end{lemma}
\begin{proof}
In this proof, all colimits refer to colimits in the underlying category $\mathcal{D}(B)$. Let $\{(M_i, \phi_i)\}$ be a diagram in $\mathcal{D}_{comp}(B[F])$. The fiber of the map $\colim_i M_i \to \widehat{\colim_i M_i}$ is uniquely $t$-divisible, and thus identifies with the fiber of the map $\colim_i M_i[1/t] \to (\widehat{\colim_i M_i})[1/t]$. As the lemma can be reformulated as the statement that $F^{\phi=1}=0$, it suffices to prove the statement for the functor $M \mapsto M^{\phi=1}$. For this, we claim that for any $(N,\phi) \in \mathcal{D}_{comp}(B[F])$, we have $N^{\phi=1} \simeq (N/t)^{\phi=1}$. First, this makes sense because we have $\phi(t) = t^p \subset tB$, so for any $N \in \mathcal{D}(B[F])$, there is an induced $\phi$-structure on $N/t$ compatible with the one on $N$. Secondly, if $N$ is derived $t$-complete, then the fibre of $N \to N/t$ is complete when endowed with the $t$-adic filtration, and the $\phi$-action on the fibre is topologically nilpotent with respect to this filtration as $\phi(t) = t^p \subset t^2 B$, so $(-)^{\phi=1}$ must vanish on the fibre, giving $N^{\phi=1} \simeq (N/t)^{\phi=1}$, as asserted.  The lemma now follows because both functors in the composition 
\[ \mathcal{D}_{comp}(B[F]) \xrightarrow{N \mapsto N/t} \mathcal{D}(B[F]) \xrightarrow{ (-)^{\phi=1} } \mathcal{D}(\mathbf{F}_p)\]
commute with all colimits.
\end{proof}

\begin{proof}[Proof of \cref{EtaleCompThm}]
Consider the following two functors $F$ and $G$ on affine formal schemes $\Spf S\in \mathrm{fSch}_{/\Spf R}$:
\[ F(S) := R\Gamma(\Spec(S[\frac{1}{p}]),\mathbf Z/p^n) \quad \text{and} \quad G(S) := (\Prism_{S/A}[\frac{1}{d}]/p^n)^{\phi=1}.\] 
Using the $\arc_p$-topology from \cite[Definition 6.14]{BhattMathew}, we shall build a comparison map $F \to G$ and then check it is an isomorphism.

First, we note that $F$ is an $\arc_p$-sheaf by \cite[Corollary 6.17]{BhattMathew}. Moreover, $F$ is $\arc_p$-locally concentrated in degree $0$ by arguing with valuation rings\footnote{It suffices to show that $F(S)$ is in degree $0$ when $S=\prod_i V_i$ is a product of $p$-complete  rank $1$ valuation rings with algebraically closed residue field. Recall that any higher \'etale cohomology class on any affine scheme can be annihilated by pullback along finitely presented finite covers. Now any finitely presented finite cover of $S[\frac{1}{p}]$ spreads out to a finitely presented finite cover of $S$. By the assumption on each $V_i$, any such cover of $S$ admits a section, so the same holds for the original cover of $S[\frac{1}{p}]$, whence all higher \'etale cohomology classes on $S[\frac{1}{p}]$ must be trivial.}, so it coincides with the $\arc_p$-sheafification of $H^0(F(-))$. Finally, for any $\Spf(S) \in \mathrm{fSch}_{/\Spf(R)}$, any clopen decomposition of $\mathrm{Spec}(S[\frac{1}{p}])$ can be lifted to a clopen decomposition of $\Spf(S)$ at the expense of replacing $S$ by an $\arc_p$-equivalent (see \cite[Definition 6.19]{BhattMathew}) subring of $S[\frac{1}{p}]$ by simply adjoining the relevant idempotents. As the Zariski topology on $\mathrm{fSch}_{/\Spf(R)}$ is coarser than the $\arc_p$-topology, it follows that we get a natural comparison map $H^0(F(-)) \to H^0_{\arc_p}(-, \mathbf{Z}/p^n)$ on $\mathrm{fSch}_{/\Spf(R)}$ (this map is in fact an isomorphism, but we do not need that here). As $F$ is the $\arc_p$-sheafification of its $H^0$, the preceding map extends extends uniquely to a natural transformation $F \to R\Gamma_{\arc_p}(-, \mathbf{Z}/p^n)$ of $\arc_p$-sheaves on $\mathrm{fSch}_{/\Spf(R)}$.

Next, we prove that $G$ is an $\arc_p$-sheaf. First, we show that the natural map
\[
G(S) := (\Prism_{S/A}[\frac{1}{d}]/p^n)^{\phi=1}\to (\Prism_{S/A,\perf}[\frac{1}{d}]/p^n)^{\phi=1}
\]
is an isomorphism. Indeed, it suffices to check this when $n=1$, and then $\Prism_{S/A,\perf}/p$ is the $d$-completed filtered colimit of
\[ \Prism_{S/A}/p \xrightarrow{\phi_S} \Prism_{S/A}/p \xrightarrow{\phi_S} \ldots\ .\]
As each map $\phi_S$ trivially induces an equivalence on applying $\big((-)[\frac{1}{d}]\big)^{\phi=1}$, Lemma~\ref{PhiFixedColimit} gives the claim. Now $\Prism_{S/A,\perf}$ is an \arc-sheaf as $S\mapsto S_\perfd$ is so (Corollary~\ref{PerfdVCohom}); this implies  that $G$ is also an \arc-sheaf. To get the stronger statement that $G$ is an $\arc_p$-sheaf, we argue as in \cite[Corollary 6.17]{BhattMathew}. Given an $\arc_p$-cover $S \to T$ of $p$-complete $R$-algebras with Cech nerve $S \to T^\bullet$, the map $S \to T \times S/p$ is an $\arc$-cover with Cech nerve $S \to T^\bullet \times T^\bullet/p \times S/p$. To deduce $\arc_p$-descent for $G$, it then suffices to show the following: for any $p$-complete $R$-algebra of the form $S \times S'$ with $pS' = 0$, the projection map $G(S \times S') \to G(S)$ is an isomorphism. By the description $G = \left(\Prism_{-/A,\perf}/p^n[\frac{1}{d}]\right)^{\phi=1}$ given earlier in this paragraph and the product compatibility for prismatic cohomology, we are reduced to checking that $\Prism_{S'/A,\perf}/p^n[\frac{1}{d}] = 0$. But $pS' = 0$, so $S'_{\perfd} = S'_\perf$ is the usual perfection (Example~\ref{PerfdPerfCharp}), whence $\Prism_{S'/A,\perf}$ identifies with $W(S'_{\perf})$. As $(d)=(p)$ as ideals in $W(S'_\perf)$, the object $W(S'_\perf)/p^n[\frac{1}{d}]$ is clearly $0$.

The previous paragraph implies that the obvious map $\mathbf{Z}/p^n \to G$ of presheaves on $\mathrm{fSch}_{/\Spf(R)}$ extends uniquely to a comparison map $R\Gamma_{\arc_p}(-, \mathbf{Z}/p^n) \to G$ of $\arc_p$-sheaves. Combining this the construction two paragraphs ago gives a  comparison map 
\[ F(S) := R\Gamma(\Spec(S[\frac{1}{p}]),\mathbf Z/p^n) \to G(S) := (\Prism_{S/A}[1/d]/p^n)^{\phi=1}\]
between the $\arc_p$-sheaves $F$ and $G$. To prove this an isomorphism, we may work $\arc_p$-locally and thus also \arc-locally. At this point, there are two ways to finish the proof. In the first approach, we can assume $S$ is perfectoid thanks to \arc-localization. One can then identify $\Prism_{S/A}[1/d]/p^n = W(S^\flat)[1/d]/p^n$, and then by Artin-Schreier-Witt we have
\[
R\Gamma(\Spec(S^\flat[\tfrac 1d]),\mathbf Z/p^n)=(\Prism_{S/A}[1/d]/p^n)^{\phi=1}\ .
\]
It remains to prove that there is a canonical quasi-isomorphism
\[
R\Gamma(\Spec(S[\tfrac 1p]),\mathbf Z/p^n)\simeq R\Gamma(\Spec(S^\flat[\tfrac 1d]),\mathbf Z/p^n)\ .
\]
This follows from the comparison between the \'etale cohomology of $\Spec(S[\frac 1p])$ and $\mathrm{Spa}(S[\frac 1p],S)$ (and similarly for $S^\flat$), cf.~\cite[Corollary 3.2.2]{HuberBook}, and \cite[Theorem 1.11]{ScholzeThesis}. 

In the second approach, one can apply further \arc-descent to reduce to a case like $S=\prod R_i$ being a product of absolutely integrally closed valuation rings of rank $\leq 1$ (Remark~\ref{ArcCoverAICVR}), where the claim can be verified by hand (as $H^i=0$ for $i>0$ and the $H^0$ becomes explicit); we leave this as an exercise to the reader.
\end{proof}

\begin{remark}
\label{EtaleCompSpecial}
There is also a variant of \cref{EtaleCompThm} without inverting $p$ or $d$: if $X$ is a $p$-adic formal scheme $R$-scheme, then there is a canonical identification
\[ \mathbf{Z}/p^n \simeq (\Prism_{X/A}/p^n)^{\phi=1}\]
of \'etale sheaves on $X$ for all $n \geq 0$; we leave this to the reader.
\end{remark}

Recall that for any integer $n \geq 0$, we have defined $p$-adic Tate twists $\mathbf{Z}_p(n)$ as sheaves on the quasisyntomic site in \cite[\S 7.4]{BMS2}; the notion is reviewed more thoroughly in \S \ref{sec:Nygaard}. For a perfectoid ring $R$, these admit a concrete description in terms of the corresponding perfect prism $(A,(d))$:
\begin{equation}
\label{ZpnR}
 \mathbf{Z}_p(n)(R) \simeq \Big(\phi^{-1}(d)^n A \xrightarrow{\frac{\phi}{d^n}-1} A\Big).
 \end{equation}
One can show that the right side is independent of the choice of generator $d$ up to quasi-isomorphism. Our goal is to describe these complexes in terms involving only the special or generic fibres.

\begin{theorem}
\label{TateTwistPerfd}
Let $R$ be a perfectoid ring.  For any integer $n \geq 1$, there is a canonical identification
\[ \mathbf{Z}_p(n)(R) \simeq R\Gamma(\mathrm{Spec}(R[1/p]), \mathbf{Z}_p(n)).\]
 For $n=0$, there is a canonical identification 
\[ \mathbf{Z}_p(0)(R) \simeq R\Gamma(\mathrm{Spec}(R), \mathbf{Z}_p) \simeq R\Gamma(\mathrm{Spec}( (R/p)_\perf), \mathbf{Z}_p).\]
\end{theorem}

\begin{proof}
The statement for $n=0$ at the end follows from Remark~\ref{EtaleCompSpecial} (and Gabber's affine analog of proper base change to pass from $R$ to $R/p$); we leave the details to the reader. Assume from now that $n \geq 1$. Write $F_n(R) = \mathbf{Z}_p(n)(R)$ and $G_n(R) = R\Gamma(\mathrm{Spec}(R[1/p]), \mathbf{Z}_p(n))$; these are $p$-complete \arc-sheaves on the category of perfectoid rings. The proof involves three steps:\\

{\em The theorem over $\mathbf Z_p[\zeta_{p^\infty}]$-algebras $R$:} We first explain why $F_n/p \simeq G_n/p$ when restricted to perfectoid $\mathbf Z_p[\zeta_{p^\infty}]$-algebras $R$; the isomorphism we produce in this case will not obviously agree with our eventual isomorphism over all perfectoid rings, but will be useful in proving structural properties of $F_n$ over all perfectoid rings. Write $q=[\epsilon]\in A_{\inf}(R)$ for the the usual element, normalized to ensure that the distinguished element $d=[p]_q$ generates $\ker(A_{\inf}(R)\to R)$. Set $\mu=q-1\in A_{\inf}(R)$. Consider the map of complexes
\[ \alpha^\prime_n: F_n(R)/p := \Big(d^{\frac{n}{p}} R^\flat  \xrightarrow{\frac{\phi}{d^n}-1} R^\flat\Big) \to G_n(R)/p := \Big(R^\flat[1/d]\xrightarrow{\phi-1} R^\flat[1/d]\Big)\]
induced by multiplication by $\mu^{-n}$ on the terms. In other words, it is given by the evident map of complexes
\[
\Big(\mu^{-n/p} R^\flat  \xrightarrow{\phi-1} \mu^{-n} R^\flat\Big) \to \Big(R^\flat[1/\mu]\xrightarrow{\phi-1} R^\flat[1/\mu]\Big)\ .
\]
The map $\alpha_n^\prime$ is a quasi-isomorphism: it is the colimit over $i$ of  the maps of complexes
\[
\Big(\mu^{-np^{i-1}} R^\flat  \xrightarrow{\phi-1} \mu^{-np^i} R^\flat\Big) \to \Big(\mu^{-np^i} R^\flat  \xrightarrow{\phi-1} \mu^{-np^{i+1}} R^\flat\Big)
\]
induced by the identity; these are homotopic to the maps induced by $\phi$, which are isomorphisms of complexes. For future references, we remark that the maps $\alpha_n^\prime$ are multiplicative, i.e., they induce an isomorphism  $\oplus_n F_n/p \to \oplus_n G_n/p$ of graded ring sheaves on perfectoid $\mathbf Z_p[\zeta_{p^\infty}]$-algebras $R$; this is easy to see from the formula given above. \\

{\em Multiplicativity:} We shall prove the following assertion that will allow us to reduce the theorem for general $n$ to the $n=1$ case:
\begin{itemize}
\item[$(\ast)$] As presheaves on all perfectoid rings, each $F_n$ (resp. $G_n$) is $p$-complete \arc-locally concentrated on degree $0$ where it is $p$-torsionfree, and that the multiplication map $F_1^{\widehat{\otimes} n} \to F_n$ (resp. $G_1^{\widehat{\otimes} n} \to G_n$) expresses the target as the $p$-complete \arc-sheafification of the source. 
\end{itemize}
As both assertions in $(\ast)$ are $p$-complete \arc-local, it suffices to prove them for perfectoid $\mathbf Z_p[\zeta_{p^\infty}]$-algebras $R$. By the first part of the proof,  it is then enough to show them only for the $G_n$'s. In other words, we reduce to showing:
\begin{itemize}
\item[$(\ast)_a$] The sheaf $R \mapsto G_n(R)/p = R\Gamma(\mathrm{Spec}(R[1/p]), \mathbf{Z}/p(n))$ on perfectoid rings is $p$-complete \arc-locally concentrated in degree $0$.
\item[$(\ast_b)$] The multiplication map $(G_1/p)^{\otimes n} \to G_n$ exhibits the target as the $p$-complete \arc-sheafification of the source.  
\end{itemize}

For $(\ast_a)$: by $p$-complete \arc-descent, it suffices to prove that $R\Gamma(\mathrm{Spec}(R[1/p]), \mathbf{Z}/p(n))$ is concentrated in degree $0$ when $R=\prod_i V_i$ is a product of $p$-complete rank $1$ valuation rings with algebraically closed fraction fields (Remark~\ref{ArcCoverAICVR}). Moreover, as each $G_n$ kills rings of characteristic $p$, we may also assume each $V_i$ has fraction field of characteristic $0$. For such an $R$, we can choose a structure map $\mathbf Z_p[\zeta_{p^\infty}] \to R$, so the previous paragraph shows that $G_n(R)/p \simeq \prod_i G_n(V_i)/p$: indeed, the analogous statement is clear for $F_n(R)/p$ as $A_{\inf}(-)$ commutes with products. Now $G_n(V_i)/p$ is visibly concentrated in degree $0$, so we conclude that each $G_n$ is $p$-completely \arc-locally concentrated in degree $0$ where it is $p$-torsionfree. 

For $(\ast_b)$, we prove a stronger statement: for $R=\prod_i V_i$ with $V_i$ a $p$-complete valuation ring with algebraically closed fraction field, the restrictions $H_i = (G_i/p)|_{\mathrm{Spf}(R)}$ to the small Zariski site of $\mathrm{Spf}(R)$ satisfy $H_1^{\otimes n} \simeq H_n$ as sheaves on the topological space $\mathrm{Spf}(R)$. As a map of sheaves on $\mathrm{Spf}(R)$ is an isomorphism exactly when it is on stacks, we must check the following: for each $x \in \mathrm{Spf}(R)$, the natural map
\[ \colim_U \left(H_1(\mathcal{O}(U))^{\otimes n} \to H_n(\mathcal{O}(U))\right)\]
is an isomorphism, where the colimit runs over all affine open neighbourhoods $U \subset \mathrm{Spf}(R)$. Now the functors $G_n(-)/p$ carry $p$-completed filtered colimits of perfectoid rings to filtered colimits of complexes: this follows from the Gabber-Fujiwara theorem. Consequently, the above map can identified with the natural map
\[ G_1(\mathcal{O}_{\mathrm{Spf}(R),x})^{\otimes n} \to G_n(\mathcal{O}_{\mathrm{Spf}(R),x}).\]
The perfectoid ring $V := \mathcal{O}_{\mathrm{Spf}(R),x}$ is the $p$-completed ultraproduct of $p$-complete absolutely integrally closed valuation rings, and hence must itself be a $p$-complete absolutely integrally closed valuation ring. But then $C = V[1/p]$ is an algebraically closed field, so the above map identifies with the tautological map
\[ H^0(\mathrm{Spec}(C), \mathbf{Z}/p(1))^{\otimes n} \to H^0(\mathrm{Spec}(C), \mathbf{Z}/p(n)),\]
which is an isomorphism by the definition $\mathbf{Z}/p(i) = \mathbf{Z}/p(1)^{\otimes i}$ as \'etale sheaves on $\mathrm{Spec}(C)$. \\

{\em Proof of theorem:} Thanks to assertion $(\ast)$ above, to prove the theorem, it remains to construct a canonical isomorphism $\alpha:F_1 \simeq G_1$ of sheaves on all perfectoid rings: any such isomorphism induces a multiplicative system of isomorphisms $\alpha_n:F_n \simeq G_n$ by $(\ast)$. By \cite[Proposition 7.17]{BMS2}, we know that $\mathbf{Z}_p(1) = \lim_n \mu_{p^n}$ as \arc-sheaves on perfectoid rings (or even as sheaves on the full quasisyntomic site), so there is a canonical map $\alpha_1:F_1 \to G_1$ determined by Kummer theory. To prove this map is an isomorphism, we may work $p$-complete \arc-locally, so we reduce to checking that $\alpha_1(R)$ is an isomorphism where  $R=\prod_i V_i$ with $V_i$ a $p$-complete and $p$-torsion free rank $1$ valuation ring with algebraically closed fraction field $C_i = V_i[1/p]$. In this case, as we saw in the proof of $(\ast_a)$ above, both sides commute with products and are concentrated in degree $0$, so we reduce to checking the statement when $R=V$ is a $p$-complete and $p$-torsion free rank $1$ valuation ring with algebraically closed fraction field $C$. Unwinding definitions, we must show that the composite map
\[ \mathbf{Z}_p(1)(V) \simeq T_p(V^*) \xrightarrow{a} T_p(C^*) \xrightarrow{b} H^0(\mathrm{Spec}(C), \mathbf{Z}_p(1))\]
is an isomorphism, where $a$ is the natural map and $b$ comes from classical Kummer theory.  In fact, both $a$ and $b$ are themselves isomorphisms: the isomorphy of $b$ is clear as $H^1(\mathrm{Spec}(C), \mathbf{G}_m) = 0$, while the isomorphy of $a$ follows as $C^*/V^*$ is uniquely $p$-divisible.
\end{proof}

\begin{remark}
The isomorphism $\mathbf{Z}_p(n)(R) \simeq R\Gamma(\mathrm{Spec}(R[1/p]), \mathbf{Z}_p(n))$ from Theorem~\ref{TateTwistPerfd} relied on the following fact about $\mathbf{Z}_p(1)(R)$: it can be described both as the Tate module of $\mathbf{G}_m$ (which allows us to construct the comparison map for $n=1$ explicitly) and as a filtered Frobenius eigenspace of prismatic cohomology (which allowed us to prove various properties of $\mathbf{Z}_p(1)(-)$ including the key multiplicativity $\mathbf{Z}_p(1)^{\widehat{\otimes}^L n} \simeq \mathbf{Z}_p(n)$ necessary to construct the comparison map for all $n$). Taking the second property as a definition, the first property is highly non-obvious and relies on homotopy-theoretic input: it is proven in \cite{BMS2} using the cyclotomic trace map to relate algebraic $K$-theory with TC\footnote{The proof in \cite{BMS2} also uses the main theorems of \cite{ClausenMathewMorrow} to equate $K$-theory and TC in the relevant context; however, once the comparison map is constructed, this input can be avoided using the method of proof of Theorem~\ref{TateTwistPerfd}.}. In \cite{BhattLurieChern}, we shall give a direct  algebraic construction of a comparison map 
\[ \log_\Prism:T_p(R^*) \to \phi^{-1}(d) A_{\inf}(R)\{1\}. \]
In fact, this is the special case of a prismatic logarithm map $\mathbf{Z}_p(1)(A/I) \to A\{1\}$ that exists over any bounded prism $(A,I)$. Using this, the proof of Theorem~\ref{TateTwistPerfd} becomes purely algebraic. Moreover, the two approaches to building comparison map are equivalent: the map $\log_\Prism$ over the $q$-de Rham prism can be shown to agree (up to a Frobenius twist) with the $q$-logarithm, which was studied in \cite{AnschutzLeBrasqLog} as an explicitification of the map coming from the cyclotomic trace construction.
\end{remark}

We next extract algebraic consequences from Theorem~\ref{TateTwistPerfd}. For this, we need the following lemma reformulating results from \cite{BMS2}; all cohomology groups appearing below are computed in the flat topology unless otherwise specified.

\begin{lemma}
\label{BMS2Zp1}
For a ring $R$, let $\mathcal{P}\mathrm{ic}(R) = \left(\tau^{\leq 1} R\Gamma(\mathrm{Spec}(R), \mathbf{G}_m)\right) [1]$ denote the groupoid of line bundles on $R$, regarded as an object of the derived category of abelian groups. For any quasisyntomic ring $R$, there is a natural identification
\[ \mathcal{P}\mathrm{ic}(R)^{\wedge}[-2] \simeq \mathbf{Z}_p(1)(R),\]
where the completion on the left side is the $p$-completion.
\end{lemma}
\begin{proof}
By quasi-syntomic descent for perfect complexes and then line bundles, the functor $\mathcal{P}\mathrm{ic}(-)$ is a sheaf of complexes on the quasi-syntomic site. Moreover, it is locally concentrated in degree $-1$: we have $\mathcal{P}\mathrm{ic}(R) = R^*[1]$ for any ring $R$ no non-trivial line bundles, and every quasisyntomic ring has a quasisyntomic cover with this property (e.g., by $p$-completing its $w$-localization). Passing to $p$-completions, and using that units are $p$-divisible locally in the quasisyntomic topology, we learn that $\mathcal{P}\mathrm{ic}(-)^{\wedge}[-2]$ is the unique sheaf of complexes on the quasisyntomic site which is locally identified by the sheaf $R \mapsto T_p(R^*) = \lim_n \mu_{p^n}(R)$ of abelian groups. But this is exactly the characterization of $\mathbf{Z}_p(1)(-)$ by \cite[Proposition 7.17]{BMS2}, so the lemma follows.
\end{proof}

\begin{corollary}\label{cor:picperfectoid}
For a perfectoid ring $R$, the groups $\mathrm{Pic}(R)$ and $\mathrm{Pic}(R[1/p])$ are uniquely $p$-divisible.
\end{corollary}
\begin{proof}
By deformation theory, we have $\mathrm{Pic}(R) \cong \mathrm{Pic}(A_{\inf}(R)) \cong \mathrm{Pic}(R^\flat)$, so the claim follows for $\mathrm{Pic}(R)$ as any perfect ring has a uniquely $p$-divisible Picard group.

For $\mathrm{Pic}(R[1/p])$, we shall use \cref{TateTwistPerfd} and Lemma~\ref{BMS2Zp1}. Consider the diagram
\[ \xymatrix{ \mathcal{P}\mathrm{ic}(R)^{\wedge}[-2] \ar[r]^{a, \ \simeq} \ar[d]^b & \mathbf{Z}_p(1)(R) \ar[d]^{c,\ \simeq} \\
\mathcal{P}\mathrm{ic}(R[1/p])^{\wedge}[-2]  \ar[r]^-d & R\Gamma(\mathrm{Spec}(R[1/p]), \mathbf{Z}_p(1)),}\]
where $a$ is the isomorphism from Lemma~\ref{BMS2Zp1}, $c$ is the isomorphism of Theorem~\ref{TateTwistPerfd}, $d$ is the classical Kummer map, and $b$ is the canonical map; by $p$-complete \arc-descent and reduction to the case $R=\mathcal{O}_C$ with $C/\mathbf{Q}_p$ complete and algebraically closed, one can check that this diagram is commutative. Moreover, the map $d[1]$ is the $p$-completion of the truncation map $\tau^{\leq 1} R\Gamma(\mathrm{Spec}(R[1/p]), \mathbf{G}_m) \to R\Gamma(\mathrm{Spec}(R[1/p]), \mathbf{G}_m)$, so $d$ is injective on all cohomology groups. As $a$ and $c$ are isomorphisms, it follows that $b$ and $d$ must also be isomorphisms. Reducing the isomorphism $b$ modulo $p$ gives an identification
\[  \mathcal{P}\mathrm{ic}(R)/p \simeq  \mathcal{P}\mathrm{ic}(R[1/p])/p.\]
Taking $H^0$ in this identification shows that $\mathrm{Pic}(R[1/p])/p \simeq \mathrm{Pic}(R)/p$; this group vanishes as $\mathrm{Pic}(R)$ is $p$-divisible, and thus $\mathrm{Pic}(R[1/p])$ is $p$-divisible. On the other hand, computing $H^{-1}$ for the previous isomorphism using the Bockstein sequences on both sides gives a surjection
\[  \mathrm{Pic}(R)[p] \to  \mathrm{Pic}(R[1/p])[p]\]
on $p$-torsion subgroups. By the unique $p$-divisibility of $\mathrm{Pic}(R)$, the left side is $0$, so the right is also $0$, whence  $\mathrm{Pic}(R[1/p])$ is uniquely $p$-divisible. 
\end{proof}

\newpage

\section{The almost purity theorem} 
\label{sec:APT}

The goal of this section is to use the perfectoidization functor $S\mapsto S_\perfd$ from \cref{generalperfoidization} to prove a general version of the almost purity theorem that includes Andr\'e's perfectoid Abhyankar lemma \cite{AndreAbhyankar}; this goal is accomplished in \S \ref{APT}. To accomplish this task, we need a notion of almost mathematics that deals with ramification with respect to any (finitely generated) ideal in a perfectoid ring; this is possible thanks to \cref{PerfectoidificationSurj}, and is the subject of \S \ref{ss:AlmostMathIdeal}.

\subsection{Almost mathematics with respect to any ideal}
\label{ss:AlmostMathIdeal}

Let $R$ be a perfectoid ring and let $J\subset R$ be a derived $p$-complete ideal with perfectoidization $J_\perfd = \ker(R\to (R/J)_\perfd)$. In this section, we study almost mathematics for $p$-complete $R$-modules with respect to the ideal $J_\perfd$.

\begin{definition} A derived $p$-complete $R$-module $M$ is $J$-almost zero if $J_\perfd M = 0$.
\end{definition}

When $J  = (p)$, we have $J_\perfd = \sqrt{pR}$, and the above notion then agrees with the usual notion of ``almost zero'' for modules over a perfectoid ring.

\begin{proposition} The subcategory of $J$-almost zero derived $p$-complete $R$-modules is stable under kernels, cokernels and extensions in the category of derived $p$-complete $R$-modules, and forms a $\otimes$-ideal. It is equivalent to the category of derived $p$-complete $(R/J)_\perfd$-modules.
\end{proposition}

\begin{proof} Consider the ``forgetful" functor from the category of derived $p$-complete $(R/J)_\perfd$-modules to the category of derived $p$-complete $R$-modules. To prove the proposition, it suffices to prove that this functor is fully faithful and preserves $\mathrm{Ext}^1$. We will in fact prove that it preserves $R\Hom$. Thus let $M$ and $N$ be two derived $p$-complete $(R/J)_\perfd$-modules. Then
\[
R\Hom_R(M,N) = R\Hom_{(R/J)_\perfd}(M\widehat{\otimes}^L_R (R/J)_\perfd,N)\ .
\]
Thus it is enough to see that $M\to M\widehat{\otimes}^L_R (R/J)_\perfd$ is an isomorphism. As $M$ is an $(R/J)_\perfd$-module, it suffices to prove this when $M=(R/J)_\perfd$, which is part of the next lemma.
\end{proof}

\begin{lemma} 
\label{GetAlmostSetup}
The multiplication maps
\[
J_\perfd\widehat{\otimes}_R^L J_\perfd\to J_\perfd\ ,\ (R/J)_\perfd\widehat{\otimes}^L_R (R/J)_\perfd\to (R/J)_\perfd
\]
are quasi-isomorphisms, and
\[
J_\perfd\widehat{\otimes}_R^L (R/J)_\perfd = 0\ .
\]
\end{lemma}

\begin{proof} It is clear that the first two maps are isomorphism precisely when $J_\perfd\widehat{\otimes}_R^L (R/J)_\perfd = 0$, so it suffices to prove that
\[
(R/J)_\perfd\widehat{\otimes}^L_R (R/J)_\perfd\to (R/J)_\perfd
\]
is a quasi-isomorphism. This is clearly true on $H^0$, so it suffices to prove more generally that for any maps $S\leftarrow R\to T$ of perfectoid rings, the ring $S\widehat{\otimes}^L_R T$ is concentrated in degree $0$ and perfectoid. Let $(B,IB)\leftarrow (A,I)\to (C,IC)$ be the diagram of perfect prisms corresponding to $S\leftarrow R\to T$. Then $D:=B\widehat{\otimes}^L_A C$ is concentrated in degree $0$ and a perfect $\delta$-ring, as this can be checked after (derived) reduction modulo $p$, where it follows from Tor-independence of perfect rings, cf.~e.g.~\cite[Lemma 3.16]{BhattScholzeWitt}. In particular, $(D,ID)$ is a perfect prism. In particular, $ID$ defines a Cartier divisor and we have $D/ID = S\widehat{\otimes}^L_R T$, giving the desired statement.
\end{proof}

\begin{proposition}\label{prop:derivedalmostzero} A derived $p$-complete complex $M$ of $R$-modules is $J$-almost zero, i.e.~$H^i(M)$ is $J$-almost zero for all $i$, if and only if
\[
J_\perfd\widehat{\otimes}^L_R M = 0\ .
\]
The category of all such is equivalent to the category $D_{p\text{-comp}}((R/J)_\perfd)$ of derived $p$-complete complexes of $(R/J)_\perfd$-modules, and is a thick tensor ideal in the category of all derived $p$-complete complexes of $R$-modules.
\end{proposition}

\begin{proof} We have already proved that $D_{p\text{-comp}}((R/J)_\perfd)\to D_{p\text{-comp}}(R)$ is fully faithful. Clearly, any object in the image is $J$-almost zero; conversely, writing any object as a limit and colimit of its truncations, and using that the inclusion commutes with all limits and colimits, shows that any $J$-almost zero complex is in the image. For all objects $M$ in the image,
\[
J_\perfd\widehat{\otimes}^L_R M = J_\perfd\widehat{\otimes}^L_R (R/J)_\perfd\otimes_{(R/J)_\perfd} M = 0\ .
\]
Conversely, if $J_\perfd\widehat{\otimes}^L_R M = 0$, then $M\to M\widehat{\otimes}^L_R (R/J)_\perfd$ is a quasi-isomorphism and so $M$ is $J$-almost zero.
\end{proof}

The following connectivity criterion in $J$-almost mathematics shall be useful later. 

\begin{lemma}
\label{ConnectivityCrit}
A $p$-complete complex $M \in D(R)$ is connective if and only if $M \widehat{\otimes}^L_R R/J_\perfd$ is connective and $M$ is $J$-almost connective (i.e., $H^i(M)$ is $J$-almost zero for $i > 0$).
\end{lemma}
\begin{proof}
The ``only if'' direction is clear as $(-) \widehat{\otimes}^L_R R/J_\perfd$ preserves connectivity. For the converse, fix a derived $p$-complete complex $M \in D(R)$. We have a triangle
\[
J_\perfd\widehat{\otimes}^L_R M\to M\to (R/J_\perfd) \widehat{\otimes}^L_R M, 
\]
so it suffices to see that if $M$ is $J$-almost connective then $J_\perfd\widehat{\otimes}^L_R M$ is connective (the converse is also true, and clear). But in the triangle
\[
J_\perfd\widehat{\otimes}^L_R\tau^{\leq 0} M\to J_\perfd\widehat{\otimes}^L_R M\to J_\perfd\widehat{\otimes}^L_R\tau^{>0} M
\]
the last term vanishes by Proposition~\ref{prop:derivedalmostzero} and the first term is connective.
\end{proof}

The following approach to Galois coverings in the $J$-almost category shall prove useful. 

\begin{definition}\label{JalmostGalois} A $p$-complete $R$-algebra $S$ equipped with an action of a finite group $G$ is a $J$-almost $G$-Galois extension if the maps
\[
R\to R\Gamma(G,S)\ ,\ S\widehat{\otimes}_R^L S\to \prod_G S
\]
are $J$-almost isomorphisms, i.e.~the cones are $J$-almost zero.
\end{definition}

\begin{remark} The conditions imply similar statements modulo $p^n$, which then by passing to $H^0$ imply that
\[
R/p^n\to (S/p^n)^G\ ,\ S/p^n\otimes_{R/p^n} S/p^n\to \prod_G S/p^n
\]
are $J$-almost isomorphisms.
\end{remark}

Note that for each $n$, $(R/p^n,J_{\perfd,n})$ defines an almost setting (called a {\em basic setup} in \cite[\S 2.1.1]{GabberRamero}), where $J_{\perfd,n}$ is the image of $J_\perfd$ in $R/p^n$ (i.e., $J_{\perfd,n}$ is an idempotent ideal of $R/p^n$). If $J=(g)$ is generated by one element, then it satisfies the hypothesis of \cite[\S 2.5.15]{GabberRamero} that $J_{\perfd,n}\otimes_{R/p^n} J_{\perfd,n}$ is a flat $R/p^n$-module: Indeed, when $g$ admits $p$-power roots, this is given by $\colim_{g^{1/p^i-1/p^{i-1}}} R/p^n$.\footnote{For this, we have to see that
\[
\colim_{g^{1/p^i - 1/p^{i-1}}} R/p^n\to J_\perfd/p^n
\]
is an isomorphism. Surjectivity is clear. For injectivity, assume that $a\in R$ is an element such that when considered as an element in the $i$-th copy of $R/p^n$ it maps to $0$ in $J_\perfd/p^n$. This means that $g^{1/p^i} a = p^n g^{1/p^k} b$ for some $k$ (which we can assume to be larger than $i$) and $b\in R$. In particular
\[
g^{1/p^k}(g^{1/p^i-1/p^k}a-p^nb) = 0\ .
\]
As $R$ is reduced, this implies that also
\[
g^{1/p^{k+1}}(g^{1/p^i-1/p^k}a-p^nb) = 0\ ,
\]
or in other words
\[
g^{1/p^i - 1/p^k + 1/p^{k+1}}a = p^n g^{1/p^{k+1}} b\in p^nR\ ,
\]
which means that the image of $a$ in $\colim_{g^{1/p^i-1/p^{i+1}}} R/p^n$ is equal to $0$.}

\begin{proposition}[{\cite[Proposition 1.9.1]{AndreAbhyankar}}]\label{GaloisAlmostFet} Let $S$ be a $p$-complete $R$-algebra with $G$-action that is a $J$-almost $G$-Galois cover. Then $R/p^n\to S/p^n$ is almost finite projective and almost unramified with respect to $(R/p^n,J_{\perfd,n})$. More generally, if for a subgroup $H\subset G$ the map $S^\prime=S^H\to S$ is a $J$-almost $H$-Galois cover, then $R/p^n\to S^\prime/p^n$ is almost finite projective and almost unramified.
\end{proposition}

Note that \cite[\S 1.1]{AndreAbhyankar} has a standing flatness assumption on $\mathfrak{m} \otimes_{\mathfrak{A}} \mathfrak{m}$, rendering it inapplicable directly in the context of Proposition~\ref{GaloisAlmostFet}. To circumvent this, one checks directly that this assumption is not necessary for the proof of \cite[Proposition 1.9.1]{AndreAbhyankar}; alternately, one can reduce to the flat case using the observation recalled before the statement of Proposition~\ref{GaloisAlmostFet} as well as the fact that an object in $D_{p-comp}(R)$ is  $J$-almost zero if and only if it is $(g)$-almost  zero for all $g \in J$.

\subsection{The almost purity theorem}
\label{APT}

Our goal is to prove the following version of the almost purity theorem (and Andr\'e's perfectoid Abhyankar lemma) over a perfectoid ring $R$, handling ramification along arbitrary closed subsets of $\Spec(R)$.

\begin{theorem}\label{GeneralAlmostPurity} Let $R$ be a perfectoid ring, $J\subset R$ a finitely generated ideal and $J_\perfd = \ker(R\to (R/J)_\perfd)$. Let $S$ be a finitely presented finite $R$-algebra such that $\Spec(S)\to \Spec(R)$ is finite \'etale outside $V(J)$. Then $S_\perfd$ is discrete and a perfectoid ring, and the map $S\to S_\perfd$ is an isomorphism away from $V(J)$. Moreover, for $n > 0$, the map $R/p^n \to S_\perfd/p^n$ is almost finite projective and almost unramified with respect to the almost setting $(R/p^n,J_{\perfd,n})$.

In fact, if $S$ admits a $G$-action for some finite group $G$ such that $\Spec(S)\to \Spec(R)$ is a $G$-Galois cover outside $V(J)$, then $R\to S_\perfd$ is a $J$-almost $G$-Galois cover in the sense of Definition~\ref{JalmostGalois}.
\end{theorem}

We first prove \cref{GeneralAlmostPurity} only in the $J$-almost setting. This $J$-almost version then implies a general discreteness statement (\cref{PerfdDiscrete}) for perfectoidizations of integral algebras over perfectoid rings, which then yields Theorem~\ref{GeneralAlmostPurity} in general. The key inputs in the argument are:
\begin{enumerate}
\item The \arc-descent properties of $S\mapsto S_\perfd$, especially Corollary~\ref{PerfdBlowup}.
\item Andr\'e's lemma (Theorem~\ref{AndreFlatness}).
\end{enumerate}

\begin{remark}
Our proof of \cref{GeneralAlmostPurity} yields, in particular, a new proof of the almost purity theorem: our arguments do not make any use of adic spaces (in particular, not of perfectoid spaces). 
\end{remark}

\begin{proof}[Proof of Theorem~\ref{GeneralAlmostPurity} in the $J$-almost setting]
With hypotheses and notation as in Theorem~\ref{GeneralAlmostPurity}, our goal is to show Theorem~\ref{GeneralAlmostPurity} holds true in $J$-almost setting, i.e., to show: 

\begin{itemize}
\item[$(\ast)$] 
$S_\perfd$ is $J$-almost connective, i.e.~$J_\perfd H^i(S_\perfd)=0$ for $i>0$, and $S\to S_\perfd$ is an isomorphism away from $V(J)$. Moreover, for all $n>0$ the map
\[
R/p^n\to H^0(S_\perfd)/p^n
\]
is almost finite projective and almost unramified with respect to the almost setting $(R/p^n,J_{\perfd,n})$.

In fact, if $S$ admits a $G$-action for some finite group $G$ such that $\Spec(S)\to \Spec(R)$ is a $G$-Galois cover outside $V(J)$, then $R\to H^0(S_\perfd)$ is a $J$-almost $G$-Galois cover in the sense of Definition~\ref{JalmostGalois}.
\end{itemize}

By Corollary~\ref{PerfdBlowup}, we are allowed to replace $S$ by an $S$-algebra $S^\prime$ that is integral over $S$ and such that $\Spec(S^\prime)\to \Spec(S)$ is an isomorphism outside $V(J)$ (as then $S_\perfd\to S^\prime_\perfd$ is a $J$-almost isomorphism). In particular, this implies that $S_\perfd$ is independent of the choice of $S$ inducing a given finite \'etale cover of $\Spec(R)\setminus V(J)$, up to $J$-almost isomorphism.

Assume first that $S$ admits an action of $G$ such that $\Spec(S)\to \Spec(R)$ is a $G$-Galois extension outside $V(J)$. We want to prove that $R\to S_\perfd$ is a $J$-almost $G$-Galois cover. By Theorem~\ref{AndreFlatness}, we can assume that $R$ is absolutely integrally closed. This implies in particular that any finite \'etale cover of any Zariski localization is Zariski locally split, so we can find generators $f_1,\ldots,f_r$ of $J$ such that $\Spec(S)\to \Spec(R)$ admits a splitting outside $V(f_i)$ for $i=1,\ldots,r$. As being a $J$-almost isomorphism is equivalent to being an $(f_i)$-almost isomorphism for all $i=1,\ldots,r$, we can assume that $J=(f)$ is a principal ideal such that $\Spec(S)\to \Spec(R)$ splits over $\Spec(R)[\frac 1f]$.\footnote{Even if $J$ was principal to start with, say $J=(p)$ as in the usual almost purity theorem, this step may change the principal ideal, and requires general elements.} But recall that we were allowed to replace $S$ by an integral $S$-algebra $S^\prime$ such that $\Spec(S^\prime)\to \Spec(S)$ is an isomorphism outside $V(J)$. This allows us to replace $S$ by the trivial $G$-torsor, where the result is clear.

Now assume more generally that $\Spec(S)\to \Spec(R)$ has constant degree $r$ outside $V(J)$. In that case, we can find a $\Sigma_r$-Galois extension of $X\setminus V(J)$ of which $\Spec(S)\setminus V(J)$ is the quotient by $\Sigma_{r-1}$, cf.~e.g.~\cite[Lemma 1.9.2]{AndreAbhyankar}. By Proposition~\ref{GaloisAlmostFet}, we can reduce to the case of a $G$-Galois extension.

In general, there is an open and closed decomposition of $\Spec(R)\setminus V(J)=\bigsqcup_{i=1}^n U_i$ over which the degree of $\Spec(S)\to \Spec(R)$ is constant. For each $i$, let $(R_i)_\perfd$ be the perfectoidization of the image $R_i$ of $R$ in $H^0(U_i,\mathcal O)$. Then $R\to \prod_{i=1}^n (R_i)_\perfd$ satisfies the hypothesis of Corollary~\ref{PerfdBlowup} (any map $R\to V$ that does not kill $J$ gives a map $\Spec(V)\to \Spec(R)\setminus V(J)=\bigsqcup_{i=1}^n U_i$, so factors over exactly one of the $U_i$, inducing a map $R_i\to V$, which then extends uniquely to the perfectoidization of $R_i$). Thus, we may replace $R$ by $(R_i)_\perfd$ for some $i$, and hence assume that $\Spec(S)\to \Spec(R)$ has constant degree outside $V(J)$.
\end{proof}

\begin{theorem}
\label{PerfdDiscrete}
Let $R$ be a perfectoid ring corresponding to a perfect prism $(A,I)$. Let $R \to S$ be the $p$-completion of an integral map. Then $\Prism_{S/A,\perf}$ is discrete and a perfect $p$-complete $\delta$-$A$-algebra. Consequently,  $S_\perfd$ is discrete and a perfectoid ring.
\end{theorem}

\begin{proof}
By passage to filtered colimits, we may assume $R \to S$ is finitely presented finite map\footnote{Note that any such $S$ is finitely presented as an $R$-module: any generator of $S$ as an $R$-algebra satisfies a monic polynomial, so any finite presentation of $S$ as an $R$-algebra immediately gives a finite presentation of $S$ as an $R$-module. In particular, such an $S$ is already $p$-complete.}. By \cref{PerfdDiscreteUniversal}, it is enough to prove $S_\perfd$ is connective. Choose a sequence $g_1,...,g_n \in R$ as in \cref{StratifyFinite}. We shall prove the connectivity of $S_\perfd$ by induction on $n$.  If $n=1$, the conditions in \cref{StratifyFinite} ensure that the map $R \to S_{\mathrm{red}}$ is finite \'etale, and thus $S_{\mathrm{red}}$ is perfectoid. But $S \to S_{\mathrm{red}}$ is an isomorphism on $\arc$-sheafification, so $S_\perfd \cong S_{\mathrm{red},\perfd} \cong S_{\mathrm{red}}$ is discrete by \cref{PerfdVCohom}.

Now assume $n > 1$. The inductive hypothesis applied to $(R/g_1R)_\perfd \to S \otimes_R (R/g_1R)_\perfd$ and \cref{PerfdSymmMon} ensure that $(S/g_1S)_\perfd$ is connective. The criterion in Lemma~\ref{ConnectivityCrit} reduces us to checking that $S_\perfd$ is $g_1$-almost connective. The conditions in \cref{StratifyFinite} imply the following:
\begin{enumerate}
\item Either $p=0$ in $R[1/g_1]$ or $p \in R[1/g_1]^*$.
\item The map $R[1/g_1] \to S_{\mathrm{red}}[1/g_1]$ factors as $R[1/g_1] \to T_1 \to S_{\mathrm{red}}[1/g_1]$, where the first map is finite \'etale and the second map is a universal homeomorphism. If $p \in R[1/g_1]^*$, then the second map is actually an isomorphism.
\end{enumerate}
We claim that $T_1 \to S_{\mathrm{red}}[1/g_1]$ is an isomorphism in general. By the last sentence of (2) above, we may assume $p=0$ in $R[1/g_1]$, so $R[1/g_1]$ is a perfect $\mathbf{F}_p$-algebra. But then $T_1$ is perfect as well, so $T_1 \to S_{\mathrm{red}}[1/g_1]$ is a universal homeomorphism from a perfect ring into a reduced ring; any such map must be an isomorphism, so $T_1 \cong S_{\mathrm{red}}[1/g_1]$.  Summarizing, the map $R[1/g_1] \to S[1/g_1] \to S_{\mathrm{red}}[1/g_1]$ is a finite \'etale cover.

If $S'$ denotes the integral closure of $R$ in $S_{\mathrm{red}}[1/g_1]$, then the $R$-finiteness of $S$ gives a map $S \to S'$ of integral $R$-algebras that is an isomorphism of $\arc$-sheaves outside $(g_1)$. By \cref{PerfdBlowup}, the map $S_\perfd \to S'_\perfd$ is a $g_1$-almost isomorphism. As we only want to show $g_1$-almost connectivity of $S_\perfd$, it suffices to check the $g_1$-almost connectivity of $S'_\perfd$. But this follows from Theorem~\ref{GeneralAlmostPurity} by approximating $S'$ by finitely presented finite $R$-algebras inducing the given finite \'etale cover $R[1/g_1] \to S_{\mathrm{red}}[1/g_1]$ on inverting $g_1$. 
\end{proof}

We needed the following lemma, stating roughly that any finitely presented finite map of qcqs schemes becomes finite \'etale over a constructible stratification of the target, at least up to universal homeomorphisms.

\begin{lemma}
\label{StratifyFinite}
Let $R \to S$ be a finitely presented finite map of commutative $\mathbf Z_{(p)}$-algebras. Then there exists a sequence of elements $g_1,...,g_n \in R$ such that, if we set $R_i = R/(g_1,..,g_{i-1})_{\mathrm{red}}[1/g_i]$ and $S_i = S/(g_1,...,g_{i-1})_{\mathrm{red}}[1/g_i]$ for $i=1,....,n$, then the following hold true:
\begin{enumerate}
\item  The ideal $(g_1,....,g_n)$ is the unit ideal of $R$; equivalently, $\cup_i \mathrm{Spec}(R_i) = \mathrm{Spec}(R)$. 
\item The map $R_i \to S_i$ factors as $R_i \to T_i \to S_i$, where $R_i \to T_i$ is finite \'etale, and $T_i \to S_i$ is a universal homeomorphism. Moreover, the map $T_i[1/p] \to S_i[1/p]$ is an isomorphism. 
\item In each $R_i$, we either have $p=0$ or $p \in R_i^*$.
\end{enumerate}
\end{lemma}
\begin{proof}
It suffices to find $g_1,...,g_n \in R$ satisfying (1) and (2): we can then stratify further if necessary to also achieve (3); explicitly, replacing $(g_1,...,g_n)$ with  $(pg_1,g_1,pg_2,g_2,...,pg_n,g_n)$ does the trick. For (1) and (2), by noetherian approximation, we may assume $R$ is a finitely presented $\mathbf{Z}_{(p)}$-algebra. Say $\mathfrak{p}_1,...,\mathfrak{p}_r$ are the minimal primes of $R$, and let $K = \prod_i K_i$ be the corresponding product of residue fields of $R$, so $K$ is the total ring of fractions of $R_{\mathrm{red}}$. As $R \to S$ is finite, the induced map $K \to (S \otimes_R K)_{\mathrm{red}}$ is then the product of maps of the form $K_i \to L_i$, where each $L_i = (S \otimes_R K_i)_{\mathrm{red}}$ is a finite reduced $K_i$-algebra. But then each $L_i$ is a finite product of finite field extensions of $K_i$, so the map $K_i \to L_i$ factors as $K_i \to M_i \to L_i$, where the first map is finite \'etale, while the second map is a universal homeomorphism and an isomorphism if $p$ is invertible in $K_i$. By an approximation argument, we can then find some $g_1 \in R$ invertible on $K$ such that $R_{\mathrm{red}}[1/g_1] \to S_{\mathrm{red}}[1/g_1]$ has the form required in (2). This constructs the open stratum, and the rest follows by noetherian induction applied to the map $R/(g_1) \to S/(g_1)$. 
\end{proof}

\begin{proof}[Proof of \cref{GeneralAlmostPurity}]
Combine the $J$-almost version proven above with \cref{PerfdDiscrete}.
\end{proof}

\begin{remark}
\cref{GeneralAlmostPurity} can be used to reprove some known results in commutative algebra relatively quickly. For example, the main theorems of \cite{HeitmannMaExtPlus} follow almost immediately. Let us explain the argument for \cite[Theorem 1.3]{HeitmannMaExtPlus}. Fix a complete regular local ring $R$ of mixed characteristic and a finite extension $R \to S$ contained in a fixed absolute integral closure $R \to R^+$. Fix some $n > 0$. To prove \cite[Theorem 1.3]{HeitmannMaExtPlus}, it is enough to show that there exists some $c \in R$ such that for all $i > 0$, the natural transformation 
\[ \mathrm{Tor}_i^{R/p^n}(S/p^n,-) \to \mathrm{Tor}_i^{R/p^n}(R^+/p^n,-)\]
is $c$-almost zero, i.e., the image is annihilated by all $p$-power roots  $c^{1/p^n} \in R^+$ of $c$. We claim that any $c \in pR$ such that $R[1/c] \to S[1/c]$ is finite \'etale will do the job. To see this, choose a faithfully flat extension $R \to R_\infty$ contained inside the $p$-adic completion of $R^+$ with $R_\infty$ perfectoid (see \cref{RegularLocalPrism}). Set $T = S \otimes_R R_\infty$, so the finite map $R_\infty \to T$ is finite \'etale after inverting $c$. \cref{GeneralAlmostPurity} implies that $R_\infty \to T_\perfd$ is $c$-almost finite \'etale. In particular, as $R \to R_\infty$ is faithfully flat, the functor $\mathrm{Tor}_i^{R/p^n}(T_\perfd/p^n,-)$ is $c$-almost zero (i.e., killed by $(cT_\perfd)_{\perfd}$) for $i > 0$; here we implicitly use that $c \in pR$ and that the $p$-torsion in $T_\perfd$ is $p$-almost zero (and thus $c$-almost zero). But the map $S/p^n \to R^+/p^n$ factors as $S/p^n \to T/p^n \to T_\perfd/p^n \to R^+/p^n$, where the last map is obtained from the universal property of perfectoidizations as the $p$-adic completion of $R^+$ is perfectoid. This gives a factorization 
\[ \mathrm{Tor}_i^{R/p^n}(S/p^n,-) \to \mathrm{Tor}_i^{R/p^n}(T_\perfd/p^n,-) \to \mathrm{Tor}_i^{R/p^n}(R^+/p^n,-),\]
which implies the claim as the middle group is $c$-almost zero for $i > 0$. 
\end{remark}

\newpage

\section{The \'etale cohomological dimension of perfectoid rings}
\label{sec:EtaleCD}

Fix a perfectoid ring $R$ corresponding to a perfect prism $(A,I)$. We apply our general version of the almost purity theorem to prove the following theorem.

\begin{theorem}[The \'etale cohomological dimension of algebraic \'etale sheaves]
\label{CohDim}
The $\mathbf{F}_p$-\'etale cohomological dimension of $X = \mathrm{Spec}(R[1/p])$ is $\leq 1$, i.e., for every \'etale $\mathbf{F}_p$-sheaf $F$ on $X$, we have $H^i(X,F)=0$ for $i>1$.
\end{theorem}

We stress that $X$ is the spectrum, not the adic spectrum, of $R[1/p]$; in particular, $F$ is a filtered colimit of Zariski-constructible sheaves. It is easy to see that the statement is wrong for general sheaves on the adic spectrum of $R[1/p]$.

\begin{proof} It suffices to prove the result when $F = j_! L$, where $j:U \hookrightarrow X$ is the inclusion of a quasicompact open subset of a constructible closed subset of $X$, and $L$ is an $\mathbf{F}_p$-local system of constant rank on $U$. Using the method of the trace \cite[Tag 03SH]{Stacks}, we can find a finite \'etale cover $V \to U$ of degree prime-to-$p$ such that $L|_V$ admits a finite filtration whose graded pieces are the constant sheaf $\mathbf{F}_p$.  Let $Y^+ := \mathrm{Spec}(S) \to X^+ := \mathrm{Spec}(R)$ be a finitely presented finite morphism lifting $V \to U$. Write $Y = Y^+[1/p]$ and let $j:V \to Y$ be the resulting open immersion. This data is summarized in the diagram
\[ \xymatrix{ V \ar[r]^-{j} \ar[d] & Y \ar[r] \ar[d] & Y^+ \ar[d] \\
		  U \ar[r]^-{j} & X \ar[r] & X^+ }\]
where all squares are cartesian. A standard devissage argument now reduces us to showing that $R\Gamma(Y, j_! \mathbf{F}_p) \in D^{\leq 1}$. Let $J \subset S$ be a finitely generated ideal cutting out the closed subset $Y^+ - V \subset Y^+$.  By the \'etale comparison theorem for $\mathrm{Spec}(S)$ and $\mathrm{Spec}(S/J)$, the complex $R\Gamma(Y, j_! \mathbf{F}_p)$ is computed by applying $(-/p[\frac{1}{I}])^{\phi=1}$ to the fibre of
\[ \Prism_{S/A,\perf} \to \Prism_{ (S/J)/A,\perf},\]
so its is enough to show this fibre is connective. This can be checked after base change along $A \to A/I$, so we are reduced to checking that the fibre of
\[ S_{\perfd} \to (S/J)_{\perfd}\] 
is connective. By \cref{PerfdDiscrete}, both objects above are discrete perfectoid rings, so we must show that the map is surjective. The universal property of perfectoidization shows that $(S/J)_\perfd \simeq (S_\perfd/JS_\perfd)_\perfd$, so the claim follows from \cref{PerfectoidificationSurj}.
\end{proof}

\newpage

\section{The Nygaard filtration for quasiregular semiperfectoid rings}
\label{sec:Nygaard}

In this section, we define and analyze the Nygaard filtration on $\Prism_S$ when $S$ is quasiregular semiperfectoid, and use it to endow prismatic cohomology with a Nygaard filtration which will in particular prove the de~Rham comparison. The main result  and its proof strategy is explained in \S \ref{ss:NygaardSPerfdStructure}; the proof, which is completed in \S \ref{ss:EndProofNygaard}, depends crucially on an analysis of the Nygaard filtration for a specific quasiregular semiperfectoid ring explained in \S \ref{ss:NygaardSpecial}. %As an application of our study, we prove in \S \ref{ss:BMS2Comp} that theory defined in this paper agrees with the one from \cite{BMS2} defined via quasisyntomic descent and topological Hochschild homology.

\subsection{Structure of the Nygaard filtration for quasiregular semiperfectoid rings}
\label{ss:NygaardSPerfdStructure}

The Nygaard filtration has a rather direct definition.

\begin{definition} Let $S$ be a quasiregular semiperfectoid ring with associated prism $(\Prism_S,(d))$. The Nygaard filtration on $\Prism_S$ is given by
\[
\mathrm{Fil}^i_N \Prism_S = \{x\in \Prism_S\mid \phi(x)\in d^i\Prism_S\}\ .
\]
It is an $\mathbf N$-indexed decreasing multiplicative filtration.
\end{definition}

We warn the reader that in general the Nygaard filtration is not separated. The primary goal of this section is to prove the following theorem. Fix any perfectoid ring $R$ (corresponding to a perfect prism $(A,(d))$) mapping to $S$. Then recall that by the Hodge-Tate comparison, $\overline{\Prism}_S = \Prism_S/d=\Prism_{S/R}/d$ is an $S$-algebra equipped with the conjugate filtration
\[
\mathrm{Fil}_i \overline{\Prism}_S\subset \overline{\Prism}_S
\]
which is an $\mathbf N$-indexed increasing exhaustive multiplicative filtration with
\[
\mathrm{gr}_i \overline{\Prism}_S = (\wedge^i L_{S/R}[-i])^\wedge\ .
\]
The conjugate filtration depends on the choice of $R$ in general (see Example~\ref{ConjFiltProduct}).

\begin{theorem}\label{ThmNygaard} The image of
\[
\mathrm{Fil}^i_N \Prism_S\xrightarrow{\frac{\phi}{d^i}} \Prism_S\to \overline{\Prism}_S
\]
agrees with $\mathrm{Fil}_i \overline{\Prism}_S$. In particular, there is a natural isomorphism
\[
\mathrm{gr}^i_N \Prism_S\cong \mathrm{Fil}_i \overline{\Prism}_S\{i\}\ .
\]
\end{theorem}

The proof of this theorem in the general case is rather indirect. We will proceed in the following steps:

\begin{enumerate}
\item The case $S=R\langle X^{1/p^\infty}\rangle/X$ (\S \ref{ss:NygaardSpecial}).
\item The case $S=R\langle X_1^{1/p^\infty},\ldots,X_n^{1/p^\infty}\rangle / (f_1,\ldots,f_m)$ for some $p$-completely regular sequence $f_1,\ldots,f_m$ (\S \ref{ss:NygaardFPSPerfd}).
\item Define a Nygaard filtration on $\Prism_{B/A}$ for any $p$-completely smooth $R$-algebra $B$ by quasisyntomic descent (\S \ref{subsec:generalnygaard}).
\item Define a derived Nygaard filtration on $\Prism_{S/A}$ via left Kan extension from smooth algebras, show that it agrees with the Nygaard filtration and finish the proof (\S \ref{ss:EndProofNygaard}).
\end{enumerate}

\begin{example}
\label{ConjFiltProduct}
We record an example of a quasiregular semiperfectoid ring $S$  where  the conjugate filtration on $\overline{\Prism}_S$ defined using a choice of a perfectoid ring mapping to $S$ depends on the choice.

Let $R=\mathcal{O}_{\mathbf{C}_p}$ and $\Gamma$ be the absolute Galois group of $\mathbf{Q}_p$, so $\Gamma$ acts naturally on $R$. Let $S=R \widehat{\otimes}_{\mathbf{Z}_p} R$, so $S$ is quasiregular semiperfectoid and has a natural $\Gamma \times \Gamma$-action. Write $i_1,i_2:R \to S$ for the inclusion of each of the two factors, and let $F_*$ and $F'_*$ be the corresponding conjugate filtrations on $\overline{\Prism}_S$. Write $R\{1\} = L_{R/\mathbf{Z}_p}^{\wedge}[-1]$, regarded as a finite free $R$-module of rank $1$ equipped with an equivariant $\Gamma$-action. Regarding $i_1$ as the map $i_1:R \widehat{\otimes}_{\mathbf{Z}_p} \mathbf{Z}_p \to R \widehat{\otimes}_{\mathbf{Z}_p} R$ shows that $i_1$ has a natural $(\Gamma \times \Gamma)$-action. This implies that the $(\Gamma \times \Gamma)$-action on $\overline{\Prism}_S$ preserves the conjugate filtration $F_*$; moreover, the standard description of $\mathrm{gr}_1$ for the conjugate filtration then gives an isomorphism 
\[ \mathrm{gr}_1^F \overline{\Prism}_S \simeq L_{i_1}^{\wedge}[-1] \simeq R \widehat{\otimes}_{\mathbf{Z}_p} R\{1\}\]
of $(\Gamma \times \Gamma)$-equivariant $S$-modules. By symmetry, we also have an isomorphism
\[ \mathrm{gr}_1^{F'} \overline{\Prism}_S \simeq  R\{1\} \widehat{\otimes}_{\mathbf{Z}_p} R\]
of $(\Gamma \times \Gamma)$-equivariant $S$-modules. 

Now assume that the conjugate filtration on $\overline{\Prism}_S$ attached to any choice of perfectoid ring mapping to $S$ is independent of the choice. Then $F_*$ and $F'_*$ are the same filtration and thus $ \mathrm{gr}_1^F \overline{\Prism}_S$ and $\mathrm{gr}_1^{F'} \overline{\Prism}_S$ are equal as subquotients of $\overline{\Prism}_S$; in particular, they are $(\Gamma \times \Gamma)$-equivariantly isomorphic $S$-modules. Thus, we learn that that  $R \widehat{\otimes}_{\mathbf{Z}_p} R\{1\}$ and $R\{1\} \widehat{\otimes}_{\mathbf{Z}_p} R$ are isomorphic as $\Gamma \times \Gamma$-modules. Applying derived invariants under $1 \times \Gamma$ to such an isomorphism would then show that
\[ R\Gamma( 1 \times \Gamma, R \widehat{\otimes}_{\mathbf{Z}_p} R\{1\}) \simeq R\Gamma( 1 \times \Gamma, R\{1\} \widehat{\otimes}_{\mathbf{Z}_p} R).\]
Using the projection formula, this implies that
\[ R \widehat{\otimes}_{\mathbf{Z}_p}^L R\Gamma(\Gamma, R\{1\}) \simeq R\{1\} \widehat{\otimes}_{\mathbf{Z}_p}^L R\Gamma(\Gamma, R).\]
Now Tate calculated that $H^0 R\Gamma(\Gamma,R) = \mathbf{Z}_p$ while $R\Gamma(\Gamma,R\{1\})$ is annihilated by a fixed power of $p$. Plugging this into the above equality gives an absurd statement: the left side is bounded $p$-torsion while the right side is nonzero after inverting $p$. Thus, we get a contradiction to the existence of such a filtration on $\overline{\Prism}_S$. 
\end{example}

\subsection{The Nygaard filtration in a special case}
\label{ss:NygaardSpecial}

In this section, we consider the semiperfectoid ring
\[
S=(\mathbf Z_p[\zeta_{p^\infty},X^{1/p^\infty}]/X)^\wedge\ .
\]
We consider it as an algebra over the perfectoid ring $R=\mathbf Z_p[\zeta_{p^\infty}]^\wedge$. The corresponding perfect prism is given by the pair $(A,I)$, where $A$ is the $(p,[p]_q)$-adic completion of $\mathbf Z_p[q^{1/p^\infty}]$ and $I=[p]_q=1+q+\ldots+q^{p-1}$; here $q^{1/p^n}\in A$ maps to $\zeta_{p^{n+1}}\in R=A/I$. From the definition of $\Prism_S = \Prism_S^{\mathrm{init}}$ as a prismatic envelope it follows that
\[
\Prism_S = A\langle Y^{1/p^\infty}\rangle\{\frac{Y^p}{[p]_q}\}^\wedge\ ,
\]
where we normalize the coordinates so that the map $S\to \overline{\Prism}_S$ takes $X^{1/p^n}$ to $Y^{1/p^{n-1}}$. (These normalizations will make the following formulas appear nicer.) The $\delta$-structure on $\Prism_S$ is determined by $\delta(q)=\delta(Y)=0$. 

It turns out that in this situation, there is an explicit description of $\Prism_S$ as a $q$-divided power algebra. Recall that for any integer $n\geq 0$, one sets
\[
[n]_q = \frac{q^n-1}{q-1} = 1 + q + \ldots + q^{n-1}\ ,\ [n]_q! = \prod_{i=1}^n [i]_q \in \mathbf{Z}_p\llbracket q-1 \rrbracket \subset A.
\]
The description as a prismatic envelope also shows that $\Prism_S$ is $(p,q-1)$-completely flat over $A$ (Lemma~\ref{PrismaticEnvSmooth}) and hence also over $\mathbf{Z}_p\llbracket q-1 \rrbracket$ as $\mathbf{Z}_p\llbracket q-1 \rrbracket \subset A$ is a flat inclusion; as the $\mathbf{Z}_p\llbracket q-1 \rrbracket$ is noetherian, we conclude that $\Prism_S$ is flat over $\mathbf{Z}_p\llbracket q-1 \rrbracket$ (see \cite[Lemma 5.15]{BhattCM}). In particular, each $[n]_q$ acts as a nonzerodivisor on $\Prism_S$, so the statement of the following lemma makes sense.

\begin{lemma}
\label{qPDEx}
We have $\frac{Y^n}{[n]_q!} \in \Prism_S$ for all integers $n \geq 0$. The resulting $A$-module map
\begin{equation}
\label{qPDmodule}
\bigoplus_{i \in \mathbf{N}[1/p]} A \cdot \frac{Y^i}{ [\lfloor i \rfloor]_q! } \to \Prism_S
\end{equation}
identifies $\Prism_S$ as the $(p,[p]_q)$-adic completion of the left side. 
\end{lemma}
\begin{proof}
Lemma~\ref{GetqPD} below implies that  $\frac{Y^n}{[n]_q!} \in \Prism_S$ for all $n \geq 0$. To check that the map is an isomorphism after completion, we may take the base change along the map $A\to A/(q-1)$ of $\delta$-rings; after this base change, $\Prism_S=A\langle Y^{1/p^\infty}\rangle\{\frac{Y^p}{[p]_q}\}$ becomes
\[
A/(q-1)\langle Y^{1/p^\infty}\rangle\{\frac{Y^p}p\}
\]
which by \cref{PDenvRegSeqLambda} agrees with the divided power envelope of $(Y)$ in $A/(q-1)\langle Y^{1/p^\infty}\rangle$ which indeed has the desired description.
\end{proof}

The following lemma will be reused below, so it is formulated in larger generality than necessary right now.

\begin{lemma}
\label{GetqPD}
Let $D$ be a $(p,[p]_q)$-completely flat $\mathbf{Z}_p\llbracket q-1 \rrbracket$-algebra. Assume we are given a map $\phi:D \to D$ and an element $x \in D$ such that the following hold true:
\begin{enumerate}
\item We have $\phi(q) = q^p$ and $\phi(x) = x^p$.
\item We have $\phi(x) \in [p]_q D$.
\end{enumerate}
Then $x^n \in [n]_q! D$ for all $n \geq 0$.
\end{lemma}
\begin{proof}
As $[i]_q$ is invertible for $i$ coprime to $p$, it suffices to prove the following statement:
\begin{itemize}
\item[$(\ast)$] Given $m \geq 0$, if $x^m \in [m]_q! D$, then $x^{mp} \in [mp]_q! D$. 
\end{itemize}
Using \cref{qMultId} (1), it suffices to show that $\phi([m]_q!) \cdot [p]_q^m \mid x^{mp}$ under the above assumption on $x$. Now $\phi([m]_q!)$ is a nonzerodivisor modulo $[p]_q$: this holds true in $\mathbf{Z}_p\llbracket q-1\rrbracket$ as the polynomial $\phi([m]_q!) = \prod_{i=1}^m \frac{q^{ip}-1}{q^p-1}$ does not vanish at the primitive $p$-th roots of unity in $\overline{\mathbf{Q}_p}$, and thus also in $D$ by flatness. Thus, $(\ast)$ reduces to showing 
\begin{itemize}
\item[$(\ast')$] Given $m \geq 0$, if $x^m \in [m]_q! D$, then $x^{mp}$ is divisible by both $[p]_q^m$ and $\phi([m]_q!)$.
\end{itemize}
Now $x^m \in [m]_q! D$ implies that $x^{mp} = \phi(x^m) \in \phi([m]_q!) D$, which proves half of $(\ast')$. For the other half, we simply observe that $x^{mp} = (x^p)^m \in [p]_q^m D$ by our assumption $x^p \in [p]_q D$.
\end{proof}

Next, we record two multiplicative identities in $\mathbf{Z}_p\llbracket q-1\rrbracket$, one of which was already used above.

\begin{lemma}
 \label{qMultId}
 In the ring $\mathbf{Z}_p\llbracket q-1\rrbracket$, we have the following identities:
 \begin{enumerate}
 \item  For any $m \in \mathbf{N}$, we have $[mp]_q! = u \cdot \phi([m]_q!) \cdot [p]_q^m$ with $u$ a unit.
 \item  For any $i \in \mathbf{N}[1/p]$, we have $[\lfloor i \rfloor p]_q! = [ \lfloor ip \rfloor]_q! \cdot v$ with $v$ a unit.
 \end{enumerate}
\end{lemma}
\begin{proof}
For (1), note that, up to multiplication by units, $[mp]_q!$ equals $\prod_{i=1}^m [ip]_q $: for any integer $k \geq 0$, the polynomial $[k]_q$ is invertible if $k$ is not divisible by $p$. Moreover, we also have an equality
\[ \prod_{i=1}^m [ip]_q! = \prod_{i=1}^m \frac{q^{ip}-1}{q-1} = \prod_{i=1}^m \big(\frac{q^{ip}-1}{q^p-1} \cdot \frac{q^p-1}{q-1}\big) = \phi([m]_q!) \cdot [p]_q^m.\]
The desired identity now easily follows.

For (2), write $i = \lfloor i \rfloor + \epsilon$ for $0 \leq \epsilon < 1$ in $\mathbf{N}[1/p]$. Then $\lfloor ip \rfloor = \lfloor i \rfloor p + \lfloor \epsilon p \rfloor$. As $0 \leq \epsilon < 1$ in $\mathbf{N}[1/p]$, we have $0 \leq \lfloor \epsilon p \rfloor < p$ in $\mathbf{N}$. But then any integer $k$ with $ \lfloor i \rfloor p < k \leq \lfloor ip \rfloor$ is coprime to $p$, so the corresponding polynomial $[k]_q$ is invertible; this easily implies the claim.
\end{proof}

Now we can describe the Nygaard filtration explicitly.

\begin{lemma}\label{NygaardKeyCase} Theorem~\ref{ThmNygaard} holds true for $S$. More precisely:
\begin{enumerate}
\item The Nygaard filtration $\mathrm{Fil}^n_N \Prism_S$ identifies with the $(p,[p]_q)$-adic completion of the $A$-submodule
\[ \bigoplus_{i \in \mathbf{N}[1/p]} [p]_{q^{1/p}}^{n-\lfloor i\rfloor } \cdot A \cdot \frac{Y^i}{ [\lfloor i\rfloor ]_{q}! } \subset \bigoplus_{i \in \mathbf{N}[1/p]} A \cdot \frac{Y^i}{  [\lfloor i\rfloor]_q! },\]
under the isomorphism in \eqref{qPDmodule}, and we follow the convention that $[p]_{q^{1/p}}^{n - \lfloor i\rfloor } =1$ for non-positive exponents (i.e., when $i \geq n+1$).

\item The image of
\[
\frac{\phi}{[p]_q^n}: \mathrm{Fil}^n_N\Prism_S\to \overline{\Prism}_S = \Prism_S/[p]_q = (\bigoplus_{i\in \mathbf N[1/p]} \mathbf Z_p[\zeta_{p^\infty}]\cdot \frac{Y^i}{[\lfloor i \rfloor]_q!})^\wedge
\]
is given by the summands with $i<p(n+1)$.

\item The conjugate filtration $\mathrm{Fil}_n \overline{\Prism}_S\subset \overline{\Prism}_S$ is also given by the summands with $i<p(n+1)$.
\end{enumerate}
\end{lemma}
\begin{proof} For (1), we note that $\phi$ is a graded map in the above $\mathbf N[1/p]$-grading, multiplying the grading by $p$. In particular, it follows that $\mathrm{Fil}^n_N \Prism_S$ is also graded. Using both parts of \cref{qMultId}, we can write
\begin{equation}
\label{FrobeniusqPD}
 \phi( \frac{Y^i}{[\lfloor i\rfloor]_q!}) = \frac{Y^{ip}}{\phi([\lfloor i\rfloor]_q!)} = \frac{Y^{ip}}{[\lfloor ip \rfloor]_q!} \cdot [p]_q^{\lfloor i\rfloor} \cdot u
 \end{equation}
for a unit $u$. It follows from this formula that $ [p]_{q^{1/p}}^{n-\lfloor i\rfloor } \cdot \frac{Y^i}{ [\lfloor i\rfloor ]_{q}! } \in \mathrm{Fil}^n_N \Prism_S$ for all $i$ (under the convention that $[p]_{q^{1/p}}^{n - \lfloor i\rfloor } =1$ for non-positive exponents). The same formula (and the fact that $\phi$ is bijective on $A$) also shows that no smaller multiple of $\frac{Y^i}{[\lfloor i\rfloor]_q!}$ can lie in $\mathrm{Fil}^n_N \Prism_S$, giving (1).

For part (2), we use the above formula to see the following:
\begin{itemize}
\item For $i < n+1$, the map $\phi$ maps $(\mathrm{Fil}^n_N \Prism_S)_{deg=i}$ isomorphically onto $[p]_q^n \cdot (\Prism_S)_{deg=ip}$.
\item For $i \geq n+1$, the map $\phi$ maps $(\mathrm{Fil}^n_N \Prism_S)_{deg=i}$ into $[p]_q^{n+1} \cdot (\Prism_S)_{deg=ip}$.
\end{itemize}
This immediately gives (2).

For part (3), recall that the natural map
\[ S := (\mathbf Z_p[\zeta_{p^\infty},X^{1/p^\infty}]/X)^\wedge  \to \overline{\Prism}_S\]
 carries $X^{1/p^n}$ to $Y^{1/p^{n-1}}$. As $[p]_q \mid Y^i$ in $\overline{\Prism}_S$ for $i \geq p$, it follows that the image of the above map is given by the summands with $i < p$, giving the claim in part (3) for $n=0$. For $n=1$, we use the end of \cref{DerivedPrismaticPrismatic} to see that $\mathrm{gr}^1 \overline{\Prism}_S$ is generated by $\frac{Y^p}{[p]_q}$. The higher graded pieces are then generated by the divided powers of $\frac{Y^p}{[p]_q}$ by the multiplicativity of the conjugate filtration, which by \cref{qMultId} (1) indeed agree up to units with $\frac{Y^{pi}}{[pi]_q!}$. This gives the desired comparison.
% the proof below had a mistake: the statement of (2) and (3) was not quite right as description of both fitrations at level $n$ was off by $1$
%
%Assume first that $i\geq n$. 
%
%As $\lfloor i\rfloor\geq n$, the expression above lies in $[p]_q^n D$, and hence $\frac{Y^i}{[\lfloor i\rfloor]_q!} \in \mathrm{Fil}^n_N D$ for $i \geq n$. Now assume $i < n$. Again using \eqref{qFactorialphi}, we can write
% \[ \phi([p]_{q^{1/p}}^{n-\lfloor i\rfloor} \cdot \frac{Y^i}{[\lfloor i\rfloor]_q!}) = [p]_q^{n-\lfloor i\rfloor} \cdot \frac{Y^{ip}}{\phi([\lfloor i\rfloor_q!)} = [p]_q^n \cdot \frac{Y^{ip}}{[\lfloor i\rfloor p]_q!} \cdot u\]
% for a unit $u$, and hence the element $[p]_{q^{1/p}}^{n-\lfloor i\rfloor} \cdot \frac{Y^i}{[\lfloor i\rfloor]_q!}$ indeed lies in $\mathrm{Fil}^n_N D$ for $i < n$. The same formula (and the fact that $\phi$ is bijective on $A$) also shows that no smaller multiple of $\frac{Y^i}{[\lfloor i\rfloor]_q!}$ can lie in $\mathrm{Fil}^n_N D$ for $i < n$.
%
%Part (2) is now immediate. For part (3), we first observe that for $i=0$ we indeed have agreement. For $i=1$, we use the end of \cref{DerivedPrismaticPrismatic} to see that $\mathrm{gr}^1$ is generated by $\frac{Y^p}{[p]_q}$. The higher graded pieces are then generated by the divided powers of $\frac{Y^p}{[p]_q}$ by the multiplicativity of the conjugate filtration, which by \eqref{qFactorialphi} indeed agree up to units with $\frac{Y^{pi}}{[pi]_q!}$. This gives the desired comparison.
\end{proof}

\subsection{The Nygaard filtration for finitely presented semiperfectoid rings}
\label{ss:NygaardFPSPerfd}

In this subsection we prove the following result.

\begin{proposition}\label{PropNygaard} Theorem~\ref{ThmNygaard} holds true for $S=R\langle X_1^{1/p^\infty},\ldots,X_n^{1/p^\infty}\rangle/(f_1,\ldots,f_m)$ where $f_1,\ldots,f_m$ is a $p$-completely regular sequence.
\end{proposition}

\begin{proof}
For the purposes of this proof, we make the following definition for a quasiregular semiperfectoid $R$-algebra $S$:

\begin{itemize}
\item[$(\ast)$] A multiplicative filtration $\{\mathrm{Fil}^i_M \Prism_S\}_{i \geq 0}$ by ideals on the ring $\Prism_S$ is {\em good} if $\phi(\mathrm{Fil}^i_M) \subset d^i \Prism_S$ for all $i$ and the induced map $\mathrm{gr}^i_M \Prism_S \xrightarrow{\frac{\phi}{d^i}} \overline{\Prism}_S$ identifies the source with $\mathrm{Fil}_i \overline{\Prism}_S$ for all $i$. 
\end{itemize}

The goodness condition has the following properties:

\begin{enumerate}
\item Given a quasiregular semiperfectoid $R$-algebra $S$ and a good filtration $\mathrm{Fil}^*_M$, we must have $\mathrm{Fil}^*_M = \mathrm{Fil}^*_N$: indeed, this follows by induction on $i$ using the map 
\[ \xymatrix{ 0 \ar[r] & \mathrm{Fil}^{i+1}_M \Prism_S \ar[r] \ar[d]^\phi & \mathrm{Fil}^i_M \Prism_S \ar[r] \ar[d]^\phi & \mathrm{gr}^i_M \Prism_S \ar[d]^\phi \ar[r] & 0 \\
0 \ar[r] & d^{i+1} \Prism_S \ar[r] & d^i \Prism_S \ar[r] & d^i \Prism_S/d^{i+1} \Prism_S \ar[r] & 0 }\]
of short exact sequences, the definition of the Nygaard filtration, and the injectivity of the rightmost vertical arrow in the above map of exact sequences.

\item If the $\Prism_S$ of a quasiregular semiperfectoid $R$-algebra $S$ admits a good filtration, then Theorem~\ref{ThmNygaard} holds true for $S$: this follows by the definition of a good filtration and (1). 

\item Say $S \to S'$ is a map of quasiregular semiperfectoid $R$-algebras which is relatively perfect in the following sense: it is obtained via base change from a map of perfectoid $R$-algebras. If  $\Prism_S$ supports a good filtration $\{\mathrm{Fil}^i_M \Prism_S\}_{i \geq 0}$, then $\{\mathrm{Fil}^i_{M'} \Prism_{S'} := \mathrm{Fil}^i_M \Prism_S \widehat{\otimes}^L_{\Prism_S} \Prism_{S'} \}_{i \geq 0}$ is a good filtration on $\Prism_{S'}$. Indeed, by the base change compatibility of the conjugate filtration along such maps $S \to S'$, we know that $\mathrm{gr}^i_{M'} \Prism_{S'} \simeq \mathrm{Fil}_i \overline{\Prism}_{S'}$ is concentrated in degree $0$ for all $i$, and hence the same is true for $ \Prism_{S'} / \mathrm{Fil}^i_{M'} \Prism_{S'}$ for all $i$. Moreover, by base change, we also know that $\Prism_{S'} \to \Prism_{S'} / \mathrm{Fil}^i_{M'} \Prism_{S'}$ is surjective for all $i$, so each $\mathrm{Fil}^i_{M'} \Prism_{S'}$ is concentrated in degree $0$ and an ideal of $\Prism_{S'}$. By the assumption on $S \to S'$, we also know that the Frobenius $\Prism_{S'} \to \Prism_{S'}$ is the base change of the Frobenius on $\Prism_S$. Using this remark, it is easy to see now that $\{\mathrm{Fil}^i_{M'} \Prism_{S'}\}_{i \geq 0}$ is a good filtration $\Prism_{S'}$.

\item Say $S \to T$ is a map of quasiregular semiperfectoid $R$-algebras which is relatively perfect in the sense of (3) and further is $p$-completely faithfully flat. If $\Prism_{T}$ supports a good filtration, so does $\Prism_S$. Indeed, consider the Cech nerve $S \to T^\bullet$ of $S \to T$. Applying (3) to $T^\bullet$, we learn that each $\Prism_{T^i}$ supports a good filtration via base change from the one for $\Prism_T$ along some structure map $T \to T^i$. By (2), this filtration on $\Prism_{T^\bullet}$ must be the Nygaard filtration, so it is independent of the structure map used in the previous sentence. In other words, the Nygaard filtration $\{\mathrm{Fil}^i_N \Prism_{T^\bullet}\}_{i \geq 0}$ defines a cartesian cosimplicial $(p,d)$-complete complex over the cosimplicial ring $\Prism_{T^\bullet}$. Now $\Prism_S \to \Prism_T$ is $(p,d)$-completely faithfully flat with Cech nerve $\Prism_{T^\bullet}$, so the Nygaard filtration $\{\mathrm{Fil}^i_N \Prism_{T^\bullet}\}_{i \geq 0}$ on $\Prism_{T^\bullet}$ must arise as the base change of a multiplicative filtration $\{\mathrm{Fil}^i_M \Prism_S\}_{i \geq 0}$ on $\Prism_S$, a priori merely in the filtered derived category. To finish, we must check that $\{\mathrm{Fil}^i_M \Prism_S\}_{i \geq 0}$ defines a good filtration on $\Prism_S$. For each $i$, the map $\Prism_S \to \Prism_S/\mathrm{Fil}^i_M \Prism_S$ base changes along $\Prism_S \to \Prism_T$ to a surjective map of modules in degree $0$. By $(p,d)$-complete faithful flatness, we conclude that each $\Prism_S/\mathrm{Fil}^i_M \Prism_S$ is connective and the map $\Prism_S \to \Prism_S/\mathrm{Fil}^i_M \Prism_S$ is surjective on $H^0$. This implies that each $\mathrm{Fil}^i_M \Prism_S$ is connective and the map to $\Prism_S$ is injective on $H^0$. On the other hand, the descent formula $\mathrm{Fil}^i_M \Prism_S \simeq \lim_\bullet \mathrm{Fil}^i_N \Prism_{T^\bullet}$ shows that $\mathrm{Fil}^i_M \Prism_S$ is also coconnective, so $\{\mathrm{Fil}^i_M \Prism_S\}_{i \geq 0}$ is indeed a multiplicative filtration by ideals on $\Prism_S$. Using $(p,d)$-complete faithful flatness and base change for the conjugate filtration, we then conclude that $\{\mathrm{Fil}^i_M \Prism_S\}_{i \geq 0}$ is a good filtration on $\Prism_S$. 
\end{enumerate}

We can now prove the proposition by devissage. 

First, consider  $S(m) =R\langle X_1^{1/p^\infty},\ldots,X_m^{1/p^\infty}\rangle/(X_1,\ldots,X_m)$. Lemma~\ref{NygaardKeyCase} gives the theorem for $S(1)$. The general $m$ case follows by the K\"{u}nneth formula: tensoring together the Nygaard filtration for each $R \langle X_i^{1/p^\infty} \rangle/(X_i)$ gives a good filtration, so the claim follow from property (2) above. 

Next, consider a quasiregular semiperfectoid ring of the form $R'/(f_1,...,f_r)$, where $R'$ is a perfectoid $R$-algebra and $f_1,...,f_r \in R'$ is a sequence of elements that is $p$-completely regular relative to $R$ and such that each $f_i$ admits a compatible system of $p$-power roots in $f_i$. Such a ring receives a relatively perfect map (in the sense of property (3) above) from $S(r)$, so the proposition follows in this case from property (3). 

Finally, by Andr\'e's lemma, any $S$ as in the proposition admits a relatively perfect and $p$-completely faithfully flat cover by the type of quasiregular semiperfectoid ring treated in the previous paragraph, so the claim follows from property (4).
\end{proof}

%
%
%\begin{proof} We may replace $R$ by $R\langle X_1^{1/p^\infty},\ldots,X_n^{1/p^\infty}\rangle$ and assume that $n=0$. By Andr\'e's lemma, Theorem~\ref{AndreFlatness}, we may assume that all $f_i$ admit a compatible sequence of $p$-power roots. We note that to prove the proposition, it suffices to exhibit some (a priori possibly different) filtration $\mathrm{Fil}^i_M \Prism_S\subset \Prism_S$ such that $\phi$ is divisible by $d^i$ on $\mathrm{Fil}^i_M \Prism_S$ and the induced map
%\[
%\mathrm{gr}^i_M\Prism_S\xrightarrow{\frac{\phi}{d^i}} \overline{\Prism}_S
%\]
%is injective with image $\mathrm{Fil}_i \overline{\Prism}_S$. Indeed, these conditions imply by induction on $i$ that $\mathrm{Fil}^i_M \Prism_S = \mathrm{Fil}^i_N \Prism_S$, which then has the desired property. In particular, we can reduce to
%\[
%S^\prime=R\langle X_1^{1/p^\infty},\ldots,X_m^{1/p^\infty}\rangle/(X_1,\ldots,X_m)
%\]
%as there is a map $R^\prime=R\langle X_1^{1/p^\infty},\ldots,X_m^{1/p^\infty}\rangle\to R$ (given by the $f_i$ and their roots) such that $S = S^\prime\widehat{\otimes}^L_{R^\prime} R$ and thus
%\[
%\Prism_S = \Prism_{S^\prime}\widehat{\otimes}^L_{\Prism_{R^\prime}} \Prism_R
%\]
%and we can set
%\[
%\mathrm{Fil}^i_M \Prism_S = \mathrm{Fil}^i_N \Prism_{S^\prime}\widehat{\otimes}^L_{\Prism_{R^\prime}} \Prism_R\ .
%\]
%Moreover, by the K\"unneth formula, we can reduce to the case $m=1$, and we can also assume that $R=\mathbf Z_p[\zeta_{p^\infty}]^\wedge$ (by Andr\'e's lemma and base change). Thus, the proposition follows from Lemma~\ref{NygaardKeyCase}.
%\end{proof}

\subsection{The Nygaard filtration on prismatic cohomology}\label{subsec:generalnygaard}

Fix a perfectoid ring $R$ corresponding to a perfect prism $(A,I)$. We want to endow $\Prism_{B/A}$ with a functorial Nygaard filtration for any $p$-completely smooth $R=A/I$-algebra $B$. Although we will later be able to give a better definition, we use the following direct recipe.

For any surjection $R\langle X_1,\ldots,X_n\rangle\to B$, the ring
\[
\tilde{B} = R\langle X_1^{1/p^\infty},\ldots,X_n^{1/p^\infty}\rangle\otimes_{R\langle X_1,\ldots,X_n\rangle} B
\]
is semiperfectoid and Zariski locally of the form considered in the previous subsection, so by localization Theorem~\ref{ThmNygaard} holds true for $\tilde{B}$. The same applies to all terms of the \v{C}ech nerve $\tilde{B}^\bullet$ of $B\to \tilde{B}$. The cosimplicial $\delta$-ring
\[
\Prism_{\tilde{B}^\bullet}
\]
computes $\Prism_{B/A}$ (for example, by the Hodge-Tate comparison and quasisyntomic descent).

\begin{definition} The Nygaard filtration
\[
\mathrm{Fil}^i_N \Prism_{B/A}\to \Prism_{B/A}
\]
is the totalization of $\mathrm{Fil}^i_N \Prism_{\tilde{B}^\bullet}\subset \Prism_{\tilde{B}^\bullet}$.
\end{definition}

\begin{proposition} There is a natural map
\[
``\frac{\phi}{d^i}": \mathrm{Fil}^i_N \Prism_{B/A}\to \Prism_{B/A}
\]
that projects to an isomorphism of $\mathrm{gr}^i_N \Prism_{B/A}\cong \tau^{\leq i} \Prism_{B/A}$.
\end{proposition}

\begin{proof} This follows from Proposition~\ref{PropNygaard} by passing to totalizations of the cosimplicial objects.
\end{proof}

It follows that the Nygaard filtration on $\Prism_{B/A}$ is independent of the choice of the surjection $R\langle X_1,\ldots,X_n\rangle\to B$: adding extra variables to the $X_i$, one gets a comparison map between the two induced Nygaard filtrations, and it induces isomorphisms on $\mathrm{gr}^i_N \Prism_{B/A}$, so by descending induction on $i$ on all $\mathrm{Fil}^i_N \Prism_{B/A}$.

\subsection{End of proof}
\label{ss:EndProofNygaard}

Finally, we can finish the proof of Theorem~\ref{ThmNygaard}. First, by left Kan extension of $B\mapsto \mathrm{Fil}^i_N \Prism_{B/A}$ we can define a ``derived Nygaard filtration" $\mathrm{Fil}^i_{N'} \Prism_{B/A}$ on $\Prism_{B/A}$ for any derived $p$-complete simplicial $R$-algebra, with $\mathrm{gr}^i_{N'} \Prism_{B/A}\cong \mathrm{Fil}_i \overline{\Prism}_{B/A}$. In case $B=S$ is quasiregular semiperfectoid, these properties imply that $\mathrm{Fil}^i_{N'} \Prism_{B/A}$ is concentrated in cohomological degrees $0$ and $-1$ for all $i\geq 0$.

\begin{proposition} For all $i\geq 0$, the complex $\mathrm{Fil}^i_{N'} \Prism_S$ sits in degree $0$ and defines a filtration of $\Prism_S$ that agrees with $\mathrm{Fil}^i_N \Prism_S$, and Theorem~\ref{ThmNygaard} holds true for $S$.
\end{proposition}

\begin{proof} We argue by induction on $i$, so assume that $\mathrm{Fil}^i_{N'} = \mathrm{Fil}^i_N$ holds true for some $i$. This is clearly true for $i=0$, giving the inductive start. First, we check that then Theorem~\ref{ThmNygaard} holds true in degree $i$, i.e.~the image of
\[
\frac{\phi}{d^i}: \mathrm{Fil}^i_N \Prism_S\to \overline{\Prism}_S
\]
is given by $\mathrm{Fil}_i \overline{\Prism}_S$. Using $\mathrm{Fil}^i_N = \mathrm{Fil}^i_{N'}$, one sees that the image is at most $\mathrm{Fil}_i \overline{\Prism}_S$. On the other hand, by picking a surjection $S^\prime=R\langle X_i^{1/p^\infty},Y_j^{1/p^\infty}\rangle/(Y_j)\to S$ that also induces a surjection on cotangent complexes and using Proposition~\ref{PropNygaard} (and passage to filtered colimits), one sees that the image is at least $\mathrm{Fil}_i \overline{\Prism}_S$. Thus, Theorem~\ref{ThmNygaard} holds true in degree $i$.

We also see that $\mathrm{Fil}^i_{N'} \Prism_S\to \mathrm{gr}^i_{N'} \Prism_S$ is surjective, and hence its derived kernel $\mathrm{Fil}^{i+1}_{N'} \Prism_S$ is still concentrated in degree $0$. Moreover, it agrees with the kernel of
\[
\frac{\phi}{d^i}: \mathrm{Fil}^i_N \Prism_S\to \overline{\Prism}_S
\]
which by definition is $\mathrm{Fil}^{i+1}_N \Prism_S$. Thus, $\mathrm{Fil}^{i+1}_{N'} \Prism_S = \mathrm{Fil}^{i+1}_N \Prism_S$, as desired.
\end{proof}

\newpage

\section{Comparison with \cite{BMS2}}
\label{sec:BMS2Comp}

The goal of this section is to compare the constructions of this paper with those in \cite{BMS2}.  Recall that for a quasiregular semiperfectoid $S$, we defined in \cite[\S 7.2]{BMS2} the ring $\widehat{\Prism}_S=\pi_0 \mathrm{TC}^-(S;\mathbf Z_p)=\pi_0 \mathrm{TP}(S;\mathbf Z_p)$ together with a Frobenius semilinear endomorphism $\phi_S$ (induced by the cyclotomic Frobenius map). The fundamental diagram
\begin{equation}
\label{BMS2THHFrob}
 \xymatrix{ \mathrm{TC}^{-}(S;\mathbf{Z}_p) \ar[r]^{\phi^{h\mathbb{T}}} \ar[d] & \mathrm{TP}(S;\mathbf{Z}_p) \ar[d] \\
\mathrm{THH}(S;\mathbf{Z}_p) \ar[r]^{\phi} & \mathrm{THH}(S;\mathbf{Z}_p)^{tC_p} }
\end{equation}
then yields on $\pi_0$ a commutative diagram
\begin{equation}
\label{BMS2PrismFrob}
 \xymatrix{ \widehat{\Prism}_S \ar^{\phi_S}[r] \ar[d]^-{a_S} & \widehat{\Prism}_S \ar[d]^{can} \\
			S\ar[r]^-{\eta_S} & \widehat{\Prism}_S/d}
\end{equation}
of commutative algebras. By construction, $\widehat{\Prism}_S$ is complete for a filtration $\mathrm{Fil}^i_N \widehat{\Prism}_S\subset \widehat{\Prism}_S$ also termed the Nygaard filtration in \cite{BMS2}. Our comparison theorem is the following:

\begin{theorem}
\label{BMS2CompNC} There is a functorial (in $S$) $\delta$-ring structure on $\widehat{\Prism}_S$ refining the endomorphism $\phi$. The induced map $\Prism_S=\Prism_S^{\mathrm{init}}\to \widehat{\Prism}_S$ identifies $\widehat{\Prism}_S$ with the Nygaard completion of $\Prism_S$, compatibly with Nygaard filtrations.
\end{theorem}

Note that the $\delta$-ring structure on $\widehat{\Prism}_S$ is uniquely determined when $\widehat{\Prism}_S$ is $p$-torsionfree, which happens for example if $S$ is $p$-torsionfree. In that case, the first part of the theorem simply says that the endomorphism $\phi$ on $\widehat{\Prism}_S$ lifts the Frobenius. Our proof of this fact is very indirect and we do not know of a good conceptual reason for this in terms of topological Hochschild homology.

\begin{proof} First we check the claim, independent of the theory of \cite{BMS2}, that the $\delta$-ring structure on $\Prism_S$ extends uniquely to a continuous $\delta$-ring structure on its Nygaard completion. This follows from the following lemma.

\begin{lemma} Let $S$ be any quasiregular semiperfectoid ring.
\begin{enumerate}
\item One has
\[
\delta(\mathrm{Fil}^i_N \Prism_S)\subset \mathrm{Fil}^{pi}_N \Prism_S + (d,p)^{i-1} \Prism_S\ .
\]
\item The completion of $\Prism_S$ with respect to the sequence of ideals $\mathrm{Fil}^i_N \Prism_S$ agrees with its completion with respect to the sequence of ideals $\mathrm{Fil}^i_N \Prism_S + (d,p)^j\Prism_S$.
\end{enumerate}
\end{lemma}

\begin{proof} In part (1), pick a perfectoid ring $R$ mapping to $S$. There is a surjection $S^\prime=R\langle X_i^{1/p^\infty},Y_j^{1/p^\infty}\rangle/(X_i)\to S$ inducing a surjection on cotangent complexes and thus on $\Prism_S$ and all $\mathrm{Fil}^i_N \Prism_S$ by the explicit description. Thus we can assume $S=S^\prime$. Replacing $R$ by $R\langle Y_j^{1/p^\infty}\rangle$ we can assume there are no $Y_j$'s. By a filtered colimit argument one can reduce to the case that there are only finitely many $X_i$. Next, we want to use the K\"unneth formula to reduce to the case of a single $X$. For this, we need to remark that it is enough to check the claim on a set of generators for the ideal $\mathrm{Fil}^i_N \Prism_S$ by the addition formula for $\delta$, and that if $i=j+k$ and $x\in \mathrm{Fil}^j_N \Prism_S$, $y\in \mathrm{Fil}^k_N \Prism_S$ satisfy
\[
\delta(x) = x_1+x_2\in \mathrm{Fil}^{pj}_N \Prism_S + (d,p)^{j-1} \Prism_S\ ,\ \delta(y) = y_1 +y_2\in \mathrm{Fil}^{pk}_N \Prism_S + (d,p)^{k-1} \Prism_S\ ,
\]
then
\[
\delta(xy) = x^p \delta(y) + y^p \delta(x) + p\delta(x)\delta(y) = x^py_1 + x^p y_2 + y^p x_1 + y^p x_2 + px_1y_1 + px_1y_2 + px_2y_1 + px_2y_2\ ,
\]
where
\[\begin{aligned}
px_2y_2&\in (d,p)^{i-1}\Prism_S\ ,\\
x^py_1,y^px_1,px_1y_1&\in \mathrm{Fil}^{pi}_N \Prism_S\ ,\\
(x^p+px_1)y_2=(\phi(x)-px_2)y_2&\in (d,p)^{i-1}\Prism_S\ ,\\
(y^p+py_1)x_2=(\phi(y)-py_2)x_2&\in (d,p)^{i-1}\Prism_S\ ,
\end{aligned}\]
so indeed $\delta(xy)\in \mathrm{Fil}^{pi}_N \Prism_S + (d,p)^{i-1}\Prism_S$. Thus, finally, we can reduce to the case of one variable. Using \cref{NygaardKeyCase} and the previous considerations, it is enough to check the claim for $x=[p]_{q^{1/p}}\in \mathrm{Fil}^1_N \Prism_S$ and $x=\frac{Y^i}{[i]_q!}\in \mathrm{Fil}^i_N \Prism_S$ (as these generate the Nygaard filtration multiplicatively). In the first case $i=1$ and the claim is trivial as $(d,p)^{i-1}\Prism_S = \Prism_S$. In the other case $\delta(x)$ is homogeneous of degree $pi$, and all such elements lie in $\mathrm{Fil}^{pi}_N \Prism_S$, as desired.

For part (2), it is enough to show that all $\mathrm{gr}^i_N \Prism_S$ are classically $(d,p)$-adically complete. Under the $\phi$-linear identification with $\mathrm{Fil}_i \overline{\Prism}_S$, it is enough to prove that the latter are classically $(\phi(d),p)$-adically complete. But $\phi(d)$ agrees up to units with $p$ modulo $d$, so it suffices to prove that the latter are classically $p$-adically complete. This follows from them being derived $p$-complete and $p$-completely flat over $S$.
\end{proof}

Now we start the proof of the theorem. Fix a perfectoid ring $R$ corresponding to a perfect prism $(A,(d))$, and restrict to $R$-algebras $S$. By Andr\'e's lemma, we may assume that there is a compatible systems of $p$-power roots of unity in $R$, and so we can take $R=\mathbf Z_p[\zeta_{p^\infty}]^\wedge$. By descent to $p$-completely smooth $R$-algebras and left Kan extension, we defined in \cite[Construction 7.12]{BMS2} a functor $S\mapsto \widehat{\Prism}_{S/A}^{\mathrm{nc}}$ (there simply denoted by $\Prism_{S/A}$ as will be justified a posteriori by the theorem; here $\mathrm{nc}$ stands for ``non-completed"). This can be endowed with a Nygaard filtration $\mathrm{Fil}^i_N \widehat{\Prism}_{S/A}^{\mathrm{nc}}$ by the same procedure, and then $\widehat{\Prism}_{S/A}$ is the Nygaard completion for quasisyntomic $R$-algebras $S$, in particular if $S$ is quasiregular semiperfectoid. By the Segal conjecture for smooth algebras, \cite[Corollary 9.12]{BMS2}, one has a Hodge-Tate comparison for $\widehat{\Prism}_{S/A}/d$ in the smooth case, and thus also in general. Moreover, the same result and the identification of $\mathrm{gr}^i_N \widehat{\Prism}_{S/A}^{\mathrm{nc}}$ with $\mathrm{gr}^i \mathrm{THH}(S;\mathbf Z_p)$ implies that for quasiregular semiperfectoid $S$ the Nygaard filtration is given by
\[
\mathrm{Fil}^i_N \widehat{\Prism}_{S/A}^{\mathrm{nc}} = \{x\in\widehat{\Prism}_{S/A}^{\mathrm{nc}}\mid \phi(x)\in d^i \widehat{\Prism}_{S/A}^{\mathrm{nc}}\}
\]
and the image of
\[
\frac{\phi}{d^i}: \mathrm{Fil}_i^N \widehat{\Prism}_{S/A}^{\mathrm{nc}}\to \widehat{\Prism}_{S/A}^{\mathrm{nc}}/d
\]
is given by the conjugate filtration $\mathrm{Fil}_i \widehat{\Prism}_{S/A}^{\mathrm{nc}}/d$ coming from the Hodge-Tate comparison. Moreover, the graded pieces $\mathrm{gr}_i \widehat{\Prism}_{S/A}^{\mathrm{nc}}/d$ are given by the $p$-completion of $\wedge^i L_{S/R}[-i]$. In other words, structurally $\widehat{\Prism}_{S/A}^{\mathrm{nc}}$ has exactly the same properties as $\Prism_S$.

Our task now is to show that $\widehat{\Prism}_{S/A}^{\mathrm{nc}} = \Prism_S$ compatibly with $\phi$, as $\widehat{\Prism}_{S/A}$ is the Nygaard completion of the left, which then inherits its functorial $\delta$-ring structure by the lemma. As both sides are defined via left Kan extensions from $p$-completely smooth algebras, and in the $p$-completely smooth case via descent, it suffices to prove the result for the quasiregular semiperfectoid algebras $S$ required in this descent. In other words, as in \S \ref{subsec:generalnygaard}, we can assume $S=R\langle X_1^{1/p^\infty},\ldots,X_n^{1/p^\infty}\rangle/(f_1,\ldots,f_m)$ where $f_1,\ldots,f_m$ is a $p$-completely regular sequence relative to $R$. In particular, in this case $\widehat{\Prism}_{S/A}^{\mathrm{nc}}$ is $p$-torsion free, so the $\delta$-ring structure is unique if it exists, and then the map $\Prism_S\to \widehat{\Prism}_{S/A}^{\mathrm{nc}}$ is unique. Following the steps in the proof Proposition~\ref{PropNygaard}, we can assume that all $f_i$ admit compatible $p$-power roots in $R$, and then to $S=R\langle X^{1/p^\infty}\rangle/(X)$ by K\"unneth and base change.

In this case, we can apply Lemma~\ref{GetqPD} to $D=\widehat{\Prism}_{S/A}^{\mathrm{nc}}$. Using the explicit description of $\Prism_S$, this gives a $\phi$-equivariant map
\[
\Prism_S\to \widehat{\Prism}_{S/A}^{\mathrm{nc}}
\]
in this case. This sits in a commutative diagram
\[\xymatrix{
\Prism_S\ar[r]\ar[d] & \widehat{\Prism}_{S/A}^{\mathrm{nc}}\ar[d]\\
\Prism_{S/p}\ar[r]^\cong & \widehat{\Prism}_{(S/p)/A}^{\mathrm{nc}},
}\]
where the lower isomorphism comes from \cite[Theorem 8.17]{BMS2} (plus descent and left Kan extension to pass to the non-completed version of $\widehat{\Prism}_{(S/p)/\mathbf F_p}^{\mathrm{nc}}$, and independence of the perfectoid base). Moreover the upper map is the derived $p$-completed base extension of $\Prism_S\to \widehat{\Prism}_{S/A}^{\mathrm{nc}}$ along $A=A_\inf(R)\to A_\crys(R/p)$ (as both $\Prism$ and $\widehat{\Prism}^{\mathrm{nc}}$ satisfy base change and using \cite[Theorem 8.17]{BMS2} again). By derived Nakayama, it suffices to see that $\Prism_S\to \widehat{\Prism}_{S/A}^{\mathrm{nc}}$ is an isomorphism after base change along $A_\inf(R)\to A_\inf(R)/(p,d)$, but this map factors over $A_\crys(R/p)$, so the result follows.
\end{proof}

\newpage

\section{$p$-adic Tate twists and the odd vanishing conjecture for $K$-theory}
\label{sec:OddVanishing}

In \cite{BMS2}, for each $n \geq 0$, we defined a sheaf $\mathbf{Z}_p(n)$ of complexes on the category of $p$-complete quasisyntomic rings. It was given by the formula
\[ \mathbf{Z}_p(n)(S) := \mathrm{fib}(\mathcal{N}^{\geq n} \widehat{\Prism}_{S}\{n\} \xrightarrow{1-\phi_n} \widehat{\Prism}_{S}\{n\}).\]
Here $\widehat{\Prism}_S$ is a version of the Nygaard completed prismatic cohomology of $S$ defined using topological Hochschild homology and quasisyntomic descent in \cite{BMS2}; when $S$ lives over a perfectoid ring, this theory coincides with the Nygaard completion of prismatic cohomology as defined in this paper, with $\mathcal{N}^{\geq \cdot} \widehat{\Prism}_S$ being the Nygaard filtration (\cref{BMS2CompNC}). Our goal in this section is to prove the following structural property of this sheaf, conjectured in \cite[Conjecture 7.18]{BMS2}:

\begin{theorem}
\label{OddVanishing}
For each $n \geq 0$, the sheaf $\mathbf{Z}_p(n)$ is discrete and $p$-torsionfree.
\end{theorem}

It was shown in \cite[Theorem 1.12 (5)]{BMS2} that ($p$-complete) topological cyclic homology of quasisyntomic rings admits a complete descending $\mathbf{N}$-indexed filtration with graded pieces given by the sheaves $\mathbf{Z}_p(n)$'s. The paper \cite{ClausenMathewMorrow} identifies $K$-theory with topological cyclic homology for a large class of rings. Theorem~\ref{OddVanishing} then has the following consequence:

\begin{corollary}
\label{OddVanishingLocally}
Locally on the quasisyntomic site, the functor $K(-;\mathbf{Z}_p)$ is concentrated in even degrees, i.e., $\pi_n K(-;\mathbf{Z}_p)$ vanishes for $n$ odd after quasisyntomic sheafification.
\end{corollary}

Another corollary of the method of the proof is the following, showing a slightly more precise form of \cref{OddVanishingLocally} in a certain situation; notably, this shows that $\pi_* K(\mathcal{O}_C/p^n;\mathbf{Z}_p)$ is concentrated in even degrees for $n \geq 0$, where $C/\mathbf{Q}_p$ is a complete and algebraically closed extension. We thank Martin Speirs for raising this question.

\begin{corollary}
\label{OddVanishingExamples}
Say $R$ is a perfectoid $\mathbf{Z}_p[\zeta_{p^\infty}]^{\wedge}$-algebra. Fix a regular sequence $f_1,...,f_r \in R$ such that each $f_i$ admits a compatible system of $p$-power roots; let $S = R/(f_1,...,f_r)$. The map $\pi_* K(R;\mathbf{Z}_p) \to \pi_* K(S;\mathbf{Z}_p)$ is surjective  in odd degrees. In particular, if $\pi_* K(R;\mathbf{Z}_p)$ is concentrated in even degrees, the same holds true for $\pi_* K(S;\mathbf{Z}_p)$. 
\end{corollary}

The proof of \cref{OddVanishing} depends on two inputs. First, we use Andr\'e's lemma (\cref{AndreFlatness})  to restrict attention to a particular nice class of quasisyntomic rings. Secondly, we use the explicit description of prismatic cohomology and its Nygaard filtration for this class of rings (coming from \S \ref{ss:NygaardSpecial}) to make calculations. Let us begin with the latter.

\begin{lemma}
\label{PerfdSPerfdSurjTate}
Let $R$ be a perfectoid $\mathbf{Z}_p[\zeta_{p^\infty}]^{\wedge}$-algebra. Set $R' = R\langle X_1^{1/p^\infty},....,X_r^{1/p^\infty}\rangle$ and $S = R'/(X_1,....,X_r)$. The natural map $\mathbf{Z}_p(n)(R') \to \mathbf{Z}_p(n)(S)$ is surjective on $H^1$.
\end{lemma}
\begin{proof}
Let $A = A_{\inf}(R)$, so $(A,[p]_q)$ is the perfect prism corresponding to $R$, and we have
\begin{equation}
\label{TateTwistA}
 A_{\inf}(R') = \widehat{\bigoplus_{\underline{i} \in \mathbf{N}[1/p]^r}} A \cdot X^{\underline{i}}, \quad \text{where} \quad X^{\underline{i}} = \prod_{j=1}^r X_j^{i_j}.
 \end{equation}
In this description, the Nygaard filtration is given by 
\[\mathrm{Fil}_N^n A_{\inf}(R') = [p]_{q^{1/p}}^n A_{\inf}(R').\]
Write 
\[ \eta_{R'}:A_{\inf}(R') \twoheadrightarrow \mathrm{coker}(\phi_n-1:\mathrm{Fil}_N^n A_{\inf}(R') \to A_{\inf}(R')) =: H^1(\mathbf{Z}_p(n)(R'))\] 
for the canonical map.

Write $D = \Prism_{S/A}$, so 
\begin{equation}
\label{TateTwistD}
 D = A_{\inf}(R')\{\frac{X_1^p}{[p]_q},....,\frac{X_r^p}{[p]_q}\}^{\wedge} \simeq \widehat{\bigoplus_{\underline{i} \in \mathbf{N}[1/p]^r}} A \cdot \frac{X^{\underline{i}}}{[\lfloor \underline{i} \rfloor]_q!}, \quad \text{where} \quad {[\lfloor \underline{i} \rfloor]_q!} = \prod_{j=1}^r [\lfloor i_j \rfloor]_q!.
 \end{equation}
Taking products of the calculation in \eqref{FrobeniusqPD}, we find
 \[ \phi(\frac{X^{\underline{i}}}{[\lfloor \underline{i} \rfloor]_q!}) = \frac{X^{\underline{i}p}}{[\lfloor \underline{i}p \rfloor]_q!} \cdot [p]_{q}^{\sum_j \lfloor i_j \rfloor} \cdot u,\]
 where $u$ is a unit. Using this description and the $[p]_q$-torsionfreeness of $A$, one checks that the Nygaard filtration is given by
\begin{equation}
\label{TateTwistFil}
 \mathrm{Fil}_N^n D = \widehat{\bigoplus_{\underline{i} \in \mathbf{N}[1/p]^r}}  [p]_{q^{1/p}}^{n - \sum_j \lfloor i_j \rfloor}   \cdot A \cdot \frac{X^{\underline{i}}}{[\lfloor \underline{i} \rfloor]_q!},
 \end{equation}
where we declare $[p]_q^a = 1$ for $a \leq 0$.  Again, we write
\[ \eta_{S}:D \twoheadrightarrow \mathrm{coker}(\phi_n-1:\mathrm{Fil}_N^n D \to D) =: H^1(\mathbf{Z}_p(n)(S))\] 
for the canonical map.

Using these explicit descriptions, we shall prove the lemma by analyzing the surjective map $\eta_S$ on various graded pieces of $D$. Before delving into the specifics, let us explain the structure of the argument. The key is to observe (by comparing  \eqref{TateTwistA} and \eqref{TateTwistD}) that the map $A_{\inf}(R') \to D$ is bijective on components of sufficiently small degree $\underline{i}$; in fact, it suffices to assume that $i_j < 1$ for all $j$ since that forces the denominator appearing in \eqref{TateTwistD} to be $1$. Thus, the lemma would follow once we knew that $H^1(\mathbf{Z}_p(n)(S))$ is already spanned by the image of components of $D$ with sufficiently small degree under the $\eta_S$. We shall verify this to be the case by using the explicit description of the Nygaard filtration given in \eqref{TateTwistFil}.

To carry out this strategy, it is convenient to use the following notation. Given a $(p,[p]_q)$-complete $A_{\inf}(R)$-module $M$ equipped with a $(p,[p]_q)$-complete $\mathbf{N}[1/p]^r$-grading 
\[M := \widehat{\bigoplus_{\underline{i} \in \mathbf{N}[1/p]^r}} M_{\underline{i}}\]
(such as $A$ or $D$ or $\mathrm{Fil}_N^n D$) and any $\alpha \in \mathbf{N}$, we write 
\[ M_{\leq \alpha, \sqcup} := \widehat{\bigoplus_{\underline{i} \in \mathbf{N}[1/p]^r, \sum_j \lfloor i_j\rfloor \leq \alpha}} M_{\underline{i}},\] 
and similarly for $M_{\geq \alpha, \sqcup}$. For future use, we observe that 
\[ M \simeq M_{\leq \alpha, \sqcup} \oplus M_{\geq \alpha+1,\sqcup} \quad \text{and} \quad M_{\leq 0,\sqcup} = \widehat{\bigoplus_{\underline{i} \in \mathbf{N}[1/p]^r \cap [0,1)^r}} M_{\underline{i}}\]
with these conventions.

We now begin the proof.  First, we claim that 
\[ \eta_S(D_{\geq n+1,\sqcup}) = 0.\]
To see this, we must check that $D_{\geq n+1,\sqcup} \subset D$ lies in the image of $\phi_n - 1:\mathrm{Fil}_N^n D \to D$. The description \eqref{TateTwistFil} shows that 
 \[ \phi_n\Big((\mathrm{Fil}_N^n D)_{\geq n+1,\sqcup}\Big) \subset [p]_q D_{\geq n+1,\sqcup}.\]  
 On the other hand, the identity map induces an isomorphism 
 \[(\mathrm{Fil}_N^n D)_{\geq n+1,\sqcup} = D_{\geq n+1,\sqcup},\] 
 again by \eqref{TateTwistFil}. Combining these, it follows that the map 
 \[(\mathrm{Fil}_N^n D)_{\geq n+1,\sqcup} \xrightarrow{\phi_n - 1} D_{\geq n+1,\sqcup}\] 
 is surjective: explicitly, for any $x \in D_{\geq n+1,\sqcup}$, the infinite sum 
\[-(x + \phi_n(x) + \phi_n^2(x) + ...)\] 
converges to an element $y \in D_{\geq n+1,\sqcup}= (\mathrm{Fil}_N^n D)_{\geq n+1,\sqcup}$ that is a lift of $x$ under $\phi_n-1$. This proves  that $\eta_S$ kills $D_{\geq n+1,\sqcup}$, so 
  \[D_{\leq n, \sqcup} \subset D \xrightarrow{\eta_S} H^1(\mathbf{Z}_p(n)(S))\] 
  is surjective.

Next, for any $\underline{i} \in \mathbf{N}[1/p]^r$ with $\sum_j \lfloor i_j \rfloor \leq n$, the map $\phi_n:\mathrm{Fil}_N^nD \to D$ induces an isomorphism
\[ \phi_n:(\mathrm{Fil}_N^nD)_{\frac{1}{p} \cdot \underline{i}} \simeq D_{\underline{i}}\] 
by \eqref{TateTwistFil}. As $\eta_S(x) = \eta_S(\phi_n(x))$ for any $x \in D$, we learn that 
\[ \eta_S(D_{\underline{i}}) \subset \eta_S(D_{\frac{1}{p} \cdot \underline{i}}).\]
Iterating this observation shows that we can divide the degree by arbitrarily many powers of $p$. In particular, after dividing finitely many times, we learn that the map 
\[ D_{\leq 0, \sqcup} \subset D \xrightarrow{\eta_S} H^1(\mathbf{Z}_p(n)(S))\] 
is surjective.

Next, we observe that the map 
\[A_{\inf}(R')_{\leq 0, \sqcup} \to D_{\leq 0, \sqcup}\]
 is surjective. Indeed, the degree $\underline{i}$ terms of \eqref{TateTwistA} and \eqref{TateTwistD} coincide since the degrees $\underline{i}$ occurring here satisfy $i_j < 1$ for all $j$. By functoriality, we conclude 
 \[H^1(\mathbf{Z}_p(n)(R')) \to H^1(\mathbf{Z}_p(n)(S))\] 
 is also surjective.
\end{proof}

The following lemma will be quite useful in making reductions.

\begin{lemma}
\label{SurjLemma}
Fix a perfectoid ring $R$. If $S' \to S$ is a surjection of quasiregular semiperfectoid $R$-algebras such that the $p$-completion of $L_{S'/R} \to L_{S/R}$ is also surjective on $H^{-1}(-)$, then ${\Prism}_{S'} \to {\Prism}_{S}$ is surjective.
\end{lemma}
\begin{proof}
This follows from the Hodge-Tate comparison.
\end{proof}

We can finally put everything together to prove \cref{OddVanishing}.

\begin{proof}[Proof of \cref{OddVanishing}]
By \cite[Remark 7.20]{BMS2} and \cref{BMS2CompNC}, it suffices to prove discreteness. Moreover, as we work locally on the quasisyntomic site, we may restrict attention to quasiregular semiperfectoid $\mathcal{O}_C$-algebras, where $C/\mathbf{Q}$ is an algebraically closed nonarchimedean field. In this setting, the Breuil-Kisin twists and the Nygaard completion may be ignored whilst calculating $\mathbf{Z}_p(n)$: more precisely, the proof of \cite[Lemma 7.22]{BMS2} shows that 
\[ \mathbf{Z}_p(n)(S) \simeq \mathrm{fib}(\mathrm{Fil}_N^n {\Prism}_{S} \xrightarrow{1-\phi_n} {\Prism}_{S}).\]
Consider the following assertion:
\begin{itemize}
\item[ $(\ast)_S$] 
Given a quasiregular semiperfectoid $\mathcal{O}_C$-algebra $S$ and an element $\alpha \in H^1(\mathbf{Z}_p(n)(S))$, there exists a quasi-syntomic cover $S \to S'$ such that  $\alpha$ maps to $0$ in $H^1(\mathbf{Z}_p(n)(S'))$.
\end{itemize}

Our goal is to prove $(\ast)_S$ for all $S$. Let us first prove this when $S = R$ is a perfectoid $\mathcal{O}_C$-algebra. \cref{TateTwistPerfd} and Artin-Schreier theory show that any class in $H^1(\mathbf{Z}_p(0)(-))$ can be annihilated by a pro-(finite \'etale) cover of $R$, which settles the $n=0$ case of $(\ast)_R$. If $n > 0$, then \cref{TateTwistPerfd} (and a trivialization of $\mathbf{Z}_p(1) \cong \mathbf{Z}_p(n)$ on the generic fibre) reduce us to the $n=1$ case. By Kummer theory (and \cite[Proposition 7.17]{BMS2} to identify $\mathbf{Z}_p(1)$ as the quasisyntomic sheaf $\lim_n \mu_{p^n}$), for any perfectoid ring $R$, the group $H^1(\mathbf{Z}_p(1)(R))$ is $H^0$ of the derived $p$-completion of $R^*$; here we use that $\mathrm{Pic}(R)$ is uniquely $p$-divisible by \cref{cor:picperfectoid}. Andr\'e's lemma (\cref{AndreFlatness}) gives a cover $R \to R'$ with $(R')^*$ being $p$-divisible, whence $H^1(\mathbf{Z}_p(1)(R')) = 0$, so we are done.

%As any such $R$ is a quotient of a $p$-torsionfree perfectoid $\mathcal{O}_C$-algebra, we may use \cref{SurjLemma} to assume $R$ is $p$-torsionfree. It is enough to prove that any $\alpha \in H^1(\mathbf{Z}/p(n)(R))$ can be killed after passage to a quasisyntomic cover of $R$. Note that $\mathrm{Fil}_N^n {\Prism}_R = \phi^{-1}(d)^n A_{\inf}(R)$ for all $n$, where $d \in A_{\inf}(R)$ is a distinguished element defining $R$. As $R$ is $p$-torsionfree, the tilt $R^\flat$ is $d$-torsionfree. Unwinding definitions then shows that an element $\alpha$ of $H^1(\mathbf{Z}/p(n)(R))$ is  represented by an element $x \in R^\flat$, and that $\alpha = 0$ exactly when there exists some $y \in R^\flat$ such that $y^p - d^ny = d^nx$ in $R^\flat$; note that this equation forces $y \in d^{\frac{n}{p}} R^\flat$, i.e., $y$ lies in the correct piece of the Nygaard filtration. The claim now follows from \cref{KillTateTwistPerfd}.

Next, we verify $(\ast)_S$ for a specific example. Set
\[S_0 = R\langle x_1^{1/p^\infty},...,x_r^{1/p^\infty}\rangle/(x_1,...,x_r),\] 
with $R$ being a perfectoid $\mathcal{O}_C$-algebra $R$. We claim that $(\ast)_{S_0}$ holds: indeed, \cref{PerfdSPerfdSurjTate} implies that $(\ast)_{S_0}$ follows from $(\ast)_{R\langle x_1^{1/p^\infty},...,x_r^{1/p^\infty}\rangle}$, which was already shown above.

Finally, we handle the general case. Fix a quasiregular semiperfectoid $S$, presented as a quotient $R'/I$ with $R'$ a perfectoid $\mathcal{O}_C$-algebra.  Fix a set $\{x_t \in I\}_{t \in T}$ of generators of $I$. By Andr\'e's lemma, we may replace $R'$ by a quasisyntomic cover if necessary to assume that each $x_t$ admits a compatible system of $p$-power roots. On fixing such a system, we obtain an evident surjection $S' := \big(R'[ \{x_t^{1/p^\infty}\}_{t \in T}]/(x_t)\big)^{\wedge} \to S$. The induced map on $p$-complete cotangent complexes is also surjective on $H^{-1}$, so we can then assume $S = S'$ by \cref{SurjLemma}. Filtering $T$ by its finite subsets then reduces us to the ring $S_0$ considered in the previous paragraph, so we are done.
\end{proof}

\begin{proof}[Proof of \cref{OddVanishingExamples}]
Via \cite[Theorem 1.12 (5)]{BMS2} and \cite{ClausenMathewMorrow}, it is enough to show that for all $n \geq 0$, applying $H^1(\mathbf{Z}_p(n)(-))$ to $R \to S$ gives a surjective map. Choose a $p$-power compatible system of roots $\{f_i^{1/p^n}\}_{n \geq 0}$ for each $i$. Via this choice, we obtain a commutative square
\[ \xymatrix{ R\langle x_1^{1/p^{\infty}},...,x_r^{1/p^{\infty}}\rangle \ar[r] \ar[d] & R\langle x_1^{1/p^{\infty}},...,x_r^{1/p^{\infty}}\rangle /(x_1,...,x_r) \ar[d] \\
   R \ar[r] & S := R/(f_1,...,f_r)}\]
where the top horizontal map is the obvious one, and left vertical map is determined by $x_i^{1/p^n} \mapsto f_i^{1/p^n}$ for all $n \geq 0$ and $i \in 1,...,r$. The top horizontal map is surjective on $H^1(\mathbf{Z}_p(n)(-))$ by \cref{PerfdSPerfdSurjTate}, while the vertical maps are surjective on $H^1(\mathbf{Z}_p(n)(-))$ by \cref{SurjLemma}. The commutativity implies the same for the bottom horizontal map, as wanted.
\end{proof}

\newpage

\section{The Nygaard filtration: Relative case}
\label{sec:RelativeNygaard}

In \S \ref{ss:NygaardRelative}, we explain how to prove a version of the results of \S \ref{sec:Nygaard} relative to any bounded base prism $(A,I)$. This yields, in particular, the Nygaard filtration on the prismatic complex on a smooth formal $A/I$-scheme, thus proving \cref{thm:C}. In \S \ref{ss:BK}, we apply these results to compare the theory constructed in this paper with the Breuil-Kisin type theory from \cite{BMS2}.

\subsection{Constructing the Nygaard filtration on relative prismatic cohomology}
\label{ss:NygaardRelative}

Let $S$ be any quasisyntomic $A/I$-algebra. Recall that this means that $S$ is a $p$-completely flat $A/I$-algebra such that $L_{S/(A/I)}$ has $p$-adic Tor amplitude in $[-1,0]$, and we always assume that $S$ is derived $p$-adically complete; such an $S$ is automatically classically $p$-adically complete by \cite[Lemma 4.7]{BMS2}.

\begin{definition} A quasisyntomic $A/I$-algebra $S$ is large if there is a surjection $A/I\langle X_i^{1/p^\infty}|i\in I\rangle\to S$ for some set $I$.
\end{definition}

For large quasisyntomic $A/I$-algebras $L_{S/(A/I)}[-1]$ is a $p$-completely flat $S$-module as $\Omega^1_{S/(A/I)}$ vanishes after $p$-completion. Note that large quasisyntomic algebras form a basis for the quasisyntomic site of $A/I$ as one can always extract compatible sequences of $p$-power roots of elements. The standard examples are $A/I\langle X_1^{1/p^\infty},\ldots,X_n^{1/p^\infty}\rangle/(f_1,\ldots,f_m)$ for some $p$-completely regular sequence $(f_1,\ldots,f_m)$ relative to $A/I$.

\begin{theorem}\label{RelativeNygaard} Let $S$ be a large quasisyntomic $A/I$-algebra.
\begin{enumerate}
\item[{\rm (1)}] The derived prismatic cohomology $\Prism_{S/A}$ is concentrated in degree $0$ and a $(p,I)$-completely flat $\delta$-$A$-algebra. It is the initial object of $(S/A)_\Prism$.
\end{enumerate}
Let
\[
\phi_{S/A}: \Prism_{S/A}^{(1)}:= \Prism_{S/A}\widehat{\otimes}^L_{A,\phi} A\to \Prism_{S/A}
\]
denote the relative Frobenius, and
\[
\mathrm{Fil}^i_N \Prism_{S/A}^{(1)} = \{x\in \Prism_{S/A}^{(1)}\mid \phi_{S/A}(x)\in I^i \Prism_{S/A}\}\ .
\]
\begin{enumerate}
\item[{\rm (2)}] The image of
\[
\phi_{S/A}: \mathrm{gr}^i_N \Prism_{S/A}^{(1)}\hookrightarrow \overline{\Prism}_{S/A}\{i\}
\]
is given by $\mathrm{Fil}_i \overline{\Prism}_{S/A}\{i\}$.
\item[{\rm (3)}] The formation of $\mathrm{Fil}^i_N$ commutes with base change in $A$.
\end{enumerate}
\end{theorem}

\begin{proof} In part (1), note that $\Prism_{S/A}$ is concentrated in degree $0$ by the Hodge-Tate comparison and the assumption that $L_{S/(A/I)}[-1]$ is $p$-completely flat, and we also see that it is $(p,I)$-completely flat. To see that $\Prism_{S/A}$ is initial, it is enough by \cref{PrismaticCohGivesPrisms} to show that a functorial idempotent endomorphism of $\Prism_{S/A}$ is the identity, which follows from the proof of \cref{QRSPPrism}.

Part (3) is immediate from the explicit description, but in fact one can see that if part (2) is true for the $A/I$-algebra $S$, then it is also true for any base change of $S$ along a map $(A,I)\to (B,J)$ of bounded prisms. Indeed, one can define a putative Nygaard filtration on $\Prism_{S\widehat{\otimes}_A B/B} = \Prism_{S/A}\widehat{\otimes}_A B$ via base change, and it follow satisfy (2), which implies that it has to be the Nygaard filtration.

The proof of part (2) follows exactly the outline of \S \ref{sec:Nygaard}. In particular, by descent to smooth algebras and left Kan extension, it suffices to treat the case
\[
S=A/I\langle X_1^{1/p^\infty},\ldots,X_n^{1/p^\infty}\rangle/(f_1,\ldots,f_m)
\]
for some $p$-completely regular sequence $(f_1,\ldots,f_m)$ relative to $A/I$. By \cref{PrismaticRefineQSyn}, we can after a flat base change in $A$ assume that all $f_i$ admit compatible $p$-power roots $f_i^{1/p^j}$. Arguing as in the proof of Proposition~\ref{PropNygaard} reduces us to the case
\[
S=A/I\langle X^{1/p^\infty}\rangle/(X)\ .
\]
Localizing on $A$, we may assume that $I=(d)$ is orientable, and then by base change in $A$ we can reduce to the universal oriented prism. In that case the perfection $A\to A_\infty$ is flat, so we can reduce to the case that $A$ is perfect, where it follows from \S \ref{sec:Nygaard}. (Note that for all terms in the \v{C}ech nerve $A_\infty\widehat{\otimes}_A A_\infty$ etc.~, the Nygaard filtration will simply be the base change from $A_\infty$ by our remark about part (3) above, so one can apply flat descent.)
\end{proof}

In particular, if $X$ is a smooth formal $A/I$-scheme, we can define sheaves $\Prism_{-/A}$ and
\[
\mathrm{Fil}^i_N \Prism_{-/A}^{(1)}\subset \Prism_{-/A}^{(1)}:=\Prism_{-/A}\widehat{\otimes}^L_{A,\phi} A
\]
on the quasisyntomic site $X_\qsyn$ by defining them on the base of large quasisyntomic $A/I$-algebras $S$, with the values defined in the theorem. We can now prove Theorem~\ref{thm:C}, whose statement we recall.

\begin{theorem}\label{thmCagain} Let $(A,I)$ be a bounded prism and let $X=\Spf R$ be an affine smooth $p$-adic formal scheme over $A/I$. There is a canonical isomorphism
\[
R\Gamma_\Prism(X/A)\cong R\Gamma(X_\qsyn,\Prism_{-/A})
\]
and we endow prismatic cohomology with the Nygaard filtration
\[
\mathrm{Fil}^i_N R\Gamma_\Prism(X/A)^{(1)} = R\Gamma(X_\qsyn,\mathrm{Fil}^i_N \Prism_{-/A}^{(1)})\ .
\]
Then there are natural isomorphisms
\[
\mathrm{gr}^i_N R\Gamma_\Prism(X/A)^{(1)} \cong \tau^{\leq i} \overline{\Prism}_{R/A}\{i\}
\]
for all $i\geq 0$. The Frobenius $\phi$ on $R\Gamma_\Prism(X/A)$ factors as
\[
\phi_A^\ast R\Gamma_\Prism(X/A)=R\Gamma_\Prism(X/A)^{(1)}\xrightarrow{\tilde{\phi}} L\eta_I R\Gamma_\Prism(X/A)\to R\Gamma_\Prism(X/A)\ ,
\]
using the d\'ecalage functor $L\eta_I$ as e.g.~in \cite{BMS1}. The map
\[
\tilde{\phi}: \phi_A^\ast R\Gamma_\Prism(X/A)\to L\eta_I R\Gamma_\Prism(X/A)
\]
is an isomorphism.
\end{theorem}

\begin{proof} The isomorphism
\[
R\Gamma_\Prism(X/A)\cong R\Gamma(X_\qsyn,\Prism_{-/A})
\]
follows from the Hodge-Tate comparison and flat descent for the cotangent complex and its wedge powers \cite[Theorem 3.1]{BMS2}. The isomorphism
\[
\mathrm{gr}^i_N R\Gamma_\Prism(X/A)^{(1)} \cong \tau^{\leq i} \overline{\Prism}_{R/A}\{i\}
\]
follows via descent from Theorem~\ref{RelativeNygaard}~(2). The Frobenius refines to a map of filtered complexes
\[
\mathrm{Fil}^\star_N R\Gamma_\Prism(X/A)^{(1)}\to I^\star R\Gamma_\Prism(X/A)\ .
\]
The filtered complex on the left is connective in the Beilinson $t$-structure by the identification of its graded pieces. Thus \cite[Proposition 5.8]{BMS2} implies that the Frobenius lifts to a map
\[
\tilde{\phi}: R\Gamma_\Prism(X/A)^{(1)}\to L\eta_I R\Gamma_\Prism(X/A)\ .
\]
To see that this is an isomorphism, it suffices by derived Nakayama to check modulo $I$. Then the right-hand side is given by $\Omega^\ast_{R/A}$ by the Hodge-Tate comparison and \cite[Proposition 6.12]{BMS1}. In particular, both sides commute with base change and satisfy \'etale localization. We can then reduce to the case of a polynomial algebra and then via base change in $A$ first to an oriented $A$, then to the universal oriented $A$, and then by Construction~\ref{UnivOrientedCrystallize} to crystalline $A$, and finally to $A=\mathbf Z_p$. Now it follows from Theorem~\ref{CrysComp} and the Cartier isomorphism.
\end{proof}

In particular, we get the de~Rham comparison in general.

\begin{corollary}\label{generaldeRham} For any bounded prism $(A,I)$ and any smooth formal $A/I$-scheme $X$, there is a canonical isomorphism of $E_\infty$-algebras in $D(X_\et,A/I)$,
\[
\Prism_{X/A}\widehat{\otimes}_{A,\phi}^L A/I\cong \Omega^\ast_{R/(A/I)}\ .
\]
\end{corollary}

\begin{proof} Take the reduction of $\tilde{\phi}$ modulo $I$ and use \cite[Proposition 6.12]{BMS1} and the Hodge-Tate comparison.
\end{proof}

Another application is the following result on the image of $\phi$ on prismatic cohomology.

\begin{corollary}\label{ImageofPhi} Let $(A,I)$ be a bounded prism and $X$ be a smooth formal $A/I$-scheme. For any $i\geq 0$, there is a natural map
\[
V_i: \tau^{\leq i}\Prism_{X/A}\otimes_A I^{\otimes i}\to \tau^{\leq i} \Prism_{X/A}^{(1)}
\]
such that $\phi V_i$ is the natural map $\tau^{\leq i}\Prism_{X/A}\otimes_A I^{\otimes i}\to \tau^{\leq i}\Prism_{X/A}$ and also the composite
\[
\tau^{\leq i} \Prism_{X/A}^{(1)}\otimes_A I^{\otimes i}\xrightarrow{\phi\otimes 1} \tau^{\leq i} \Prism_{X/A}\otimes_A I^{\otimes i}\xrightarrow{V_i} \tau^{\leq i} \Prism_{X/A}^{(1)}
\]
is the natural map. In particular, $V_i$ induces a map
\[
V_i: H^i(X_\et,\Prism_{X/A}^{(1)})\otimes_A I^{\otimes i}\to H^i(X_\et,\Prism_{X/A}) = H^i_\Prism(X/A)
\]
that is an inverse of $\phi$ up to $I^{\otimes i}$.
\end{corollary}

\begin{proof} This follows from the isomorphism
\[
\tilde{\phi}: \Prism_{X/A}^{(1)}\simeq L\eta_I \Prism_{X/A}
\]
and \cite[Lemma 6.9]{BMS1}.
\end{proof}

\begin{remark}
\label{AbsCCNygaard}
By left Kan extension, \cref{AbsCCPrismatic} combined with Theorem~\ref{thm:C} (2) (which is a part of \cref{thmCagain}) imply the following: for any formal $A/I$-scheme $X$, there is a canonical identification
\[ L_{X/A} \simeq \mathrm{gr}^1_N \Prism_{X/A}^{(1)}.\]
When $I=(p)$, this identification has been proven (independently) recently by Illusie (to appear).
\end{remark}

\subsection{Comparison with the Breuil-Kisin type theory from \cite{BMS2}}
\label{ss:BK}

One of the goals of  \cite{BMS2} was to give a construction of Breuil-Kisin-type cohomology theories. Let us verify that the theory defined in \cite{BMS2} agrees with the present construction. Thus, let $K$ be a complete discretely valued extension of $\mathbf Q_p$ with perfect residue field with ring of integers $\mathcal O_K$ and residue field $k$, and fix a uniformizer $\pi\in \mathcal O_K$. Let $\mathfrak S=W(k)[[u]]$ which surjects onto $\mathcal O_K$ via $u\mapsto \pi$. Let $I\subset \mathfrak S$ be the kernel of this map; then $(\mathfrak S,I)$ is a prism. We fix a generator $d\in I$.

In \cite{BMS2}, we used relative $\mathrm{THH}$ for the base $\mathbf S[u]$. The key comparison is now the following.

\begin{proposition} 
\label{BKComp}
For any quasiregular semiperfectoid quasisyntomic $\mathcal O_K$-algebra $S$, the cyclotomic Frobenius on $\pi_0 \mathrm{TP}(S/\mathbf S[u];\mathbf Z_p)$ refines to a $\delta$-ring structure, functorial in $S$, and identifies with the Nygaard completion $\widehat{\Prism_{S/\mathfrak S}^{(1)}}$ of $\Prism_{S/\mathfrak S}^{(1)}$, compatibly with Nygaard filtrations.
\end{proposition}

\begin{proof} The ring $\pi_0 \mathrm{TP}(S/\mathbf S[u];\mathbf Z_p)$ is $p$-torsion free, so we need to check that the cyclotomic Frobenius is a Frobenius lift. This can be checked after the $(p,u)$-completed base change along $\mathfrak S\to W(k)[[u^{1/p^\infty}]]$, which gives $\pi_0 \mathrm{TP}(S\langle \pi^{1/p^\infty}\rangle;\mathbf Z_p)$ by \cite[Corollary 11.8]{BMS2}. Thus, we can apply Theorem~\ref{BMS2CompNC}. Now the universal property of $\Prism_{S/\mathfrak S}$ and the formal properties of $\pi_0\mathrm{TP}(S/\mathbf S[u];\mathbf Z_p)$ (namely, the relative to $\mathbf{S}[u]$ analog of the diagrams \eqref{BMS2THHFrob} and \eqref{BMS2PrismFrob}) give a map
\[
\Prism_{S/\mathfrak S}^{(1)}\to \pi_0 \mathrm{TP}(S/\mathbf S[u];\mathbf Z_p)\ .
\]
Checking its compatibility with the Nygaard filtration can again be done after base change to $W(k)[[u^{1/p^\infty}]]$ where it follows from Theorem~\ref{BMS2CompNC}. In particular, the map extends to the Nygaard completion, and then is an isomorphism, again via reduction to Theorem~\ref{BMS2CompNC}.
\end{proof}

Now note that for $i$ at least the dimension of $X$, the map
\[
\phi_{S/A}: \mathrm{Fil}^i_N \widehat{\Prism_{-/A}^{(1)}}\to I^i \Prism_{-/A}
\]
induces an isomorphism
\[
R\Gamma(X_\qsyn,\mathrm{Fil}^i_N \widehat{\Prism_{-/A}^{(1)}})\to R\Gamma(X_\qsyn,\Prism_{-/A})\otimes_A I^{\otimes i}
\]
as both are complete for compatible filtrations ($\mathrm{Fil}^j_N$ respectively $I^j \Prism_{-/A}$) and one has isomorphisms on graded pieces by Theorem~\ref{thmCagain}. The left-hand side can be expressed in terms of $\pi_0 \mathrm{TP}(-/\mathbf S[u];\mathbf Z_p)$ and its Nygaard filtration by the previous proposition. This is how $R\Gamma_{\mathfrak S}(X)$ was defined in \cite{BMS2} (cf.~\cite[Proposition 11.5]{BMS2}, noting that the Nygaard filtration is what one finds on higher homotopy groups in $\mathrm{TC}^-$), so we get a canonical isomorphism
\[
R\Gamma_{\mathfrak S}(X)\cong R\Gamma_\Prism(X/\mathfrak S)\ ,
\]
as desired.
\newpage

\section{$q$-crystalline and $q$-de Rham cohomology}
\label{sec:qcrys}

In this section, we construct a canonical $q$-deformation of de Rham cohomology: given a formally smooth $\mathbf{Z}_p$-scheme $X$, we construct a ringed site  --- the $q$-crystalline site of $X$ --- whose cohomology yields a deformation of the de Rham cohomology of $X/\mathbf{Z}_p$ across the map $\mathbf{Z}_p\llbracket q-1\rrbracket \xrightarrow{q\mapsto 1} \mathbf{Z}_p$, and can be computed in local co-ordinates via a $q$-de Rham complex; this verifies some conjectures from \cite{ScholzeqdeRham}. The main innovation here is the introduction of a $q$-analog of the notion of divided power thickenings (defined in a co-ordinate free fashion) in the category of $\delta$-rings over $\mathbf{Z}_p\llbracket q-1\rrbracket$; this notion is introduced in \S \ref{ss:qPD}, and the basic example is the pair $(\mathbf{Z}_p\llbracket q-1\rrbracket, (q-1))$. With this ingredient, the $q$-crystalline site is defined in an evident fashion in \S \ref{ss:qcryssite} and the comparison with $q$-de Rham complexes is the subject of \S \ref{ss:qdR}; our definitions are set up to work over any $q$-divided power thickening as a base. Along the way, we also check that $q$-crystalline cohomology is closely related to prismatic cohomology (\cref{QCrysPrism}), so the comparison with $q$-de Rham complexes gives an explicit complex computing prismatic cohomology in many cases.

\begin{notation}
Set $A = \mathbf{Z}_p\llbracket q-1\rrbracket$ with $\delta$-structure given by $\delta(q) = 0$, and let $[p]_q = \frac{q^p-1}{q-1} \in A$ be the $q$-analog of $p$. Note that $\phi(q-1) = q^p-1 \in [p]_q A$. We shall often use without comment the congruence $[p]_q = p \mod (q-1)$ and that $(q-1)^{p-1}$ and $p$ differ by a multiplicative unit in $A/[p]_q \cong \mathbf{Z}_p[\zeta_p]$, where $\zeta_p$ is a primitive $p$-th root of $1$. In particular, derived $(p,[p]_q)$-completion coincides with derived $(p,[p]_q)$-completion for any complex of $A$-modules. Finally, if $x$ is an element of a $[p]_q$-torsionfree $\delta$-$A$-algebra $D$ such that $\phi(x) \in [p]_q D$, then we write
\[
\gamma(x) := \frac{\phi(x)}{[p]_q} - \delta(x) \in D\ ;
\]
if $q=1$ in $D$, then we have $\gamma(x) = \frac{x^p}{p}$ is (up to the unit $(p-1)!$) the usual divided $p$-th power, and in general we think of $\gamma(x)$ as the ``divided $[p]_q$-th power". (A true $q$-analog of the $p$-th divided power would be obtained by further dividing the expression for $\gamma(x)$ above by the unit $\prod_{j=1}^{p-1} [j]_q$; we avoid doing this to keep formulas simple.)
\end{notation}

\subsection{$q$-divided power thickenings}
\label{ss:qPD}

The key innovation of our approach to $q$-crystalline cohomology is the following definition.

\begin{definition}[$q$-divided power algebras]
\label{DefqPD}
A {\em $q$-PD pair} is given by a derived $(p,[p]_q)$-complete $\delta$-pair $(D,I)$ over $(A,(q-1))$ satisfying the following conditions:
\begin{enumerate}
\item The ideal $I \subset D$ satisfies $\phi(I) \subset [p]_qD$ (so that $\gamma$ is defined on $I$) and $\gamma(I) \subset I$.
\item The pair $(D,([p]_q))$ is a bounded prism over $(A,([p]_q))$, i.e., $D$ is $[p]_q$-torsionfree, and $D/([p]_q)$ has bounded $p^\infty$-torsion.
\item The ring $D/(q-1)$ is $p$-torsionfree with finite $(p,[p]_q)$-complete Tor-amplitude\footnote{We could also say finite $p$-complete Tor amplitude since $[p]_q = p \mod (q-1)$.} over $D$. 
\end{enumerate}
The corresponding map $D \to D/I$ is sometimes called a {\em $q$-PD thickening} and the ideal $I$ is sometimes called a {\em $q$-PD ideal}. There is an obvious category of $q$-PD pairs. If $(D,I)$ is a $q$-PD pair with $q-1=0$ in $D$, then we call $(D,I)$ a {\em $\delta$-PD pair}. The collection of $\delta$-PD pairs forms a full subcategory of the category of all $q$-PD pairs.
\end{definition}

Unlike the classical crystalline theory, being a $q$-PD pair is a property of a $\delta$-pair $(D,I)$ rather than extra structure. Moreover, condition (3) above is a technical condition imposed to facilitate some arguments below. Likewise, condition (2) might reasonably be weakened to merely asking that $(D,([p]_q))$ is a prism, i.e.~that $D$ is $[p]_q$-torsionfree. However, as with prismatic cohomology above, it is more convenient to work with bounded prisms. We note that one way to satisfy (2) and (3) is for $D$ to be $(p,[p]_q)$-completely flat over $A$; another is for $q=1$ in $D$ and $D$ being $p$-torsionfree. These will be the cases of interest below, and one can show that if $D$ is noetherian, these are the only possibilities, at least for $D$ local, using the Buchsbaum-Eisenbud criterion. 

\begin{remark}[Characterizing $\delta$-PD pairs]
\label{DeltaPDPair}
Say $D$ is a $\delta$-ring regarded as an $A$-algebra via $q=1$ and $I \subset A$ is an ideal. Then $(D,I)$ is a $\delta$-PD pair if and only if the following hold true:
\begin{enumerate}
\item[(a)] The ring $D$ is $p$-torsionfree, and both $D$ and $I$ are $p$-adically complete.  
\item[(b)] The ideal $I$ admits divided powers, i.e., for each $x \in I$, we have $\frac{x^n}{n!} \in I$ for all $n > 0$.
\end{enumerate}
Indeed, it is easy to see that any pair $(D,I)$ satisfying the above conditions is a $\delta$-PD pair. Conversely, if $(D,I)$ is a $\delta$-PD pair, then condition (a) above is automatic. For (b), we note that condition (1) in Definition~\ref{DefqPD} ensures that $\frac{x^p}{p} \in I$ for all $x \in I$. As $(p-1)!$ is a unit, this means $\frac{x^p}{p!} \in I$ for all $x \in I$. Via the formula $\gamma_{kp}(x)=u\gamma_k(\gamma_p(x))$ for some unit $u$, one sees inductively that all divided powers stay in $I$.
\end{remark}

\begin{remark}[Relating $q$-PD pairs to $\delta$-PD pairs]
The hypotheses are designed to ensure that if $(D,I)$ is a $q$-PD pair, then $(D/(q-1), ID/(q-1))$ is a $\delta$-PD pair; this construction gives a left adjoint to the inclusion of $\delta$-PD pairs into all $q$-PD pairs.
\end{remark}

\begin{lemma}[Homological properties of $q$-PD pairs]
\label{qPDCommAlg}
Let $(D,I)$ be a $q$-PD pair. 
\begin{enumerate}
\item The ring $D$ is derived $f$-complete for every $f \in I$. 
\item The functor $M \mapsto M \widehat{\otimes}_D^L D/(q-1)$ of $(p,[p]_q)$-completed base change along $D \to D/(q-1)$ is conservative on $(p,[p]_q)$-complete complexes.
\item The functor in (2) commutes with totalizations of cosimplicial $(p,[p]_q)$-complexes in $D^{\geq 0}$.
\item Let $I' \subset D$ be the derived $(p,[p]_q)$-complete ideal generated by $I$. Then $(D,I')$ is a $q$-PD pair.
\item Let $D \to D'$ be a $(p,[p]_q)$-completely flat map of $\delta$-$A$-algebras. Let $I'$ be the $(p,[p]_q)$-complete ideal of $D'$ generated by $I$. Then $(D',I')$ is a $q$-PD pair.
\end{enumerate}
\end{lemma}
\begin{proof}
For (1), we simply note that for any $f \in I$, we have $f^p \in (p,[p]_q)$ since $\phi(f) \in [p]_qD$ by hypothesis, so derived $f^p$-completeness (and hence derived $f$-completeness) follows from derived $(p,[p]_q)$-completeness of $D$.

For (2), it suffices to show that base change along the composite  $D \to D/(q-1) \to D/(q-1,p)$ is conservative on derived $(p,[p]_q)$-complete complexes. If $M \otimes_D^L D/(q-1,p)=0$, then $M \otimes_D^L \mathrm{Kos}(D;[p]_q, q-1)= 0$ as well, since $\mathrm{Kos}(D;[p]_q,q-1)$ has at most two nonzero cohomology groups, each of which is a $D/(q-1,p)$-module. But this forces $M=0$ by derived Nakayama, so the claim follows.

Part (3) follows immediately Definition~\ref{DefqPD} (3) and Lemma~\ref{TorAmpBC}.

For part (4), we must check that condition (1) from Definition~\ref{DefqPD} is satisfied for $I'$. The containment $\phi(I') \subset [p]_q D$ is clear since $\phi^{-1}([p]_q D)$ is derived $(p,[p]_q)$-complete. To check $\gamma(I') \subset I'$, recall that $I'$ is defined as the image of $I^{\wedge} \to D$, where the source is the derived $(p,[p]_q)$-completion of $I$. As $q-1 \in I$, the ideal $I'$ is the preimage of its image in $D/(q-1)$. Moreover, the image of $I'$ in $D/(q-1)$ coincides with the image of $I^{\wedge} \to \overline{I}^{\wedge} \to D/(q-1)$, where $\overline{I} = \mathrm{im}(I \to D/(q-1))$. The  map $I^{\wedge} \to \overline{I}^{\wedge}$ is surjective as the derived completion preserves surjections, so $I'$ is the preimage of the $p$-complete ideal in $D/(q-1)$ generated by $\overline{I}$. Thus, we can reduce to case $q=1$, in which case the claim follows from $p$-adic continuity of divided powers in a $p$-torsionfree ring.

For part (5), we check the conditions from Definition~\ref{DefqPD}. Conditions (1) and (2) in Definition~\ref{DefqPD} come from part (4) above and \cref{BoundedPrismProp} respectively. For (3) in Definition~\ref{DefqPD}, we observe that the $p$-completed derived base change $D/(q-1) \to (D' \otimes^L_D D/(q-1))^{\wedge}$ of $D \to D'$ along $D \to D/(q-1)$ is $p$-completely flat with source being $p$-torsionfree. Consequently, by \cref{BoundedPrismProp} again, the derived $p$-completion $(D' \otimes^L_D D/(q-1))^{\wedge}$ is concentrated in degree $0$ and $p$-torsionfree in that degree. But then this object must coincide with $D'/(q-1)$ by derived $p$-completeness of $D'$ and stability of derived $p$-completeness under cokernels, so $D'/(q-1)$ is indeed $p$-torsionfree. This reasoning also shows that $D' \to D'/(q-1)$ is the $p$-completed base change of $D \to D/(q-1)$ and hence must have finite $p$-complete Tor amplitude by the assumption on $D \to D/(q-1)$, thus verifying the conditions in Definition~\ref{DefqPD} (3).
\end{proof}

To check condition (1) from Definition~\ref{DefqPD} in examples, the following remark is useful. 

\begin{remark}[Additivity and multiplicative properties of $\gamma$]
\label{qPDelements}
Let $D$ be a $[p]_q$-torsionfree $\delta$-$A$-algebra. Given $x,y \in \phi^{-1}([p]_qD)$, one easily checks that
\[ \gamma(x+y) = \gamma(x) + \gamma(y) + \frac{(x+y)^p - x^p - y^p}{p}.\]
Similarly, given $x \in \phi^{-1}([p]_qD)$ and $f \in D$, one checks that 
\[ \gamma(fx) = \phi(f)\gamma(x) - x^p \delta(f). \]
It follows immediately from the shape of these formulas that if $I$ is any ideal of $D$, then the subset $J := \{x \in I \mid \phi(x) \in [p]_qD, \ \gamma(x) \in I\}$ is an ideal of $D$. In particular, to check $I = J$, it suffices to check that $\phi(x) \in [p]_q D$ and $\gamma(x)\in I$ as $x$ runs through a generating set for $I$. 
\end{remark}

The next lemma roughly states that if $\gamma(x)$ makes sense, so does $\gamma(\gamma(x))$.

\begin{lemma}[Existence of higher $q$-divided powers]
\label{qPDgenerate}
Let $D$ be a $[p]_q$-torsionfree $\delta$-$A$-algebra. The ideal $\phi^{-1}([p]_q D)$ is stable under $\gamma$. 
\end{lemma}
\begin{proof}
We must show that if $f \in D$ with $\phi(f) \in [p]_q D$, then $\phi(\gamma(f)) \in [p]_q D$ as well.  It suffices to prove this in the universal case $D = A\{f,\frac{\phi(f)}{[p]_q}\}$. In particular, we may assume that $D$ is $A$-flat by \cref{qPDEnvRegular}. Our goal is to show that 
\[ \frac{\phi^2(f)}{\phi([p]_q)} \equiv \phi(\delta(f)) \mod [p]_q D\]
Note that $\phi([p]_q)$ equals $p$ in $A/[p]_qA$. In particular, it is a nonzerodivisor in this ring. Flatness implies that $\phi([p]_q)$ equals $p$ and is a nonzerodivisor in $D/[p]_qD$ as well. It thus suffices to check that
\[ \phi^2(f) \equiv p\phi(\delta(f)) \mod [p]_q D.\]
Now $\phi(f) = f^p + p\delta(f)$, so $\phi^2(f) = \phi(f)^p + p\phi(\delta(f))$. Our claim follows since $\phi(f) \in [p]_q D$ by assumption.
\end{proof}

\begin{corollary}[Smallest and largest $q$-PD ideals]
\label{SmallLargeqPD}
Say $D$ is a derived $(p,[p]_q)$-complete $A$-algebra satisfying conditions (2) and (3) from Definition~\ref{DefqPD}. Then $(q-1)$ is the smallest $q$-PD ideal in $D$, and $\phi^{-1}([p]_qD)$ is the largest $q$-PD ideal in $D$.
\end{corollary}
\begin{proof}
We first show that $(q-1) \subset D$ is a $q$-PD ideal. This ideal is derived $(p,[p]_q)$-complete as it is the image of the map $D \xrightarrow{q-1} D$ between derived $(p,[p]_q)$-complete $A$-modules. For the rest, using Remark~\ref{qPDelements}, it suffices to check that $\phi(q-1) \in [p]_qD$ and $\gamma(q-1) \in (q-1)$. The first containment is clear: $\phi(q-1) = q^p-1 = (q-1)[p]_q \in [p]_qA$. For the second, note that
\[ \gamma(q-1) = \frac{\phi(q-1)}{[p]_q} - \delta(q-1) = (q-1) - \delta(q-1).\]
It thus suffices to check that $\delta(q-1) \in (q-1)$. As $D/(q-1)$ is $p$-torsionfree by hypothesis, it is enough to show that $p\delta(q-1) \in (q-1)$. But we have 
\[ p\delta(q-1) = \phi(q-1) - (q-1)^p = (q^p-1) - (q-1)^p = (q-1)([p]_q - (q-1)^{p-1}),\]
which lies in $(q-1)$, so the claim follows.

We now show that $\phi^{-1}([p]_qD)$ is a $q$-PD ideal. Derived $(p,[p]_q)$-completeness of the ideal follows as above:  $\phi^{-1}([p]_q D)$ is the limit of the diagram $D \xrightarrow{\phi} \phi_* D \xleftarrow{\phi_*([p]_q)} \phi_*(D)$ of derived $(p,[p]_q)$-complete $A$-modules. The containment $\phi(\phi^{-1}([p]_qD)) \subset [p]_qD$ is clear, while the containment $\gamma(\phi^{-1}([p]_qD)) \subset \phi^{-1}([p]_qD)$ follows from Lemma~\ref{qPDgenerate}. 
\end{proof}

We now give the most important examples of $q$-PD pairs for our purposes.

\begin{example}[Examples of $q$-PD pairs]
The key examples are:
\begin{enumerate}
\item (The initial object). The pair $(A,(q-1))$ is a $q$-PD pair. Indeed, conditions (2) and (3) from Definition~\ref{DefqPD} are clear, while (1) follows from Corollary~\ref{SmallLargeqPD}. Note that the pair $(A,(q-1))$ is the initial object in the category of all $q$-PD pairs. More generally, the same reasoning shows that if $D$ is a $(p,[p]_q)$-completely flat $A$-algebra, then $(D,(q-1))$ is a $q$-PD pair.

\item (A perfect object). Let $A_{\inf}$ be the $(p,[p]_q)$-completed perfection of $A$, and set $\xi := \phi^{-1}([p]_q) \in A_{\inf}$. Then $(A_{\inf},(\xi))$ is a $q$-PD pair. Using Remark~\ref{qPDelements}, the only non-trivial statement one must check is that $\gamma(\xi) \in (\xi)$. Unwinding definitions, this amounts to showing that $\delta(\xi) \equiv 1 \mod (\xi)$. As $p$ is a nonzerodivisor on $A_{\inf}/(\xi)$, it suffices to check that $p\delta(\xi) \equiv p \mod (\xi)$. As $\xi^p \equiv 0 \mod (\xi)$, it is enough to show that $\phi(\xi) \equiv p \mod (\xi)$. But 
\[ \phi(\xi) = [p]_q = \frac{q^p-1}{q-1} = 1+q+...+q^{p-1} \equiv p \mod (q-1),\]
so the claim follows as $\xi \mid q-1$. 

\item (The classical case). If $D$ is a $p$-torsionfree and $p$-complete $\delta$-ring equipped with a $p$-complete ideal $I$, then $(D,I)$ is a $\delta$-PD pair exactly when each $x \in I$ admits all divided powers in $D$ (Remark~\ref{DeltaPDPair}).
\end{enumerate}
\end{example}

\begin{lemma}[Existence of $q$-PD envelopes]
\label{qPDEnvConstruct}
Let $(D,I)$ be a $q$-PD pair. Let $P$ be a $(p,[p]_q)$-completely flat $\delta$-$D$-algebra. Let $x_1,...,x_r \in P$ be a sequence that is $(p,[p]_q)$-completely regular relative to $D$. 
\begin{enumerate}
\item The $(p,[p]_q)$-complete $\delta$-$D$-algebra 
\[E := P\{\frac{\phi(x_1)}{[p]_q},\ldots,\frac{\phi(x_r)}{[p]_q}\}^{\wedge}\] 
obtained by freely adjoining $\frac{\phi(x_i)}{[p]_q}$ is $(p,[p]_q)$-completely flat over $D$. In particular, it is discrete and $[p]_q$-torsionfree. 
\item Write $J \subset P$ for the ideal generated by $I$ and the $x_i$'s and let $K\subset E$ be the minimal $(p,[p]_q)$-complete ideal containing $J$ and stable under the operation $\gamma$. Then the natural map $P/J\to E/K$ is an isomorphism, and in particular we get a natural map $E\to E/K\cong P/J$. Then $(E,K)$ is a $q$-PD pair and the induced map $(P,J) \to (E,K)$ of $\delta$-pairs is the universal map from $(P,J)$ to a $q$-PD pair.
\end{enumerate}
In this situation, we often write $D_{J,q}(P) = E$ and call the pair $(D_{J,q}(P),K)$ (or just $D_{J,q}(P)$ if there is no potential for confusion) the $q$-PD envelope of $(P,J)$; note that it only depends on the ideal $J$ and not the generators $x_1,...,x_r$ by (2) above. We then have the following:
\begin{enumerate}[resume]
\item The functor $(P,J) \mapsto (D_{J,q}(P),K)$ commutes with $(p,[p]_q)$-completed derived base change along maps $(D,I) \to (D',I')$ of $q$-PD pairs. In particular, $D_{J,q}(P) \widehat{\otimes}_D^L D/(q-1)$ is the pd-envelope of $J/(q-1) \subset P/(q-1)$.
\end{enumerate}
\end{lemma}
\begin{proof}
Let $\widetilde{E}$ be the $(p,[p]_q)$-complete simplicial commutative $\delta$-$P$-algebra obtained by freely adjoining $\frac{\phi(x_i)}{[p]_q}$ to $P$. We claim that $\widetilde{E}$ is $(p,[p]_q)$-completely flat over $D$; this will imply that $\widetilde{E}$ is discrete and thus coincides with the ring $E$ above, proving (1). To show $(p,[p]_q)$-complete flatness of $\widetilde{E}$ over $D$, it suffices to do so after derived base change along $D \to D/(q-1)$. After this base change, the claim follows from Lemma~\ref{PDFlatp}, so we have proven (1).

To prove (2), we first construct a map $E \to P/J$. To define this map, observe that the previous paragraph shows that $E/(q-1)$ is the $p$-completely flat $D/(q-1)$-algebra obtained as the pd-envelope of $p$-completely flat $D/(q-1)$-algebra $P/(q-1)$ along the ideal generated by the sequence $x_1,...,x_r$ that is $p$-completely regular relative to $D/(q-1)$. In particular, we have an obvious map $E/(q-1) \to P/(q-1,x_1,...,x_r)$, and hence also an obvious map $E \to E/(q-1) \to P/(q-1,x_1,...,x_r) \to P/J$. Let $K^\prime$ denote the kernel of this map. Note that the kernel $\overline{K}^\prime$ of $E/(q-1) \to P/J$ has divided powers by the previous paragraph (and because $I \cdot D/(q-1)$ has divided powers), and that $K^\prime$ is the preimage of $\overline{K}^\prime$. 

We now prove that $(E,K^\prime)$ is a $q$-PD pair. For this, it suffices to check that $\phi(K^\prime)\in [p]_q E$ and that $\gamma(K^\prime) \subset K^\prime$. Equivalently, if we set $K^{\prime\prime} = K^\prime \cap \phi^{-1}([p]_q E)$, we must show that $K^{\prime\prime} = K^\prime$ and that $\gamma(K^{\prime\prime}) \subset K^{\prime\prime}$. For the latter, we must check that if $f \in K''$, then $\gamma(f) \in K''$, which is immediate: the containment in $\phi^{-1}([p]_qE)$ follows from Lemma~\ref{qPDgenerate}, while the containment in $K'$ follows as $\gamma(f)$ maps under $E \to E/(q-1)$ to the element $\frac{f^p}{p}$, which lies in $\overline{K}' \subset E/(q-1)$ as the latter ideal has divided powers (and thus $\gamma(f)$ itself must lie in the inverse image of $\overline{K}'$, which is $K'$). It remains to show $K'' = K' = K$. We already have $(I,x_1,...,x_r) \subset K''$ and $\gamma(K'') \subset K''$. We shall deduce that $K''=K'=K$ from Lemma~\ref{PDgenerate}. To apply this lemma, it suffices to check that $K''$ is the preimage of its image under $E \to E/(q-1)$. But this is easy to see: if $x \in K''$ and $y \in E$, then $x + (q-1)y \in K''$ since $\phi(q-1) \in [p]_q D$ and $q-1 \in K$.

This also proves (2) by the definition of $K$ and the identification $K'=K$. Finally, (3) is immediate from the construction of $E \cong \widetilde{E}$ and the fact that the hypotheses on the $x_i$'s commute with base change.
\end{proof}

\begin{lemma}
\label{PDgenerate}
Let $B$ be $p$-torsionfree $\mathbf{Z}_p$-algebra equipped with an ideal $I$ that is regular modulo $p$. Let $(D,J)$ be the pd-pair obtained as the pd-envelope of $(B,I)$. Then the ideal $J \subset D$ can described as the smallest ideal of $D$ that contains $ID$ and is stable under the operation $f \mapsto \frac{f^p}{p}$.
\end{lemma}
\begin{proof}
The regularity hypothesis ensures that $D$ is $p$-torsionfree. The rest follows by \cite[Tag 07GS]{Stacks}.
\end{proof}

\subsection{The $q$-crystalline site}
\label{ss:qcryssite}

In this section, we fix a $q$-PD pair $(D,I)$ over $(A,(q-1))$ as well as a $p$-completely smooth $D/I$-algebra $R$. 

\begin{definition}[The $q$-crystalline site and its cohomology]
\label{Defqcryssite}
The {\em $q$-crystalline site of $R$ relative to $D$}, denoted $(R/D)_{\qc}$, is the (opposite of the) category of $q$-PD thickenings of $R$ relative to $D$, i.e., the category of $q$-PD pairs $(E,J)$ over $(D,I)$ equipped with a $D/I$-algebra map $R \to E/J$; we give this category the indiscrete topology, so all presheaves are sheaves. Let $\mathcal{O}_{\qc}$ be the presheaf on $(R/D)_{\qc}$ determined by $(E,J) \mapsto E$; this presheaf is naturally valued in $\delta$-$B$-algebras. The {\em $q$-crystalline cohomology of $R$ relative to $D$}, denoted $q\Omega_{R/D}$, is defined as $R\Gamma( (R/D)_{\qc}, \mathcal{O}_{\qc})$, viewed as a $(p,[p]_q)$-complete commutative algebra object in $D(D)$ equipped with a $\phi_D$-semilinear endomorphism $\phi_{R/D}$. 
\end{definition}

Let us explain how to compute $q$-crystalline cohomology explicitly in a manner analogous to Construction~\ref{PrismaticFunctorialCC}.

\begin{construction}[Computing $q$-crystalline cohomology via \v{C}ech-Alexander complexes]
\label{CechAlexanderqCrys}
Choose a derived $(p,[p]_q)$-complete polynomial $D$-algebra $P$ with a surjection $P \to R$ over $D$ with kernel $J_0$. Let $F$ be the free derived $(p,[p]_q)$-complete $\delta$-$D$-algebra on $P$, and write $J=(J_0 F)^{\wedge}$ for the derived $(p,[p]_q)$-complete ideal in $F$ generated by $J_0$. Form the Cech nerves $P^\bullet$ and $F^\bullet$ of $D \to P$ and $D \to F$ respectively. We then have a map $P^\bullet \to F^\bullet$ expressing $F^\bullet$ as the free derived $(p,[p]_q)$-complete $\delta$-$D$-algebra on $P^\bullet$. Let $J_0^\bullet \subset P^\bullet$ be the kernel of the surjection $P^\bullet \to P \to R$, and let $J^\bullet := (J_0^\bullet F^\bullet)^{\wedge} \subset F^\bullet$ be the corresponding derived $(p,[p]_q)$-complete ideal of $F^\bullet$ (so $J^0 = J$). Then the natural $D$-algebra map $R \simeq P^\bullet/J_0^\bullet \to F^\bullet/J^\bullet$ is identified with the derived $p$-completed Cech nerve of the map $R \to F/J$ obtained via base change from $P \to F$: this follows by pushing out $P^\bullet \to F^\bullet$  first along $P^\bullet \to P$ and then $P \to R$. The kernel of each augmentation $P^{n} \to P \to R$ is generated by $I$ and, locally on $\Spec(P)$, by a filtered colimit of ideals generated by sequences that are $(p,[p]_q)$-completely regular relative to $D$. We may then apply Lemma~\ref{qPDEnvConstruct} to the ideals $J^\bullet \subset F^\bullet$ obtain a cosimplicial $q$-PD thickening $D_{J^\bullet, q}(F^\bullet) \to F^\bullet/J^\bullet$ with kernel $K^\bullet$. In particular, this yields a cosimplicial object $(D_{J^\bullet,q}(F^\bullet), K^\bullet) \in (R/D)_{\qc}$. By construction, this cosimplicial object is the \v{C}ech nerve in $(R/D)_{\qc}$ of its $0$-th term $D_{J,q}(F)$. Moreover, using the freeness of $P$ (as a $D$-algebra) and $F$ (as a $\delta$-$D$-algebra), one checks (as in Construction~\ref{cons:weakinitialsite}) that this $0$-th term $D_{J,q}(F)$ is a weakly initial object. Consequently, by \v{C}ech theory, we learn that the limit of $D_{J^\bullet,q}(F^\bullet)$ computes $q\Omega_{R/D}$.
%Choose a $\delta$-$D$-algebra $P$ that is ind-smooth as a $D$-algebra as well as a surjection $P \to R$ of $D$-algebras; for example, we may choose $P$ to be the free $\delta$-$D$-algebra on the set $W(R)$ to obtain a strictly functorial choice. Let $P^\bullet$ be the \v{C}ech nerve of $D \to P$. The kernel of each augmentation $P^{\otimes_D (n+1)} \to P \to R$ is generated by $I$ and, locally on $\Spec P$, by a filtered colimit of ideals generated by sequences that are $(p,[p]_q)$-completely regular relative to $D$. Consequently, applying Lemma~\ref{qPDEnvConstruct}, we obtain a cosimplicial diagram $D_{J^\bullet,q}(P^\bullet)$ in $(R/D)_{\qc}$. When $P$ is chosen to be free as a $\delta$-$D$-algebra, the limit of this diagram computes $q\Omega_{R/D}$ by \v{C}ech theory.
\end{construction}

\begin{theorem}[$q$-crystalline and crystalline cohomology]
\label{qCrysCrysComp}
There is a canonical identification 
\[q\Omega_{R/D} \widehat{\otimes}_D^L D/(q-1) \simeq R\Gamma_{\crys}(R/(D/(q-1))).\]
\end{theorem}
\begin{proof}
We use the complex $D_{J^\bullet,q}(F^\bullet)$ from Construction~\ref{CechAlexanderqCrys} to compute $q\Omega_R$. Applying the functor $- \widehat{\otimes}_D^L D/(q-1)$ of $p$-completed derived base change along $D \to D/(q-1)$, and using Lemma~\ref{qPDCommAlg} (3), we learn that $q\Omega_D \widehat{\otimes}_D^L D/(q-1)$ is computed by the cosimplicial $D/(q-1)$-complex $D_{J^\bullet,q}(F^\bullet) \widehat{\otimes}_D^L D/(q-1)$. Lemma~\ref{qPDEnvConstruct} (3) then shows that this base change is a \v{C}ech-Alexander complex computing $R\Gamma_{\crys}(R/(D/(q-1)))$, proving the theorem.
\end{proof}

\begin{remark}[Globalization]
Let us briefly explain two equivalent approaches to defining $q$-crystalline complexes for not necessarily affine smooth $p$-adic formal schemes $\mathfrak{X}/(D/I)$.
\begin{enumerate}
\item Theorem~\ref{qCrysCrysComp} together with Zariski descent for crystalline cohomology implies that the functor carrying a $p$-completely smooth $D/I$-algebra $R$ to the $(p,[p]_q)$-complete $E_\infty$-$D$-algebra $q\Omega_{R/D}$ is naturally a sheaf for the Zariski topology: by derived Nakayama, the isomorphy of a map of derived $(p,[p]_q)$-complete $D$-complexes can be checked after $p$-completed base change along $D \to D/(q-1)$ as $D$ is derived $(q-1)$-complete. Consequently, we obtain a sheaf $q\Omega_{\mathfrak{X}/D}$ of $E_\infty$-$D$-algebras on the Zariski site of $\mathfrak{X}$ admitting a comparison analogous of Theorem~\ref{qCrysCrysComp} with the pushforward of the crystalline structure structure along the standard comparison map $(\mathfrak{X}/(D/(q-1)))_\crys \to \mathfrak{X}$ between the crystalline and Zariski sites.

\item One can define a ringed $q$-crystalline site $\left((\mathfrak{X}/D)_{\qc},\mathcal{O}_{\qc}\right)$ with a morphism $u_{\mathfrak{X}/D}:(\mathfrak{X}/D)_{\qc} \to \mathfrak{X}$, and set $q\Omega_{\mathfrak{X}/D} := Ru_{\mathfrak{X}/D,*} \mathcal{O}_{\qc}$. The objects of $(\mathfrak{X}/D)_{\qc}$ are given by $q$-PD pairs $(E,J)$ over $(D,I)$ together with $D/I$-maps $\mathrm{Spf}(E/J) \to \mathfrak{X}$; the topology is determined by finite families of maps $\{(E,J) \to (E_i,J_i)\}_{i=1,...,n}$ of $q$-PD pairs in $(\mathfrak{X}/D)_{\qc}$ with the property that the map $E \to \prod_i E_i$ is a $(p,[p]_q)$-completed Zariski cover and such that $E/J \widehat{\otimes}^L_E E_i \simeq E_i/J_i$ for all $i$. The morphism $u_{\mathfrak{X}/D}$ is obtained by deforming Zariski covers along $q$-PD thickenings, using Lemma~\ref{qPDCommAlg} (5) to extend $q$-PD structures along $(p,[p]_q)$-completeted Zariski localizations.
\end{enumerate}
As the affine case suffices for our applications, and in view of the relatively simple construction in (1), we do not develop the definition in (2) further in this paper.
\end{remark}

\begin{remark}[Computing $q$-crystalline cohomology via ``small'' \v{C}ech-Alexander complexes]
There is a variant of Construction~\ref{CechAlexanderqCrys} with much smaller objects that produces a quasi-isomorphic output. For instance, choose a $(p,[p]_q)$-completely smooth $\delta$-$D$-algebra $P$ with a surjection $P \to R$ of $D$-algebras. Let $P^\bullet$ be the derived $(p,[p]_q)$-completed \v{C}ech nerve of $D \to P$, and let $J^\bullet \subset P^\bullet$ be the kernel of $P^\bullet \to R$. The kernel of each augmentation $P^{n} \to P \to R$ is generated by $I$ and, locally on $\Spec P$, by sequences that are $(p,[p]_q)$-completely regular relative to $D$. Consequently, applying Lemma~\ref{qPDEnvConstruct}, we obtain a cosimplicial diagram $D_{J^\bullet,q}(P^\bullet)$ in $(R/D)_{\qc}$. The limit of this complex still computes $q\Omega_{R/D}$. Indeed, this follows by running through the proof of Theorem~\ref{qCrysCrysComp} with  $D_{J^\bullet,q}(P^\bullet)$ instead, and observing that classical crystalline cohomology can be computed by the cosimplicial ring $D_{J^\bullet}(P^\bullet) \simeq D_{J^\bullet,q}(P^\bullet)/(q-1)$ formed by the $p$-completed PD-envelope of $J^\bullet \subset P^\bullet$: this cosimplicial ring is the \v{C}ech-Alexander complex of the PD-thickening $D_{J^0}(P^0)$ of $R$ relative to $D$, and this latter object is weakly initial amongst all PD-thickenings of $R$ relative to $D$. In particular, it is often convenient to let $P$ simply be a $(p,I)$-completely smooth lift of $R$ along $D \to D/I$. 
\end{remark}

\begin{theorem}[Invariance under $q$-PD thickenings of the base]
\label{qCrysInvqPD}
Let $(D,J) \to (D,I)$ be a morphism of $q$-PD pairs that is the identity on underlying $\delta$-rings. (For example, take $J = (q-1)$.) Let $\widetilde{R}$ be a $p$-completely smooth $D/J$-algebra lifting $R$ along $D/J \to D/I$. Then  there is a canonical identification $q\Omega_{\widetilde{R}/D} \simeq q\Omega_{R/D}$.
\end{theorem}
\begin{proof}
Let us first observe that if $(E,K) \in (\widetilde{R}/D)_{\qc}$, then the natural map $E \to E/K \simeq \widetilde{R} \to R$ is surjective and a $q$-PD thickening: surjectivity follows from that of $\widetilde{R} \to R$ (which holds true as it does so modulo $p$), while the rest follows as the kernel can be described as the smallest derived $(p,[p]_q)$-complete ideal of $E$ generated by $K$ and $IE$.  This construction gives a functor $(\widetilde{R}/D)_{\qc} \to (R/D)_{\qc}$ that does not change the underlying $\delta$-$D$-algebra, and thus yields a canonical map $q\Omega_{\widetilde{R}/D} \to q\Omega_{R/D}$. It now follows by inspection of the proof of Lemma~\ref{qPDEnvConstruct} that the \v{C}ech-Alexander complex used to compute $q\Omega_{\widetilde{R}/D}$ also computes $q\Omega_{R/D}$. 
\end{proof}

\begin{theorem}[$q$-crystalline and prismatic cohomology]
\label{QCrysPrism} Let $R$ be a $p$-completely smooth $D/I$-algebra. Let $R^{(1)}$ be the $p$-completely smooth $D/([p]_q)$-algebra defined via $p$-completed base change along the map $D/I \to D/([p]_q)$ induced by $\phi_D$. Writing $\Prism_{R^{(1)}/D}$ for the prismatic cohomology of $R^{(1)}$ relative to the bounded prism $(D,[p]_q)$, there is a canonical isomorphism $\Prism_{R^{(1)}/D} \simeq q\Omega_{R/D}$.
\end{theorem}

\begin{proof}
Let us first describe a canonical map $\Prism_{R^{(1)}/D} \to q\Omega_{R/D}$. For this, it is enough to explain how each $(E,J) \in (R/D)_{\qc}$ functorially yields an object $(E \to E/([p]_q) \gets R^{(1)})$ of $(R^{(1)}/D)_\Prism$. As $E\to E/J \cong R$ is a $q$-PD thickening, we have $\phi_E(J) \subset [p]_q E$, so $\phi_E$ yields a map $R \cong E/J \to E/([p]_q)$ that is linear over $\phi_D$. By linearization, this can be viewed as a $D$-linear map $R^{(1)} \to E/([p]_q)$, which then gives the desired object $(E\to E/([p]_q) \gets R^{(1)})$ of $(R^{(1)}/D)_\Prism$.

We have constructed the map $\alpha:\Prism_{R^{(1)}/D} \to q\Omega_{R/D}$. It remains to check that $\alpha$ is an isomorphism. We are now allowed to make choices. In particular, using Theorem~\ref{qCrysInvqPD}, we may assume $I=(q-1)$; here we implicitly use that $R$ lifts to a $p$-completely smooth algebra along $D/(q-1) \to D/I$. To check $\alpha$ is an isomorphism, it suffices to do so after $(p,[p]_q)$-completed derived base change along $D \to D/(q-1)$ by Lemma~\ref{qPDCommAlg} (2). Now $D \to D/(q-1)$ refines to a map $(D,([p]_q)) \to (D/(q-1),(p))$ of prisms by our assumptions on $D$. So base change for prismatic cohomology identifies $\Prism_{R^{(1)}/D} \widehat{\otimes}^L_D D/(q-1)$ with $\Prism_{\overline{R}^{(1)}/(D/(q-1))}$, where $\overline{R}^{(1)}$ is the smooth $D/(p, q-1)$-algebra defined by $R^{(1)}$ via base change. Using base change for prismatic cohomology again, this can also be written as $\widehat{\phi_D^*} \Prism_{\overline{R}/(D/(q-1))}$, where $\overline{R}$ is the $D/(p, q-1)$-algebra defined by $R$ via base change (i.e, $\overline{R} = R/p$, so $\overline{R}^{(1)}$ is the Frobenius twist of $\overline{R}$ relative to $D/(q-1,p)$, as the notation suggests). But Theorem~\ref{CrysComp} identifies this with $R\Gamma_\crys(\overline{R}/(D/(q-1)))$, which coincides with $R\Gamma_\crys(R/(D/(q-1)))$ and hence also $q\Omega_{R/D} \widehat{\otimes}_D^L D/(q-1)$ by Theorem~\ref{qCrysCrysComp}.
\end{proof}

\subsection{$q$-de Rham cohomology}
\label{ss:qdR}

In this section, we fix a $q$-PD pair $(D,I)$ with $D$ being $A$-flat as well as a $p$-completely smooth $D/I$-algebra $R$.

\begin{construction}[The $q$-de Rham complex of a framed $D$-algebra]
\label{FramedqdR}
A {\em framed} $D$-algebra is a pair $(P,S)$, where $P$ is a $p$-completely ind-smooth $D$-algebra and $S \subset P$ is a set of $p$-completely ind-\'etale co-ordinates, i.e., the map $\square:D[S] := D[\{X_s\}_{s \in S}] \to P$ is $p$-completely ind-\'etale. There are unique (and compatible) $\delta$-$D$-algebra structures on $D[S]$ and $P$ determined by the requirement $\delta(X_s) = 0$ for all $s \in S$. For each $s \in S$, there is a unique $\delta$-$D$-algebra automorphism $\gamma_s$ of $D[S]$ given by scaling $X_s \mapsto qX_s$ and $X_t \mapsto X_t$ for $t \neq s$. As this automorphism is congruent to the identity modulo the topologically nilpotent element $qX_s-X_s$, it extends uniquely to a $\delta$-$D$-algebra automorphism $\gamma_s$ of $P$ that is also congruent to the identity modulo $qX_s-X_s$. Since $D$ was $A$-flat, the same holds true for $P$. In particular, $qX_s-X_s$ is a nonzerodivisor on $P$, so we obtain a $q$-derivation $\nabla_{q,s}:P \to P$ given by the formula
\[\nabla_{q,s} := \frac{\gamma_s(f) - f}{qX_s-X_s}.\]
Note that $\nabla_{q,s}$ lifts the derivative $\frac{\partial}{\partial X_s}(-)$ on both $D[S]/(q-1)$ and $P/(q-1)$: this is clear for $D[S]$ by explicit computation and follows for $P$ by formal \'etaleness\footnote{More precisely, as $(q-1)$ is a nonzerodivisor on $P$, it suffices to check that the endomorphism $f \mapsto \gamma_s(f) - f$ of $P$ coincides with $X_s \cdot \frac{\partial}{\partial X_s}(-)$ modulo $(q-1)^2$, which follows from formal \'etaleness of $D[S] \to P$ and the interpretation of the endomorphism $f \mapsto \gamma_s(f) - f$ of $D[S]/(q-1)^2$ as the element of $H^0(T_{D[S]/D} \otimes_D (q-1)/(q-1)^2)$ classifying the infinitesimal automorphism $\gamma_s$ of the square-zero extension $D[S]/(q-1)^2 \to D[S]/(q-1)$.}. Varying $s$ we can assemble these $q$-derivations into a map
\[ \nabla_q:P \to \Omega^1_{P/S} := \widehat{\bigoplus}_{s \in S} P dX_s\]
given by the rule $\nabla_q(f) := \sum \nabla_{q,s}(f) dX_s$. Taking the Koszul complex on the commuting endomorphisms $\nabla_{q,s}$, we obtain the $q$-de Rham complex
\[ q\Omega_{P/D}^{\ast,\square} := \Big(P \xrightarrow{\nabla_q} \Omega^1_{P/D} \xrightarrow{\nabla_q} \Omega^2_{P/D} \to ...\Big),\]
regarded as a chain complex of $D$-modules. This construction has the following functoriality property: given two framed $D$-algebras $(P,S)$ and $(P',S')$ as well as a map $P \to P'$ of $D$-algebras carrying $S$ into $S'$, we obtain an induced morphism $q\Omega_{P/D}^{\ast,\square} \to q\Omega_{P'/D}^{\ast,\square}$ of chain complexes.
\end{construction}

\begin{construction}[The $q$-de Rham complex of a framed $q$-PD pair]
\label{FramedqPDqdR}
A {\em framed $q$-PD datum} is a triple $(P,S,J)$, where $(P,S)$ is a framed $D$-algebra as in Construction~\ref{FramedqdR} and $J \subset P$ is the kernel of a $D$-algebra surjection $P \to R$.   Let $D_{J,q}(P)$ be the $q$-PD envelope as in Construction~\ref{qPDEnvConstruct}. We have the following:
\begin{lemma}
For each $s \in S$, the automorphism $\gamma_s$ of $P$ extends uniquely to an automorphism of $D_{J,q}(P)$ that is congruent to the identity modulo $qX_s - X_s$.
\end{lemma}
\begin{proof}
We will check that $\gamma_s$ extends to $D_{J,q}$ as an endomorphism that is congruent to the identity modulo $qX_s-X_s$; this will force the extension to be an automorphism, proving the lemma. Consider the composition $P \xrightarrow{\gamma_s} P \to D_{J,q}(P)$. As all rings in sight are $A$-flat, it suffices to show that this composition carries $\phi(f)$ into $[p]_qD_{J,q}(P)$ for all $f \in J$. As $\gamma_s$ is congruent to the identity modulo $qX_s-X_s$, for any $f \in P$, we can write
\[ \gamma_s(f) = f + (q-1) X_s g\]
for suitable $g \in P$. Applying $\phi$ shows 
\[ \phi(\gamma_s(f)) = \phi(f) + [p]_q \cdot (q-1) \cdot X_s^p \cdot \phi(g).\] 
As $\phi(f) \in [p]_q D_{J,q}(P)$ for $f \in J$, the same must hold true for $\phi(\gamma_s(f))$ by the previous formula. This proves that $\gamma_s$ extends to $D_{J,q}$ as an endomorphism. In fact, this formula also shows that the resulting endomorphism of $D_{J,q}$ is congruent to the identity modulo $qX_s-X_s$, as wanted. 
\end{proof}
We can thus define $q$-derivatives $\nabla_{q,s}$ of $D_{J,q}$ by the same formula as in Construction~\ref{FramedqdR}; note that these derivatives continue to lift the operator $\frac{\partial}{\partial X_s}$ on $D_{J,q}/(q-1)$: this follows by identifying the latter as the pd-envelope of of $P/(q-1) \to R$ and reducing to the analogous statement for $P$ using torsionfreeness and density. Taking the Koszul complex of the commuting endomorphisms $\nabla_{q,s}$ as in Construction~\ref{FramedqdR}, we obtain a $q$-dR complex
\[ q\Omega_{D_{J,q}(P)/D}^{\ast,\square} := \Big(D_{J,q}(P) \xrightarrow{\nabla_q} D_{J,q}(P) \widehat{\otimes}_P \Omega^1_{P/D} \xrightarrow{\nabla_q} D_{J,q}(P) \widehat{\otimes}_P \Omega^2_{P/D} \to...\Big), \]
regarded as a chain complex of $D$-modules. The construction $(P,S,J) \mapsto q\Omega_{D_{J,q}(P)/D}^{\ast,\square}$ has the following functoriality property: given two framed $q$-PD data $(P,S,J)$ and $(P',S',J')$ and a map $P \to P'$ of $D$-algebras carrying $S$ into $S'$ and $J$ into $J'$, we obtain an induced morphism $q\Omega_{D_{J,q}(P)/D}^{\ast,\square} \to q\Omega_{D_{J',q}(P')/D}^{\ast,\square}$. To see this, it suffices to show that for any $t \in S' - \mathrm{im}(S)$, we have an equality $\gamma_t \circ \alpha = \alpha$ of morphisms, where $\alpha:D_{J,q}(P) \to D_{J',q}(P')$ is the map induced by the given map $P \to P'$ and $\gamma_t$ denotes the endomorphism of $D_{J',q}(P')$ attached to $t \in S'$ coming from Construction~\ref{FramedqPDqdR}. This identity holds true after evaluation on each of the \'etale co-ordinates $X_s \in P$ as $t \in S' - \mathrm{im}(S)$; the rest follows by $(p,I)$-complete \'etaleness of the framing $A[\{X_s\}]_{s \in S} \to P$, the $[p]_q$-torsionfreeness of $D_{J',q}(P')$, and the fact that $D_{J,q}(P)$ is generated as a $(p,I)$-complete $\delta$-$A$-algebra by elements of the form $f/[p]_q$ for suitable $f \in P$.
\end{construction}

\begin{theorem}[$q$-de Rham and $q$-crystalline comparison]
\label{qFramedPDqdR}
Let $(P,S,J)$ be a framed $q$-PD datum in the sense of Construction~\ref{FramedqPDqdR}. Then there is a canonical quasi-isomorphism $q\Omega_{R/D} \simeq q\Omega_{D_{J,q}(P)/D}^{\ast,\square}$.
\end{theorem}

The proof given below follows the cosimplicial proof of the crystalline-de Rham comparison in \cite[Tag 07LG]{Stacks}.

\begin{proof}
Consider the \v{C}ech nerve $P^\bullet$ of $D \to P$. The \'etale co-ordinates $S$ on $P = P^0$ naturally define \'etale co-ordinates $S^\bullet$ on $P^\bullet$ by taking coproducts, yielding a cosimplicial framed $D$-algebra $(P^\bullet,S^\bullet)$. Likewise, letting $J^\bullet \subset P^\bullet$ be the cosimplicial ideal defined as the kernel of the augmentation $P^n \to P \to R$, we obtain a cosimplicial framed $q$-PD datum $(P^\bullet, S^\bullet, J^\bullet)$. Applying Construction~\ref{FramedqPDqdR}, we obtain a cosimplicial complex $M^{\bullet,\ast} := q\Omega^{\ast,\square}_{D_{J^\bullet,q}(P^\bullet)}$. The cosimplicial $D$-module $M^{\bullet,i}$ given by the ``$i$-th row'' is acyclic (even homotopy equivalent to $0$ as a cosimplicial module) for $i > 0$ by a standard argument (see \cite[Tag 07JP]{Stacks} for the analogous assertion in crystalline cohomology). On the other hand, for any face map $P^i \to P^j$ in the cosimplicial ring $P^\bullet$, the induced map $M^{i,\ast} \to M^{j,\ast}$ of chain complexes is a quasi-isomorphism: this holds true modulo $(q-1)$ by the Poincar\'e lemma, and thus holds true by $(q-1)$-completeness. It formally follows that the chain complex attached to the cosimplicial $D$-module $M^{\bullet,0}$ is quasi-isomorphic to the $0$-th column $M^{0,\ast}$. The former computes $q\Omega_{R/D}$, while the latter is $ q\Omega_{D_{J,q}(P)/D}^{\ast,\square}$, so we obtain the desired quasi-isomorphism.
\end{proof}

\newpage

\section{Comparison with $A\Omega$}
\label{sec:AOmegaComp}

In this section, we fix a perfectoid field $C$ of characteristic $0$ containing $\mu_{p^\infty}$ as well as a $p$-completely smooth $\mathcal{O}_C$-algebra $R$. The goal of this section is to prove the comparison between the prismatic cohomology of $R$ and the $A\Omega$-complexes defined in \cite{BMS1} via the intermediary of $q$-de Rham cohomology. Let us introduce the relevant notation first.

\begin{notation}
Choosing a compatible system $\{\zeta_{p^n} \in \mu_{p^n}(C)\}$ of primitive $p$-power roots of $1$, we obtain the rank $1$ element $q = [\underline{\epsilon}] \in A := A_{\inf}(\mathcal{O}_C)$, where $\epsilon=(1,\zeta_p,\ldots)\in \mathcal O_C^\flat$. This allows us to view $A$ as a perfect $\delta$-$\mathbf{Z}_p\llbracket q-1\rrbracket$-algebra that is $(p,q-1)$-completely flat and thus also genuinely flat (see \cite[Lemma 5.15]{BhattCM}). Let $\xi = \phi^{-1}([p]_q)$, so $(A,\xi)$ is a $q$-PD pair corresponding to the $q$-PD thickening $A \to A/(\xi) \cong \mathcal{O}_C$. 
\end{notation}

With this notation, our goal is to show the following:

\begin{theorem}
\label{PrismaticAOmegaComp}
There is a canonical isomorphism $q\Omega_{R/A} \simeq A\Omega_R$, and thus
\[
A\Omega_R\simeq q\Omega_{R/A}\simeq \Prism_{R^{(1)}/A} = \phi_A^\ast \Prism_{R/A}\ .
\]
All these maps are isomorphisms of $E_\infty$-$A$-algebras compatible with the Frobenius. 
\end{theorem}

\begin{remark} The proof will a priori give an isomorphism of $E_1$-$A$-algebras functorially in $R$ and compatibly with the Frobenius. However, this can be upgraded to an isomorphism of $E_\infty$-$A$-algebras a posteriori: By left Kan extension, one can extend both sides to quasiregular semiperfectoid $R$-algebras, where it induces an isomorphism (as associative rings) of discrete commutative rings, which then tautologically upgrades to an isomorphism as $E_\infty$-rings. Now by quasisyntomic descent, one gets the desired isomorphism of $E_\infty$-$A$-algebras.
\end{remark}

\begin{proof}
We shall build a functorial map $\mu_R:q\Omega_{R/A} \to A\Omega_R$ by constructing one between functorial representatives of both sides; to check this is a quasi-isomorphism, we give an abstract argument (ultimately reducing to the Hodge-Tate comparison theorem for both sides). In fact, it suffices to build a strictly functorial map under the additional assumption that $R$ is {\em very small}, i.e., when $R/p$ is generated over $\mathcal{O}_C/p$ by units; the comparison map for general $R$ is then obtained by glueing as any smooth formal $\mathcal{O}_C$-scheme has a basis of opens of the form $\mathrm{Spf}(R)$ with $R$ very small. For the rest of the proof, we assume $R$ is very small.  

Before introducing the complexes, we need some notation. Let $S=R^*$ and let $\mathcal{C}_S$ be the collection of all finite subsets $\Sigma \subset S$ large enough to generate $R$, i.e., such that the induced map $A[\{x_s^{\pm 1}\}_{s \in \Sigma}]^{\wedge} \to R$ (or equivalently $\mathcal{O}_C[\{x_s^{\pm 1}\}_{s \in \Sigma}]^{\wedge} \to R$) is surjective; thus, the inclusion relation amongst subsets of $S$ turns $\mathcal{C}_S$ into a filtered poset. For each $\Sigma \in \mathcal{C}_S$, let $P_\Sigma = A[\{x_s^{\pm 1}\}_{s \in \Sigma}]^{\wedge}$ be the formal torus over $A$ with variables indexed by $\Sigma$; we endow $P_\Sigma$ with the $\delta$-structure defined by $\delta(x_s) = 0$ for all $s \in \Sigma$. Moreover, for each $s \in \Sigma$, we have the automorphism $\gamma_s$ of $P_\Sigma$ given by scaling $x_s$ by $q$ and fixing $x_t$ for $t \neq s$.

Elaborating on the strategy above, in the proof below, we shall build, for each $\Sigma \in \mathcal{C}_S$, explicit functorial complexes computing $q\Omega_{R/A}$ and $A\Omega_{R/A}$ and a natural (in $R$ and a choice of $\Sigma \in \mathcal{C}_S$) quasi-isomorphism between them. By naturality in $\Sigma$, we may pass to the colimit over all $\Sigma$ will yield a natural (in $R$) quasi-isomorphism $q\Omega_{R/A} \to A\Omega_{R/A}$ intrinsic to the very small algebra $R$. \\

{\em Computing $q\Omega_{R/A}$ by an explicit and functorial complex.} 
Fix $\Sigma \in \mathcal{C}_S$. Write $D_{J_\Sigma,q}(P_\Sigma)$ for the $q$-PD envelope of the surjection $P_\Sigma \to R$.  By Theorem~\ref{qFramedPDqdR}, we can then compute $q\Omega_{R/A}$ by the complex $\mathrm{Kos}_c(D_{J_\Sigma,q}(P_\Sigma); \{\nabla_{q,s}\}_{s \in \Sigma})$. As the formation of these complexes is functorial in $\Sigma \in \mathcal{C}_S$ (see discussion before Theorem~\ref{qFramedPDqdR}), we can take a colimit to see that $q\Omega_{R/A}$ is also computed by the complex 
\begin{equation}
\label{Expq}
 \colim_{\Sigma \in \mathcal{C}_S} \mathrm{Kos}_c(D_{J_\Sigma,q}(P_\Sigma); \{\nabla_{q,s}\}_{s \in \Sigma}),
 \end{equation}
which is an explicit complex functorial in the very small algebra $R$. \\

{\em Computing $A\Omega_R$ by an explicit and functorial complex.} 
Fix $\Sigma \in \mathcal{C}_S$, and set $\Delta_\Sigma := \prod_{s \in \Sigma} \mathbf{Z}_p(1)$ be the displayed profinite group. Formally extracting all $p$-power roots of the $x_s$'s gives a pro-\'etale $\Delta_\Sigma$-torsor over $R[1/p]$. Let  $R_{\Sigma,\infty}$ be the $p$-completed integral closure of $R$ in this extension of $R[1/p]$. It is then standard that $R_{\Sigma,\infty}$ is perfectoid\footnote{This follows from \cite[\S II.2]{ScholzeTorsion} in general. For our application, it suffices to consider sufficiently large $\Sigma \in \mathcal{C}_S$, so we may assume that $\Sigma$ contains a subset $\{x_1,...,x_n\} \subset \Sigma$ of units defining a formally \'etale map $\mathcal{O}_C[x_1^{\pm 1},...,x_n^{\pm 1}]^{\wedge} \to R$; in this case, the perfectoidness claim also follows from almost purity.}, and the map $\mathrm{Spa}(R_{\Sigma,\infty}[1/p]) \to \mathrm{Spa}(R[1/p])$ is a pro-\'etale $\Delta_\Sigma$-torsor. Consider the complex 
\[ \eta_{q-1} \mathrm{Kos}_c(A_{\inf}(R_{\Sigma,\infty}), \{\sigma_s-1\}_{s \in \Sigma}),\]
where $\sigma_s$ denotes the automorphism induced by the $s$-th basis vector of $\Delta_\Sigma$. Note that the complex $\mathrm{Kos}_c(A_{\inf}(R_{\Sigma,\infty}), \{\sigma_s-1\}_{s \in \Sigma})$ computes the continuous group cohomology of $\Delta_\Sigma$ acting on $A_{\inf}(R_{\Sigma,\infty})$ by \cite[Lemma 7.3]{BMS1}.  Moreover, this construction is evidently functorial in enlarging $\Sigma$, so we may pass to the limit to form the complex
\begin{equation}
\label{ExpA}
 \colim_{\Sigma \in \mathcal{S}_S} \eta_{q-1} \mathrm{Kos}_c(A_{\inf}(R_{\Sigma,\infty}), \{\sigma_s-1\}_{s \in \Sigma}),
 \end{equation}
which is an explicit complex functorial in the very small algebra $R$. Using the arguments in \cite[\S 9]{BMS1}, one checks that this explicit functorial complex computes $A\Omega_R$: in fact, for a cofinal collection of $\Sigma \in \mathcal{C}_S$ (more precisely, any $\Sigma$ containing a subset $\{x_1,...,x_n\} \subset \Sigma$ of units defining a formally \'etale map $\mathcal{O}_C[x_1^{\pm 1},...,x_n^{\pm 1}]^{\wedge} \to R$), the complex $\eta_{q-1} \mathrm{Kos}_c(A_{\inf}(R_{\Sigma,\infty}), \{\sigma_s-1\}_{s \in \Sigma})$ already computes $A\Omega_R$ by \cite[Theorem 9.4 (iii)]{BMS1} and the criterion in \cite[Lemma 5.14]{BhattBMS}. \\

{\em Constructing the comparison map $\mu_R:q\Omega_{R/A} \to A\Omega_R$.} To construct a functorial comparison map using the complexes described above in \eqref{Expq} and \eqref{ExpA}, it suffices to construct, for each $\Sigma \in \mathcal{C}_S$, a natural  (in $R$ and the chosen $\Sigma \in \mathcal{C}_S$) map
\[ \mu_0:D_{J_\Sigma,q}(P_\Sigma) \to A_{\inf}(R_{\Sigma,\infty})\]
of $A_{\inf}$-algebras that intertwines $\gamma_s$ on the left with $\sigma_s$ on the right and is compatible with enlarging $\Sigma$. Indeed, we will then get a map
\[
\mathrm{Kos}_c(D_{J_\Sigma,q}(P_\Sigma); \{\nabla_{q,s}\}_{s\in \Sigma}) \cong  \eta_{q-1} \mathrm{Kos}_c(D_{J_\Sigma,q}(P_\Sigma); \{\gamma_s-1\}_{s\in \Sigma})\to \eta_{q-1} \mathrm{Kos}_c(A_{\inf}(R_{\Sigma,\infty}); \{\sigma_s-1\}_{s \in \Sigma})
\]
also compatible with enlarging $\Sigma$, so we can then pass to the limit over $\Sigma \in \mathcal{C}_S$ to obtain a desired comparison map between the explicit complexes constructed earlier.

To construct the comparison map $\mu_0$ desired above, let $P_\Sigma \to P_{\Sigma,\infty}$ be the $(p,I)$-completed perfection. Write $\overline{P_\Sigma} := P_\Sigma/(\xi)$, etc. Note that $P_{\Sigma,\infty}$ is a perfect $(p,\xi)$-completely flat $\delta$-ring over $A_{\inf}$, so $\overline{P_{\Sigma,\infty}}$ is perfectoid and $P_{\Sigma,\infty} \simeq A_{\inf}(\overline{P_{\Sigma,\infty}})$. Moreover, the perfectoid ring $\overline{P_{\Sigma,\infty}}$ is obtained from $\overline{P_\Sigma}$ by formally extracting all $p$-power roots of the $x_s$'s, so there is a natural $\Delta_\Sigma$-action on $\overline{P_{\Sigma,\infty}}$. Applying $A_{\inf}(-)$, we obtain a $\Delta_\Sigma$-action on $P_{\Sigma,\infty}$ with the property that the automorphism $\sigma_s$ of $P_{\Sigma,\infty}$ induced by the $s$-th basis vector $e_s \in \Delta_\Sigma$ lies over the automorphism $\gamma_s$ of $P_\Sigma$. We can relate the objects over $P_\Sigma$ and over $R$ in the following commutative diagram:
\[ \xymatrix{P_\Sigma \ar[r] \ar[d] & P_{\Sigma,\infty} \simeq A_{\inf}(\overline{P_{\Sigma,\infty}}) \ar[r] \ar[d] & A_{\inf}(R_{\Sigma,\infty}) \ar[dd] \\
		\overline{P_\Sigma} \ar[r] \ar[d] & \overline{P_{\Sigma,\infty}} \ar[d] & \\
		R \ar[r] & R \widehat{\otimes}_{\overline{P_\Sigma}} \overline{P_{\Sigma,\infty}} \ar[r] & R_{\Sigma,\infty}, }\]
here all maps in the top horizontal row are $\delta$-maps, all vertical maps are surjective, and the two smaller squares on the left are pushout squares. For each $s \in \Sigma$, the automorphism $\sigma_s$  of $A_{\inf}(R_{\Sigma,\infty})$ is compatible with corresponding automorphism of $P_{\Sigma,\infty}$ and hence also with the automorphism $\gamma_s$ of $P_\Sigma$. 

 As $A_{\inf}(R_{\Sigma,\infty}) \to R_{\Sigma,\infty}$ is a $q$-PD thickening, the outer square in the above commutative diagram shows that the $\delta$-map $P_\Sigma \to A_{\inf}(R_{\Sigma,\infty})$ extends uniquely to a $\delta$-$A$-map $D_{J_\Sigma,q}(P_\Sigma) \to A_{\inf}(R_{\Sigma,\infty})$. It also follows by uniqueness that the induced $\gamma_s$-action on $D_{J_\Sigma,q}(P_\Sigma)$ is compatible with the $\sigma_s$-action on $A_{\inf}(R_{\Sigma,\infty})$: the maps 
\[ D_{J_\Sigma,q}(P_\Sigma) \xrightarrow{\gamma_s} D_{J_\Sigma,q}(P_\Sigma) \to A_{\inf}(R_{\Sigma,\infty}) \quad \text{and} \quad D_{J_\Sigma,q}(P_\Sigma) \to A_{\inf}(R_{\Sigma,\infty}) \xrightarrow{\sigma_s} A_{\inf}(R_{\Sigma,\infty})\] 
of $\delta$-$A$-algebras are induced by the universal property of $q$-PD envelopes from the maps 
\[ P_\Sigma \xrightarrow{\gamma_s} P_\Sigma \to A_{\inf}(R_{\Sigma,\infty}) \quad \text{and} \quad P_\Sigma \to A_{\inf}(R_{\Sigma,\infty}) \xrightarrow{\sigma_s} A_{\inf}(R_{\Sigma,\infty}),\]
so the former pair must coincide as the latter pair does. This gives the promised map  $\mu_0:D_{J_\Sigma,q}(P_\Sigma) \to A_{\inf}(R_{\Sigma,\infty})$ of $\delta$-rings that intertwines the $\gamma_s$ automorphism of the source with the $\sigma_s$ automorphism of the target, and thus also a comparison map
\[ \mu_R:q\Omega_{R/A} \to A\Omega_R\]
by the explicit complexes constructed above. \\

{\em Proving that $\mu_R$ is an isomorphism.} To check that $\mu_R$ is an isomorphism, we use the criterion in Lemma~\ref{ComparisonCrit} with $d = [p]_q$. By Theorem~\ref{QCrysPrism}, we can identify $q\Omega_{R/A} \simeq \Prism_{R^{(1)}/A}$, where $R^{(1)} := R \otimes_{A_{\inf},\phi} A$. By the Hodge-Tate comparison for $\Prism_{R^{(1)}/A}$, we then obtain a canonical map 
\[ \eta_{q\Omega_{R/A}}:R^{(1)} \to H^0(q\Omega_{R/A}/[p]_q)\] 
that induces an isomorphism 
\[ (\Omega^*_{R^{(1)}/(A/[p]_q)}, d_{dR}) \simeq (H^*(q\Omega_{R/A}/[p]_q),\beta_{[p]_q})\]
of commutative differential graded $A/[p]_q$-algebras. On the other hand, by \cite[Theorem 9.2]{BMS1} for $A\Omega_R$, we have a canonical identification\footnote{The Frobenius twist in $\widetilde{\Omega}_{R^{(1)}}$ is implicit but not mentioned in \cite[Theorem 9.2]{BMS1}: the $p$-completely smooth $\mathcal{O}$-algebra $R$ (corresponding to the formal scheme $\mathfrak{X}$) is regarded in {\em loc.\ cit.} as an $A_{\inf}$-algebra via the map $A_{\inf} \xrightarrow{\widetilde{\theta}} \mathcal{O} \to R$, which is identified in the present context with the map $A_{\inf} \xrightarrow{\theta} \mathcal{O} \to R \stackrel{\phi}{\simeq} R^{(1)}$.} \[ \widetilde{\Omega}_{R^{(1)}} \simeq A\Omega_R/[p]_q\]
of commutative $A/[p]_q$-algebras. This gives a comparison map 
\[ \eta_{A\Omega_R}:R^{(1)} \to H^0(A\Omega_R/[p]_q).\] 
By the Hodge-Tate comparison \cite[Theorem 8.3]{BMS1}, the map $\eta_{A\Omega_R}$ induces an isomorphism
\[ (\Omega^*_{R^{(1)}/(A/[p]_q)}, d_{dR}) \simeq (H^*(A\Omega_{R}/[p]_q),\beta_{[p]_q})\]
of commutative differential graded $A/[p]_q$-algebras. We leave it to the reader to check that the comparison map $\mu_R$ intertwines $\eta_{q\Omega_{R/A}}$ with $\eta_{A\Omega_R}$ by reduction to the case where $R = \mathcal{O}_C[x^{\pm 1}]^{\wedge}$. Lemma~\ref{ComparisonCrit} then implies that $\mu_R$ is an isomorphism.
\end{proof}

The following abstract criterion for proving that comparison maps are isomorphisms often allows us to bypass tedious checks (e.g., as in the proof above).

\begin{lemma}
\label{ComparisonCrit}
Let $A$ be a commutative ring equipped with a nonzerodivisor $d$. Let $R$ be a smooth $A/d$-algebra.  Let $E$ and $F$ be derived $d$-complete commutative algebras in $D(A)$ with $E/d$ and $F/d$ being coconnective and $H^*(E/d)$ and $H^*(F/d)$ being strictly graded commutative.  Assume we are given $A/d$-algebra maps $\eta_E:R \to H^0(E/d)$ and $\eta_F:R \to H^0(F/d)$ such that the map $(\Omega^*_{R/(A/d)}, d_{dR}) \to (H^*(E/d), \beta_d)$ induced by the universal property of the de Rham complex is an isomorphism, and similarly for $F$. Then any commutative $A$-algebra map $\alpha:E \to F$ that intertwines $\eta_E$ with $\eta_F$ must be an isomorphism. 

A similar statement holds true if $(A,d)$ is a bounded prism, $R$ is assumed to be $p$-completely smooth over $A/d$, and $E$ and $F$ are assumed to be derived $(p,d)$-complete.
\end{lemma}
\begin{proof}
To show $\alpha$ is an isomorphism, it is enough to show that $\bar{\alpha}:E/d \to F/d$ induces an isomorphism on $H^*(-)$. Now $H^*(\bar{\alpha})$ is a map of graded $A/d$-algebras that lifts to a map of commutative differential graded $A/d$-algebras $(H^*(E/d), \beta_d) \to (H^*(F/d),\beta_d)$. This latter map intertwines $\eta_E$ with $\eta_F$ by assumption, so we obtain a commutative diagram
\[ \xymatrix{ (\Omega^*_{R/(A/d)}, d_{dR}) \ar@{=}[r] \ar[d] & (\Omega^*_{R/(A/d)}, d_{dR}) \ar[d] \\
		  (H^*(E/d), \beta_d) \ar[r]^-{H^*(\bar{\alpha})} & (H^*(F/d), \beta_d) }\]
of maps of commutative differential graded $A/d$-algebras. The left and right vertical maps are induced by $\eta_E$ and $\eta_F$, and are isomorphisms by assumption. It follows that the bottom horizontal map is also an isomorphism, as wanted.
\end{proof}

\newpage

\section{A uniqueness criterion for comparison isomorphisms}
\label{sec:UniqueComp}

Let $(A,I)$ be a perfect prism corresponding a perfectoid ring $R$. In this section, we show that the prismatic cohomology functor $S\mapsto \Prism_{S/R}$ has no non-trivial automorphisms that are compatible with the Hodge-Tate structure map $\eta_S:S\to \overline{\Prism}_{S/R}$. In particular, the comparison isomorphisms relating prismatic cohomology to $q$-crystalline cohomology (Theorem~\ref{QCrysPrism}), the $A\Omega$-theory (Theorem~\ref{PrismaticAOmegaComp}), or the $p$-adic Nygaard complexes arising from topological Hochschild homology (Theorem~\ref{BMS2CompNC}) are essentially unique. To formulate our result  precisely, we introduce some $\infty$-categorical notation.

\begin{notation}
Let $\mathrm{Sm}_R$ denote the category of $p$-completely smooth $R$-algebras. Let $\mathcal{C}_{\mathrm{Sm}_R}$ be the $\infty$-category of pairs $(G,\eta)$, where $G:\mathrm{Sm}_R \to \mathcal{D}_{(p,I)\text{-comp}}(A)$ is a symmetric monoidal functor and $\eta:\mathrm{id} \to G \otimes_A^L R$ is a natural transformation of symmetric monoidal functors $\mathrm{Sm}_R\to \mathcal{D}_{p\text{-comp}}(R)$; here $G \otimes_A^L R$ is the $\mathcal{D}_{p\text{-comp}}(R)$-valued functor determined by $S \mapsto G(S) \otimes_A^L R$. Prismatic cohomology together with the Hodge-Tate comparison naturally gives an object $\Prism_{-/A} \in \mathcal{C}_{\mathrm{Sm}_R}$.
\end{notation}

\begin{theorem}
The object $\Prism_{-/A}\in \mathcal{C}_{\mathrm{Sm}_R}$  has no nontrivial endomorphisms, i.e., $\mathrm{End}(\Prism_{-/A}) = \{1\}$.
\end{theorem}

Somewhat surprisingly, we do not need to {\em a priori} impose compatibility with the Frobenius endomorphism of prismatic cohomology in formulating the above uniqueness assertion.

\begin{proof}
Let $\mathrm{rsPerfd}_R$ denote the category of $R$-algebras $S$ that are quotients of $p$-completely flat perfectoid $R$-algebras by $p$-completely regular sequences relative to $R$. In particular, $\mathrm{Kos}(S;p)$ is flat over $\mathrm{Kos}(R;p)$ for any such $S$. Let $\mathcal{C}_{\mathrm{rsPerfd}_R}$ be the $\infty$-category defined the same way as $\mathcal{C}_{\mathrm{Sm}_R}$ with $\mathrm{rsPerfd}_R$ replacing $\mathrm{Sm}_R$.  Derived prismatic cohomology and its Hodge-Tate comparison then similarly defines an object $\Prism'_{-/A}\in \mathcal{C}_{\mathrm{rsPerfd}_R}$. As any $p$-completely smooth $\mathcal{O}$-algebra $S$ admits a quasisyntomic hypercover (or even a \v{C}ech cover) $S \to S^\bullet$ with each $S^i \in \mathrm{rsPerfd}_R$, it suffices by descent to show that $\Prism'_{-/A}$ has no nontrivial endomorphisms. As the prismatic cohomology of any $S\in \mathrm{rsPerfd}_R$ takes on discrete values, this reduces to the following concrete statement:

\begin{lemma}
Assume that for each $S\in \mathrm{rsPerfd}_R$, we are given an $A$-algebra endomorphism $\epsilon_S$ of $\Prism_{S/A}$ with the following properties:
\begin{enumerate}
\item The endomorphism $\epsilon_S$ is functorial in $S$.
\item The endomorphism $\epsilon_S$ is compatible with the identity map on $S$ under the natural map $\eta_S:S \to \overline{\Prism}_{S/A}$.
\end{enumerate}
Then $\epsilon_S = \mathrm{id}$ for all $S$.
\end{lemma}
\begin{proof}
Let us first check that $\epsilon_S = \mathrm{id}$ when $S$ is perfectoid. In this case, we have $\Prism_{S/A} \cong A_{\inf}(S)$, so $\epsilon_S$ can be regarded as an $A$-algebra endomorphism of $A_{\inf}(S)$. By $A$-linearity, this endomorphism carries $IA_{\inf}(R)$ to itself. As any endomorphism of a perfect $p$-complete $\delta$-ring is automatically compatible with the $\delta$-structure (Corollary~\ref{PerfectLambda}), we may regard $\epsilon_S$ as an endomorphism of the perfect prism $(A_{\inf}(S), IA_{\inf}(S))$. By assumption (2), this endomorphism gives the trivial endomorphism of $S\cong A_{\inf}(S)/IA_{\inf}(S)$. By Theorem~\ref{PerfdPrism}, it follows that $\epsilon_S = \mathrm{id}$ for perfectoid $S$.

Now fix some $S \in \mathrm{rsPerfd}_R$. By assumption, we can write $S = R^\prime/(f_1,...,f_r)$, where $R^\prime$ is a $p$-completely flat perfectoid $R$-algebra and $f_1,...,f_r$ is a sequence in $S$ that is $p$-completely regular relative to $R$. Then $\Prism_{R^\prime/A} \cong A_{\inf}(R^\prime)$, so $\Prism_{S/A}$ is naturally a $\delta$-$A_{\inf}(R^\prime)$-algebra. Moreover, if we fix a distinguished element $d \in I$ as well as $g_i \in A_{\inf}(R^\prime)$ lifting $f_i \in S$ for $i=1,...,r$, then the presentation $S = R^\prime/(f_1,...,f_r)$ gives a presentation of $\Prism_{S/A}$ as the $\delta$-$A_{\inf}(R^\prime)$-algebra $A_{\inf}(S)\{\frac{g_1}{d},...,\frac{g_r}{d}\}^{\wedge}$ obtained by freely adjoining $\frac{g_i}{d}$ for $i=1,...,r$ to $A_{\inf}(S)$ in the world of $(p,d)$-complete simplicial $\delta$-rings. Applying assumption (1) to the map $R^\prime\to S$, we obtain the commutative diagram
\[ \xymatrix{ A_{\inf}(R^\prime) \ar[r]^-{can} \ar[d]^-{\epsilon_{R^\prime}} & A_{\inf}(R^\prime)\{\frac{g_1}{d},...,\frac{g_r}{d}\}^{\wedge} \ar[d]^-{\epsilon_S} \\
 		     A_{\inf}(R^\prime) \ar[r]^-{can} & A_{\inf}(R^\prime)\{\frac{g_1}{d},...,\frac{g_r}{d}\}^{\wedge}. }\]
We must show that the $A$-module endomorphism $\epsilon_S - \mathrm{id}$ of $A_{\inf}(R^\prime)\{\frac{g_1}{d},...,\frac{g_r}{d}\}^{\wedge}$ is identically $0$.  Equivalently, if $K$ denotes the kernel of this map, we must show $K = A_{\inf}(R^\prime)\{\frac{g_1}{d},...,\frac{g_r}{d}\}^{\wedge}$.  Note that $K$ is an $A$-submodule of $A_{\inf}(R^\prime)\{\frac{g_1}{d},...,\frac{g_r}{d}\}^{\wedge}$ that has the following stability properties:

\begin{enumerate}
\item $K$ is stable under taking powers, i.e., if $f \in K$ then $f^n \in K$ for all $n \geq 0$: the kernel of a difference of ring homomorphisms always has this property. 
\item $K$ is naturally linear over $A_{\inf}(R^\prime)$, not merely $A$: we have $\epsilon_{R^\prime} = \mathrm{id}$ by the previously settled perfectoid case, so both $\epsilon_S$ and $\mathrm{id}$ are $A_{\inf}(R^\prime)$-algebra maps. In particular, the image of $A_{\inf}(R^\prime)$ in $A_{\inf}(R^\prime)\{\frac{g_1}{d},...,\frac{g_r}{d}\}^{\wedge}$ lies in $K$. 
\item $K$ is saturated in $A_{\inf}(R^\prime)\{\frac{g_1}{d},...,\frac{g_r}{d}\}^{\wedge}$ with respect to multiplication by both $p$ as well as $\phi^n(d)$ for all $n \geq 0$. Indeed, $A_{\inf}(R^\prime)\{\frac{g_1}{d},...,\frac{g_r}{d}\}^{\wedge}$ is $(p,d)$-completely flat over $A$ (\cref{qPDEnvRegular}). In particular, by Lemma~\ref{BoundedPrismProp} (2), the element $p$ as well as all the Frobenius powers $\phi^n(d)$ for $n \geq 0$ are nonzerodivisors in this $A$-module, and hence the kernel of any endomorphism of this $A_{\inf}(R^\prime)$-module has the stated saturation property.
\end{enumerate}

We now show that $K=A_{\inf}(R^\prime)\{\frac{g_1}{d},...,\frac{g_r}{d}\}^{\wedge}$. As $K$ is derived $(p,[p]_q)$-complete and an $A_{\inf}(R')$-module by property (2), it suffices to check that a topological generating set for the $A_{\inf}(R^\prime)$-module $A_{\inf}(R^\prime)\{\frac{g_1}{d},...,\frac{g_r}{d}\}^{\wedge}$ lies in $K$. Using the operations $\{\delta_n\}_{n \geq 0}$ of Joyal (Remark~\ref{JoyalOps} and Remark~\ref{JoyalOpsGenerate}), it suffices to check that $\delta_n(\frac{g_i}{d}) \in K$ for all $n \geq 0$ and all $i$. We check this by induction on $n$ (and fixed $i$). For $n=0$, we must check that $\frac{g_i}{d} \in K$. But $g_i \in K$ by property (2), and hence $\frac{g_i}{d} \in K$ by the $d$-saturatedness of $K$ from property (3).  Assume now that $\delta_j(\frac{g_i}{d}) \in K$ for $j \leq n$. To show $\delta_{n+1}(\frac{g_i}{d}) \in K$, we use the formula
\[ \phi^{n+1}(\frac{g_i}{d}) = (\frac{g_i}{d})^{p^n} + p \delta_1(\frac{g_i}{d})^{p^{n-1}} + ... + p^n \delta_n(\frac{g_i}{d})^p + p^{n+1} \delta_{n+1}(\frac{g_i}{d}).\]
The term on the left lies in $K$: this follows after multiplication by $\phi^{n+1}(d)$ using property (2), and hence holds true on the nose as $K$ is $\phi^{n+1}(d)$-saturated by property (3).  Moreover, all terms on the right except the last one also lie in $K$: this follows by induction and property (1). Thus, the last term $p^{n+1} \delta_{n+1}(\frac{g_i}{d})$ also lies in $K$. As $K$ is $p$-saturated by property (3), it follows that $\delta_{n+1}(\frac{g_i}{d}) \in K$ as well, which finishes the induction.
\end{proof}
\end{proof}

\newpage

\bibliographystyle{amsalpha}
\bibliography{prisms}

\end{document}